\documentclass[aos]{imsart}

\usepackage{amsthm,amsmath,amsfonts,amssymb}
\RequirePackage[numbers]{natbib}
\RequirePackage[colorlinks,citecolor=blue,urlcolor=blue]{hyperref}
\RequirePackage{graphicx}

\usepackage[mathscr]{euscript}
\usepackage{amsthm}
\usepackage{amsmath}
\usepackage{amsfonts}
\usepackage{bbm}
\usepackage{xcolor}
\usepackage[numbers]{natbib}
\usepackage{comment}
\usepackage{float}
\usepackage{bbm}
\usepackage{subfig}
\usepackage{mathtools}

\usepackage{xr}
\DeclareMathOperator*{\argmax}{arg\,max}
\DeclareMathOperator*{\argmin}{arg\,min}

\newcommand{\bR}{\mathbb{R}}
\newcommand{\bN}{\mathbb{N}}

\newcommand{\cA}{\mathcal{A}}

\newcommand{\cC}{\mathcal{C}}
\newcommand{\cD}{\mathcal{D}}
\newcommand{\cG}{\mathcal{G}}
\newcommand{\cH}{\mathcal{H}}

\newcommand{\cK}{\mathcal{K}}

\newcommand{\cR}{\mathcal{R}}
\newcommand{\cM}{\mathcal{M}}
\newcommand{\cE}{\mathcal{E}}

\newcommand{\cP}{\mathcal{P}}

\newcommand{\cF}{\mathcal{F}}
\newcommand{\cQ}{\mathcal{Q}}

\newcommand{\cI}{\mathcal{I}}
\newcommand{\cV}{\mathcal{V}}

\newcommand{\sfE}{{\sf E}}

\newcommand{\sfQ}{{\sf Q}}
\newcommand{\sfP}{{\sf P}}

\newcommand{\rt}{\right}
\newcommand{\lt}{\left}

\startlocaldefs
\newtheorem{theorem}{Theorem}[section]
\newtheorem{lemma}[theorem]{Lemma}

\newtheorem{corollary}{Corollary}[section]
\newtheorem{proposition}{Proposition}[section]

\theoremstyle{remark}
\newtheorem{remark}{Remark}[section]

\newtheorem{example}{Example}[section]  %%we made this change here from the original annals template for the examples

\endlocaldefs

\begin{document}

\begin{frontmatter}
\title{Joint Sequential Detection and Isolation \\ for Dependent  Data Streams}
%An asymptotic optimality theory}
%\title{A sample article title with some additional note\thanksref{t1}}
\runtitle{Joint Sequential Detection and Isolation}
%\thankstext{T1}{A sample additional note to the title.}

\begin{aug}
%%%%%%%%%%%%%%%%%%%%%%%%%%%%%%%%%%%%%%%%%%%%%%
%%Only one address is permitted per author. %%
%%Only division, organization and e-mail is %%
%%included in the address.                  %%
%%Additional information can be included in %%
%%the Acknowledgments section if necessary. %%
%%%%%%%%%%%%%%%%%%%%%%%%%%%%%%%%%%%%%%%%%%%%%%
\author{\fnms{Anamitra} \snm{Chaudhuri}\ead[label=e1,mark]{ac34@illinois.edu}}
\and
\author{\fnms{Georgios} \snm{Fellouris}\ead[label=e2,mark]{fellouri@illinois.edu}}
%\and
%\author[B]{\fnms{Third} \snm{Author}\ead[label=e3,mark]{third@somewhere.com}}
%%%%%%%%%%%%%%%%%%%%%%%%%%%%%%%%%%%%%%%%%%%%%%
%% Addresses                                %%
%%%%%%%%%%%%%%%%%%%%%%%%%%%%%%%%%%%%%%%%%%%%%%
\address{Department of Statistics,
University of Illinois, Urbana-Champaign,
\printead{e1,e2}}
\end{aug}

\begin{abstract} %should not be longer than 200 words
The problem of joint sequential detection and isolation  is considered  in the context of multiple, not necessarily independent, data streams. A multiple testing framework is proposed, where each  hypothesis corresponds to a different subset of data streams,  the sample size is a  stopping time of the observations, and the probabilities of  four kinds of error are controlled below distinct, user-specified levels. Two of these  errors  reflect the  detection component  of the formulation, whereas the other two the  isolation component.  The optimal expected sample size is characterized to a first-order asymptotic approximation as the error probabilities go to 0. Different asymptotic regimes, expressing different prioritizations of the detection and isolation tasks, are considered.  A novel, versatile family of testing procedures  is proposed, in which  two distinct, in general, statistics are computed  for each hypothesis, one addressing the detection task and the other  the isolation task. Tests in this family, of various  computational complexities,
are shown to be  asymptotically  optimal under different setups. The  general theory is  applied to the detection and isolation of anomalous, not necessarily independent,  data streams, as well as to the detection and isolation of an unknown dependence structure.  
\end{abstract}

%where it is shown that the most practical rule with quadratic complexity is able to achieve asymptotic optimality under a wide range of probability distributions.
%: (i) the number of free parameters (thresholds) is equal to the number of error constraints and (ii) it ranges from the most practical rule to the asymptotically optimal rule based on their design that also incorporates the aspect of its computational complexity. 

\begin{keyword}[class=MSC2020]
\kwd[Primary ]{62L10}
\kwd{62L05}
\kwd[; secondary ]{62J15}
\end{keyword}

\begin{keyword}
\kwd{Sequential multiple testing}
\kwd{Detection and isolation}
\kwd{Anomaly detection}
\kwd{Asymptotic optimality}
\kwd{Dependence structure}
%\kwd{Multiple testing for correlation}
%\kwd{Familywise error rates}
\end{keyword}

\end{frontmatter}
%%%%%%%%%%%%%%%%%%%%%%%%%%%%%%%%%%%%%%%%%%%%%%
%% Please use \tableofcontents for articles %%
%% with 50 pages and more                   %%
%%%%%%%%%%%%%%%%%%%%%%%%%%%%%%%%%%%%%%%%%%%%%%
%\tableofcontents
%ranging from public health, engineering, cyber-security, signal detection etc., data are generated by mutliple sources (sensor networks, etc)involve multiple testing problem regarding higher dimensional dependence structure, for example, testing pairwise independence between data streams. 
 
\section{Introduction}
The field of sequential analysis  deals with statistical procedures whose sample size  is not fixed in advance, but it is a function of the collected observations. For recent textbooks regarding developments  on this area since the  groundbreaking work of Abraham Wald \cite{wald1945sequential,Wald_book}, we refer to  \cite{tartakovsky2014sequential,Bartroff_Book_Clinicaltrials}.  While earlier works, like Wald’s, were motivated by problems in inspection sampling and quality control, in modern industrial and engineering setups it is quite common that there are multiple data streams that are generated by a multitude of data sources (e.g., sensors, channels, etc.), are monitored sequentially, and need to be utilized for real-time decisions.  This is the case, for example, in multiple access wireless network \cite{rappaport1996wireless}, multisensor surveillance systems \cite{foresti2003multisensor}, multichannel signal detection \cite{tartakovsky2003sequential},  environmental monitoring \cite{env-mon}, blind source separation \cite{blind-source}, biometric authentication \cite{biom-auth}, sensor networks \cite{sens-net}, power systems \cite{power-grid}, \cite{Hey_Taj2018smart_grid}, neural coding \cite{neur-cod}, etc. In many of  these applications,   the goal is to   identify  data streams that are  ``anomalous", in the sense that their marginal distributions deviate substantially from a certain, ``normal'' behavior. In others,  the goal is  to  identify data streams that share a certain dependence structure.  In both cases, the problem  admits   a multiple testing formulation, in which control of the classical  familywise error rates   guarantees, with  high probability, the precise identification of  the anomalous data sources  \cite{holm1979simple, de2012step, de2012sequential, bartroff2014sequential,
Cohen2015,song2017asymptotically,tsopela} or of the unknown  dependence structure \cite{Hey_Taj2016det_markov, Hey_Taj2016lin_search, loc-str, chau_fell20}  (see also  \cite{li2012, cai2011, cai2013} %cai3
for fixed-sample-size formulations). 

% tends to be the one determining the sample size.   
%where in the first one the hypotheses refer to the marginal distributions  of the data streams,
%\cite{de2012sequential} \cite{de2012step,De201588,2014arXiv_Bartroff,bartroff2014sequential,malloy2014sequential},  \cite{song2017asymptotically},  and in the second one to their pairwise distributions.
%which was introduced by Wald in his revolutionary work of \textit{sequential hypothesis testing} \cite{wald1945sequential} and has been an extensive area of research since then (see, \cite{tartakovsky2014sequential}),

%A different formulation, motivated by a range of applications \cite{env-mon,blind-source,biom-auth,sens-net,power-grid, Hey_Taj2018smart_grid,neur-cod} is to test for  pairwise independence between data streams and identify the dependent pairs  with a certain degree accuracy. This problem has been considered in the  fixed sample size setting \cite{cai3}, \cite{li2012}, \cite{cai2011}, \cite{cai2013}, it is not well-explored how to design efficient sequential procedures that are computationally feasible as well as minimizes the sampling cost.  and in the second one the hypotheses refer to pairwise dependence.  % both in sequential and non-sequential (fixed sample size) setup

Familywise error  metrics are, in general,  determined by  the most difficult hypothesis problem under consideration. As a result, even the  average  sample size  for  their  control can be excessive in practice. One approach for addressing this  issue is to  employ alternative  error metrics,  such as 
%by considering some other popular error metrics, such as %(i) familywise error rates \cite{bartroff2014sequential}, \cite{de2012sequential}, \cite{de2012step},
%\cite{holm1979simple}, \cite{song2017asymptotically}, 
generalized familywise error rates \cite{bartroff2010multistage, 2014arXiv_Bartroff, De201588, song_fell2019}, %\cite{hommel1988controlled, LehmannRomano2005, romano2006stepup, romano2007control},  
the generalized misclassification rate \cite{Li2014, malloy2014sequential, song_fell2019}, or  the  false discovery rate \cite{He_Bart_2021}. %  among others. 
%\cite{benjamini1995controlling, benjamini2001, guo2014further, pena2011power}. 
%in \cite{bartroff2014sequential}, \cite{de2012sequential}, \cite{de2012step},
%\cite{holm1979simple}, \cite{song2017asymptotically},\cite{2014arXiv_Bartroff}, \cite{bartroff2010multistage}, \cite{De201588}, \cite{song_fell2019}, \cite{Li2014}, \cite{malloy2014sequential}, \cite{song_fell2019}, \cite{He_Bart_2021} etc. 
A different approach  is  to focus  on detecting  the presence of  anomalous   \cite{det_vvv1993, det_vvv1994, det_Tart2001, Inv_seq_det2001, det_Tart2002, tartakovsky2003sequential, fellouris2017multichannel}   or dependent \cite{Castro2012, Hey_Taj2016det_markov, Hey_Taj2016lin_search}  data streams, without identifying them with guaranteed precision. This \textit{detection} task leads to a binary  testing problem and suffices when  an action needs to be taken  as long as an anomalous data stream is present, no matter which one  it is, or  there is some kind of dependence, no matter where exactly it is located. 

The first goal of the present work is to introduce a  novel  multiple testing formulation that unifies  the pure detection and the pure  isolation problems. Specifically, the proposed formulation requires control of  four distinct error metrics, two of which  refer to the detection task, whereas the other  two refer  to the isolation task. Thus, it allows 
for  strict  control of \textit{false alarms} and \textit{missed detections}, without having  to  either escalate the prescribed control of \textit{false positives} and \textit{false negatives}, as in the traditional familywise error formulation, or   ignore it, as in the pure detection problem.

The second goal of  this work is  %main  methodological contribution  is the 
the proposal and asymptotic analysis of a general, versatile family of testing procedures that   satisfy the prescribed error control in the above formulation.
 %Specifically,  two distinct, in general, statistics are computed  for each hypothesis, where one of them the detection task and the other to the isolation task.The most challenging task however, and the main contribution from a theoretical point of view,   is  an asymptotic analysis with which 
Specifically, we show that  tests in the proposed family, of various  degrees of computational complexity, achieve the optimal expected sample size as the target error probabilities go to 0 under various asymptotic regimes  that  correspond to different prioritizations  of the detection and isolation tasks.

The  above asymptotic analysis is conducted under  a general framework that allows not only for  temporal dependence, but also  for dependence among the  data streams. Thus, the third goal of the paper is to  apply the above results to  two  problems  that arise as  special cases of the general setup considered  in this work:  (i) the detection and isolation of anomalous  data streams, and (ii) the detection and isolation of  an unknown dependence structure. 

For the first of these problems, in the special case that the data streams are assumed to be  \text{independent}, we obtain asymptotically optimal rules  that are easy to implement, for any number of data sources, under a general setup that allows for temporal dependence. Thus, we improve upon existing rules  for the pure detection and isolation problems  \cite{de2012step, de2012sequential, bartroff2014sequential, song2017asymptotically}. In the general case for the first problem, where the data streams can be dependent,  as well as for the  second problem,  the  asymptotically optimal rules we obtain in this work are not always easy to implement. For this reason, we put special emphasis on low-complexity  testing procedures, evaluate their asymptotic relative efficiencies, and we further show that they do achieve  asymptotic optimality under various setups. Finally, we apply  the above results to the special case that the data streams follow a multivariate Gaussian distribution. In this way, we also generalize certain classical results on the problem of sequentially testing for the  correlation  coefficient of a bivariate Gaussian distribution  \cite{choi, kowalski, sat-pra, wolde, wood}. 
 
  %problems ofthe detection and isolation of anomalous streams  testing  the means and   and dependentstructure of 

 %that focus on the isolation problem and we t and in this way, the theory developed in this paper can explain the origin of these rules, which were designed in some ad-hoc manners in the past literature. On the other hand, in the presence of dependence,  \textcolor{red}{When the problem  is to detect and isolate a dependence structure...} 

%These results are illustrated to the problem of testing pairwise correlation in dependent streams whose underlying distribution is multivariate Gaussian.   

 %Finally, we provide sufficient conditions for this optimal asymptotic expected sample size to be  achieved by the proposed testing rules. 

%As before, the asymptotically optimal rules we derive are not always computationally convenient, and we focus again on the most practical rule to apply. We show that under a wide range of probability distributions, it achieves asymptotic optimality.   

%  of dependent streams whose underlying distribution is%To the best of our knowledge, asymptotic optimality of these existing rules is understood only under the strict assumption that the streams are independent and in fact, they fail to be asymptotically optimal, in general, when

%Finally, the  general theory we develop is applied  to the problem of joint detection and isolation of a dependence structure. 

The remainder of the paper is organized as follows:  we formulate the  general  framework of this paper  in Section \ref{sec:prob_form}, we introduce the proposed family of testing procedures  in Section \ref{sec:prop_proced}, and  we establish our  general asymptotic optimality theory in Section \ref{sec:asymp_opt}. Next, we apply the previous results to the detection and isolation of anomalous data  streams in Section \ref{sec:appl_anomalous}, and   of  an unknown dependence structure in Section \ref{sec:appl_depen_struc}. We present the results of two simulation studies  in Section \ref{sec:sim_stud}. Finally, we conclude and discuss potential extensions of this work in Section \ref{sec:conclusion}. All proofs are presented in the Appendix, %(Supplementary material \cite{chau_fell_aos_supp})
where we also include  auxiliary results of independent interest.

Finally, we collect some notations that  are  used  throughout the paper.  %Furthermore, for any event $\Gamma$, we use the following notation:$$\sfE_\sfP\lt[Y; \Gamma\rt] = \int_{\Gamma}Yd\sfP.$$
%We use the convention that maximum and minimum of the empty set are $-\infty$ and $+\infty$ respectively. : $C_k^J$ denotes the binomial coefficient ${J \choose k}$, that is, the number of subsets of size $k$ from a set of size $J$
We denote the set of natural numbers by $\bN$, i.e., $\bN := \{1, 2, \dots\}$, and for any $n \in \bN$ we  set $[n] := \{1, \dots, n\}$. We denote by $\bR$ the set of real numbers and for any $a,b \in \bR$ we set 
 $a \vee b \equiv \max\{a, b\}$ and $a \wedge  b \equiv \min\{a, b\}$.   $|A|$ represents the  cardinality  of a finite set $A$. %$\mathscr{P}(A)$ denotes its power set and , 
For a square matrix $M$, we denote by $det(M)$ its determinant and by $tr(M)$  its trace.  For any $p \in \bN$, we denote by $N_p(\mu, \Sigma)$ the $p$-variate Gaussian distribution with mean vector $\mu$ and covariance matrix $\Sigma$, and omit the subscript when $p = 1$.
% i.e., in case of univariate Gaussian distribution.
For a probability measure $\sfP$,  we denote by $\sfE_\sfP$, $\sf{Var}_\sfP$ and ${\sf Corr}_{\sfP}$ the corresponding expectation, variance and correlation, respectively. Finally, for any sequences of positive numbers $(x_n)$ and $(y_n)$, we write   $x_n \sim y_n$ for  $\lim_n  (x_n/y_n) = 1$ and  $x_n \lesssim  y_n$ for  $\limsup_n (x_n/y_n) \leq 1.$
 
%\newpage

\section{Problem formulation}\label{sec:prob_form}
\subsection{Data}
We consider $K$ data  sources that generate data simultaneously in discrete time and denote by $X^k_n$ the   random element whose value is generated by data source $k \in [K]$  at  time $n \in \bN$.  For any  subset of data sources, $s \subseteq [K]$, we set 
\begin{align*}
X^s_n &:=\{X^{k}_n: k \in s\} \qquad 
\text{and}  \qquad   \cF_n^{s} := \sigma\lt(X^{s}_m:  m \in [n]\rt), \quad n \in \bN,
\end{align*}
i.e.,  $X_n^s$ is the set of  observations by  the data  sources in $s$ at  time $n$ and $\cF_n^{s}$  the $\sigma$-field generated by the observations from the  data sources in $s$ up to time $n$, and 
\begin{align*}
X^s &:=\{X^{s}_n, n \in \bN\} \qquad 
\text{and}  \qquad \cF^{s} := \sigma\lt( \cF^{s}_n, n \in \bN\rt),
\end{align*}
i.e., $X^s$ is the complete sequence of observations  by the data sources in  $s$ and $\cF^s$ the  $\sigma$-field it generates. 
In what follows, we omit the superscipt in the above notations  whenever  $s=[K]$, and replace it by $k$ whenever   $s=\{k\}$ for some $k \in [K]$.
%In the rest of the paper we similarly omit any superscript of subscript whenever it is $[K]$ and replace it by $k$ whenever it is $\{k\}$ for some $k \in [K]$. %, e.g., when $s = [K]$ and $s = \{k\}$, we also denote $\cI(\sfP, \sfQ; s)$ by $\cI(\sfP, \sfQ)$ and $\cI(\sfP^k, \sfQ^k)$ respectively. In the same spirit, we obtain similar notations $\Lambda_k$, $s_k$, $s'_k$, $A_k, B_k, \cH^k, \cG^k$ etc.

\subsection{Global distributions}
 We consider  a finite family, $\cP$,  of  plausible distributions of    $X$,  to the elements of which we refer as \textit{global} distributions.   We say that two subsets of data sources $s_1\subseteq [K]$ and $s_2\subseteq [K]$  are independent under some $\sfP \in \cP$   if $X^{s_1}$ and $X^{s_2}$ are independent under $\sfP$. For any global distribution, $\sfP \in \cP$, and any  subset  of sources,  $s\subseteq [K]$, we denote   by $\sfP^s$ the restriction of $\sfP$ on  $\cF^s$, i.e., the marginal distribution of $X^s$ under $\sfP$.  For any family of global distributions, $\cQ \subseteq \cP$, and any subset of data sources,  $s \subseteq [K]$,  we set 
$$\cQ^s := \{\sfP^s : \sfP \in \cQ\},$$ 
e.g., $\cP^s$ is the family of distributions of $X^s$ induced by $\cP$. 
As before, we omit the superscript whenever $s = [K]$, and replace it by $k$ whenever   $s=\{k\}$ for some $k \in [K]$.

\subsection{A family of local hypotheses}
We consider  a   finite family of subsets of $[K]$,  $\cK$, to the    elements of which we  refer as   \textit{units}. For each  $e \in \cK$, we consider two hypotheses for the marginal distribution of $X^e$,
\begin{equation*} % \label{testing_problem}
\cH^e \;  \text{(null)} \qquad \text{  versus }  \qquad  \cG^e \;  \text{(alternative)},
\end{equation*}
 where $\cH^e$,  $\cG^e$ are non-empty, disjoint subfamilies of $\cP^e$.  For any global distribution, $\sfP \in \cP$, we 
we say that $e\in \cK$  is  a \textit{signal} under $\sfP$  if $\sfP^e \in \cG^e$, and  
we   denote by  $\cA(\sfP)$  the set of signals  under $\sfP$, i.e., 
$$\cA(\sfP) := \{e \in \cK : \;  \sfP^e \in \cG^e\}.$$

%\begin{example}
\begin{remark}
When the units are the pairs of data sources,   i.e., 
\begin{equation} \label{pairs}
\cK = \big\{\{i, j\} : 1 \leq i < j \leq K\big\},
\end{equation}
each   testing problem  concerns  the  marginal distribution  of a pair of data sources.  On the other hand,  when the units are the individual data sources,  i.e., 
\begin{equation} \label{singletons}
\cK = \{\{k\} : k \in [K]\},
\end{equation}
each  testing problem  concerns the  marginal distribution of an individual data source. In the latter case,   we  replace the superscript in $\cH^e$ and $ \cG^e$  by $k$ when $e = \{k\}$ for some $k \in [K]$.
\end{remark}
%\end{example}\begin{example}

%\end{example}
%When $\cK$ is given by \eqref{singletons}, we replace the superscript in the above notations by $k$ whenever $e = \{k\}$ for some $k \in [K]$.

\subsection{A Gaussian example}\label{gauss_prob_setup}
Before continuing with the proposed problem formulation, we illustrate the already introduced notation with a  concrete example. For this, we let    
$$\cM_k \subset \bR, \qquad \cV_k \subset \bR_+, \qquad \cR_{kj} \subset (-1, 1)$$ be arbitrary finite sets,  with $0 \in \cM_k$  and   $0 \in \cR_{kj}$ for every $k, j \in [K] \; \text{and}  \; k < j$, and  consider the  family of distributions, $\cP$,  for which  $\sfP \in \cP$ if and only if   $X$  is a  sequence of iid  Gaussian random vectors under $\sfP$  with 
$$ \sfE_{\sfP}[X_n^k] \in \cM_k, \qquad {\sf Var}_{\sfP}[X_n^k] \in \cV_k , \qquad {\sf Corr}_{\sfP}(X_n^k,  X_n^j) \in \cR_{kj} $$
for every $n \in \bN$ and $k, j \in [K]$  with $ k < j$. 
  
\subsubsection{Detection and isolation of non-zero Gaussian means}\label{test_mean}
When  testing whether the mean of each source is  equal to  zero or not, then $\cK$ is given by \eqref{singletons} and for every $e=\{k\}$, where $k \in [K]$, the hypotheses take the following form: 
 %for $k$ take the form:Now, for every $k \in [K]$, suppose $0 \in \cM_k$ and we further fix some non-zero $\mu_k \in \cM_k$ to consider the following two hypotheses for $X^k$,:
% which is a sequence of independent and identically distributed (iid) Gaussian random variables 
\begin{align*}
\cH^e  \equiv \cH^k &= \left\{X^k \overset{\text{iid}}{\sim} N(0, \sigma^2) : \sigma^2 \in \cV_k\right\} \\
\cG^e \equiv \cG^k &= \left\{ X^k \overset{\text{iid}}{\sim} N(\theta, \sigma^2) : \theta \in \cM_k \setminus \{0\}, \sigma^2 \in \cV_k \right\}, 
%\quad k \in [K],
\end{align*}
independently of the  dependence structure of the data sources. 

\subsubsection{Detection and isolation of a Gaussian dependence structure}\label{test_corr}
When testing whether each pair of sources is independent or not, then $\cK$ is given by \eqref{pairs} and  for every  $e = \{i, j\} \in \cK$  the hypotheses take the following form:
%suppose $0 \in \cR_{ij}$ and we %further fix some $\cR'_e \subset \cR_e$, such that $0 \notin \cR'_e$, to 
%consider the following two hypotheses for $X^e$, which is a sequence of iid bivariate Gaussian random vectors
\begin{align*}
\cH^e &= \left\{X^e \overset{\text{iid}}{\sim} N_2\left((\theta_1, \theta_2), \Sigma\right) : \theta_1 \in \cM_i, \theta_2 \in \cM_j, \Sigma_{11} \in \cV_i, \Sigma_{22} \in \cV_j, \Sigma_{12} = 0\right\},\\
\cG^e &= \left\{X^e \overset{\text{iid}}{\sim} N_2\left((\theta_1, \theta_2), \Sigma\right) : \theta_1 \in \cM_i, \theta_2 \in \cM_j, \Sigma_{11} \in \cV_i, \Sigma_{22} \in \cV_j, \Sigma_{12}  \in \cR_{ij} \setminus\{ 0\} \right\}.
\end{align*}

\subsection{Prior information about the signal configuration}
We assume that the set of signals is known \textit{a priori} to belong to a  family of subsets of $\cK$,  $\Psi$, and denote by $\cP_\Psi$ the  family of global distributions that are compatible with $\Psi$, i.e.,
\begin{align*}
\cP_\Psi :=  \{\sfP \in \cP:  \quad \sfP^e \in \cH^e \cup \cG^e \; \;  \text{for every}\ e \in \cK \quad \text{and} \quad  \cA(\sfP) \in \Psi\}.
\end{align*}
We refer to  $\cH_0 \equiv \cP_\emptyset$ as the family of \textit{global null distributions}.   Clearly, $\cH_0 \subseteq  \cP_\Psi \Leftrightarrow
  \emptyset \in \Psi.$

The family  $\Psi$ is to be specified by the practitioner on the basis of the available information regarding the signal set.  For example,  if it is assumed that  there are at least $l$ and at most $u$ signals,  then $\Psi$ takes the following form:
  $$\Psi_{l, u}:= \{A \subseteq \cK : l \leq |A| \leq u\}.$$
When, in particular, $l=0$ and $u=K$,  
$\Psi_{l, u}$ coincides with the powerset of $\cK$ and  corresponds to the case of no prior information about the signals whatsoever.  

 Alternatively, the available prior information may  be of  topological nature. For example, if the signals are assumed  to form a \textit{cluster},  in the sense of being  stable   with respect to symmetric differences, then  $\Psi$ takes the following form:
$$\Psi_{clus} := \{A \subseteq \cK : \;\; \text{if} \;\; e, e' \in A\;\; \text{and}\;\; e \triangle  e' \in \cK,\;\; \text{then}\;\; e \triangle  e' \in A\}.$$
Alternatively, if  the signals  are assumed to be disjoint, then  $\Psi$ becomes
$$\Psi_{dis} := \{A \subseteq \cK : \;\; \text{if} \;\; e, e' \in A\;\; \text{then}\;\; e \cap e' = \emptyset\}.$$

\begin{remark} When $\cK$ is given by \eqref{pairs} or  \eqref{singletons}, then $\Psi_{dis} \subseteq \Psi_{clus}$. However,   this inclusion does not   hold  in general.  Indeed,  if $K = 2$,  \,  $\cK = \{\{1\}, \{2\}, \{1,2\}\}$ and 
$ A:= \{\{1\}, \{2\}\}$, then $A\in   \Psi_{dis} $, but $A \notin \Psi_{clus}$, \, since   $\{1\} \triangle \{2\} = \{1, 2\} \in \cK$,  but \,   $\{1, 2\} \notin A$.
\end{remark}

\subsection{Tests}
We assume that  observations are collected sequentially from all sources. At each time  instance, we decide whether  to stop sampling  or not, and in the former case we select one of the two hypotheses for each unit.  Thus, 
a \text{testing procedure} is a pair $(T,D)$, where 
  $T$ is  the, possibly random, time at which  sampling is terminated in \text{all} data sources and  $D$ is  the set of units  identified  upon stopping as  the set of signals. 
  
  More formally, we say that $(T,D)$ is a   \textit{test} if 
\begin{itemize}
\item $T$ is an $\{\cF_n\}$-stopping time, i.e.,  $\{T=n\} \in \cF_n$ for every $n \in \bN$,  
\item  $D$ is $\cF_T$-measurable and  $\Psi$-valued,  i.e.,  $\{T=n, D=A\} \in \cF_n$ for every $n \in \bN$,  $A \in \Psi$,
\end{itemize}  
 and we denote by $\cC_\Psi$ the family of all tests. We next restrict this family, focusing on tests that  control the  probabilities of certain errors of interest.

\subsection{Errors}
We say that a test   $(T,D) \in \cC_\Psi$ commits  under some $\sfP \in \cP_\Psi$  
 \begin{itemize} 
\item  a  \textit{false alarm}  when there is no signal but at least one unit is identified as such, i.e., when $\cA(\sfP) = \emptyset$ and the event $D \neq \emptyset$ occurs,
\item  a  \textit{missed detection}  when  there is at least one signal   but no unit is identified as such,  i.e., when $\cA(\sfP) \neq \emptyset$   and  the event $D = \emptyset$ occurs,

\item a  \textit{false positive}  when  there is at least one true signal and  at least one mistakenly identified signal, i.e., when  $\cA(\sfP) \neq \emptyset$ and the event $D \setminus \cA(\sfP) \neq \emptyset$ occurs,

\item  a  \textit{false negative} when at least one unit is identified as a  signal and there is also at least one missed signal, i.e.,  when  the events  $D\neq \emptyset$ and  $\cA(\sfP) \setminus D \neq \emptyset$ both occur. 

\end{itemize}

\subsection{Testing formulations} 
 False alarms and missed detections are  of special importance when it is  critical to take a certain action  if there is at least one signal,  \textit{no matter which}. The relevance and importance of these two  errors, relative to false positives/negatives, can give rise to different  testing formulations.

\subsubsection{Pure isolation}
When $\emptyset \notin\Psi$,  false alarms and missed detection are not possible. In this case, 
we  consider  tests that control the  probabilities of  false positives and  false negatives,  as in  classical familywise error control. That is,  the  class of tests of interest is
\begin{align*}
\cC_{\Psi}(\gamma, \delta) := \{  (T, D) \in \cC_\Psi  :  \,  & \; \sfP \lt(D \setminus \cA(\sfP) \neq \emptyset\rt) \leq \gamma \\ 
& \& \; \; \sfP\lt(\cA(\sfP) \setminus D \neq \emptyset, \, D \neq \emptyset  \rt) \leq \delta  \quad    \forall  \;\sfP \in \cP_\Psi \},
\end{align*}
where $\gamma$ and $\delta$ are user-specified tolerance levels in $(0,1)$.

\begin{remark}
If it is assumed that there are exactly $m$    signals, i.e., $\Psi=\Psi_{m,m}$,  where $m \in (0, |\cK|)$, then 
$$\cC_{\Psi}(\gamma, \delta)= \{(T, D) \in \cC_\Psi : \; \;  \sfP\lt(D \neq \cA(\sfP)\rt) \leq \gamma \wedge \delta \quad   \forall \;  \sfP \in \cP_\Psi \},
$$ 
since, in this case, whenever a false positive occurs so does a false negative,  and vice versa. 
\end{remark}

\subsubsection{Pure detection}

When $\emptyset \in\Psi$, or equivalently $\cH_0 \subseteq \cP_\Psi$,  false alarms and missed detections are possible. If, also,  these are  the only  errors we want to explicitly control,   the class of tests of interest  becomes  
\begin{align*}
\cD_{\Psi}(\alpha, \beta) := \{(T, D) \in \cC_\Psi  :  \; \; \sfP\lt(D \neq \emptyset\rt) & \leq \alpha \quad    \forall \; \;  \sfP \in \cH_0 \\ 
 \&  \quad
 \sfP\lt(D = \emptyset\rt) &\leq \beta  \quad   \forall \; \; \sfP \in \cP_\Psi \setminus \cH_0\},
\end{align*}
where $\alpha$ and $\beta$ are user-specified tolerance levels in $(0,1)$.

\subsubsection{Familywise error control}
\label{subsec:familywise}
Even if  false alarms and missed detections are possible and of special interest, it may still be desirable to also identify the signals,  if any, with some degree of accuracy.  This is achieved, for example, by the classical familywise error control, according to which the class of tests of interest is 
\begin{align*}
\cE_{\Psi} (\gamma, \delta) &:= \{  (T, D) \in \cC_\Psi  :  \;  \sfP \lt(D \setminus \cA(\sfP) \neq \emptyset\rt) \leq \gamma  
 \quad \& \quad \sfP\lt(\cA(\sfP) \setminus D \neq \emptyset \rt) \leq \delta  \; \; \;   \forall  \;\sfP \in \cP_\Psi \},
\end{align*}
where, as before, $ \gamma$ and $\delta$ are user-specified tolerance levels in $(0,1)$.  When $\emptyset \notin \Psi$, $\cE_{\Psi} (\gamma, \delta) $ coincides with  $\cC_{\Psi} (\gamma, \delta)$.  When $\emptyset \in \Psi$, it controls for a false alarm/missed detection by  essentially treating  it  as any other false positive/negative. Indeed, the probability of a false alarm is bounded above by $\gamma$, that  of a missed detection by $\delta$ and, as a result,
\begin{equation}\label{fwer_and_det}
\cE_{\Psi} (\gamma, \delta) \subseteq  \cD_{\Psi}(\alpha, \beta) \quad \Leftrightarrow \quad  \alpha\geq \gamma, \; \beta\geq \delta.
\end{equation}
 Thus,  in order  to lower the false alarm/missed detection rate in this context, the practitioner needs to  lower  the false positive/negative rate,  even if the latter is not of equally major  concern, and vice versa. 
 % Similarly,  to lower the  false positive/negative rate,  the practitioner needs to  lower the  false alarm/missed detection rate, even if the latter errors are not of equal importance.
%This results however in a substantial increase of the required resulting sample size,   

\subsubsection{Joint detection and isolation}
Motivated by the above limitation of familywise error control, we propose a more general and, to the best of our knowledge, novel  multiple testing   formulation, which  generalizes  the pure detection and  pure isolation problems.  Specifically, the proposed class of tests is
\begin{align*}
\cC_{\Psi}(\alpha, \beta, \gamma, \delta) := \cD_{\Psi}(\alpha, \beta) \cap \cC_{\Psi \setminus \emptyset} (\gamma, \delta), 
\end{align*}
where $\alpha, \beta, \gamma, \delta$ are user-specified tolerance levels in $(0,1)$. When $\emptyset \notin \Psi$, this class  reduces to  $\cC_{\Psi}(\gamma, \delta)$, similarly to
$\cE_{\Psi}(\gamma, \delta)$. Unlike the latter, however,  it  also reduces to  $\cD_{\Psi}(\alpha, \beta)$ when  $\emptyset \in \Psi$ and $\gamma = \delta = 1$. More interestingly, with a non-trivial choice of $\gamma$ and $\delta$ it leads to  the same false alarm and missed detection rates as in the pure detection problem, while providing an explicit and non-trivial control  of  the false positive and false negative rates. In general, the proposed formulation  allows for an asymmetric treatment of  the four kinds of error that were introduced earlier, through the the  selection  of the four distinct,  user-specified tolerance levels, $\alpha, \beta, \gamma, \delta$.

\subsection{Goals}
A main goal of this work is  to propose tests,  for each of the above formulations, that achieve the  optimal expected sample size
under every possible distribution  to a first-order asymptotic approximation as the corresponding error probabilities go to $0$. This is achieved under quite weak distributional assumptions that allow  both for temporal dependence, as well as  for dependence among the sources.  Since asymptotically optimal tests are not always computationally simple when the sources are dependent, another  main goal  is to propose practical tests,   for each of the above formulations, and to evaluate their performance loss, if any, relative to asymptotically optimal schemes.  

We next introduce the proposed   family of tests (Section \ref{sec:prop_proced}) and state  our asymptotic  results (Section \ref{sec:asymp_opt}). We  subsequently apply this general methodology and theory to two special cases: the  detection and isolation of, possibly dependent,  anomalous sources (Section \ref{sec:appl_anomalous}),  and the detection and isolation of a dependence structure among the sources (Section \ref{sec:appl_depen_struc}).

\section{A family of tests} \label{sec:prop_proced}

In this section we introduce families of tests
 that can be designed to belong to the families  we  introduced earlier. In order to do so, for each unit $e \in \cK$ we denote    by  $\cH_{\Psi,e}$  the family of non-null global distributions  under which the null is  correct  in unit $e$,  and by $\cG_{\Psi,e}$ the family of global distributions   under which $e$ is a signal,  i.e., 
\begin{align*}
\cH_{\Psi,e}&:= \{\sfP \in \cP_\Psi \setminus \cH_0 : \; e \notin \cA(\sfP) \}  \qquad \text{and} \qquad 
\cG_{\Psi,e}:= \{\sfP \in \cP_\Psi: \; e \in \cA(\sfP)\}. 
 \end{align*}
We assume that for every  
$$e \in \cK,   \quad  \sfP \in \cG_{\Psi,e} \cup \cH_0, \quad \sfQ \in \cH_{\Psi,e} \cup \cH_0,  \quad e \subseteq s \subseteq  [K],   \quad n \in \bN,$$
   $\sfP^s$ and $\sfQ^s$, i.e, the marginal distributions of $X^s$ under $\sfP$ and $\sfQ$ respectively, are  mutually absolutely continuous when restricted to $\cF_n^{s}$, and  we denote the corresponding likelihood ratio (Radon-Nikodym derivative) and log-likelihood ratio  as follows:
$$\Lambda_n(\sfP^s, \sfQ^s ) := \frac{d\sfP^s}{d\sfQ^s}\lt(\cF_n^{s}\rt) \qquad \text{and}  \qquad 
Z_n(\sfP^s, \sfQ^s ) :=\log \Lambda_n(\sfP^s, \sfQ^s).
$$

\subsection{Detection and isolation statistics}
For each unit $e \in \cK$ we select  two subsets of $[K]$,  $ s_e , s'_e \subseteq [K]$, that both  include $e$,   i.e.,  $e \subseteq s_e \cap s'_e,$ 
to which we refer as the \textit{subsystems of unit $e$}. Then, at each time $n \in \bN$,  we compute the values of  the following GLR-type  statistics,
\begin{align*}
\Lambda_{e, \rm det}(n) & := 
\frac{ \max  \{ \Lambda_n(\sfP, \sfP_0) : \sfP \in \cG_{\Psi,e}^{s_e} \} }
 {\max \{ \Lambda_n(\sfP, \sfP_0): \sfP \in \cH_0^{s_e}  \} }, \quad \quad 
 \Lambda_{e, \rm iso}(n)  := 
\frac{ \max  \{ \Lambda_n(\sfP, \sfP'_0) : \sfP \in \cG_{\Psi,e}^{s'_e}\} }
 {\max \{ \Lambda_n(\sfP, \sfP'_0): \sfP \in \cH_{\Psi,e}^{s'_e} \} },
\end{align*}
where   $\sfP_0$ and $\sfP'_0$ are arbitrary   distributions from   $\cH_0^{s_e}$ and $\cH_0^{s'_e}$, respectively. We note that
\begin{itemize}
\item  $\Lambda_{e, \rm det}(n)$  is based on data from  $s_e$ up to  time $n$ and the  larger its value, the stronger the evidence that  there exists \textit{at least one signal}. Thus,  we refer   to  $s_e$   as the \textit{detection  subsystem}  for  unit $e$ and to $\Lambda_{e, \rm det}$ as the   \textit{detection} statistic  for  unit $e$.   
 
\item    $\Lambda_{e, \rm iso}(n)$  is based on data from  $s'_e$ up to  time $n$, and the larger its value, the stronger the evidence that  unit $e$ \textit{is a signal}. Thus, we  refer to  $s'_e$  as the \textit{isolation  subsystem}  for  unit $e$ and to $\Lambda_{e, \rm iso}$ as the   \textit{isolation} statistic  for  unit $e$.   

   \end{itemize}

For each $n \in \bN$ we also denote by $D_{\rm iso}(n)$ the  set of units whose  isolation statistics at time $n$  have values larger than $1$, i.e.,
\begin{equation} \label{D_iso}
D_{\rm iso}(n) := \lt\{e \in \cK : \Lambda_{e, \rm iso}(n) > 1\rt\}.
\end{equation}
Therefore,  $D_{\rm iso}(n)$  can be interpreted as the estimated signal set based on the from the isolation subsystems in the first $n$ time instances. 
 %\qquad \text{and} \qquad p(n)=|D_{\rm iso}(n)|.$$

When the detection and isolation subsystems of unit $e \in \cK$ are selected as small as possible, i.e.,  $s_e = s'_e =e$, then its  isolation and detection statistics coincide and  reduce  to 
\begin{align} \label{LLR}
  \Lambda_e(n) := \frac{\max \{\Lambda_n(\sfP, \sfP_0):\sfP \in \cG^e\}}
{\max \{ \Lambda_n(\sfP, \sfP_0): \sfP \in \cH^e \} }.
\end{align}

We refer to  $\Lambda_e(n)$ as the \textit{local} statistic of unit $e$ at time $n$, as it is computed based solely on data from unit $e$ up to time $n$. 
Furthermore, we denote the ordered versions of the local statistics  by  $\Lambda_{(1)}(n) \geq \dots \geq \Lambda_{(|\cK|)}(n)$, and   the corresponding indices by  $i_1(n), \dots, i_{|\cK|}(n)$. We also denote by $p(n)$ the number of units whose local statistics at time $n$ have values larger than $1$, i.e., 
%$\Lambda_{(p(n))}(n) > 1 \geq \Lambda_{(p(n) + 1)}(n)$.
$p(n):=|\{e \in \cK:  \Lambda_e(n)>1\}|$. 

%\textcolor{red}{define $p(n)$ here}.

\begin{remark}
When  $e = \{k\}$ for some $k \in [K]$,
% e.g., when $\cK$ is given by \eqref{singletons}, 
we replace $e$ in $\Lambda_{e, \rm det}(n)$,   $\Lambda_{e, \rm iso}(n)$,  $\Lambda_{e}(n)$ by $k$, and write $\Lambda_{k, \rm det}(n)$,   $\Lambda_{k, \rm iso}(n)$,  $\Lambda_{k}(n)$, instead.
\end{remark}

%\begin{remark}The  selection of the subsystems  will clearly have an impact on both the computational requirements of the resulting tests, as well as on  their statistical properties.  For the time being, unless otherwise specified, we consider these subsystems as  arbitrary. \end{remark}

%\begin{equation*} %\label{stream_indi} 
%\Lambda_{(k)}(n) = \Lambda_{i_k(n)}(n), \quad  k \in [K].
%\end{equation*}
%\textcolor{red}{Finally, we denote by $p(n)$ the number of likelihood ratio statistics which are greater than 1 at time $n$.}
% i.e.,$$\Lambda_{(1)}(n) \geq \dots \geq \Lambda_{(p(n))}(n) > 1 \geq \Lambda_{(p(n) + 1)}(n) \geq \dots \geq \Lambda_{(K)}(n).$$

%The above statistics are determined based on the data generated in the local subsystems $S_e$ and $S'_e$ respectively.  When  $S_e= S'_e=[K]$ the computation of these statistics requires data from the whole system. 

%\subsection{The family of tests}\label{sec:prop_rule} Before we introduce the proposed procedure in the general case of joint detection and isolation, We start by considering  the special cases of the pure isolation problem and the pure detection problem. 

\subsection{The proposed tests}
We next  utilize  the above statistics to devise  tests that belong to each of the  families we introduced in the previous section.  

\subsubsection{The test for pure isolation}  
\label{subsubsec: pure_isolation_test}
When $\emptyset \notin \Psi$, we  stop sampling when the value of every \textit{isolation}  statistic is either too large or too small, and we identify as signals those units for which  these values are large.  Therefore, the proposed test in this case, $\chi^{*} := (T^{*}, D^{*})$,   is:
\begin{align*}
T^{*} &:= \inf \left\{ n \in \bN :   D_{\rm iso}(n) \in \Psi,  \; \Lambda_{e, \rm iso}(n) \notin \lt(1/A_e, B_e\rt) \; \forall \;   e \in \cK   \right\} \quad \text{and} \quad
D^{*} := D_{\rm iso}(T^{*}), 
\end{align*}
where $\{A_e, B_e\}_{e \in \cK}$  are  thresholds to be determined so that  $\chi^{*} \in \cC_\Psi(\gamma, \delta)$.

\subsubsection{The test for pure detection} \label{subsubsec: pure_detection_test}
When  $\emptyset \in \Psi$ and the focus is on the pure detection problem, we stop sampling  when either  all detection statistics are small, in which case the null hypothesis is selected in all units,   or  \textit{at least one} detection statistic is large enough, in which case   some non-empty subset of units is identified as the set of signals.   
Specifically,  the proposed test  in this case,
$\chi^*_{det} := (T^*_{det},  D^*_{det})$,  is:
\begin{align*}
T^*_{det} &:= \min\{T_0, T_{det}\} \qquad  \text{and} \qquad 
D_{det}^* = \emptyset \quad \Leftrightarrow \quad  T_0 < T_{det},\\
\text{where} \quad T_0 &:= \inf\lt\{n  \in \bN : \Lambda_{e, \rm det}(n) \leq 1/C_e \quad  \text{for every}\ e \in \cK\rt\},\\
T_{det} &:= \inf\left\{n \in \bN: \Lambda_{e, \rm det}(n) \geq D_e\quad \text{for some}\ e \in \cK \right\},
\end{align*}
and  $\{C_e, D_e\}_{e \in \cK}$ are thresholds to be determined so that   $ \chi^*_{det}  \in \cD_\Psi(\alpha, \beta)$.

%When $T_1 < T_{2, D}$, that is, we declare upon stopping that null hypothesis is true in all the units, that is, $D^* = \emptyset$. When $T_{2, D} < T_1$, we declare that there exist at least one unit in which the alternative hypothesis is true, that is, $D^* \neq \emptyset$. We refer to $\chi^*_{D}$ as the \textit{Pure Detection Rule}.

\subsubsection{The test for joint detection and isolation}\label{subsubsec: joint_test}

%\subsubsection{Joint detection and isolation}
When $\emptyset \in \Psi$ and the goal is to control all four kinds of error, the proposed test  in this case, $\chi^* := \lt(T^*, D^*\rt)$,  is based on a combination of the two previous ones:
%we combine the two previous %Given the above statistics, the proposed testing procedure continuous sampling until  either every  local detection statistic is small enough, in which the   or  \textit{at least one} local detection statistic is large enough and \textit{at the same time} \textit{all}  local isolation statistics  are either small enough or large enough, below a threshold $1/A_e$ or above a different threshold,  while at the same time with $D_{\rm iso}(n)$ belonging to $\Psi$ in adherence to the prior information, but not being $\emptyset$. Thus, 
%the  proposed stopping rule takes the following  form 
\begin{align*}
T^* &:= \min\{T_0, T_{joint}\}, \qquad 
D^* := \begin{cases}
\emptyset\ &\text{if} \quad T_0< T_{joint} \\
D_{\rm iso}\lt(T^*\rt)\ &\text{if} \quad  T_{joint} < T_0\end{cases}, 
\end{align*}
where $T_{0}$ is defined as before,  
\begin{align*}
T_{joint} := \inf\{ n \in \bN : & \; \Lambda_{e, \rm det}(n) \geq D_e \quad \text{for some}\ e \in \cK,\\
&  \Lambda_{e, \rm iso}(n) \notin \lt(1/A_e, B_e\rt) \quad \text{for every}  \, e \in \cK, \quad D_{\rm iso}(n) \in \Psi \setminus \emptyset\}, 
\end{align*}
and $\{A_e, B_e, C_e, D_e\}_{e \in \cK}$ are thresholds to be determined so that $\chi^* \in \cC_\Psi(\alpha, \beta, \gamma, \delta)$.

\subsubsection{The test for familywise error control}\label{jdi_to_fwer}
When $\emptyset \in \Psi$ and we are  interested in  the familywise error control of Subsection \ref{subsec:familywise}, 
we utilize  
 $$\chi^* := \lt(T^*, D^*\rt) \quad \text{with} \quad  A_e=C_e, B_e=D_e \quad  \text{for every} \; e \in \cK,$$
and  denote the resulting test   by  
$\chi_{fwer}^* := (T_{fwer}^*, D_{fwer}^*).$

\begin{remark}
When
%\begin{remark} \label{rem:int_rule}In case of $\chi^*_{fwer}$ if we further consider 
$s_e = s'_e = e$ for every $e \in \cK$,   then 
%$(T_{fwer}^*, D_{fwer}^*).$  reduces to
\begin{align}\label{intersection_rule}
\begin{split}
T^*_{fwer} &= \inf\{ n \in \bN : \, D_{\rm iso}(n) \in \Psi, \;  \Lambda_{e}(n) \notin \lt(1/A_e, B_e\rt)  \; \forall \; e \in \cK \}, \\
 D^*_{fwer} &= D_{\rm iso}(T^*_{fwer})
 \end{split}
\end{align}
%Furthermore, for every $e \in \cK$, it we let $A_e = A$ and $B_e = B$, that is, the local statistics in all units are required to cross the same thresholds for the sampling to be terminated, then 
and if also  $\Psi$ is the power set of $\cK$, 
$\chi^*_{fwer}$ coincides with the  \textit{``Intersection Rule''}   in \cite{de2012sequential, de2012step}.
\end{remark}
\begin{comment}

%and the thresholds in $\chi^*$ are chosen such that $A_e = C_e$ and $B_e = D_e$, 
then
 $T_0$ and $T_{joint}$ take the following form
\begin{align*}
T_0 &= \inf\lt\{n  \in \bN : \Lambda_{e}(n) \leq 1/A_e \quad  \text{for every}\ e \in \cK\rt\},\\
T_{joint} &= \inf\{ n \in \bN : \; \Lambda_{e}(n) \geq B_e\quad \text{for some}\ e \in \cK,\\
&\qquad \qquad \qquad \; \Lambda_{e}(n) \notin \lt(1/A_e, B_e\rt) \; \; \text{for every}  \, e \in \cK \quad  \text{and} \quad D_{\rm iso}(n) \in \Psi \setminus \emptyset\}.
\end{align*}
Since the stopping events of $T_{joint}$ satisfy the following that,
$$\{\Lambda_{e}(n) \notin \lt(1/A_e, B_e\rt) \; \; \text{for every}  \, e \in \cK \;\;  \text{and} \;\; D_{\rm iso}(n) \in \Psi \setminus \emptyset\} \subseteq \{\Lambda_{e}(n) \geq B_e\;\; \text{for some}\ e \in \cK\},$$
\end{comment}
%\end{remark}

\subsection{Error control}\label{sec:err_con}

\begin{comment}
If condition \eqref{infty_det_iso} holds and $\emptyset \notin \Psi$, for any choice of thresholds, the rule $\chi^*_{iso}$ terminates almost surely under every distribution $\sfP \in \cP(\Psi)$. Also due to the same condition, $\chi^*_{det}$ and $\chi^*$ terminate almost surely under every distribution $\sfP \in \cP(\Psi)$ when $\emptyset \in \Psi$.
\end{comment}

 We next  select the thresholds in each of the above tests in order to satisfy the desired error constraints.  We emphasize  that no additional  distributional assumptions are needed for this purpose. To state these results, we set
 \begin{equation}\label{a,b,c,d}
 a_e := |\cH_{\Psi,e}^{s'_e}|, \quad b_e :=  |\cG_{\Psi,e}^{s'_e}|, \quad c_e := |\cH_0^{s_e}|, \quad d_e := |\cG_{\Psi,e}^{s_e}| \quad e \in \cK. 
 \end{equation}

% We start with the case of pure isolation, i.e., when $\emptyset \notin \Psi$, and  continue with the case that $\emptyset \in \Psi$.  

\begin{theorem}\label{th:errcon}
 Suppose  that  $\emptyset \notin \Psi$ and let
 $\gamma, \delta \in (0, 1)$.  
 \begin{itemize}
\item[(i)]  $\chi^* \in \cC_{\Psi}(\gamma, \delta)$ when 
\begin{align*}  
A_e \geq a_e \, |\cK|/\delta \qquad \text{and} \qquad 
B_e \geq b_e \, |\cK|/\gamma  \quad  \forall \,  e \in \cK.
\end{align*}
\item[(ii)] If, also,     $\Psi = \Psi_{m,m}$ \,   for some $m \in (0, |\cK|)$, then \,  $\chi^* \in \cC_{\Psi}(\gamma, \delta)$ when 
\begin{align*}  
A_e \geq a_e \, |\cK|/(\gamma \wedge \delta) \quad \text{and} \qquad 
B_e \geq b_e \, |\cK|/(\gamma  \wedge \delta)  \quad  \forall \,  e \in \cK.
\end{align*}
\end{itemize}
\end{theorem}

\begin{proof}
The proof can be found in Appendix \ref{pf:th:errcon}.
\end{proof}

\begin{theorem}\label{th:errcon2}
 Suppose that $\emptyset \in \Psi$ and let $\alpha, \beta, \gamma, \delta \in (0, 1)$.
 \begin{enumerate}
\item[(i)]   $ \chi^*_{det} \in \cD_{\Psi}(\alpha, \beta)$ when 
\begin{align*} %\label{thresholds_D}
C_e \geq  c_e/\beta \qquad \text{and} \qquad
D_e \geq  d_e \, |\cK|/\alpha \quad \forall \,  e \in \cK.
\end{align*}
\item[(ii)]   $\chi_{fwer}^*\in \cE_{\Psi}(\gamma, \delta)$ when
\begin{align*} 
A_e &\geq  (a_e \vee c_e) \, (|\cK| + 1)/ \delta  \qquad \text{and } \qquad 
B_e \geq  (b_e \vee d_e) \, |\cK|/\gamma 
\quad \forall \,   e \in \cK.
\end{align*}
 If, also,   $s_e = s'_e = e$ \, for every $e \in \cK$,  then \, 
$\chi_{fwer}^*\in \cE_{\Psi}(\gamma, \delta)$ when 
\begin{align*}  
A_e \geq a_e \, |\cK|/\delta \qquad \text{and} \qquad 
B_e \geq b_e \, |\cK|/\gamma  \quad  \forall \,   e \in \cK.
\end{align*}
\item[(iii)]
   $\chi^* \in \cC_\Psi(\alpha, \beta, \gamma, \delta)$ when  $\{C_e, D_e\}_{e \in \cK}$ are as in (i)  and $\{A_e, B_e\}_{e \in \cK}$ as in Theorem \ref{th:errcon}.

\end{enumerate}
\end{theorem}

\begin{proof}
The proof can be found in Appendix \ref{pf:th:errcon2}.
\end{proof}

\subsection{Selection of thresholds}

Based on Theorems \ref{th:errcon} and \ref{th:errcon2}, in what follows we assume that 
\begin{itemize}
\item  when $\emptyset \notin \Psi$,  $\{A_e, B_e\}_{e \in \cK}$  are selected, as functions of  $\gamma, \delta \in (0,1)$, so that 
\begin{align} \label{A,B,asym}
\begin{split}
& \chi^* \in \cC_\Psi(\gamma, \delta) \quad \forall \;
 \gamma, \delta \in (0,1), \\
&\log A_e \sim |\log\delta|,  \quad  \log B_e \sim |\log\gamma|
\quad \forall   \;  e \in \cK \quad \text{as} \;  \; \gamma,\delta \to 0,
\end{split}
\end{align}
\item  when $\emptyset \in \Psi$,  $\{A_e, B_e, C_e, D_e\}_{e \in \cK}$ are selected, as functions of
 $\alpha, \beta, \gamma, \delta \in (0,1)$, so that 
\begin{align} \label{C,D,asym}
\begin{split}
&\chi^*_{det} \in \cD_{\Psi}(\alpha, \beta), \quad \chi^*_{fwer} \in \cE_{\Psi}(\gamma, \delta), \quad \chi^* \in \cC_\Psi(\alpha, \beta, \gamma, \delta) \quad \forall \,  \alpha, \beta, \gamma, \delta \in (0,1), \\
&\log A_e \sim |\log\delta|,  \quad \log B_e \sim |\log\gamma|, \quad \log C_e \sim |\log\beta|, \quad \log D_e \sim |\log \alpha|  \\
&\qquad \qquad \qquad \qquad \forall \; \ e \in \cK
\quad  \text{as} \quad     \alpha, \beta, \gamma, \delta  \to 0.
\end{split}
\end{align}
\end{itemize}
We  further assume that
\begin{equation} \label{uniform}
A_e= A, \quad B_e = B, \quad C_e = C, \quad D_e = D \quad \text{ for every } e \in \cK.
\end{equation}
The latter is not  needed  for any of the results in Section \ref{sec:asymp_opt}, but simplifies the implementation of the proposed tests  in Section \ref{subsec:appl_anomalous_indep}, and is also used in Theorem \ref{pure_det_dep}. 
% in Section \ref{sec:appl_depen_struc}.

\section{Asymptotic optimality} \label{sec:asymp_opt}

In this section we show that, at least when the detection and/or isolation subsystems are  large enough,    the tests  introduced in the previous section achieve the minimum expected number of observations   until stopping, simultaneously for every  $\sfP \in \cP_\Psi$, to  a first-order asymptotic approximation as the corresponding error probabilities go to 0. Moreover, we establish sufficient conditions for   asymptotic optimality  using  smaller subsystems.  All these results are established under weak distributional assumptions, which  allow both for temporal  dependence as well as  for dependence among  the sources.
% and are stated next. 

\subsection{Distributional assumptions}
To  state  the  main results of this section we assume that,  for every  
$$e \in \cK,   \quad  \sfP \in \cG_{\Psi,e}, \quad \sfQ \in \cH_{\Psi,e} \cup \cH_0,  \quad  e  \subseteq s \subseteq  [K],$$
 there are positive  numbers, $ \cI \lt(\sfP^s, \sfQ^s\rt)$ and  $ \cI \lt(\sfQ^s, \sfP^s \rt)$, such that  as $n \to \infty$ 
\begin{align}\label{max_series}
\begin{split}
&\sfP\lt(\max_{1 \leq m \leq n} Z_m\lt(\sfP^s, \sfQ^s\rt) \geq n \rho  \rt) \to 0 \quad \text{for all} \; \;  \rho> \cI \lt(\sfP^s, \sfQ^s \rt), \\ 
 &\sfQ\lt(\max_{1 \leq m \leq n} Z_m\lt(\sfQ^s, \sfP^s\rt) \geq  n \rho \rt) \to 0
  \quad \text{for all} \; \;  \rho > \cI \lt(\sfQ^s, \sfP^s \rt) ,
\end{split} 
\end{align}
%and also 
\begin{align}\label{comp_conv_onesided}
\begin{split}
& \sum_{n = 1}^{\infty}\sfP\lt( Z_n\lt(\sfP^{s}, \sfQ^{s}\rt) < n \rho \rt) < \infty  \quad \text{for all} \; \;  \rho< \cI \lt(\sfP^s, \sfQ^s \rt), \\ %\label{comp_conv_onesided1}
&\sum_{n = 1}^{\infty}\sfQ\lt( Z_n\lt(\sfQ^{s}, \sfP^{s}\rt)  < n \rho \rt) < \infty  \quad \text{for all} \; \;   \rho< \cI \lt(\sfQ^s, \sfP^s \rt).  %\label{comp_conv_onesided2}
\end{split}
\end{align}

%\begin{remark}
Conditions   \eqref{max_series} and \eqref{comp_conv_onesided}  imply that 
\begin{align} \label{LLN}
\frac{1}{n}  Z_n \lt(\sfP^s, \sfQ^s\rt)  & \overset{\sfP} \rightarrow  \cI \lt(\sfP^s, \sfQ^s \rt)  \qquad \text{and} \qquad 
\frac{1}{n}  Z_n \lt(\sfQ^s, \sfP^s\rt)    \overset{\sfQ} \rightarrow   \cI \lt(\sfQ^s, \sfP^s \rt)
\end{align}
thus,  the limits,  $ \cI \lt(\sfP^s, \sfQ^s\rt)$ and  $ \cI \lt(\sfQ^s, \sfP^s \rt)$, can be interpreted  as (Kullback-Leibler)  divergences between the global distributions $\sfP$ and $\sfQ$ based on the data  from the sources in $s$.   Several useful properties regarding these numbers are   presented  in  Appendix \ref{app:info_num}.  %\end{remark}

\begin{remark}
When the convergences in \eqref{LLN}  hold not only in probability but also almost surely, 
then condition \eqref{max_series} is satisfied, but  not necessarily    \eqref{comp_conv_onesided}.
\end{remark}

\begin{comment}
\begin{remark} \label{remark:iid}
When  $\{Z_n\lt(\sfP^s, \sfQ^s\rt) , n \in \bN\}$ is a random walk whose increments have a finite \textit{first}  moment under both $\sfP$ and $\sfQ$, then conditions   \eqref{max_series} and \eqref{comp_conv_onesided} hold with
 $$\cI \lt(\sfP^s, \sfQ^s\rt) = \sfE_{\sfP}[Z_1\lt(\sfP^s, \sfQ^s\rt)] \qquad \text{and} \qquad  \cI \lt(\sfQ^s, \sfP^s \rt)= \sfE_{\sfQ}[Z_1\lt(\sfQ^s, \sfP^s\rt)],$$
i.e.,  $ \cI \lt(\sfP^s, \sfQ^s\rt)$ and  $ \cI \lt(\sfQ^s, \sfP^s \rt)$ in this case are the  \textit{Kullback-Leibler} divergences of the increments of the random walk.  Indeed, by   Kolmogorov's Strong Law of Large Numbers  it follows  that   \eqref{LLN}  holds with almost sure convergence, and by the  Chernoff bound it follows that \eqref{comp_conv_onesided}  holds. 
\end{remark}
\end{comment}

\begin{remark} \label{remark:iid}
If  $X^s$ is an iid sequence   and 
 $Z_1\lt(\sfP^s, \sfQ^s\rt)$ has  \textit{non-zero, finite}  expectation under both $\sfP$ and $\sfQ$, then conditions   \eqref{max_series}-\eqref{comp_conv_onesided} hold with
 $$\cI \lt(\sfP^s, \sfQ^s\rt) = \sfE_{\sfP}[Z_1\lt(\sfP^s, \sfQ^s\rt)] \qquad \text{and} \qquad  \cI \lt(\sfQ^s, \sfP^s \rt)= \sfE_{\sfQ}[Z_1\lt(\sfQ^s, \sfP^s\rt)].$$
%i.e.,  $ \cI \lt(\sfP^s, \sfQ^s\rt)$ and  $ \cI \lt(\sfQ^s, \sfP^s \rt)$ in this case are the  \textit{Kullback-Leibler} divergences between the distributions of the increments of the random walk under $\sfP$ and $\sfQ$.  
Indeed,   \eqref{LLN}  holds with almost sure convergence by   Kolmogorov's Strong Law of Large Numbers, and  \eqref{comp_conv_onesided}  by the  Chernoff bound. As a result,  in the special case of  Subsection \ref{gauss_prob_setup}, conditions  \eqref{max_series}-\eqref{comp_conv_onesided}  hold with 
\begin{equation}\label{gauss_info}
\begin{split}
\cI(\sfP^s, \sfQ^s) &= \frac{1}{2}\lt\{\log\frac{det(\Sigma_\sfQ)}{det(\Sigma_\sfP)} - |s| + tr(\Sigma_\sfQ^{-1}\Sigma_\sfP) + (\mu_\sfQ - \mu_\sfP)^t\Sigma_\sfQ^{-1}(\mu_\sfQ - \mu_\sfP)\rt\}, \\
\cI(\sfQ^s, \sfP^s) &= \frac{1}{2}\lt\{\log\frac{det(\Sigma_\sfP)}{det(\Sigma_\sfQ)} - |s| + tr(\Sigma_\sfP^{-1}\Sigma_\sfQ) + (\mu_\sfP - \mu_\sfQ)^t\Sigma_\sfP^{-1}(\mu_\sfP - \mu_\sfQ)\rt\},
\end{split}
\end{equation}
where $\mu_\sfP$ (resp. $\mu_\sfQ$) is the mean vector and $\Sigma_\sfP$ (resp. $\Sigma_\sfQ$)  the covariance matrix of $X_1^s$ under $\sfP$ (resp. $\sfQ$). In general, when $X^s$ is non-iid, these conditions hold under various models, such as autoregressive time-series models, state-space models, etc., see \cite[Section 3.4]{tartakovsky2014sequential} for more details.
\end{remark}

\subsection{Notations}
For any  $\sfP \in \cP_\Psi$,  $\cQ \subseteq \cP_\Psi$ and  $s \subseteq [K]$ we set  
$$\cI(\sfP^s, \cQ^s) := \inf_{\sfQ \in \cQ} \cI(\sfP^s, \sfQ^s),$$
whenever the quantities in the right-hand side are well defined.  We  also  set
$$ \cI(\sfP, \sfQ; s) := \cI(\sfP^s, \sfQ^s) \qquad \text{and}  \qquad 
\cI(\sfP, \cQ; s) :=   \min_{\sfQ \in \cQ}\cI(\sfP, \sfQ; s).$$
As before, we omit  $s$ when $s=[K]$ and replace it by $k$ whenever $s=\{k\}$ for some $k \in [K]$.  

\begin{example}  %(\textcolor{red}{modification here with some intuitions})
Consider the setup of  Subsection \ref{test_mean},  fix $\sfP \in \cP_\Psi$ and, for each $k \in [K]$ denote by $\sigma_k^2$ the variance of $X_1^k$ under 
$\sfP $. Moreover, denote by $\mu_k$  the expectation of  $X_1^k$ under 
$\sfP $ when $\{k\} \in \cA(\sfP)$. Then
 \begin{align*}
\cI(\sfP^k, \cH^k) &= \frac{1}{2} \left(\frac{\mu_k}{ \sigma_k} \right)^2   \qquad \text{if} \; \; \{k\} \in \cA(\sfP),\\
 \cI(\sfP^k, \cG^k) &= \frac{1}{2}
\;  \min_{\theta \in \cM_k \setminus \{0\}} \left( \frac{\theta}{\sigma_k} \right)^2 \qquad \text{if} \; \; \{k\} \notin \cA(\sfP).
 \end{align*}
%where $\mu_k$ is the mean and $\sigma_k^2$ the variance of $X_n^k$ under $\sfP$.   \textcolor{red}{Some comment?}
\end{example}

\begin{example}
Consider the setup of  Subsection \ref{test_corr}
%In this problem, \eqref{indep_subs} holds, and for any $\Psi$ and any $s \subseteq [K]$,  \eqref{max_series}-\eqref{comp_conv_onesided} are satisfied with $\cI(\sfP^s, \sfQ^s)$ and $\cI(\sfQ^s, \sfP^s)$ as defined in \eqref{gauss_info}. Furthermore, If for every $e = \{i, j\} \in \cK$ we consider
with $\cR_{kj}$ equal to  either  $\{0, \rho\}$ or $\{0, \pm\rho\}$ for every $k, j \in [K]$ with $k < j$, for some $\rho \in (0, 1)$. Then, for any $\sfP \in \cP_\Psi$, 
 \begin{align*}
\cI(\sfP^e, \cH^e) &= -\log \sqrt{1-\rho^2}  \qquad \forall  \;  e \in \cA(\sfP), \\
\cI(\sfP^e, \cG^e) &= \log\sqrt{1-\rho^2} + \frac{\rho^2}{1 - \rho^2} \qquad  \forall  \; e \notin \cA(\sfP).
 \end{align*}
%\textcolor{red}{Some comment?}
\end{example}

\subsection{Upper bound on the expected sample size}
We continue with an asymptotic  upper bound on the expected sample size of the proposed tests as their thresholds go to infinity.

\begin{theorem}\label{ub_I}
Suppose that  \eqref{comp_conv_onesided} holds,  $\emptyset \notin \Psi$, and let $\sfP \in \cP_\Psi$. Then, as $\gamma, \delta \to 0$, 
\begin{equation*} %\label{ub_I_1}
\sfE_\sfP\lt[T^*\rt] \lesssim \max_{e \in \cA(\sfP)}   \lt\{ \frac{|\log \gamma|}{\cI\lt(\sfP, \cH_{\Psi, e}; s'_e \rt)}  \rt\} \bigvee
 \max_{e \notin \cA(\sfP)}   \lt\{ \frac{|\log \delta|}{ \cI \lt(\sfP, \cG_{\Psi, e} ; s'_e \rt)}  \rt\}.
\end{equation*}
%When  $\Psi = \Psi_{m, m}$ for some $0<m < |\cK|$, the  upper bound takes the following form:
%\begin{equation*} 
 %\frac{|\log(\gamma \wedge \delta)|}{\cI(\sfP, \cP_\Psi \setminus \{\sfP\} ;s'_e)}.
%\end{equation*}
\end{theorem}

\begin{proof}
The proof can be found in Appendix \ref{pf:ub_DI}.
\end{proof}

\begin{theorem}\label{ub_DI}
Suppose that  \eqref{comp_conv_onesided} holds,  $\emptyset \in \Psi$, and let $\sfP \in \cP_\Psi$.
\begin{enumerate}
\item[(i)] If   $\sfP \in \cH_0$, then  
%as $\beta \to 0$ we have
\begin{align*} %\label{ub_D_1}
\sfE_\sfP\lt[T^*_{det}\rt], \sfE_\sfP\lt[T^*\rt] & \lesssim \;  \max_{e \in \cK}   \lt\{ \frac{|\log \beta|}{\cI \lt(\sfP, \cG_{\Psi,e}; s_e\rt)}  \rt\} \quad \text{as} \quad \beta \to 0, \\
\sfE_\sfP\lt[T^*_{fwer}\rt] & \lesssim \;  \max_{e \in \cK}   \lt\{ \frac{|\log \delta|}{\cI \lt(\sfP, \cG_{\Psi,e}; s_e\rt)}  \rt\}  \quad \text{as} \quad \delta \to 0.
\end{align*}
\item[(ii)] If   $\sfP \notin  \cH_0$, then
\begin{align*} %\label{ub_D_2}
\sfE_\sfP\lt[T^*_{det}\rt] \lesssim &\min_{e \in \cA(\sfP)} \lt\{ \frac{|\log \alpha|}{\cI\lt(\sfP, \cH_0; s_e \rt)} \rt\}  \quad \text{as} \quad \alpha \to 0,\\
\sfE_\sfP\lt[T^*_{fwer}\rt]  \lesssim  &\min_{e \in \cA(\sfP)}  \lt\{ \frac{|\log \gamma|}{\cI\lt(\sfP, \cH_0; s_e \rt)} \rt\} \bigvee \max_{e \in \cA(\sfP)}  \lt\{\frac{|\log \gamma|}{\cI \lt(\sfP, \cH_{\Psi, e}; s'_e \rt) }   \rt\}  \\
&\bigvee \max_{e \notin \cA(\sfP)}  \lt\{ \frac{|\log \delta|}{\cI \lt(\sfP, \cG_{\Psi,e}; s'_e \rt)}  \rt\}    \quad \text{as} \quad \gamma, \delta \to 0,\\
\sfE_\sfP\lt[T^*\rt] \lesssim  &\min_{e \in \cA(\sfP)}  \lt\{ \frac{|\log \alpha|}{\cI\lt(\sfP, \cH_0; s_e \rt)} \rt\} \bigvee \max_{e \in \cA(\sfP)}  \lt\{\frac{|\log \gamma|}{\cI \lt(\sfP, \cH_{\Psi, e}; s'_e \rt) }   \rt\}  \\
&\bigvee \max_{e \notin \cA(\sfP)}  \lt\{ \frac{|\log \delta|}{\cI \lt(\sfP, \cG_{\Psi,e}; s'_e \rt)}  \rt\}   \quad \text{as} \quad \alpha, \gamma, \delta \to 0.
\end{align*}
\end{enumerate}
\end{theorem}

\begin{proof}
The proof can be found in Appendix \ref{pf:ub_DI}.
\end{proof}

\subsection{Asymptotic optimality}

We next show that the proposed tests achieve the optimal expected sample size in their corresponding classes, to a first-order asymptotic approximation as the corresponding error rates go to $0$, when the detection and/or isolation subsystems are selected to be equal to $[K]$, and in certain cases  even smaller than $[K]$.  Our standing assumption in the remainder of this section is that conditions \eqref{max_series}-\eqref{comp_conv_onesided} hold.\\

%\subsection{Pure isolation}
%We now focus on the case that the main focus is either on the false positive or the false negative control, or both. We start with the special case that the  in the sense that $\gamma \to 0$ is much smaller than $\delta$, and also  $\alpha$  if  the global null is possible.

%The next theorem establishes asymptotic optimality of $\chi^*$ in the \textit{Pure Isolation Regime for False Positive} when the true probability distribution $\sfP \in \cP(\Psi) \setminus \cP_0$. %and the levels $\alpha, \beta, \gamma, \delta \to 0$ in such a way that the control of false positive is much stricter than the control of false alarm and false negative, or in other words, $\gamma \to 0$ in a much faster rate than $\alpha$ and $\delta$. 
%We show that in this case, the ESS of $\chi^*$ is asymptotically equal to that of the pure isolation rule $\chi^*_{iso}$ when $\gamma \to 0$ much faster than $\delta$. %For that reason, we define this asymptotic regime as the \textit{}.
We start with the case that  $\emptyset \notin \Psi$, where  there is no detection task and the only test under consideration is $\chi^*$, defined in  Subsection \ref{subsubsec: pure_isolation_test}.

\begin{theorem}\label{P1P2}
Suppose that   $\emptyset \notin \Psi$,  $\sfP \in \cP_\Psi$, and let the isolation subsystems be selected such that 
\begin{align}
 \min_{e \in \cA(\sfP)} \cI \lt( \sfP , \cH_{\Psi,e}; s'_e \rt)  &= \min_{e \in \cA(\sfP)} \cI \lt( \sfP , \cH_{\Psi,e} \rt),   \label{condP1} \\
 \min_{e \notin \cA(\sfP)} \cI \lt( \sfP , \cG_{\Psi,e}; s'_e \rt)  
&=  \min_{e \notin \cA(\sfP)} \cI \lt( \sfP , \cG_{\Psi,e} \rt) ,
  \label{condP2}
\end{align} 
which is trivially the case when $s'_e=[K]$ for every $e \in \cK$.
Then,  as $\gamma, \delta \to 0$, 
\begin{align*}
\sfE_\sfP\lt[T^*\rt] &\sim \max_{e \in \cA(\sfP)} \lt\{ \frac{|\log\gamma|}{  \cI(\sfP, \cH_{\Psi,e}) }\rt\}  \bigvee 
\max_{e \notin \cA(\sfP)} \lt\{  \frac{|\log\delta|}{  \cI(\sfP, \cG_{\Psi,e} )} \rt\}  \\
& \sim  \inf \lt\{ \sfE_{\sfP}[T]:\; 
(T, D) \in \cC_{\Psi}(\gamma, \delta) \rt\} .
\end{align*}
When, in particular,  $\Psi = \Psi_{m, m}$ for some $m \in (0, |\cK|)$, 
then,  as $\gamma, \delta \to 0$,
\begin{equation*} 
\sfE_\sfP\lt[T^*\rt] \sim
 \frac{|\log(\gamma \wedge \delta)|}{\cI(\sfP, \cP_\Psi \setminus \{\sfP\})} \sim  \inf \lt\{ \sfE_{\sfP}[T]:\; 
(T, D) \in \cC_{\Psi}(\gamma, \delta) \rt\} .
\end{equation*}

%\item If $\emptyset \in \Psi$, then as  $\alpha, \beta, \gamma, \delta \to 0$ such that either $\gamma << \alpha$ or $\delta << \alpha$ we have $$\sfE_\sfP\lt[T^*\rt] \sim \frac{|\log\gamma|}{\min_{e \in \cA(\sfP)}  \cI(\sfP, \cH_{\Psi,e} ) }  \bigvee \frac{|\log\delta|}{\min_{e \notin \cA(\sfP)}  \cI(\sfP, \cG_{\Psi,e} ) }  \sim N^*(\sfP; \alpha, \beta, \gamma, \delta; \Psi).$$\end{itemize}
%\end{enumerate}
\end{theorem}

\begin{proof}
The proof can be found in Appendix \ref{pf:P1P2}.
\end{proof}

%\subsection{Joint detection and isolation} \label{optimality_md}

%We start with the case that all nulls are correct and we  next consider the case where  the global null hypothesis is possible, i.e., $\emptyset \in \Psi$, but there is at least one correct alternative  and  the main focus is on the  detection task, in the sense that $\alpha$ is much smaller than $\gamma$ and $\delta$. 

%Asymptotic optimality under the global null.

We continue with the case that  there is also a detection task, i.e.,
$\emptyset \notin \Psi$,  where we  compare the  three  tests, $\chi^*_{det}$, $\chi^*$ and $\chi^*_{fwer}$, introduced in Subsections \ref{subsubsec: pure_detection_test}, \ref{subsubsec: joint_test} and \ref{jdi_to_fwer}, respectively.  We start by establishing sufficient conditions for their asymptotic optimality when there are no signals present. 

\begin{theorem}\label{PminusP0}
Suppose that   $\emptyset \in \Psi$, $\sfP \in \cH_0$,  and let
the detection subsystems be selected such that 
\begin{equation}\label{condPminusP0}
 \min_{e \in \cK} \; \cI \lt( \sfP , \cG_{\Psi,e} ; s_e \rt) =\min_{e \in \cK} \; \cI \lt( \sfP , \cG_{\Psi,e}  \rt),
\end{equation}
which is trivially the case when $s_e=[K]$ for every $e \in \cK$.
Then,   as $ \alpha, \beta \to 0 $,
\begin{align*}
\sfE_\sfP\lt[T^*_{det}\rt] & \; \sim  \; \max_{e \in \cK} \lt\{\frac{|\log\beta|}{ \cI \lt( \sfP , \cG_{\Psi,e} \rt)} \rt\} \; \sim \;   \inf \lt\{ \sfE_{\sfP}[T]:\; 
(T, D) \in \cD_{\Psi}(\alpha, \beta) \rt\},
\end{align*}
as $\alpha, \beta \to 0$, while  $\gamma$ and $\delta$ are either  fixed or go to $0$, 
$$\sfE_\sfP\lt[T^*\rt] \; \sim \; \max_{e \in \cK}  \lt\{ \frac{|\log\beta|}{ \cI \lt( \sfP , \cG_{\Psi,e} \rt)}  \rt\}\; \sim  \;
  \inf \lt\{ \sfE_{\sfP}[T]:\; 
(T, D) \in \cC_{\Psi}(\alpha, \beta,\gamma, \delta) \rt\},$$
and  as $\gamma, \delta \to 0$,
\begin{align*}
\sfE_\sfP\lt[T^*_{fwer}\rt]\; \sim \; \max_{e \in \cK}  \lt\{ \frac{|\log\delta|}{ \cI \lt( \sfP , \cG_{\Psi,e} \rt)}  \rt\}\; \sim  \;
  \inf \lt\{ \sfE_{\sfP}[T]:\; 
(T, D) \in \cE_{\Psi}(\gamma, \delta) \rt\}.
\end{align*}

\end{theorem}

\begin{proof}
The proof can be found in Appendix \ref{pf:P0}.
\end{proof}

%\begin{remark}

Theorem \ref{PminusP0} shows that, with sufficiently large subsystems so that  \eqref{condPminusP0} holds, $\chi^*$ has
 the same expected sample size  as $\chi^*_{det}$, to a first-order approximation as   $\alpha$ and $\beta$ go to $0$, independently of how $\gamma$ and $\delta$ are selected.  On the other hand,   this is the case for  $\chi^*_{fwer}$ when   $\delta = \beta$,  in which case its  missed detection rate does not exceed $\beta$  (recall \eqref{fwer_and_det}).   Therefore, when $\delta=\beta$ and   \textit{there are not any signals present}, 
 $\chi^*_{fwer}$  does not  lead to  deterioration in  performance   in comparison to   $\chi^*_{det}$,  when $\alpha$ and $\beta$ are small enough. 
 
 We continue with the case that there is at least one signal.

% \end{remark}
 %  $\beta$, $\delta \leq \beta$,  in order to . Thus, if $\delta  \sim  \beta^r$ for some $r > 1$, then the asymptotic relative efficiency  of $\chi^*_{fwer}$ is $r$ times larger than that of $\chi^*$ toa first-order approximation as   $\gamma$ and $\delta$ go to $0$. ????\end{remark}
%We  next  establish sufficient conditions for the  asymptotic optimality of $\chi^*_{det}$, $\chi^*$ and $\chi^*_{fwer}$  when there is at least one signal and the strictest error control is that of  false alarms. 

\begin{theorem}\label{P0}
Suppose that $\emptyset \in \Psi$, $\sfP \in \cP_\Psi \setminus  \cH_0$, and let the detection subsystems be selected so that 
\begin{equation}\label{condP0}
\max _{e \in \cA (\sfP)} \cI(\sfP, \cH_0; s_e) = \cI(\sfP, \cH_0),
\end{equation}
which is trivially the case when $s_e = [K]$ for some $e \in \cA(\sfP)$. Then, as $\alpha, \beta \to 0$,
\begin{align*}
\sfE_\sfP\lt[T^*_{det}\rt] &\sim \frac{|\log\alpha|}{\cI\lt(\sfP, \cH_0\rt)} \sim  \inf \lt\{ \sfE_{\sfP}[T]:\; 
(T, D) \in \cD_{\Psi}(\alpha, \beta) \rt\},
\end{align*}
and,  as   $\alpha, \beta, \gamma, \delta \to 0$  such that 
$|\log \alpha| >> |\log \gamma| \vee |\log \delta|,$ 
$$
\sfE_\sfP\lt[T^*\rt] \sim \frac{|\log\alpha|}{\cI\lt(\sfP, \cH_0\rt)} \sim  \inf \lt\{ \sfE_{\sfP}[T]:\; 
(T, D) \in \cC_{\Psi}(\alpha, \beta,\gamma, \delta) \rt\}.$$
If, also, the isolation subsystems are selected so that \eqref{condP1} holds,  then, as $\gamma, \delta \to 0$ such that $|\log \gamma| >> |\log \delta|$, 
\begin{align*}
\sfE_\sfP\lt[T^*_{fwer}\rt] \sim \max_{e \in \cA(\sfP)} \lt\{ \frac{|\log\gamma|}{\cI(\sfP, \cH_0)  \wedge  \cI(\sfP, \cH_{\Psi,e} ) }\rt\} %\bigvee \max_{e \notin \cA(\sfP)}  \lt\{ \frac{|\log\delta|}{ \cI(\sfP, \cG_{\Psi,e} ) }  \rt\} \\
\sim   \inf \lt\{ \sfE_{\sfP}[T]:\;  (T, D) \in \cE_{\Psi}(\gamma, \delta) \rt\}.
\end{align*}

\end{theorem}
\begin{proof}
The proof can be found in Appendix \ref{pf:P0}. 
\end{proof}

%\begin{remark}
Theorem \ref{P0} shows that, when  the subsystems are large enough so that \eqref{condP0} holds,  $\chi^*$ has the same expected sample size  as $\chi^*_{det}$ to a first-order approximation as  $\alpha$ goes to $0$ much faster than $\gamma$ and $\delta$. 
On the other hand,  when   $\gamma = \alpha$,  in which case the  false alarm rate of 
 $\chi^*_{fwer}$  does not exceed $\alpha$  (recall \eqref{fwer_and_det}),   then from the the above theorem it follows that 
 $$\frac{\sfE_\sfP[T^*_{det}]}{\sfE_\sfP[T_{fwer}^*]}  \to  \min_{e \in \cA(\cP)}  \frac{\cI(\sfP, \cH_{\Psi, e})}{\cI(\sfP,  \cH_0)} \wedge 1
 $$
as   $\alpha, \beta, \gamma, \delta \to 0$  such that 
$|\log \alpha| >> |\log \gamma| \vee |\log \delta|$.   Intuitively speaking,  this means that  $\chi^*_{fwer}$  has the same asymptotic performance as
$\chi^*_{det}$  when there is at least one signal and   $\gamma = \alpha$ is much smaller than  $\delta$,   as long as the detection problem is harder than the isolation problem in the sense that 
\begin{equation*}
\cI(\sfP, \cH_0) \leq \min_{e \in \cA(\sfP)} \cI(\sfP, \cH_{\Psi, e}).
\end{equation*}
Typically, however, especially when there are more than one signals under $\sfP$, the isolation problem is (much) harder, and hence, the above limit will be (much) smaller than 1, suggesting that when $\gamma=\alpha$  is much smaller than $\delta$,   $\chi^*_{fwer}$  will perform much worse than $\chi^*_{det}$. 
\\

%however, the reverse is true,  which case its  false alarm rate does not exceed $\alpha$  (recall \eqref{fwer_and_det}), as per the equivalent approximation as  $\gamma$ goes to $0$ much faster than $\delta$.   Therefore,   $\chi^*_{fwer}$  leads to a significant deterioration in  performance, in comparison to  $\chi^*$ and  $\chi^*_{det}$, when \textit{false alarm control is the strictest} but \textit{the corresponding isolation problem is more difficult than the detection problem}. The latter arises, in general, when the true number of signals is larger than $1$. %This is not the case, however, when there is at least one signal, as we show next. 

% we know that it  guarantees that the  false alarm rate of  $\chi^*_{fwer}$ does not exceed that of  $\chi^*_{det}$, implies that.....

%\end{remark}

% must not exceed $\alpha$, i.e.,$\gamma \leq \alpha$. Thus, if $\gamma = O(\alpha^r)$, \textcolor{red}{(or, $\gamma \lesssim \alpha^r$)} for some $r > 1$ then the asymptotic performance of $\chi^*_{fwer}$ is $r$ times worse than that of $\chi^*$ as $\gamma$ goes to $0$ much faster than $\delta$.

Finally, we consider the general case where  all  error rates go to $0$ at arbitrary rates, and there is at least one signal.

%establish sufficient conditions for the  asymptotic optimality of $\chi^*$ and $\chi^*_{fwer}$ when there is at least one signal and 

\begin{theorem}\label{P0P1P2}
Suppose that $\emptyset \in \Psi$,  $\sfP \in \sfP_\Psi \setminus  \cH_0$, and let the detection subsystems  be selected so that \eqref{condP0}  holds, 
 and let  the isolation subsystems  be selected so that \eqref{condP1} - \eqref{condP2}  hold.  Then,  as $\gamma, \delta \to 0$,
\begin{align*}
\sfE_\sfP\lt[T^*_{fwer}\rt] &\sim \max_{e \in \cA(\sfP)} \lt\{ \frac{|\log\gamma|}{\cI(\sfP, \cH_0)  \wedge  \cI(\sfP, \cH_{\Psi,e} ) }\rt\} \bigvee \max_{e \notin \cA(\sfP)}  \lt\{ \frac{|\log\delta|}{ \cI(\sfP, \cG_{\Psi,e} ) }  \rt\} \\
& \sim   \inf \lt\{ \sfE_{\sfP}[T]:\;  (T, D) \in \cE_{\Psi}(\gamma, \delta) \rt\},
\end{align*}
and,  as $\alpha, \beta, \gamma, \delta \to 0$, %we have%If also the  detection subsystems are selected so that  \eqref{condP0} holds, then as   $\alpha, \beta, \gamma, \delta \to 0$ we have 
\begin{align*}
\sfE_\sfP\lt[T^*\rt] &\sim   \frac{|\log\alpha|}{\cI(\sfP, \cH_0)} \bigvee \max_{e \in \cA(\sfP)}  \lt\{\frac{|\log\gamma|}{ \cI(\sfP, \cH_{\Psi,e} ) } \rt\} \bigvee   \max_{e \notin \cA(\sfP)}  \lt\{ \frac{|\log\delta|}{ \cI(\sfP, \cG_{\Psi,e} ) } \rt\} \\
& \sim 
  \inf \lt\{ \sfE_{\sfP}[T]:\; 
(T, D) \in \cC_{\Psi}(\alpha, \beta,\gamma, \delta) \rt\}.
\end{align*}
\end{theorem}

\begin{proof}
The proof can be found in Appendix \ref{pf:P0P1P2}.
\end{proof}

Theorem \ref{P0P1P2}   generalizes 
Theorem \ref{P0}  and allows for
a  comparison between  $\chi^*$ and $\chi^*_{fwer}$ as $\alpha, \beta, \gamma, \delta$ go to 0 at arbitrary rates.  For example, an implication of this theorem is that, 
when the  detection and isolation subsystems are large enough so that  conditions  \eqref{condP0}  and  \eqref{condP1}-\eqref{condP2} hold,    $\chi^*$ and $\chi^*_{fwer}$ have the same expected sample size, to a first-order asymptotic approximation,  when there is at least one signal and either $\alpha = \gamma, \delta \to 0$  or $\delta$ goes to $0$ much faster than $\alpha$ and $\gamma$.

\begin{remark}\label{rem:P1orP2}
In Theorems \ref{P1P2} and  \ref{P0P1P2},  the isolation subsystems do not need to satisfy  \eqref{condP2}  (resp.  \eqref{condP1})  when $|\log \gamma| >> |\log \delta|$ (resp.     $|\log \delta| >> |\log \gamma|$).   The proof for this claim is presented in Appendix \ref{pf:rem:P1orP2}.  
\end{remark}

\begin{remark}\label{rem:fwer}
For the asymptotic approximation regarding  the expected sample size of  $T^*_{fwer}$ in Theorems \ref{P0} and \ref{P0P1P2}, the isolation subsystems do not need to satisfy  \eqref{condP1}  when 
\begin{equation}\label{FWERcondP0}
\cI(\sfP, \cH_0) \leq \min_{e \in \cA(\sfP)} \cI(\sfP, \cH_{\Psi, e}) \quad \text{and} \quad  \max _{e \in \cA (\sfP)} \cI(\sfP, \cH_0; s_e) \leq \min_{e \in \cA(\sfP)} \cI \lt( \sfP , \cH_{\Psi,e}; s'_e \rt),
\end{equation}
and the detection subsystems do not need to satisfy \eqref{condP0} when 
\begin{equation}\label{FWERcondP1}
\cI(\sfP, \cH_0) \geq \min_{e \in \cA(\sfP)} \cI(\sfP, \cH_{\Psi, e}) \quad \text{and} \quad  \max _{e \in \cA (\sfP)} \cI(\sfP, \cH_0; s_e) \geq \min_{e \in \cA(\sfP)} \cI \lt( \sfP , \cH_{\Psi,e}; s'_e \rt).
\end{equation}
The proof for this claim is presented  in Appendix \ref{pf:rem:fwer}.
\end{remark}

\begin{remark}
The conditions of the  previous   theorems may  be satisfied when the  detection and isolation subsystems  are  strict subsets of $[K]$,  even when $s_e=s'_e=e$ for every $e \in \cK$.  We elaborate  this point  in Proposition  \ref{prop:P1} of the supplementary file, as well as in the two following sections, where we specialize the general setup which is considered so far. 
% to two concrete problems.  
\end{remark}

\section{Detection and isolation of anomalous sources} \label{sec:appl_anomalous}  In this section we  apply the results of   Sections  \ref{sec:prob_form}-\ref{sec:asymp_opt} to the  problem of  detecting and isolating anomalous sources. That is, throughout this section  the units are the data sources themselves, i.e., $ \cK = \{\{k\} : k \in [K]\}$ and, as a result,   the hypotheses refer to the marginal distributions of the data sources. Moreover,  we assume that the distributional assumptions \eqref{max_series}-\eqref{comp_conv_onesided} are satisfied, thus,  all asymptotic optimality results in Section  \ref{sec:asymp_opt} hold when the  subsystems are large enough. 
% i.e.,    $s_k=[K]$ and/or $s'_k=[K]$ for every $k \in [K]$.  

In Subsection \ref{subsec:appl_anomalous_indep}  we consider the case that the   data sources are independent, i.e., 
\begin{equation} \label{product}
\sfP = \sfP^{1} \otimes  \ldots \otimes \sfP^{K} \quad \text{ for every } \;
\sfP \in \cP_\Psi,
\end{equation} 
and there are  a priori bounds on the  number of signals,  i.e.,  $\Psi$  is of the form 
$\Psi_{l,u}$, % for some positive integers $l<u$,  
and we show that the asymptotically optimal procedures in this context are  also \text{easily implementable}. In Subsection \ref{subsec:appl_anomalous_depen} we consider the general case where  the  data sources are not necessarily independent, 
 and  asymptotically optimal procedures  are not, in general, easy to implement.  In this context,  we focus on computationally simple versions of the proposed tests,  upper bound their asymptotic relative efficiencies, and establish sufficient conditions for their asymptotic optimality.

\begin{remark}
An implication of the independence assumption \eqref{product} is that  conditions \eqref{max_series}-\eqref{comp_conv_onesided},
 which refer to \textit{global} distributions, are satisfied as long as similar  conditions are satisfied by the marginal distributions of the  data  sources.  For example, they hold  when,
  for every $k \in [K]$,  $X^k$ is an iid sequence and $Z_1 (\sfP^k, \sfQ^k)$ has  finite, non-zero  expectation  under  both $\sfP^k$ and  $\sfQ^k$.  For more details on this point, we refer to  Proposition \ref{propo:sufficient}. 
\end{remark}

\subsection{The case of independence}\label{subsec:appl_anomalous_indep} 
Throughout this subsection we assume that the  data sources are independent,  i.e., \eqref{product} holds,  and that arbitrary lower and upper bounds, $l$ and  $u$, are given on the number of signals, i.e.,  $\Psi=\Psi_{l,u}$. In this setup, we show that  asymptotically optimal procedures  are always easy to implement, as they  only require  \textit{ordering at each time $n \in \bN$  the  local  statistics $\Lambda_k(n), k \in [K]$}, defined in \eqref{LLR}. %, and since the units are singletons, we  denote  the latter simply by  .%,  their ordered versions  by  $\Lambda_{(1)}(n) \geq \dots \geq \Lambda_{(K)}(n)$, and   the corresponding indices by  $i_1(n), \dots, i_K(n)$.\\
%Our second goal in this section is to recover  certain asymptotically optimal testing procedures  that have been obtained in the context of familywise error control. Even for this well-studied problem however,  we will obtain certain new, computationally efficient,  asymptotically optimal tests. All these test though are to be compared with the corresponding asymptotically optimal rule  for joint detection and isolation, which is  is equally easy to  apply and is much more flexible. 

We start by showing that, even when the isolation subsystems are selected as large as possible, i.e.,  $s'_k=[K]$ for every $k \in [K]$,  the estimated  signal set, defined in \eqref{D_iso}, at  $T^*$ or $T^*_{fwer}$, admits a very simple form.

\begin{theorem}\label{decision_indep}
%Let $(T, D)$ be either $\chi^*$ or $ \chi^*_{fwer}$. 
Suppose that either $s'_k = \{k\} \; \forall \; k \in [K]$, or  $s'_k = [K] \; \forall \; k \in [K]$. When $l = 0$, suppose also  that either $s_k = \{k\}  \; \forall \; k \in [K]$, or  $s_k = [K]
\; \forall \; k \in [K]$. Then:
%$$D = \lt\{i_1(T), \dots, i_{l \vee p(T) \wedge u} (T)\rt\},$$
$$D_{\rm iso}(T) = \lt\{i_1(T), \dots, i_{l \vee p(T) \wedge u} (T)\rt\} \quad \text{for}  \; \; T \in \{T^*, T^*_{fwer}\},$$
with the understanding that if  $l = p(T) = 0$, then $D_{\rm iso}(T) = \emptyset$.
%\textcolor{red}{Where is $p(n)$ defined?}
%and if $s'_k = \{k\}$ for every $k \in [K]$ then
%$D = \lt\{i_1(T), \dots, i_{p(T)} (T)\rt\}.$
\end{theorem}

\begin{proof} 
When  $s'_k = \{k\}$ for every $k \in [K]$, the proof is straightforward.  Indeed, in this case the isolation statistics reduce to the local statistics and  $D_{\rm iso}(T) = \lt\{i_1(T), \dots, i_{p(T)} (T)\rt\}$. At the same time,  at least $l$ and at most $u$ many local statistics are greater than $1$ at the time of stopping, i.e., $l \leq p(T) \leq u$. When $s'_k = [K]$ for every $k \in [K]$, on the other hand,  the proof is much more involved and  is presented in  Appendix \ref{pf:decision_indep}.
\end{proof}

In the rest of this subsection we  set either $s'_k = \{k\}$ for every $k \in [K]$ or $s'_k =[K]$ for every $k \in [K]$, and focus  on the form of the stopping rules.    We start with the  case that $l\geq 1$, where there is no detection task.

%\subsection{Pure isolation}
% and we show that  setting $s'_k = [K]$,  $k \in [K]$ leads to a very simple representation for $T^*$. %which is convenient for practical implementation.

\begin{theorem}\label{th:iso_omega}
Suppose that   $ l\geq 1$ and let $s'_k = [K]$ for every  $k \in [K]$.
\begin{enumerate}
\item[(i)] If $\ell=u <K$, then
\begin{align*}
T^{*} &= \inf\lt\{n \in \bN : \Lambda_{(u)}(n) \geq \Lambda_{(u+1)}(n) \;  \max\{A, B\}\rt\}. %\quad D^{*} = \{i_1(T^*), \dots, i_u(T^*)\}.
\end{align*}
\item[(ii)] If $ l<u\leq K$, then
\begin{align*}
T^* &=
 \begin{cases}
 T_1 \wedge  T_2  \wedge  T_4 \wedge  T_5 \wedge  T_6   &\text{if} \quad   u<K \\
 T_1 \wedge  T_2  \wedge T_3   &  \text{if} \quad   u=K \\
\end{cases}, %\quad 
 %D^* =  \lt\{i_1(T^*), \dots, i_{l \vee p(T^*) \wedge u} (T^*)\rt\}, 
\end{align*}
where
\begin{align*}
T_1 &:= \inf\lt\{n \in \bN : \;  p(n) < l  \quad  \text{and} \quad  \Lambda_{(l)}(n) \geq \Lambda_{(l+1)}(n) \; \max\{A, B\}\rt\},\\
T_2 &:= \inf\lt\{n \in \bN : \;  p(n) = l \quad \text{and} \quad \Lambda_{(l+1)}(n) \leq \min \lt\{  1/A,  \; \Lambda_{(l)}(n) / B \rt\}  \rt\},\\
T_3 &:= \inf\lt\{n \in \bN : \;    p(n) >l  \quad   \text{and} \quad  \Lambda_k(n) \notin \lt(1/A, B\rt) \quad \forall \; k \in [K] \rt\}, \\
T_4 &:= \inf\lt\{n \in \bN : \;   p(n) \in (l,u)   \quad  \text{and}  \quad  \Lambda_k(n) \notin \lt(1/A, B\rt) \quad \forall  \;  k \in [K] \rt\},\\
T_5 &:= \inf\lt\{n \in \bN :\;   p(n) = u  \quad  \text{and} \quad  \Lambda_{(u)}(n) \geq  \max \lt\{ B, \; A \, \Lambda_{(u+1)}(n)  \rt\}  \rt\},\\
T_6 &:= \inf\lt\{n \in \bN :\;  p(n) > u \quad \text{and} \quad  \Lambda_{(u)}(n) \geq \Lambda_{(u + 1)}(n) \;  \max\{A, B\}\rt\}.
\end{align*}
\end{enumerate}
\end{theorem}

\begin{proof}
The proof can be found in Appendix \ref{pf:th:iso_omega}.
\end{proof}

\begin{remark}
When $l=u$, the stopping time  in Theorem \ref{th:iso_omega}   coincides with  the so-called ``gap rule'', which was proposed in \cite{song2017asymptotically} and  where its  asymptotic optimality was established only in the case that the sources generate iid data.   When $l<u$, on the other hand,  the stopping time in Theorem \ref{th:iso_omega}   differs from the  ``gap-intersection rule''  proposed  in \cite{song2017asymptotically}. In particular, it is  defined up to only 2, instead of 4, distinct thresholds,  while being the  minimum of 6, instead of 3, stopping times.
\end{remark}

%\subsection{Joint detection and isolation}
We continue with the  case that $l = 0$, where there is also a detection task. 
% detection task is present. %and  consider the stopping rules for  the problems of pure detection and  joint detection and isolation. 

\begin{theorem}\label{th:dete_indep}
Suppose that $l=0$ and $s_k = [K]$  for every $k \in [K]$. Then, the stopping times, $T_0$ and $T_{det}$,  defined as in Subsection \ref{subsubsec: pure_detection_test}, take the following form:
 %$T^*_{det} = T_0 \wedge T_{det}$, where
\begin{align*}
T_{0} &= \inf\lt\{n \in \bN : \; \Lambda_{(1)}(n) \leq 1/C\rt\}, \quad T_{det} =  \inf \{n \in \bN : \; \prod_{i = 1}^{p(n) \wedge u} \Lambda_{(i)}(n) \geq D \}.
\end{align*}
If, also, $s'_k= [K]$  for every $k \in [K]$, then
\begin{align*}
T_{joint} &=
\begin{cases} 
T_1 \wedge T_3 \wedge T_4 \wedge T_5  &  \text{if} \quad u<K\\
T_1 \wedge T_2   & \text{if} \quad u=K 
 \end{cases},
%D^* &= 
%\begin{cases}
%\emptyset &\text{if} \quad T^* = T_0, \\
%\lt\{i_1(T^*), \dots, i_{p(T^*)}(T^*)\rt\} & \text{if} \quad u=K \\ 
%\lt\{i_1(T^*), \dots, i_{p(T^*) \wedge u}(T^*)\rt\}\ & \text{otherwise}
%\end{cases},
\end{align*}
where %$T_0$ is as in Theorem \ref{th:dete_indep}, and 
\begin{align*}
T_{1} &= \inf \{n \in \bN : \; \Lambda_{(1)}(n) \geq D, \quad \Lambda_{(2)}(n) \leq \min\{1/A, \; \Lambda_{(1)}(n) / B\} \} ,\\
T_{2} &= \inf \{n \in \bN : \; p(n) >1, \quad \prod_{i = 1}^{p(n)} \Lambda_{(i)}(n) \geq D,  \quad \Lambda_k(n) \notin \lt(1/A, B\rt)\quad  \forall \; k \in [K] \} , \\
%T_{0} &= \inf \{n \in \bN : \Lambda_{(1)}(n) \leq 1/C \},\\
%T_{1} &= \inf \{n \in \bN : \Lambda_{(1)}(n) \geq D, \quad \frac{\Lambda_{(1)}(n)}{\Lambda_{(2)}(n)} \geq B, \quad \Lambda_{(2)}(n) \leq 1/A  \},\\
T_{3} &= \inf \{n \in \bN :  \; p(n) \in (1, u), \quad \prod_{i = 1}^{p(n)} \Lambda_{(i)}(n) \geq D, \quad \Lambda_k(n) \notin \lt(1/A, B\rt) \quad  \forall \; k \in [K] \},\\
T_{4} &= \inf \{n \in \bN : \;  p(n) = u, \quad \prod_{i = 1}^{u} \Lambda_{(i)}(n) \geq D, \quad  \Lambda_{(u)}(n) \geq  \max \lt\{ B, \; A \, \Lambda_{(u+1)}(n)  \rt\} \},\\
T_{5} &= \inf \{n \in \bN : \;  p(n) > u, \quad \prod_{i = 1}^{u} \Lambda_{(i)}(n) \geq D, \quad   \Lambda_{(u)}(n) \geq 
\Lambda_{(u+1)}(n) \max\{A , B\}   \}.
\end{align*}
%Furthermore, if $A, B, C$ and $D$ are selected such that  $\chi^* \in \cC(\alpha, \beta, \gamma, \delta; \Psi)$ with
%$\log A \sim |\log\delta|$, $\log B \sim |\log\gamma|$, $\log C \sim |\log\beta|$, and $\log D \sim |\log \alpha|$,
%for example, according to the maximum of the corresponding quantities in \eqref{thresholds_I} and \eqref{thresholds_D}, %as in Theorem \ref{errcon}. 
%then the rule is AO.
%with $$A = \frac{K}{\delta}\sum_{m=1}^{u \wedge K-1} C_m^{K-1}, \quad \quad B = \frac{K}{\gamma}\sum_{m=1}^u C_{m-1}^{K-1}, \quad \quad C = \frac{1}{\beta}, \quad \quad D = \frac{K}{\alpha}\sum_{m=1}^u C_{m-1}^{K-1}.$$
%$A = C = \frac{K+1}{\delta}\sum_{m=1}^{u \wedge K-1} C_m^{K-1}, \quad B = D = \frac{K}{\gamma}\sum_{m=1}^u C_{m-1}^{K-1},$$. %\in \cC_{fwer}(\gamma, \delta; \Psi_{0, u})$
\end{theorem}

\begin{proof}
The proof can be found in Appendix \ref{pf:th:dete_indep}.
\end{proof}

\begin{remark}
The corresponding rule for the pure detection problem has been  considered in \cite{Inv_seq_det2001, det_Tart2002, tartakovsky2003sequential,  fellouris2017multichannel}.%, \textcolor{red}{however, when there is no prior information. }
\end{remark}

Subsection \ref{jdi_to_fwer} implies that  setting $A = C$ and $B = D$  in Theorem \ref{th:dete_indep} leads to an asymptotically optimal test with respect to \textit{familywise error control},  since $s_k = s'_k = [K]$ for every $k \in [K]$.   In the next theorem we show that such an asymptotically optimal,  and \textit{equally feasible}, test can also be obtained even with 
$s_k=\{k\}$ and $s'_k = \{k\}$ (resp. $[K]$) for every $k \in [K]$ when $u = K$ (resp. $u < K$).

\begin{theorem}\label{th:indep_fwer}
Suppose that  $l=0$ and $s_k = \{k\}$  for every $k \in [K]$. Then $\chi^*_{fwer}$ is asymptotically optimal, as $\gamma, \delta \to 0$, under every $\sfP \in \cP_\Psi$
\begin{enumerate}
\item[(i)]  when $u=K$ and  $s'_k =\{k\}$    for every $k \in [K]$,  in which case 
\begin{align*}
T^*_{fwer} &= \inf \lt\{n \in \bN : \; \Lambda_k(n) \notin \lt(1/A, B\rt) \quad \forall \; k \in [K ]\rt\},
\end{align*}
\item[(ii)]  when  $u<K$ and  $s'_k =[K]$    for every $k \in [K]$, in which case 
\begin{align*}
T^*_{fwer} &= T_0 \wedge T_1 \wedge T_2 \wedge T_3 \wedge T_4,
\end{align*}
where
\begin{align*}
T_{0} &= \inf\lt\{n \in \bN : \; \Lambda_{(1)}(n) \leq 1/A\rt\},\\
T_{1} &= \inf\lt\{n \in \bN : \; \Lambda_{(1)}(n) \geq B, \quad \Lambda_{(2)}(n) \leq \min\{1/A, \; \Lambda_{(1)}(n) / B\} \rt\},\\
T_{2} &= \inf\lt\{n \in \bN : \;  p(n) \in (1, u), \quad \Lambda_{(1)}(n) \geq B,  \quad \Lambda_k(n) \notin \lt(1/A, B\rt) \quad \forall  \; k \in [K] \rt\},\\
T_{3} &= \inf\lt\{n \in \bN : \; p(n) = u,  \quad \Lambda_{(u)}(n) \geq \max \lt\{B, \;  A \, \Lambda_{(u+1)}(n) \rt\}\rt\},\\ %\quad \textcolor{red}{\Lambda_{(1)}(n) \geq B},
T_{4} &= \inf\lt\{n \in \bN :\;  p(n) > u, \quad  \Lambda_{(1)}(n) \geq B, \quad \Lambda_{(u)}(n) \geq \Lambda_{(u + 1)}(n) \;  \max\{A, B\}   \rt\}.
\end{align*}
\end{enumerate}
\end{theorem}

\begin{proof}
The proof can be found in Appendix \ref{pf:th:indep_fwer}.
\end{proof}

\begin{remark}
When  $u=K$, the stopping rule in Theorem \ref{th:indep_fwer}  coincides with the    ``Intersection rule'',  proposed in \cite{de2012sequential, de2012step}, whose asymptotic optimality was established in  \cite{song2017asymptotically, song_fell2019}. 
 When $u<K$, on the other hand,    the stopping rule  in Theorem \ref{th:indep_fwer}    differs from the ``gap-intersection'' rule,  which was proposed  and analyzed in the iid setup in \cite{song2017asymptotically}, in that it is defined up to 2, instead of 4, distinct thresholds, while being the  minimum of 5, instead of 3, stopping times.
\end{remark}

\subsection{The general case}  \label{subsec:appl_anomalous_depen} 
In this subsection we do not assume that the independence assumption \eqref{product} holds,  and we do not make any assumption regarding the form of the prior information, $\Psi$.  We only assume that $\emptyset \in \Psi$, so that the detection problem is relevant, and  focus on the  joint detection and isolation problem, restricting our attention to the test  $\chi^*$, introduced in Subsection \ref{subsubsec: joint_test}. 

As mentioned earlier, if the subsystems are large enough, $\chi^*$ is asymptotically  optimal, but may not always  be easy to implement in this general setup. For this reason,  we focus on  the case that  the subsystems  are selected   as small as possible, i.e.,  $s_k =s'_k = \{k\}$ for every $k \in [K]$, so that the  implementation of  $\chi^*$  requires only the ordering of the local statistics, introduced in  \eqref{LLR}. For this  version of  $\chi^*$,  we upper bound the  asymptotic relative efficiency, 
%In order to analyze the asymptotic performance of $\chi^*$ in general, we introduce the following definition which is extensively used in Section \ref{sec:appl_dep} and \ref{sec:appl_depen_struc}. \begin{definition}We define the asymptotic relative efficiency of $\chi^*$ to be the following
\begin{equation} \label{ARE}
{\sf{ARE}}_{\sfP}[\chi^*] := \limsup \; \frac{\sfE_{\sfP}[T^*]}{\inf \lt\{ \sfE_{\sfP}[T]: \; 
(T, D) \in \cC_{\Psi}(\alpha, \beta,\gamma, \delta) \rt\}},
\end{equation}
%as $\alpha, \beta, \gamma, \delta \to 0$. %or $\alpha, \beta \to 0$, while $\gamma$ and $\delta$ are fixed.\end{definition}
as $\alpha, \beta, \gamma, \delta$ go to 0 at various relative rates,  for different  $\sfP \in \cP_{\Psi}$ and $\Psi$, and  establish sufficient conditions for its asymptotic optimality. Moreover, we specialize the above  results to the problem of detecting and identifying non-zero Gaussian means. \\

%By definition, this quantity is greater than one and it measures the performance loss implied by $\chi$ with respect to the optimal one.

%\begin{definition}%[Isolated pair]
%A data source is called \textit{isolated} \textcolor{red}{under $\sfP$} if it is independent of the rest of the sources  \textcolor{red}{under $\sfP$}.
%\end{definition}
%The following theorem provides the condition for asymptotic optimality of the proposed rules in the pure isolation regime for false positive when the true distribution belongs to $\cP(\Psi) \setminus \cP_0$.

We start with the asymptotic regime  where $\gamma$ (resp. $\delta$) goes to 0 much faster than $\alpha$ and $\delta$ (resp. $\gamma$). If, in particular, it is assumed  that there can be at most $u$ signals, i.e.,  $\Psi = \Psi_{0, u}$, we show that  the computationally simplest version of $\chi^*$ is asymptotically optimal  when the  true number of signals is larger than $1$ (resp. smaller than $u$), as long as  one of the  most difficult to isolate sources
%based on its own observations o,
 is a signal (resp. non-signal) independent of the other sources.

%The above theorem states that under the scenario when false negative rate is much smaller than the other possible error rates, asymptotic optimality is guaranteed when at least one source, in which null is true and is the most difficult to isolate based only on the observations from that source, is independent of the rest, and in this case asymptotic optimality is achieved even by setting the isolation subsystems $s'_k = \{k\}$ for every $k \in [K]$. This condition is often satisfied when the dependence structure is \textit{sparse}. 

%The second result states that in testing of means under Gaussian setup asymptotic optimality is achieved even by setting $s'_k = \{k\}$ for every $k \in [K]$ and it is not affected by any dependence between the streams.

\begin{theorem}\label{w_eP1_marg}
Suppose that $\emptyset \in \Psi$, $\sfP \in \cP_\Psi \setminus \cH_0$ and $s'_k = \{k\}$ for every $k \in [K]$. 

\begin{enumerate}
\item[(i)] If $\alpha, \beta, \gamma, \delta \to 0$ such that $|\log \gamma| >> |\log \alpha| \vee |\log \delta|$, then 
$${\sf{ARE}}_{\sfP}[\chi^*] \leq \frac{\min_{\{k\} \in \cA(\sfP)} \cI \lt( \sfP , \cH_{\Psi, k} \rt)}{\min_{\{k\} \in \cA(\sfP)} \cI \lt( \sfP^k , \cH^k \rt)}.$$
If, also, $\Psi = \Psi_{0, u}$, $|\cA(\sfP)| > 1$, and there is a source  that is independent of the other ones under $\sfP$ and achieves
$\min_{\{k\} \in \cA(\sfP)} \; \cI(\sfP^k, \cH^k)$, then $${\sf{ARE}}_{\sfP}[\chi^*] =1.$$

\item[(ii)]
If  $\alpha, \beta, \gamma, \delta \to 0$ such that $|\log \delta| >> |\log \alpha| \vee |\log \gamma|$, then
$${\sf{ARE}}_{\sfP}[\chi^*] \leq \frac{\min_{\{k\} \notin \cA(\sfP)} \cI \lt( \sfP , \cG_{\Psi,k} \rt)}{\min_{\{k\} \notin \cA(\sfP)} \cI \lt( \sfP^k , \cG^k \rt)}.$$
If also $\Psi = \Psi_{0, u}$, $|\cA(\sfP)| < u$, and there is a source that is independent of the other ones under  $\sfP$ and  achieves  $\min_{\{k\} \notin \cA(\sfP)} \; \cI(\sfP^k, \cG^k)$, then  $${\sf{ARE}}_{\sfP}[\chi^*]=1.$$
\end{enumerate}
\end{theorem}

\begin{proof}
The proof can be found in Appendix \ref{pf:w_eP1_marg}.
\end{proof}

We next consider the asymptotic regime where  $\alpha$  is much smaller than  $\gamma$ and $\delta$, and we show that 
the computationally simplest version of $\chi^*$
 is asymptotically optimal when  (i) there is exactly one  signal that is independent of the other sources, as well as   when (ii)  $\Psi_{1,1} \subseteq \Psi$,  there is no signal, and one of the most difficult to isolate  sources
%based on the observations from that source only, 
is independent of the rest.

%and in this case asymptotic optimality is achieved even by setting the detection subsystems $s_k = \{k\}$ for every $k \in [K]$. 

%This condition is often satisfied when the dependence structure is \textit{sparse}. 

\begin{theorem}\label{w_ePminusP0_marg}
Suppose that $\emptyset \in \Psi$ and $s_k = \{k\}$ for every $k \in [K]$. 
\begin{enumerate}
\item[(i)] If   $\sfP \in \cP_{\Psi} \setminus \cH_0$ and $\alpha, \beta, \gamma, \delta \to 0$  such that 
$|\log \alpha| >> |\log \gamma| \vee |\log \delta|,$ then 
$${\sf{ARE}}_{\sfP}[\chi^*] \leq \min_{\{k\} \in \cA(\sfP)}  \left\{ \frac{\cI(\sfP, \cH_0) }{\cI(\sfP^k, \cH^k)} \right\}.$$
If, also, there is  exactly one signal that is independent of the other sources under $\sfP$, then 
$${\sf{ARE}}_{\sfP}[\chi^*]=1.$$ 
\item[(ii)] If $\sfP \in \cH_0$ and $\alpha, \beta \to 0$, while  $\gamma$ and $\delta$ are either  fixed or go to $0$,  then 
$${\sf{ARE}}_{\sfP}[\chi^*] \leq \frac{\min_{k \in [K]} \cI \lt( \sfP , \cG_{\Psi,k} \rt)}{\min_{k \in [K]} \cI \lt( \sfP^k , \cG^k \rt)}.$$
If, also, $\Psi_{1,1} \subseteq \Psi$, and there is a source that is independent of the other ones under $\sfP$ and achieves 
$ \min_{k \in [K]} \; \cI(\sfP^k, \cG^k)$, then $${\sf{ARE}}_{\sfP}[\chi^*]=1.$$
\end{enumerate}
\end{theorem}

\begin{proof}
The proof can be found in Appendix \ref{pf:w_ePminusP0_marg}.
\end{proof}

\subsection{Detection and isolation of non-zero Gaussian means}
We next specialize  Theorems \ref{w_eP1_marg} and \ref{w_ePminusP0_marg} to the  testing problem in Subsection \ref{test_mean}. % when  $\Psi = \Psi_{0, u}$. 
%Then the following result establishes upper bounds on ${\sf{ARE}}$ of $\chi^*$ under the asymptotic regimes considered in
\begin{corollary}\label{cor:dep_source_gauss}
Consider the setup of  Subsection \ref{test_mean}, and suppose that $\Psi = \Psi_{0, u}$ and
$s_k = s'_k = \{k\}$ for every $k \in [K]$.  Fix $\sfP \in \cP_\Psi$ and let $R$ be the correlation matrix and $\sigma_k^2$ (resp. $\mu_k$) be the variance (resp. mean) of $X_1^k$ for every $k \in [K]$ (resp. $\{k\} \in \cA(\sfP)$) under $\sfP$. 
%where $\mu_k$ is the mean  under $\sfP$.  
\begin{itemize}
\item[(i)] If $\alpha, \beta, \gamma, \delta \to 0$ such that $|\log \gamma| >> |\log \alpha| \vee |\log \delta|$ and $|\cA(\sfP)| > 1$, then
$${\sf{ARE}}_{\sfP}[\chi^*] \leq \min\lt\{R^{-1}_{jj} : \{j\} \in \argmin_{\{k\} \in \cA(\sfP)} \frac{\mu_k^2}{\sigma_k^2}\rt\}.
$$
Moreover,   ${\sf{ARE}}_{\sfP}[\chi^*] =1$ when there is a source that is independent of others and achieves 
$$\min_{\{k\} \in \cA(\sfP)} ({\mu_k^2}/{\sigma_k^2}).$$
\item[(ii)] If $\alpha, \beta, \gamma, \delta \to 0$ such that $|\log \delta| >> |\log \alpha| \vee |\log \gamma|$ and $1 \leq |\cA(\sfP)| < u$, then
$${\sf{ARE}}_{\sfP}[\chi^*] \leq \min\lt\{R^{-1}_{jj} : \{j\} \in \argmin_{\{k\} \notin \cA(\sfP)} \min_{\theta \in \cM_k \setminus \{0\}} \frac{\theta^2}{\sigma_k^2}\rt\}.$$ 
Moreover,   ${\sf{ARE}}_{\sfP}[\chi^*] =1$ when there is a source that is independent of others and achieves 
$$\min_{\{k\} \notin \cA(\sfP)} \min_{\theta \in \cM_k \setminus \{0\}} ({\theta^2}/{\sigma_k^2}).$$
\item[(iii)] If   $\sfP \in \cP_{\Psi} \setminus \cH_0$ and $\alpha, \beta, \gamma, \delta \to 0$  such that 
$|\log \alpha| >> |\log \gamma| \vee |\log \delta|,$ then 
$${\sf{ARE}}_{\sfP}[\chi^*] \leq \sum_{\{i\}, \{j\} \in \cA(\sfP)} \frac{\mu_iR^{-1}_{ij}\mu_j}{\sigma_i\sigma_j}\bigg/\max_{\{k\} \in \cA(\sfP)} \frac{\mu_k^2}{\sigma_k^2}.$$
\item[(iv)] If $\sfP \in \cH_0$ and $\alpha, \beta \to 0$, while  $\gamma$ and $\delta$ are either  fixed or go to $0$, then
$${\sf{ARE}}_{\sfP}[\chi^*] \leq \min\lt\{R^{-1}_{jj} : j \in \argmin_{k \in [K]} \min_{\theta \in \cM_k \setminus \{0\}} \frac{\theta^2}{\sigma_k^2}\rt\}.$$
Moreover,   ${\sf{ARE}}_{\sfP}[\chi^*] =1$ when there is a source that is independent of others and achieves 
$$\min_{k \in [K]} \min_{\theta \in \cM_k \setminus \{0\}} ({\theta^2}/{\sigma_k^2}).$$
\end{itemize}
\end{corollary}
\begin{proof}
The proof of the upper bounds in (i)-(iv) can be found in Appendix \ref{pf:cor:dep_source_gauss}. To show when these upper bounds are equal to 1 and, as a result,  ${\sf{ARE}}_{\sfP}[\chi^*] = 1$,  we recall that no diagonal entry of the inverse of correlation matrix is smaller than $1$, i.e., $R^{-1}_{jj} \geq 1$ for every $j \in [K]$, and that the equality holds if and only if source $j$ is independent of the other sources. In fact, this condition for asymptotic optimality 
is more generally stated, for example, in Theorem \ref{w_eP1_marg}(i).
\end{proof}

\begin{comment}
On the other hand, as $\alpha, \beta, \gamma, \delta \to 0$  such that 
$|\log \alpha| >> |\log \gamma| \vee |\log \delta|,$
$${\sf{ARE}}_{\sfP}[\chi^*] \leq \frac{1}{\Sigma_{kk}^2}\sum_{\{i\}, \{j\} \in \cA(\sfP)}\Sigma^{-1}_{ij}.$$
\end{comment}
 
\section{Detection and isolation of a  dependence structure} \label{sec:appl_depen_struc}
We next  apply the general theory,  developed in  Sections  \ref{sec:prob_form}-\ref{sec:asymp_opt}, to the problem of  detecting and/or identifying  dependent  sources.  To be  specific,  throughout this section the units are pairs of data sources, i.e., 
$$\cK = \big\{\{i, j\} : 1 \leq i < j \leq K\big\},$$
and  the components of $X^e$   are  independent (resp. dependent) under $\sfP$ when   $\sfP^e$ belongs to $\cH^e$ (resp. $\cG^e$)
for  every  $e \in \cK$ and $\sfP \in \cP_\Psi$.  We  also assume that conditions  \eqref{max_series} and \eqref{comp_conv_onesided} are satisfied and,  as a result, all asymptotic optimality results in Section  \ref{sec:asymp_opt}  hold when the detection and isolation subsystems are equal to $[K]$.   The next theorem provides a setup where this choice of subsystems leads to a test whose implementation has relatively low complexity.  
For this result, as well as some others in this section, we  also  assume that, under any $\sfP \in \cP_\Psi$,
\begin{align}\label{indep_subs}
& \;\; s_1, s_2 \subseteq [K]  \;\; \text{are disjoint and   independent} \;\; & \Leftrightarrow  \quad  \{i, j\} \notin \cA(\sfP) \quad \forall \; i \in s_1,  j \in s_2.
\end{align}
This assumption is satisfied, for example,  in the Gaussian setup of  Subsection \ref{test_corr}.  % \textcolor{red}{This theorem uses the notations for the order statistics of $\Lambda_e(n)$, which was defined in the case of independent sources. So what should we do here? Introduce here}

\begin{comment}
\begin{theorem}\label{pure_det_dep} 
Suppose that \eqref{indep_subs} holds and that  $\Psi$ is either 
 $\Psi_{clus}$ or $ \Psi_{dis}$. If   $s_e = [K]$ for every $e \in \cK$,  then
\begin{align*}
T_0 &= \inf\{n \in \bN : \Lambda_{(1)}(n) \leq 1/C\}.
\end{align*}
When in particular $\Psi = \Psi_{dis}$,  then
\begin{align*}
T_{det} = \inf \{n \in \bN : \; \max_{\mathcal{B} \subseteq 
\{i_1(n), \dots i_{p(n)}(n)  \}  , \; \mathcal{B} \in \Psi_{dis} }  \prod_{e \in \mathcal{B}} \Lambda_{e}(n) \geq D \}.
\end{align*}
\end{theorem}
\end{comment}

\begin{theorem}\label{pure_det_dep} 
Suppose that \eqref{indep_subs} holds and $\Psi = \Psi_{dis}$. If $s_e = [K]$ for every $e \in \cK$,  then
\begin{align*}
T_0 &= \inf\{n \in \bN : \Lambda_{(1)}(n) \leq 1/C\},\\
T_{det} &= \inf \{n \in \bN : \; \max_{\mathcal{B} \subseteq 
\{i_1(n), \dots i_{p(n)}(n)  \}  , \; \mathcal{B} \in \Psi}  \prod_{e \in \mathcal{B}} \Lambda_{e}(n) \geq D \},
\end{align*}
where $T_0$ and $T_{det}$ are defined as in Subsection \ref{subsubsec: pure_detection_test}.
\end{theorem}

\begin{proof}
The proof can be found in Appendix \ref{pf:pure_det_dep}.
\end{proof}

In general,  however, setting $s_e=s'_e=[K]$  for every $e \in \cK$ may not  lead to computationally simple  tests.  For this reason,  in the remainder of this section we  obtain sufficient conditions  for asymptotic optimality when the subsystems are smaller than  $[K]$, and especially when  $s_e =s'_e =e$ for every $e \in \cK$.  When these conditions are not satisfied, we upper bound the asymptotic relative efficiency  of  $\chi^*$, in its simplest version.  
Moreover, we focus on the joint detection and isolation problem,  thus, we  assume that  $\emptyset \in \Psi$,  and restrict our attention to the test $\chi^*$, introduced in Subsection \ref{subsubsec: joint_test}. We also note that,  for every $\sfP \in \cP_\Psi$,
there is a disjoint partition of $[K]$,   $v_0, v_1, \dots, v_{L(\sfP)}$, such that 
$$\sfP = \sfP^{v_0} \otimes \sfP^{v_1} \otimes \dots \otimes \sfP^{v_{L(\sfP)}},$$
%Next, we denote by $\{V_1, \dots, V_{L(\sfP)} \}$ the set of independent subsystems each of which consist of dependent streams. 
where   $v_0$ is the subset of $[K]$ that consists of the independent sources under $\sfP$. This decomposition is used in the statement of the main results of this section.

%, %\textcolor{red}{i.e., $\sfP^{v_0} = \sfP_0^{v_0}$ for some $\sfP_0 \in \cH_0$ : this relation holds when \eqref{indep_subs} holds}, with the understanding that $v_0 = \emptyset$ when there are not any independent sources under $\sfP$.  %for any $e \in \cA(\sfP)$ there exists $l \in [L(\sfP)]$ such that $e \subseteq V_l$. 

\begin{remark}
Suppose that \eqref{indep_subs} holds. If $\Psi=\Psi_{clus}$, then $v_1, \dots, v_{L(\sfP)}$ are clusters themselves.  If $\Psi=\Psi_{dis}$, then $L(\sfP) = |\cA(\sfP)|$ and $|v_l| = 2$ for every $l \in [L(\sfP)]$.
\end{remark}

\subsection{General results}
 We start with the asymptotic regime where $\gamma$ (resp. $\delta$) goes to 0 much faster than $\alpha$ and $\delta$ (resp. $\gamma$). In this setup, we show that  $\chi^*$ is asymptotically optimal  when the isolation subsystem that corresponds to each dependent (resp. independent) pair contains the subsets in the partition that include that  pair. Moreover, we show that  $\chi^*$ is asymptotically optimal even with $s'_e = e$ for every $e \in \cK$, whenever one of the  most difficult to isolate,   dependent (resp. independent) pairs  is  independent of all  other data sources. 
%\end{enumerate}

\begin{theorem}\label{w_eP1}
Suppose that $\emptyset \in \Psi$,  $\sfP \in \cP_\Psi \setminus \cH_0$, and  $\alpha, \beta, \gamma, \delta \to 0$ such that $|\log \gamma| >> |\log \alpha| \vee |\log \delta|$.  
\begin{enumerate}
\item[(i)]  If, for every $e \in \cA(\sfP)$, there is an  $ l \in [L(\sfP)]$ so that  $s'_e \supseteq v_l  \supseteq e$, then \  ${\sf{ARE}}_{\sfP}[\chi^*]=1$.

%  when  $L(\sfP) > 1$.%\begin{equation*}s'_e \supseteq v_l \;\; \text{where} \;\; v_l  \supseteq e \;\; \text{for some?} \;\; l \in [L(\sfP)]. \end{equation*}

\item[(ii)]  If   $s'_e = e$ for every $e \in \cK$, then 
$${\sf{ARE}}_{\sfP}[\chi^*] \leq \frac{\min_{e \in \cA(\sfP)} \cI \lt( \sfP , \cH_{\Psi,e} \rt)}{\min_{e \in \cA(\sfP)} \cI \lt( \sfP^e , \cH^e \rt)}.$$
Moreover,  ${\sf{ARE}}_{\sfP}[\chi^*]=1$   when one of the following set of conditions is satisfied:
\begin{itemize}
\item  \eqref{indep_subs} holds, \,  $\Psi = \Psi_{dis}$, \, $|\cA(\sfP)| > 1$,
\item   $\Psi $ is either $\Psi_{dis}$, or $\Psi_{clus}$, or the powerset of $\cK$, $L(\sfP) > 1$, and there is a pair  of data sources that is independent of all other data sources under $\sfP$ and  achieves 
$$\min_{e \in \cA(\sfP)} \; \cI(\sfP^e, \cH^e).$$
\end{itemize}
\end{enumerate}
\end{theorem}
\begin{proof}
The proof can be found in Appendix \ref{pf:w_eP1}.
\end{proof}

\begin{comment}
\begin{example}\label{gaussian_P1}
Consider the  setup of  Example \ref{gaussian_P0}. In that case, one can obtain
$$\min_{e \in \cA(\sfP)} \cI(\sfP, \cH_{\Psi, e}) = -\frac{1}{2}\log\lt\{\frac{1 + \rho_1(m - 2) - \rho_1^2(m-1)}{1 + \rho_1(m - 2)}\rt\} \leq -\frac{1}{2}\log(1 - \rho_1),$$
where the inequality is obtained by letting $m  \to \infty$ since it is an increasing function. Thus, the ARE in Theorem \ref{w_eP1} can be further bounded by
\begin{align*}
\frac{\log(1 - \rho_1)}{\log(1 - \rho_1^2)},
\end{align*}
as it is further illustrated in Figure \ref{ubP1}.
\end{example}
\end{comment}

\begin{comment}
The next theorem states that under the scenario when false negative rate is much smaller than the other possible error rates, asymptotic optimality is guaranteed if any of the following two conditions hold. First is, if for every independent pair, its corresponding isolation subsystem happens to contain at least all sources which are dependent to the elements of that pair, and therefore, carries the maximum possible information to isolate that pair as an independent one. The second is, when the independent pair, that is the most difficult to isolate based only on the observations from the pair, is independent of the rest and in this case asymptotic optimality is achieved even by setting the isolation subsystems $s'_e = e$ for every $e \in \cK$. Again, the second condition is often satisfied when the dependence structure is sparse.
\end{comment}

\begin{theorem}\label{w_eP2}
Suppose that $\emptyset \in \Psi$,  $\sfP \in \cP_\Psi \setminus \cH_0$, and  $\alpha, \beta, \gamma, \delta \to 0$ such that $|\log \delta| >> |\log \alpha| \vee |\log \gamma|$.
\begin{enumerate}
\item[(i)] ${\sf{ARE}}_{\sfP}[\chi^*]=1$  when, for every $e = \{i, j\} \notin \cA(\sfP)$, % we have
\begin{equation*}
s'_e \supseteq
\begin{cases}
v_l \quad &\text{if} \quad v_l \supseteq e \quad \text{for some} \quad l \in [L(\sfP)],\\
v_{l} \cup v_{l'} \quad &\text{if} \quad i \in v_l \quad \text{and} \quad j \in v_{l'} \quad \text{for some} \quad l \neq l',  \quad l, l' \in [L(\sfP)],\\
e \cup v_l \quad &\text{if} \quad i \;\; \text{or} \;\; j \in v_0 \quad \text{and} \quad e \setminus v_0 \in v_l \quad \text{for some} \quad l \in [L(\sfP)],\\
e \quad &\text{if} \quad v_0 \supseteq e.
\end{cases}
\end{equation*}

\item[(ii)] If $s'_e = e$ for every $e \in \cK$, then 
$${\sf{ARE}}_{\sfP}[\chi^*] \leq \frac{\min_{e \notin \cA(\sfP)} \cI \lt( \sfP , \cG_{\Psi,e} \rt)}{\min_{e \notin \cA(\sfP)} \cI \lt( \sfP^e , \cG^e \rt)}.$$
Moreover,  ${\sf{ARE}}_{\sfP}[\chi^*]=1$ when $\Psi$ is either $\Psi_{dis}$, or $\Psi_{clus}$, or the powerset of $\cK$, and there is a pair of data sources that is independent of all other data sources under $\sfP$ and  achieves 
$$\min_{e \notin \cA(\sfP)} \; \cI(\sfP^e, \cG^e).$$

\end{enumerate}
\end{theorem}

\begin{proof}
The proof can be found in Appendix \ref{pf:w_eP2}.
\end{proof}

We next consider the asymptotic regime where $\alpha$ goes to 0 much faster than $\gamma$ and $\delta$. In this setup, we show that   $\chi^*$ is asymptotically optimal 
when the detection subsystem that corresponds to some dependent pair contains all dependent sources. Moreover, we show that  $\chi^*$ is asymptotically optimal even with $s_e = e$ for every $e \in \cK$, whenever there is exactly one dependent pair that is independent of all other data sources. %\textcolor{red}{which are also independent under $\sfP$}. \end{enumerate}

\begin{theorem}\label{w_eP0}
Suppose that $\emptyset \in \Psi$,  $\sfP \in \cP_\Psi \setminus \cH_0$, and  $\alpha, \beta, \gamma, \delta \to 0$ such that $|\log \alpha| >> |\log \gamma| \vee |\log \delta|$.
\begin{enumerate}
\item[(i)] If   there is an $e \in \cA(\sfP)$ such that   $s_{e} \supseteq v_1 \cup \ldots \cup v_{L(\sfP)}$, then  \; ${\sf{ARE}}_{\sfP}[\chi^*]=1$.
 %Condition \eqref{condP0} is satisfied and all conclusions in Theorem \ref{P0} hold when there is an  $e \in \cA(\sfP)$ such that 
%, e.g., when  condition \eqref{indep_subs} holds, $|\cA(\sfP)| = 1$ and  $s_e = e$ for every $e \in \cK$.
\item[(ii)]  If $s_e = e$ for every $e \in \cK$, then  %Then as $\alpha, \beta, \gamma, \delta \to 0$  such that 
%$|\log \alpha| >> |\log \gamma| \vee |\log \delta|,$
$${\sf{ARE}}_{\sfP}[\chi^*] \leq \min_{e \in \cA(\sfP)} \left\{ \frac{\cI(\sfP, \cH_0)}{\cI(\sfP^e, \cH^e)} \right\}.$$
Moreover,  ${\sf{ARE}}_{\sfP}[\chi^*] = 1$  when there is exactly one dependent pair that is independent of all other data sources under $\sfP$, for example, when
\eqref{indep_subs} holds and $|\cA(\sfP)| = 1$.
\end{enumerate}
\end{theorem}
\begin{proof}
The proof can be found in Appendix \ref{pf:w_eP0}.
\end{proof}

We next consider the case that there are no signals and we show that $\chi^*$ is asymptotically  optimal, with $s_e = e$ for every $e \in \cK$, if it is \textit{a priori} possible that there exists exactly one dependent pair of data sources, i.e.,  $\Psi_{1,1} \subseteq \Psi$, and one of the  most difficult  pairs to detect is  independent of all other data sources.

\begin{theorem}\label{depcase_P0}
Suppose that $\emptyset \in \Psi$, $\sfP \in \cH_0$, and $\alpha, \beta \to 0$ while $\gamma$ and $\delta$ are either fixed or go to $0$. 
If $s_e = e$ for every $e \in \cK$ then
$${\sf{ARE}}_{\sfP}[\chi^*] \leq \frac{\min_{e \in \cK} \cI \lt( \sfP , \cG_{\Psi,e} \rt)}{\min_{e \in \cK} \cI \lt( \sfP^e , \cG^e \rt)}.$$
If, also, $\Psi_{1,1} \subseteq \Psi$, and there is a pair of data sources that is independent of all other data sources under $\sfP$ and achieves 
$\min_{e \in \cK} \; \cI(\sfP^e, \cG^e)$, e.g., when \eqref{indep_subs} holds, then \,${\sf{ARE}}_{\sfP}[\chi^*]=1.$

%, $\emptyset \in \Psi$, $\Psi \supseteq \Psi_{1,1}$ and $\sfP \in \cH_0$. If  $s_e=e$ for every $e \in \cK$,  then 
%${\sf{ARE}}_{\sfP}[\chi^*] = 1$ as $\alpha, \beta, \gamma, \delta \to 0$.
%condition \eqref{condPminusP0} is satisfied and  all  conclusions of Theorem \ref{PminusP0} hold. 
\end{theorem}
\begin{proof}
The proof can be found in Appendix \ref{pf:depcase_P0}.
\end{proof}

Finally, we  consider the case that it is a priori known that %there is at least one signal and 
all signals are disjoint pairs, i.e., $\Psi = \Psi_{dis}$, and we establish the asymptotic optimality of the 
% the error  control of interest is the FWER and we show that asymptotic optimality is achieved by the 
\textit{Intersection rule}, defined  in \eqref{intersection_rule}, %, i.e., with $s_e = s'_e = e$ for every $e \in \cK$ 
whenever \eqref{indep_subs} holds and one of the  most difficult to isolate,  independent pairs is  independent of all other data sources.

\begin{comment}
The next theorem, we provide the asymptotic optimality result of the practical rule when $\Psi = \Psi_{dis}$ and we are interested to control the FWER. The above theorem states that when the prior information allows the dependent sources to exist only as disjoint pairs, intersection rule is asymptotically optimal if the independent pair, that is the most difficult to isolate based only on the observations from the pair, is independent of the rest. 
\end{comment}

\begin{theorem}\label{dep_case_psi_di_prac_rule}
Suppose  that  $\Psi = \Psi_{dis}$, $\sfP \in \cP_\Psi$ and $s_e = s'_e = e$ for every $e \in \cK$. If \eqref{indep_subs} holds, and  there is a pair of data sources that is independent of all other data sources under $\sfP$ and achieves  $\min_{e \notin \cA(\sfP)} \; \cI(\sfP^e, \cG^e)$, then 
$\chi^*_{fwer}$ is asymptotically optimal as
 $\gamma, \delta \to 0$. 
\end{theorem}

\begin{proof}
The proof can be found in Appendix \ref{pf:dep_case_psi_di_prac_rule}.
\end{proof}

\subsection{Detection and isolation of a Gaussian dependence structure}
We next specialize the  results of this section to the setup  of Subsection \ref{test_corr}, when it is  assumed  a priori that a subset of data sources follow an equicorrelated multivariate Gaussian distribution with  common positive correlation $\rho$, and are independent of the rest of the sources.  Specifically, we show that when there is only one such  subset, of size $m$,   the other  sources are independent, and $\gamma$ (resp. $\delta$) goes to $0$ faster than $\alpha$ and $\delta$ (resp. $\gamma$),   then  the asymptotic relative efficiency of $\chi^*$,  defined as in \eqref{ARE}, with  $s_e = s'_e = e$ for every $e \in \cK$,
 is bounded above by a quantity that depends only on $\rho$ and $m$.  This upper bound approaches 1 as   $\rho$ increases, while  $m$ is fixed, and  some number strictly greater than 1 as $m$ increases, while $\rho$ is fixed. We illustrate this upper bound in Figure \ref{ubP1}. 

%To be more specific, %Specifically, %To incorporate this structure, %we fix $\rho \in \cR$ such that $\rho > 0$ and we set $\Psi = \Psi_{clus}$ and  $\cR_{ij} = \{0, \rho\}$ for every $e = \{i, j\} \in \cK$, and  

\begin{corollary}\label{cor:dep_struc}
Consider the setup of  Subsection \ref{test_corr} with   $\cR_{ij} = \{0, \rho\}$ for every $i, j \in [K] \; \text{and}  \; i < j$. Suppose that $\Psi = \Psi_{clus}$ and $s_e = s'_e = e$ for every $e \in \cK$. Let  $\sfP \in \cP_\Psi$  such that $L(\sfP) = 1$ and $|v_1| = m$.
%then %\begin{align*} %\cI(\sfP, \cH_0) &\leq -(m-1) \log \sqrt{1-\rho} - \log \sqrt{1 + (m - 1)\rho_1}.  %\end{align*} %Then 
\begin{enumerate}
\item[(i)] If {$2 < m \leq K$} and $\alpha, \beta, \gamma, \delta \to 0$ such that $|\log \gamma| >> |\log \alpha| \vee |\log \delta|$, then
\begin{align*}
{\sf{ARE}}_{\sfP}[\chi^*] &\leq  \log\lt( 1- \rho \frac{ \rho(m-1)}{1 + \rho(m - 2)}\rt) / \log(1 - \rho^2)  \leq \frac{\log(1 - \rho)}{\log(1 - \rho^2)}.
\end{align*}

\item[(ii)]  If {$2 \leq m < K$} and $\alpha, \beta, \gamma, \delta \to 0$ such that $|\log \delta| >> |\log \alpha| \vee |\log \gamma|$, then
\begin{align*}
{\sf{ARE}}_{\sfP}[\chi^*] 
&\leq  \lt\{ \log\sqrt{1 - \frac{\rho^2m}{1 + \rho(m - 1)} } + \frac{\rho^2 m}{(1 + \rho m)(1 - \rho)} \rt\} \bigg/\lt\{\log\sqrt{1-\rho^2} + \frac{\rho^2}{1 - \rho^2}\rt\}\\
&\leq \lt( \log(1 - \rho) + \frac{2\rho}{1 - \rho}\rt)\bigg/
 \lt( \log(1 - \rho^2) + \frac{2\rho^2}{1 - \rho^2}\rt).
\end{align*}

\item[(iii)]  If {$2 \leq m \leq K$}  and $\alpha, \beta, \gamma, \delta \to 0$  such that  $|\log \alpha| >> |\log \gamma| \vee |\log \delta|$,  then
$${\sf{ARE}}_{\sfP}[\chi^*] \leq \frac{(m-1)\log (1-\rho) + \log (1 + (m - 1)\rho)}{\log (1-\rho^2)}.$$
\end{enumerate}
\end{corollary}
\begin{proof}
The proof can be found in Appendix \ref{pf:cor:dep_struc}.
\end{proof}

%The above result shows that under this setup, if $\gamma$ (resp. $\delta$) goes to $0$ faster than $\alpha$ and $\delta$ (resp. $\gamma$) and a subset of $m$ data sources follows equicorrelated multivariate Gaussian with common positive correlation $\rho$ whereas the rest of the sources are independent, then even with considering $s_e = s'_e = e$ for every $e \in \cK$, the asymptotic relative efficiency of $\chi^*$ is bounded above by a quantity \textit{independent of} both $m$ and $K$, and it is closer to one when $\rho$ is larger. Although it shifts further from one with increase in $m$, the extent of shift is not significant when $m$ is large.We illustrate the first two upper bounds in Figure \ref{ubP1}. 

\begin{figure}[H]
\centering
\includegraphics[scale = 0.32]{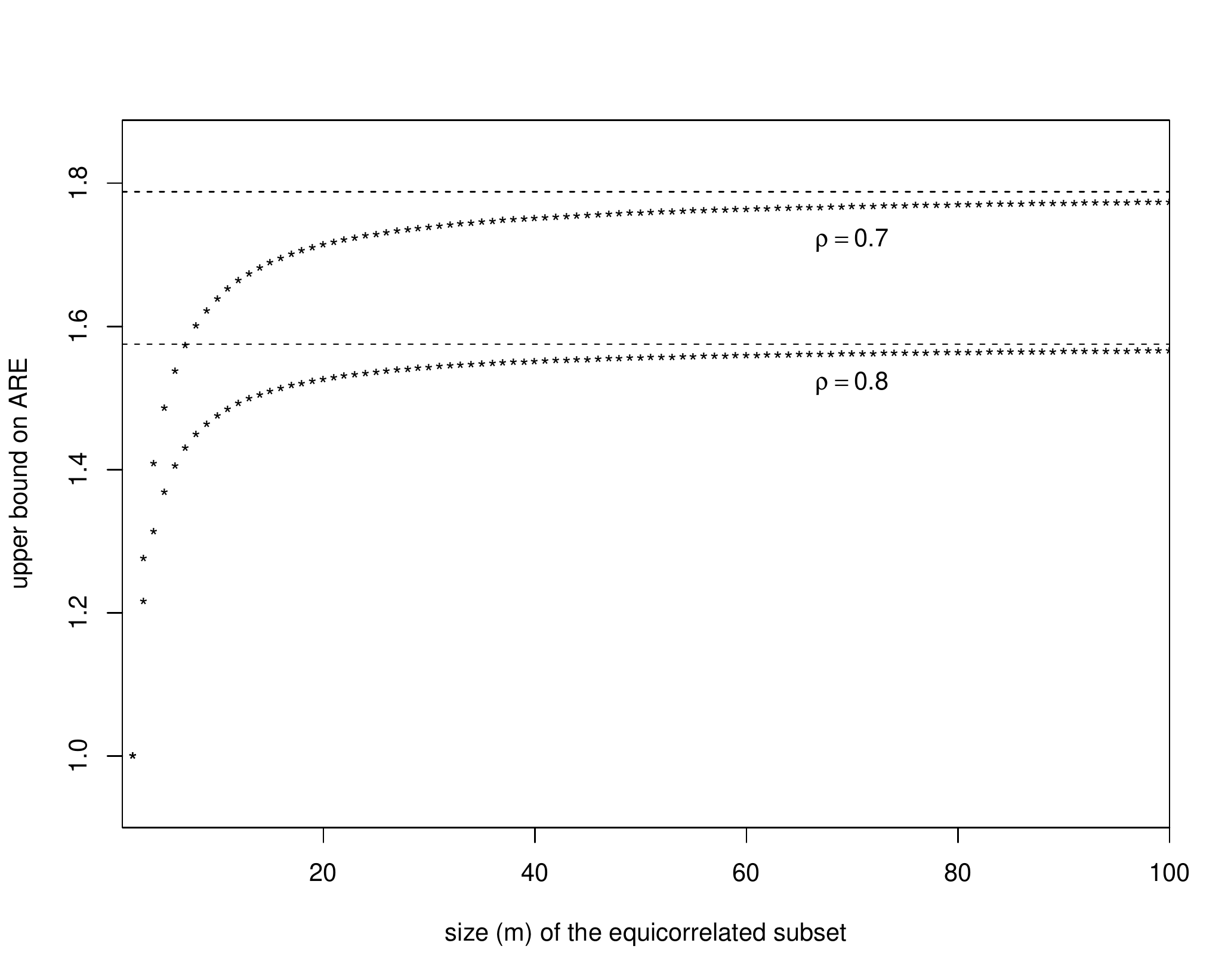} \includegraphics[scale = 0.32]{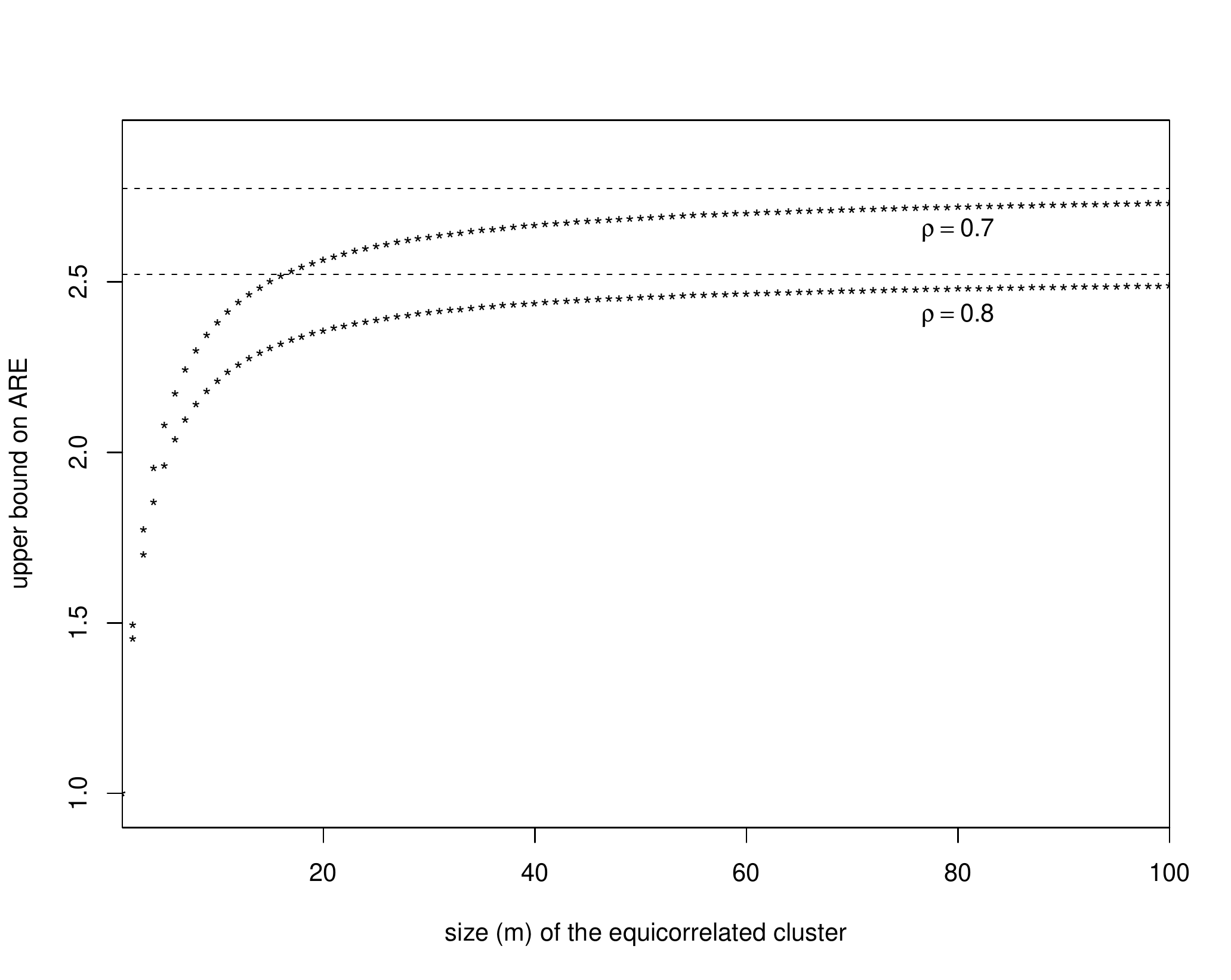}
\caption{The  graph in the left/right refers to the upper bounds of part (i)/(ii)  of Corollary \ref{cor:dep_struc}(i).  In both of them, the curves represent the bounds that depend on both $m$ and $\rho$, whereas the straight lines  the bounds that depend only on $\rho$.}
\label{ubP1}
\end{figure}

\section{Simulation Study}\label{sec:sim_stud}
In this section  we present the results of two  simulation studies for the problem of  joint detection and isolation  of a dependence structure  in Section \ref{sec:appl_depen_struc}. In both of them
 we  consider the Gaussian setup of Subsection \ref{test_corr},  fix the marginal distributions of the sources to be standard Gaussian, i.e.,  we set $\cM_k = \{0\}$, $\cV_k = \{1\}$ for every $k \in [K]$, and we consider a  \textit{two-sided} testing problem for each correlation coefficient, i.e.,  we set  $\cR_{ij}=\{0, \pm \rho\}$ for each $i, j \in [K]$ with $i < j$, where $\rho = 0.7$.   
 
 In the first study, there are  $K = 7$ data sources and  two clusters,  $\{4, 5\}$ and   $\{1, 2, 3\}$.  In particular, $\{4, 5\}$  is a dependent pair, with correlation $\rho$,  independent of all other sources, $\{6, 7\}$ is an independent pair,  also independent of all other sources, and the covariance matrix that corresponds to  $\{1, 2, 3\}$ is
\begin{equation}\label{sig_mat}
\begin{bmatrix}
1 & \rho &-\rho\\
\rho & 1 &-\rho\\
-\rho &-\rho &1
\end{bmatrix}.
\end{equation}
In the second study there are $K = 3$  data sources, whose  covariance matrix  is \eqref{sig_mat}. 

In both studies we  consider two  versions of the  testing procedure, $\chi^*$,  in both of which  the isolation and detection subsystems coincide, i.e.,   $s_e = s'_e$ for every $e \in \cK$. In the first version, the subsystems are of size  2, i.e.,  $s_e = e$ for every $e \in \cK$. In the second, they are of size 3 (the third source in each subsystem, other than the pair itself, is chosen  in some arbitrary way).  For each of these two versions of $\chi^*$,  
we consider two further  cases depending on  the available prior information. In the first one, there is no prior information whatsoever, i.e., $\Psi$ is equal to the powerset of $\cK$, whereas in the second, it is known a priori that there is a dependence structure of cluster type, i.e.,  $\Psi=\Psi_{clus}$.

 % , $s_e = s'_e$, each with size $3$, and %which implies $a_e = a$, $b_e = b$, $c_e = c$ and $d_e = d$. The values of $a, b, c$ and $d$ are specified in Table \ref{tab:coeff} when $K =3, 7$.
%This implies $a_e = 9$, $b_e = d_e = 14$ and $c_e = 1$ under $\Psi = \; \text{powerset of} \; \cK$, and $a_e = 5$, $b_e = d_e = 6$ and $c_e = 1$ under $\Psi = \Psi_{clus}$.\item[2.] 
%and thus, $a_e =  c_e = 1$ and $b_e = d_e = 2$ for any $\rho$, $\Psi$ and $K$.
%This implies $a_e =  c_e = 1$, $b_e = d_e = 2$ both under $\Psi = \; \text{powerset of} \; \cK$, and $\Psi = \Psi_{clus}$.
%\end{itemize} 

%\subsection{Design of the testing procedures}
For each of the resulting four versions of  $\chi^*$  and each  value of the target error probabilities, $\alpha, \beta, \gamma, \delta$, we select the thresholds  according to Theorem \ref{th:errcon2}(iii). 
%i.e., $$A_e = a_e|\cK|/\gamma, \quad B_e = b_e|\cK|/\delta, \quad C_e = c_e/\beta, \quad  D_e = d_e|\cK|/\alpha \quad \text{for all} \quad e \in \cK,$$ 
 When the subsystems are of size $2$,  for the constants $\{a_e, b_e, c_e, d_e\}_{e \in \cK}$   in \eqref{a,b,c,d} we have $a_e =  c_e = 1$,  $b_e = d_e = 2$  for every $e \in \cK$, regardless of $\Psi$ and $K$.  When the subsystems are of size $3$, we have $a_e = a$, $b_e = b$, $c_e = c$, $d_e = d$
for every $e \in \cK$,  where  $a, b, c, d$ are specified in Table \ref{tab:coeff}. % for  $K \in\{3, 7\}$  and for each of the two cases we consider for $\Psi$.  
Moreover, we  set 
 $\alpha = \beta = \gamma = \delta$, i.e.,  all error metrics  are of the same order of magnitude.

\begin{table}[H]
\caption{Coefficients for the thresholds}
\label{tab:coeff}
\begin{tabular}{@{}lccrr|rrrrrr@{}}
\hline
&\multicolumn{4}{c}{$K = 3$} & \multicolumn{4}{c}{$K = 7$}\\
\cline{2-9}
 $\Psi$ & $a$ & $b$ & $c$ & $d$ & $a$ & $b$ & $c$ & $d$\\[1pt]
\hline
$\text{powerset of} \; \cK$ & $8$ & $14$ & $1$ & $14$ & $9$ & $14$ & $1$ & $14$ \\
$\Psi_{clus}$ & $4$ & $6$ & $1$ & $6$ & $5$ & $6$ & $1$ & $6$ \\[1pt]
\hline
\end{tabular}
\end{table}

%In the following setups we consider the true probability distributions $\sfP$ in such a way that $\cA(\sfP) \in \Psi_{clus}$, and thus, we also consider both tests mentioned above when $\Psi = \Psi_{clus}$.

In the context of the first study,  conditions  \eqref{condP1} and \eqref{condP2} hold for all four versions of $\chi^*$, 
since the  conditions of  Theorem \ref{w_eP1}(ii) and Theorem \ref{w_eP2}(ii) are satisfied.    Furthermore, since  $\alpha = \gamma$ and the true distribution $\sfP$ satisfies (see, in particular, Lemma \ref{info_ineq1} in Appendix \ref{app:info_num})
$$\cI(\sfP, \cH_0) \geq \min_{e \in \cA(\sfP)} \cI(\sfP^e, \cH^e) = \min_{e \in \cA(\sfP)} \cI(\sfP, \cH_{\Psi, e}) ,$$ 
by Theorem \ref{P0P1P2} it follows that all four versions of $\chi^*$  are asymptotically optimal as $\alpha \to 0$.  

In the context of the second study, the versions of $\chi^*$, with subsystems of size 3 remain asymptotically optimal, as in this case we  have $s_e=s'_e=[K]$ for every $e \in \cK$. However, this is not true  for  the versions of $\chi^*$ with subsystems of size 2, which use only local data to isolate the global dependence structure. % Indeed,  for every pair of sources, test 2 uses observations only from that pair, whereas with bigger subsystems test 1 also uses observation from another source to make decision, and therefore stops faster than test 2.

%\subsection{Results}
For each study, we vary   $\alpha$  from $10^{-8}$ to $10^{-1}$, and  estimate the expected sample size of each test using $16,000$ Monte Carlo replications. 
We observe that the simplest version of $\chi^*$ performs the same regardless of the presence of prior information since the number of sources is small in both studies.
We plot the results for the first study in Figure \ref{fig:K7-rh07} and  for the second study in Figure \ref{fig:K3-rh07}. 
%which refer to the first and the second study respectively.  %In part (c) we plot the asymptotic relative efficiency of the tests, as defined in \eqref{ARE}.
The x-axis in all figures is $|\log_{10} \alpha|$. On the other hand, in each figure, the y-axis is   the expected sample size of the tests in part (a),  the ratio of the expected sample size of 
the versions without prior information with respect to that of the version with bigger subsystem as well as prior information in part (b), and the ratio of expected sample size of all tests over the first-order approximation of the corresponding optimal performance  in part (c).

From parts (a) and (b)  we see, not surprisingly,  that the expected sample size is  smaller when  using bigger subsystems and utilize prior information, and that this  effect  is more substantial  in the second study and  when $\alpha$ is small. %Furthermore, both studies show that a test, which utilizes prior information, performs better than the corresponding version which assumes there is no prior information, as can be seen in  part (a) of Figure \ref{fig:K7-rh07} and \ref{fig:K3-rh07}. 
Moreover, we see that  parts (b) and (c) 
%of Figure \ref{fig:K7-rh07} and \ref{fig:K3-rh07}  
are  consistent with the asymptotic optimality  results mentioned above, according to which all versions of $\chi^*$ are asymptotic optimal in the first study, and only those with the subsystems of size 3 in the second study.
% under both studies as well as the asymptotic optimality (resp. suboptimality) of the other tests only in the first (resp. second) study. 
%  in order to illustrate their rate of converge as $\alpha$ becomes small.
The Monte Carlo standard error in the estimation of every expectation did not exceed $2\%$ of the corresponding estimate.

\begin{comment}
\begin{table}[H]
\caption{Coefficients of thresholds for test $1$ when $K = 3, 7$}
\label{tab:coeff}
%
\begin{tabular}{@{}lccrrr|rrrrr@{}}
\hline
& &\multicolumn{4}{c}{$K = 7$} & \multicolumn{4}{c}{$K = 3$}\\
\cline{3-10}
$\rho$ & $\Psi$ & $a$ & $b$ & $c$ & $d$ & $a$ & $b$ & $c$ & $d$\\[1pt]
\hline
$0.4$ & $\text{powerset of} \; \cK$ & $9$ & $18$ & $1$ & $18$ & $8$ & $18$ & $1$ & $18$\\
          & $\Psi_{clus}$ & $5$ & $10$ & $1$ & $10$ & $4$ & $10$ & $1$ & $10$\\[1pt]
          \hline
$0.7$ & $\text{powerset of} \; \cK$ & $9$ & $14$ & $1$ & $14$ & $8$ & $14$ & $1$ & $14$\\
          & $\Psi_{clus}$ & $5$ & $6$ & $1$ & $6$ & $4$ & $6$ & $1$ & $6$\\[1pt]
\hline
\end{tabular}
%
\end{table}
\end{comment}

\begin{comment}
Note that the values of the above constants, specifically $b$ and $d$, change when $\rho$ changes from $0.4$ to $0.7$. This is because for every $e \in \cK$ when the subsystems $s_e = s'_e$ remain the same, then for different values of $\rho$ the number of possible covariance matrices of $X^{s_e}$ are also different, which is subsequently reflected in change in the value of $|\cG_{\Psi, e}^{s_e}|$.  
\end{comment}

%\subsubsection{Practical rule is asymptotically optimal}

\begin{figure}%[H]
\centering
\makebox[\linewidth]{
\subfloat[]{\label{fig:ESS-K7-rh07} 
\includegraphics[width=0.32\linewidth]{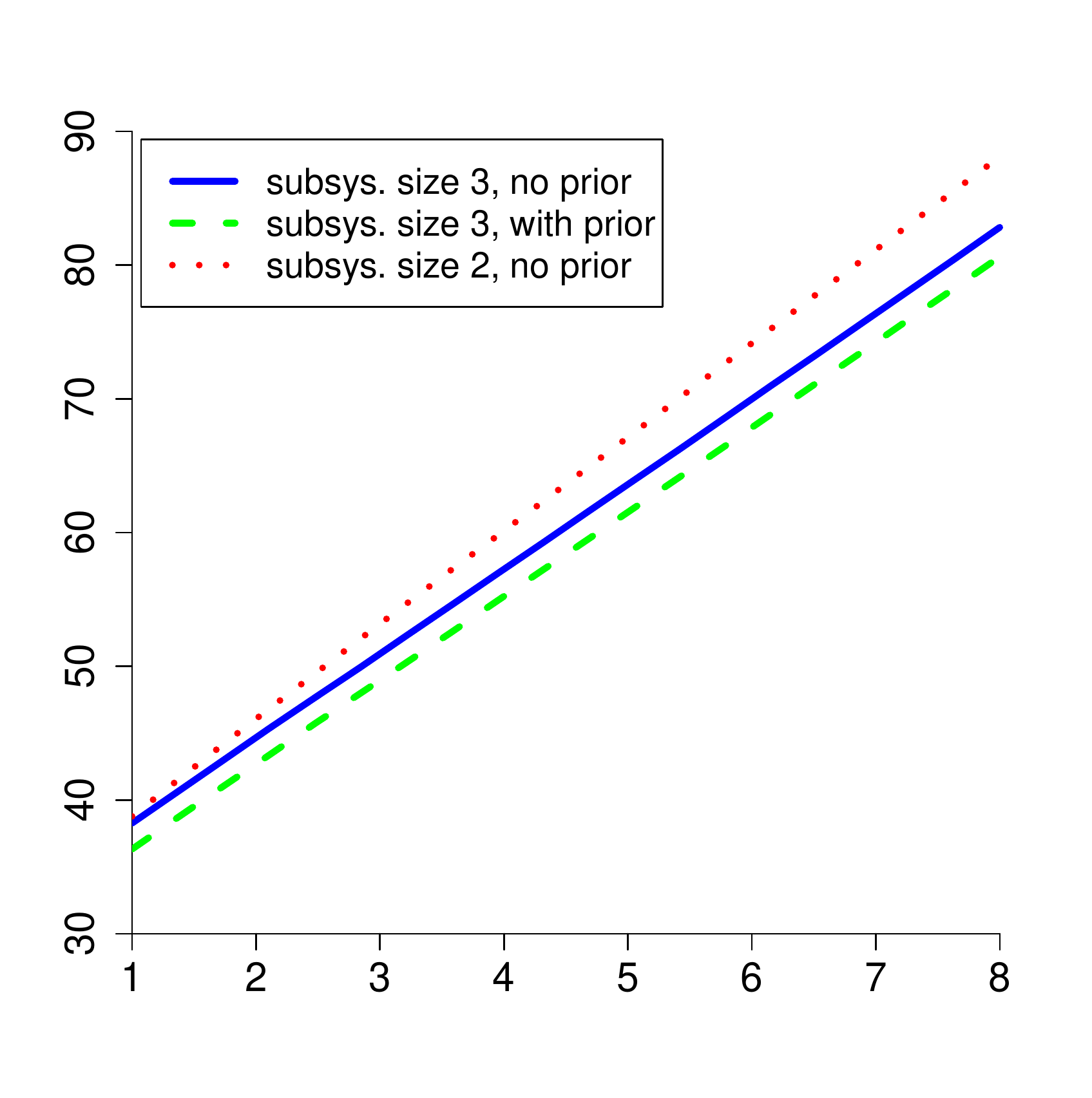}  
}
%\hspace{1cm}
\subfloat[]{\label{fig:Ratio-K7-rh07} 
\includegraphics[width=0.32\linewidth]{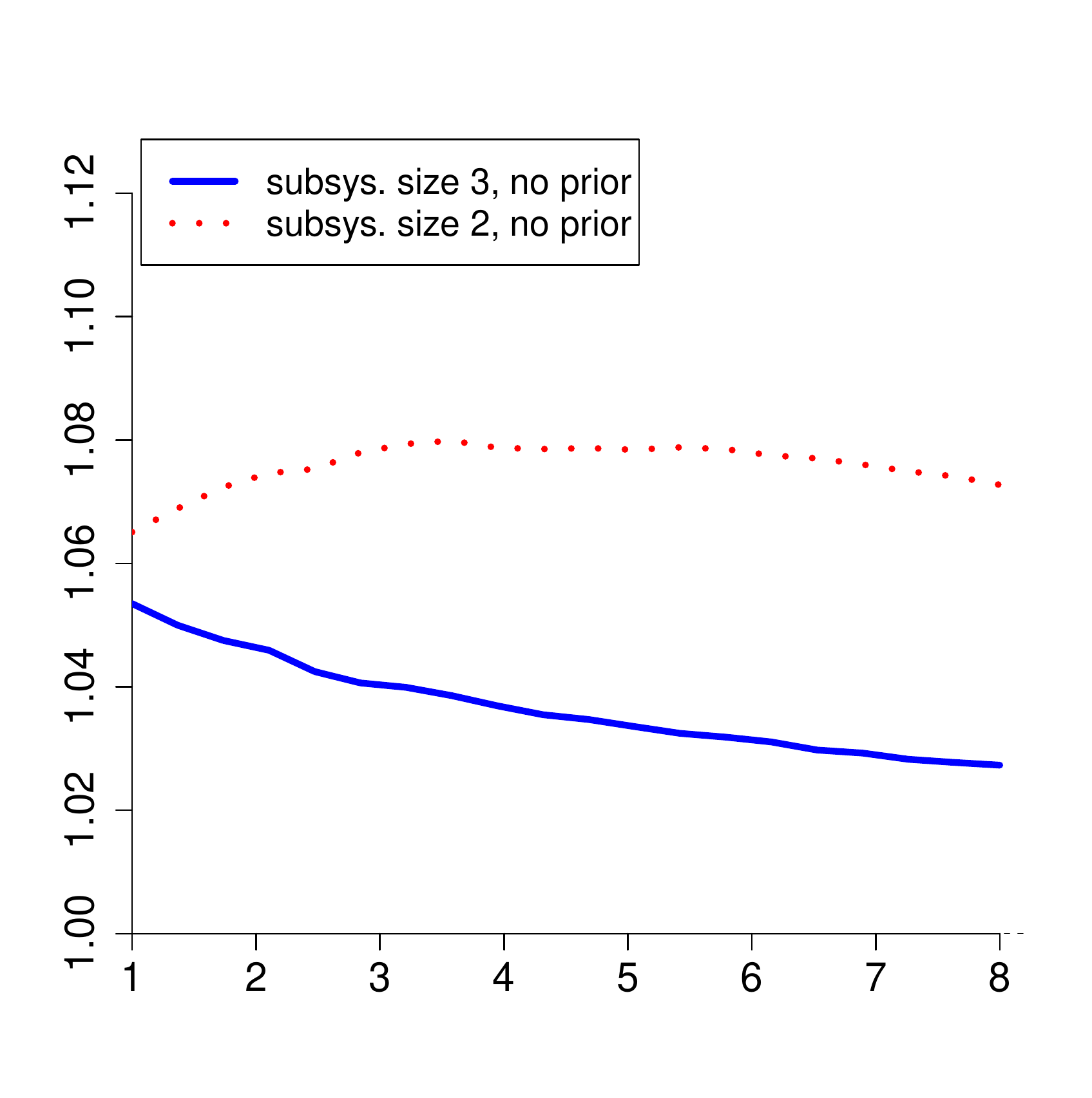}  
}
\subfloat[]{\label{fig:ARE-K7-rh07}  
\includegraphics[width=0.32\linewidth]{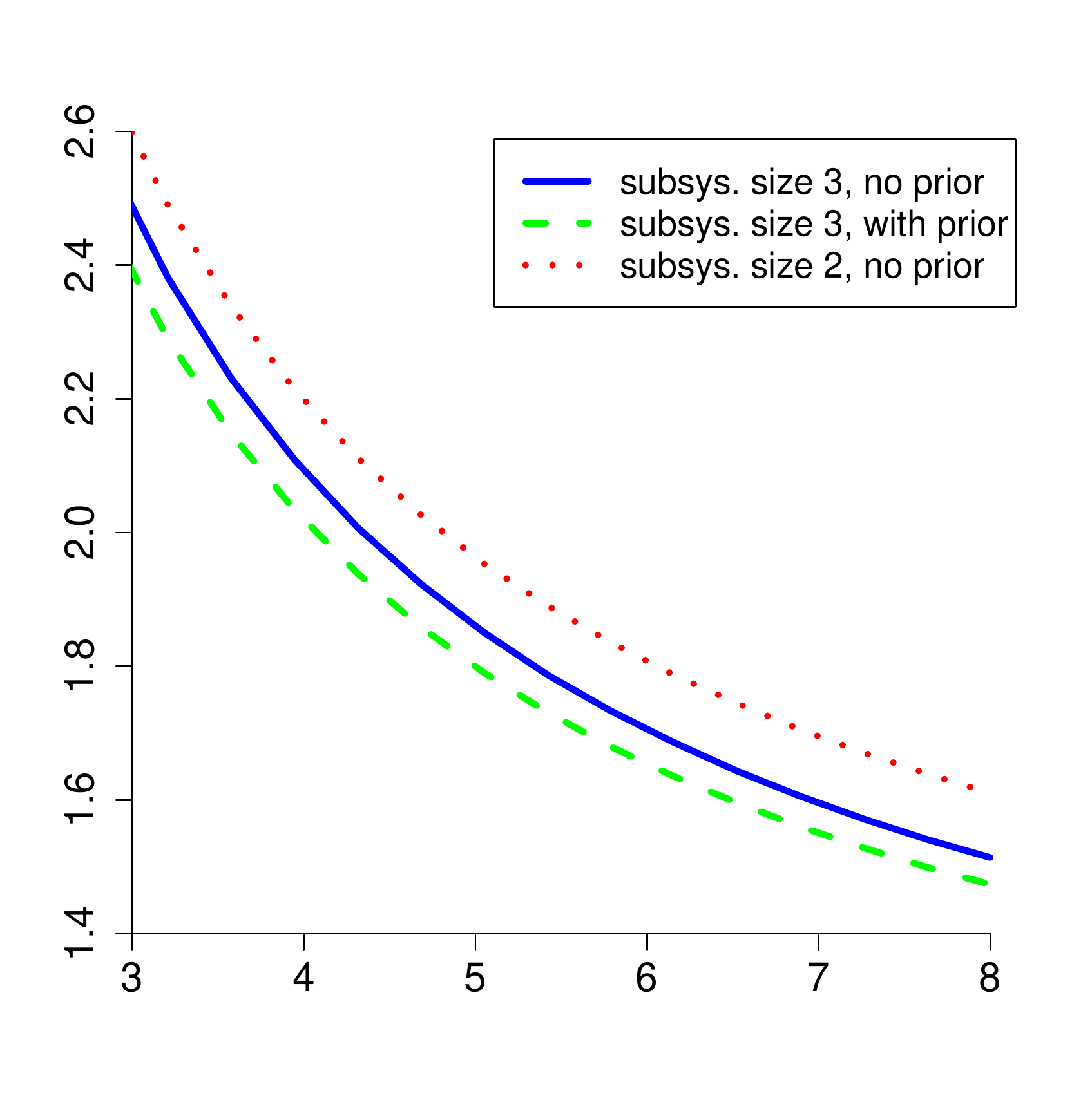}  
}
}

\caption{Performance of the tests under the first study, where $K = 7$. %(a)-(c) corresponds to the case when $\rho = 0.4$ and (d)-(f) to the case when $\rho = 0.7$. 
%The x-axis in all figures is $|\log_{10} \alpha|$ and the y-axis for (a) is the expected sample size of the tests, for (b) the ratio of the expected sample size of 
%the versions without prior information with respect to that of the version with bigger subsystem as well as prior information, and for (c) the ratio of expected sample size of all tests with respect to the first order approximation of the corresponding optimal performance.
}
\label{fig:K7-rh07}
\vspace{-0.3cm}
\end{figure}

\begin{figure}[H]
\centering
\makebox[\linewidth]{
\subfloat[]{\label{fig:ESS-K3-rh07} 
\includegraphics[width=0.32\linewidth]{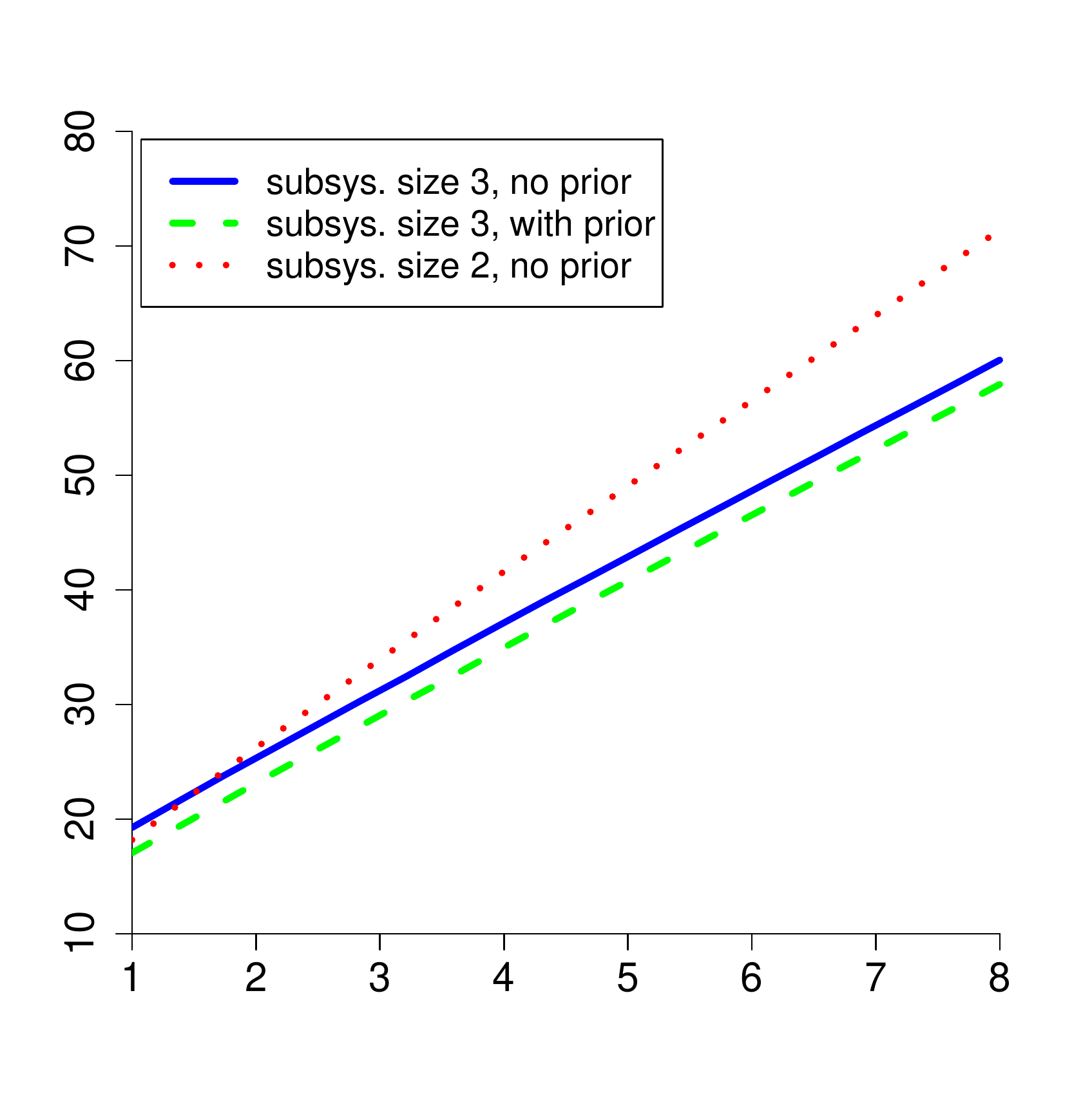}  
}
%\hspace{1cm}
\subfloat[]{\label{fig:Ratio-K3-rh07} 
\includegraphics[width=0.32\linewidth]{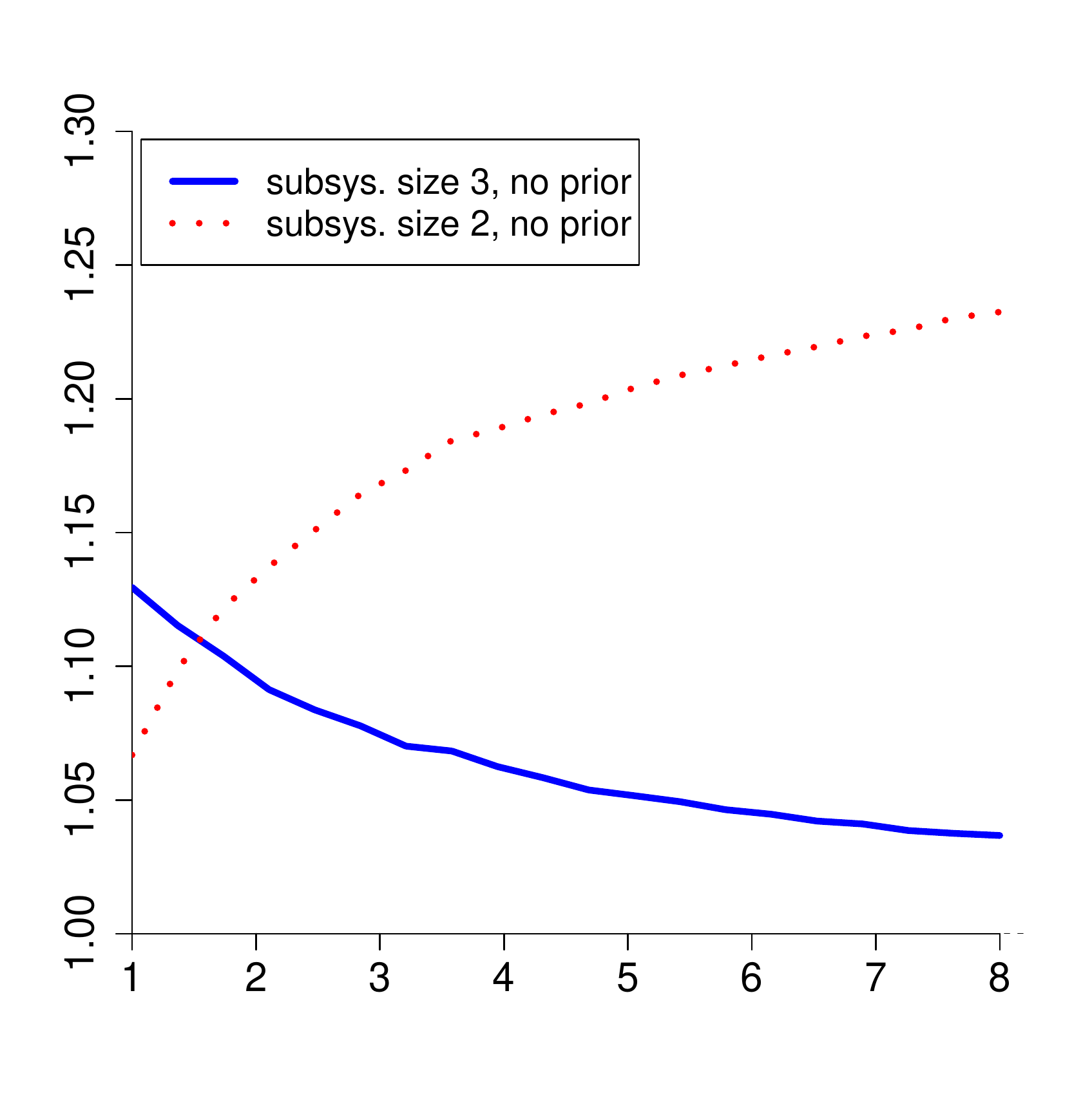}  
}
\subfloat[]{\label{fig:ARE-K3-rh07}  
\includegraphics[width=0.32\linewidth]{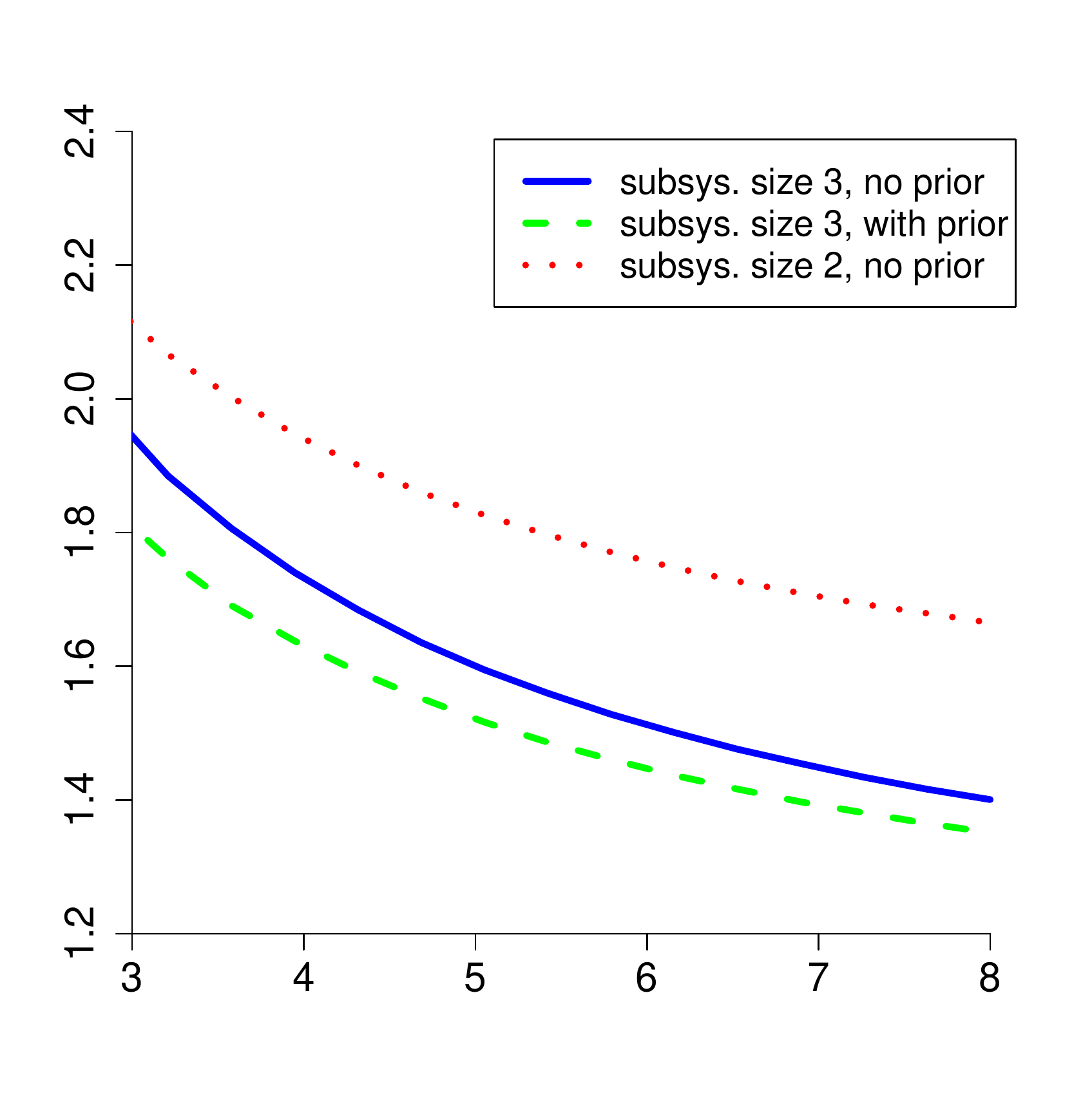}  
}
}

\caption{
Performance of the tests under the second study, where $K = 3$. %(a)-(c) corresponds to the case when $\rho = 0.4$ and (d)-(f) to the case when $\rho = 0.7$. 
%The x-axis in all figures is $|\log_{10} \alpha|$ and the y-axis for (a) is ESS of all tests under $\sfP$, for (b) is ratio of ESS of test 1 and test 2 without prior info with respect to ESS of test 1 with prior info, and for (c) is the ratio of ESS of all test with respect to the first order approximation of the optimal performance.
}
\label{fig:K3-rh07}
\vspace{-0.3cm}
\end{figure}

\section{Conclusion}\label{sec:conclusion}
In the present paper we consider a setup where multiple  data streams are   generated by distinct data sources and we propose a novel multiple testing formulation  that  unifies the  problems of detection and isolation. Specifically, we introduce a family of  testing procedures,  of various computational complexities, that control 4 distinct error metrics, two of which are related to   the pure detection problem and  two  to the pure isolation problem.  We characterize the optimal expected sample size in the class of all sequential tests that control  these error probabilities to a first-order asymptotic approximation as the latter go to 0.  Moreover, 
we evaluate the asymptotic relative efficiency of low-complexity rules in the proposed family, and establish  sufficient conditions for their asymptotic optimality.   

Unlike most previous works in the literature, the results  in the present paper  do not rely on the assumption of independence among the data sources.  Thus,  we apply them   to the problem of detecting and isolating  (i) anomalous, possibly dependent, data streams, (ii) an unknown dependence structure. For the first of these problems, in the case of independence we  obtain asymptotically optimal rules that are easily implementable in practice, under general conditions that allow for temporal dependence. Thus, we  generalize  previous results in the  literature, which focus  either  on the  pure detection or the pure isolation problem.

%optimal asymptotic performance;  and when these are not satisfied, we  obtain upper bounds on their asymptotic relative efficiencies. % by compromising its asymptotic performance to some allowable extent. 

%We show that the practical rule is asymptotically optimal under a wide range of probability distributions. In these problems, we establish sufficient conditions for computationally tractable  rules in the proposed  family of tests to achieve asymptotic optimality and when this is not possible, we  obtain bounds on its asymptotic relative efficiency.

% and show that under certain conditions it is asymptotically optimal. Finally, we apply the general theory to the  consider the problem of joint detection and isolation 
 % A simulation study is also performed to further illustrate these results.

 There are various natural generalizations of the present work. For example, we can obtain a similar asymptotic theory using alternative metrics for the false positive and false negative rates, such as FDR, pFDR, etc., working similarly to \cite{He_Bart_2021}. Moreover, we can extend the results of the current work  to the case that  the distribution of each  unit  under each hypothesis  belongs to a parametric family, working similarly to \cite[Section 6]{song_fell2019}.
 
%  replacit is interesting to develop optimal procedures for

% One direction is to allow the families of distributions (hypotheses) associated to each unit to be infinite, for example, if we consider a parametric family and allow the parameter of interest to belong to some interval

Finally, there are many open questions and  directions for   further  research,  such as  a more precise description of the optimal performance,  the proof of   stronger optimality properties, 
%in the present work we have established a first-order asymptotic optimality theory. 
%Another direction of interest is to   consider 
asymptotic regimes where the number of sources also goes to infinity as the error probabilities vanish, modeling the  data  streams using spatial models, e.g., a  Markov random field, where the special  underlying dependence  can lead to further insights, 
closed-form expressions  for the thresholds that are less conservative than the ones we obtain in this work, more efficient design of subsystems, etc.  Finally, another direction of interest is  to allow for the sampling to be terminated  at a different time in each data stream, as in  \cite{2014arXiv_Bartroff, bartroff2014sequential, malloy2014sequential}.

%Moreover, although in practice the  thresholds of the proposed testing procedures can be determined using Monte-Carlo simulation, it would be convenient to have Another potential direction of interest would be to consider the case . 

%Moreover, in our asymptotic analysis we have considered the number of sources to be constant. It would be of interest to % hroughout the paper but it can be very large in practice, for example, in high dimensional settings. 

%\newpage

%Here put the example with testing of $\rho$ in multivariate normal with unknown mean and variance.
%Compare to example with known mean and variance

%%%%%%%%%%%%%%%%%%%%%%%%%%%%%%%%%%%%%%%%%%%%%%
%% Example with single Appendix:            %%
%%%%%%%%%%%%%%%%%%%%%%%%%%%%%%%%%%%%%%%%%%%%%%
%%%%%%%%%%%%%%%%%%%%%%%%%%%%%%%%%%%%%%%%%%%%%%
%% Example with multiple Appendixes:        %%
%%%%%%%%%%%%%%%%%%%%%%%%%%%%%%%%%%%%%%%%%%%%%%

%%%%%%%%%%%%%%%%%%%%%%%%%%%%%%%%%%%%%%%%%%%%%%
%% Support information, if any,             %%
%% should be provided in the                %%
%% Acknowledgements section.                %%
%%%%%%%%%%%%%%%%%%%%%%%%%%%%%%%%%%%%%%%%%%%%%%

\begin{comment}
\begin{acks}[Acknowledgments]
The authors would like to thank the anonymous referees, an Associate
Editor and the Editor for their constructive comments that improved the
quality of this paper.
\end{acks}
\end{comment}

%%%%%%%%%%%%%%%%%%%%%%%%%%%%%%%%%%%%%%%%%%%%%%
%% Funding information, if any,             %%
%% should be provided in the                %%
%% funding section.                         %%
%%%%%%%%%%%%%%%%%%%%%%%%%%%%%%%%%%%%%%%%%%%%%%
\begin{funding}
This work was supported in part by NSF Grant DMS-1737962.
\end{funding}

\newpage

\begin{appendix}

\section{Proofs regarding Error Control}\label{app:error_control}

\subsection{Some important lemmas}
In this subsection we establish some lemmas which are critical in establishing the error control in Theorems \ref{th:errcon} and \ref{th:errcon2}.
\begin{lemma}\label{D_e}
%Assume that \eqref{infty_det_iso} holds. 
Suppose $\emptyset \in \Psi$. Fix any $e \in \cK$ and $D_e > 1$. If $\sfP \in \cH_0$ then
\begin{align*}
\sfP\lt(\Lambda_{e, \rm det}\lt(n\rt) \geq D_e\ \;\; \text{for some}\ \;\; n \in \bN\rt) \leq \frac{|\cG_{\Psi, e}^{s_e}|}{D_e}.
\end{align*}
\end{lemma}
\begin{proof}
We have,
\begin{align*}
&\sfP\lt(\Lambda_{e, \rm det}\lt(n\rt) \geq D_e\ \;\; \text{for some}\ \;\; n \in \bN\rt)\\
&= \sfP^{s_e}\lt(\Lambda_{e, \rm det}\lt(n\rt) \geq D_e\ \;\; \text{for some}\ \;\; n \in \bN\rt)\\
&\leq \sfP^{s_e}\lt(\max_{\sfQ \in \cG_{\Psi, e}^{s_e}}\frac{\Lambda_n(\sfQ, \sfP_0)}{\Lambda_n(\sfP^{s_e}, \sfP_0)} \geq D_e\  \;\; \text{for some}\ \;\; n \in \bN\rt),\ \text{since}\ \;\; \sfP^{s_e} \in \cH_0^{s_e}\\
&\leq \sfP^{s_e}\lt(\bigcup_{\sfQ \in \cG_{\Psi, e}^{s_e} }\lt\{\frac{\Lambda_n(\sfQ, \sfP_0)}{\Lambda_n(\sfP^{s_e}, \sfP_0)} \geq D_e\ \;\; \text{for some}\ \;\; n \in \bN\rt\}\rt)\\
&\leq \sum_{\sfQ \in \cG_{\Psi, e}^{s_e}}\sfP^{s_e}\lt(\Lambda_n(\sfQ, \sfP^{s_e}) \geq D_e\ \;\; \text{for some}\ \;\; n \in \bN\rt)\\
&\leq \sum_{\sfQ \in \cG_{\Psi, e}^{s_e}}\frac{1}{D_e} = \frac{|\cG_{\Psi, e}^{s_e}|}{D_e},
\end{align*}
where the third inequality follows from Boole's inequality and the last one follows from Ville's inequality.
\end{proof}

\begin{lemma}\label{C_e}
%Assume that \eqref{infty_det_iso} holds. 
Suppose $\emptyset \in \Psi$. Fix any $\sfP \in \cP_\Psi \setminus \cH_0$, $e \in \cA(\sfP)$ and $C_e > 1$. Then
\begin{align*}
\sfP\lt(\Lambda_{e, \rm det}\lt(n\rt) \leq \frac{1}{C_e}\ \;\; \text{for some}\ \;\; n \in \bN\rt) \leq \frac{|\cH_0^{s_e}|}{C_e}.
\end{align*}
\end{lemma}
\begin{proof}
Since $e \in \cA(\sfP)$ and $\sfP \in \cP_\Psi \setminus \cH_0$, we have $\sfP \in \cG_{\Psi, e}$. Now,
\begin{align*}
&\sfP\lt(\Lambda_{e, \rm det}\lt(n\rt) \leq \frac{1}{C_e}\ \;\; \text{for some}\ \;\; n \in \bN\rt)\\
&= \sfP^{s_e}\lt(\Lambda_{e, \rm det}\lt(n\rt) \leq \frac{1}{C_e}\ \;\; \text{for some}\ \;\; n \in \bN\rt)\\
&\leq \sfP^{s_e}\lt(\max_{\sfQ \in \cH_0^{s_e}}\frac{\Lambda_n(\sfQ, \sfP_0)}{\Lambda_n(\sfP^{s_e}, \sfP_0)} \geq C_e\ \;\; \text{for some}\ \;\; n \in \bN\rt),\ \text{since}\ \;\; \sfP^{s_e} \in \cG_{\Psi, e}^{s_e}\\
&\leq \sfP^{s_e}\lt(\bigcup_{\sfQ \in \cH_0^{s_e}}\lt\{\frac{\Lambda_n(\sfQ, \sfP_0)}{\Lambda_n(\sfP^{s_e}, \sfP_0)} \geq C_e\ \;\; \text{for some}\ \;\; n \in \bN\rt\}\rt)\\
&\leq \sum_{\sfQ \in \cH_0^{s_e}}\sfP^{s_e}\lt(\Lambda_n(\sfQ, \sfP^{s_e}) \geq C_e\ \;\; \text{for some}\ \;\; n \in \bN\rt)\\
&\leq \sum_{\sfQ \in \cH_0^{s_e}}\frac{1}{C_e} = \frac{|\cH_0^{s_e}|}{C_e},
\end{align*}
where the third inequality follows from Boole's inequality and the last one follows from Ville's inequality.
\end{proof}

\begin{lemma}\label{B_e}
%Assume that \eqref{infty_det_iso} holds. 
Suppose $\emptyset \in \Psi$.
Fix any $\sfP \in \cP_\Psi \setminus \cH_0$, $e \notin \cA(\sfP)$ and $B_e > 1$. Then
\begin{align*}
\sfP\lt(\Lambda_{e, \rm iso}\lt(n\rt) \geq B_e\ \;\; \text{for some}\ \;\; n \in \bN\rt) \leq \frac{|\cG_{\Psi, e}^{s'_e}|}{B_e}.
\end{align*}
\end{lemma}

\begin{proof}
Since $e \notin \cA(\sfP)$ and $\sfP \in \cP_\Psi \setminus \cH_0$, we have $\sfP \in \cH_{\Psi, e}$. Now,
\begin{align*}
&\sfP\lt(\Lambda_{e, \rm iso}\lt(n\rt) \geq B_e\ \;\; \text{for some}\ \;\; n \in \bN\rt)\\
&= \sfP^{s'_e}\lt(\Lambda_{e, \rm iso}\lt(n\rt) \geq B_e\ \;\; \text{for some}\ \;\; n \in \bN\rt)\\
&\leq \sfP^{s'_e}\lt(\max_{\sfQ \in \cG_{\Psi, e}^{s'_e}}\frac{\Lambda_n(\sfQ, \sfP_0)}{\Lambda_n(\sfP^{s'_e}, \sfP_0)} \geq B_e\ \;\; \text{for some}\ \;\; n \in \bN\rt),\ \text{since}\ \;\; \sfP^{s'_e} \in \cH_{\Psi, e}^{s'_e}\\
&\leq \sfP^{s'_e}\lt(\bigcup_{\sfQ \in \cG_{\Psi, e}^{s'_e}}\lt\{\frac{\Lambda_n(\sfQ, \sfP_0)}{\Lambda_n(\sfP^{s'_e}, \sfP_0)} \geq B_e\ \;\; \text{for some}\ \;\; n \in \bN\rt\}\rt)\\
&\leq \sum_{\sfQ \in \cG_{\Psi, e}^{s'_e}}\sfP^{s'_e}\lt(\Lambda_n(\sfQ, \sfP^{s'_e}) \geq B_e\ \;\; \text{for some}\ \;\; n \in \bN\rt)\\
&\leq \sum_{\sfQ \in \cG_{\Psi, e}^{s'_e}}\frac{1}{B_e} = \frac{|\cG_{\Psi, e}^{s'_e}|}{B_e},
\end{align*}
where the third inequality follows from Boole's inequality and the last one follows from Ville's inequality.
\end{proof}

\begin{lemma}\label{A_e}
%Assume that \eqref{infty_det_iso} holds.
Suppose $\emptyset \in \Psi$. 
Fix any $\sfP \in \cP_\Psi \setminus \cH_0$, $e \in \cA(\sfP)$ and $A_e > 1$. Then
\begin{align*}
\sfP\lt(\Lambda_{e, \rm iso}\lt(n\rt) \leq \frac{1}{A_e}\ \;\; \text{for some}\ \;\; n \in \bN\rt) \leq \frac{|\cH_{\Psi, e}^{s'_e}|}{A_e}.
\end{align*}
\end{lemma}
\begin{proof}
Since $e \in \cA(\sfP)$ and $\sfP \in \cP_\Psi \setminus \cH_0$, we have $\sfP \in \cG_{\Psi, e}$. Now,
\begin{align*}
&\sfP\lt(\Lambda_{e, \rm iso}\lt(n\rt) \leq \frac{1}{A_e}\ \;\; \text{for some}\ \;\; n \in \bN\rt)\\
&= \sfP^{s'_e}\lt(\Lambda_{e, \rm iso}\lt(n\rt) \leq \frac{1}{A_e}\ \;\; \text{for some}\ \;\; n \in \bN\rt)\\
&\leq \sfP^{s'_e}\lt(\max_{\sfQ \in \cH_{\Psi, e}^{s'_e}}\frac{\Lambda_n(\sfQ, \sfP_0)}{\Lambda_n(\sfP^{s'_e}, \sfP_0)} \geq A_e\ \;\; \text{for some}\ \;\; n \in \bN\rt),\ \text{since}\ \;\; \sfP^{s'_e} \in \cG_{\Psi, e}^{s'_e}\\
&\leq \sfP^{s'_e}\lt(\bigcup_{\sfQ \in \cH_{\Psi, e}^{s'_e}}\lt\{\frac{\Lambda_n(\sfQ, \sfP_0)}{\Lambda_n(\sfP^{s'_e}, \sfP_0)} \geq A_e\ \;\; \text{for some}\ \;\; n \in \bN\rt\}\rt)\\
&\leq \sum_{\sfQ \in \cH_{\Psi, e}^{s'_e}}\sfP^{s'_e}\lt(\Lambda_n(\sfQ, \sfP^{s'_e}) \geq A_e\ \;\; \text{for some}\ \;\; n \in \bN\rt)\\
&\leq \sum_{\sfQ \in \cH_{\Psi, e}^{s'_e}}\frac{1}{A_e} = \frac{|\cH_{\Psi, e}^{s'_e}|}{A_e},
\end{align*}
where the third inequality follows from Boole's inequality and the last one follows from Ville's inequality.
\end{proof}

\subsection{Proof of Theorem \ref{th:errcon}}\label{pf:th:errcon}

\begin{proof}
The first part of the proof is similar to the proof of Theorem \ref{th:errcon2}(iii), which is done later.

For the second part, note that when $\Psi = \Psi_{m, m}$ for some $0 < m < |\cK|$, then for any $(T, D) \in \cC_\Psi(\gamma, \delta)$, we have
$\sfP \lt(D \setminus \cA(\sfP) \neq \emptyset\rt) = \sfP\lt(\cA(\sfP) \setminus D \neq \emptyset  \rt) = \sfP\lt(D \neq \cA(\sfP)\rt) \leq \gamma \wedge \delta$.
Therefore, replacing both $\gamma$ and $\delta$ by $\gamma \wedge \delta$ in the lower bounds of $A_e$ and $B_e$ gives us the result.
\end{proof}

\subsection{Proof of Theorem \ref{th:errcon2}}\label{pf:th:errcon2}

\begin{proof}
First, we prove (iii) since the proof of (i) follows similarly. Fix the thresholds $\{A_e, B_e, C_e, D_e\}_{e \in \cK} > 1$. We detect that there is at least one unit in which alternative is correct only when $T_{joint} < T_0$ which means $\lt\{\Lambda_{e, \rm det}\lt(T^*\rt) \geq D_e\rt\}$ for some $e \in \cK$. Formally,
$$\lt\{D^* \neq \emptyset\rt\} \subseteq \bigcup_{e \in \cK}\lt\{\Lambda_{e, \rm det}\lt(T^*\rt) \geq D_e\rt\}.$$
Now, consider any arbitrary $\sfP \in \cH_0$. Then, we have
\begin{align*}
\sfP\lt(D^* \neq \emptyset\rt) &\leq \sum_{e \in \cK} \sfP\lt(\Lambda_{e, \rm det}\lt(T^*\rt) \geq D_e\rt)\\
&\leq \sum_{e \in \cK} \sfP\lt(\Lambda_{e, \rm det}\lt(n\rt) \geq D_e\ \;\; \text{for some}\ \;\; n \in \bN\rt)\\
&\leq\sum_{e \in \cK}\frac{|\cG_{\Psi, e}^{s_e}|}{D_e} %= \sum_{e \in \cK}\frac{1}{D},\ \text{letting}\ \frac{|\cP_{e, 1}^{s_e}|}{D_e} = \frac{1}{D}\ \text{for all}\ e \in \cK\\
%&=\frac{|\cK|}{D},
\end{align*}
where the first inequality follows from Boole's inequality and the third one follows from Lemma \ref{D_e}.
Thus, if we select $D_e$ according to (i), the probability of false alarm is controlled. 

For the rest of the proof of (iii), consider any arbitrary $\sfP \in \cP_\Psi \setminus \cH_0$.
We detect that null hypothesis is true in all units only if $T_0 < T_{joint}$, which means $\lt\{\Lambda_{e, \rm det}\lt(T^*\rt) \leq \frac{1}{C_e}\rt\}$ for all $e \in \cK$, and therefore, for
%which means
%$$\lt\{D^* = \emptyset\rt\} \subseteq \bigcap_{e \in \cK}\lt\{\Lambda_{e, \rm det}^{s_e}\lt(T^*\rt) \leq \frac{1}{C}\rt\} \subseteq $$
any arbitrary $e \in \cA(\sfP)$. Formally,
$$\lt\{D^* = \emptyset\rt\} \subseteq \lt\{\Lambda_{e, \rm det}\lt(T^*\rt) \leq \frac{1}{C_e}\rt\},$$ which implies
\begin{align*}
\sfP\lt(D^* = \emptyset\rt) &\leq \sfP\lt(\Lambda_{e, \rm det}\lt(T^*\rt) \leq \frac{1}{C_e}\rt)\\
&\leq \sfP\lt(\Lambda_{e, \rm det}\lt(n\rt) \leq \frac{1}{C_e}\ \;\; \text{for some}\ \;\; n \in \bN\rt)\\
&\leq \frac{|\cH_0^{s_e}|}{C_e},
\end{align*}
where the third inequality follows from Lemma \ref{C_e}.
Thus, if we select $C_e$ according to (i), the probability of missed detection is controlled.

Note that, at least one false positive is made only if $T_{joint} < T_0$ and we mistakenly isolate a unit in which null is true. Formally, the event $\lt\{\Lambda_{e, \rm iso}(T^*) \geq B_e\rt\}$ happens for some $e \notin \cA(\sfP)$, i.e.,
$$\lt\{D^* \setminus \cA(\sfP) \neq \emptyset\rt\} \subseteq \bigcup_{e \notin \cA(\sfP)}\lt\{\Lambda_{e, \rm iso}(T^*) \geq B_e\rt\},$$ which implies
\begin{align*}
\sfP\lt(D^* \setminus \cA(\sfP) \neq \emptyset\rt) &\leq \sum_{e \notin \cA(\sfP)}\sfP\lt(\Lambda_{e, \rm iso}\lt(T^*\rt) \geq B_e\rt)\\
&\leq \sum_{e \notin \cA(\sfP)}\sfP\lt(\Lambda_{e, \rm iso}\lt(n\rt) \geq B_e\ \;\; \text{for some}\ \;\; n \in \bN\rt)\\
&\leq\sum_{e \notin \cA(\sfP)}\frac{|\cG_{\Psi, e}^{s'_e}|}{B_e} \leq \sum_{e \in \cK} \frac{|\cG_{\Psi, e}^{s'_e}|}{B_e},%= \sum_{e \notin \cA(\sfP)}\frac{1}{B},\ \text{letting}\ \frac{|\cP_{e, 1}^{s'_e}|}{B_e} = \frac{1}{B}\ \text{for all}\ e \in \cK\\
%&=\frac{|\cA(\sfP)^c|}{B} \leq \frac{|\cK|}{B},
\end{align*}
where the first inequality follows from Boole's inequality and the third one follows from Lemma \ref{B_e}.
Thus, if we select $B_e$ as in Theorem \ref{th:errcon} we obtain the desired error control. 

Next, observe that at least one false negative is made \textit{and} $D^* \neq \emptyset$ only if $T_{joint} < T_0$ and we fail to isolate a unit in which alternative is true. Formally, the event $\lt\{\Lambda_{e, \rm iso}(T^*) \leq \frac{1}{A_e}\rt\}$ happens for some $e \in \cA(\sfP)$, i.e.,
$$\lt\{\cA(\sfP) \setminus D^* \neq \emptyset, D^* \neq \emptyset\rt\} \subseteq \bigcup_{e \in \cA(\sfP)}\lt\{\Lambda_{e, \rm iso}(T^*) \leq \frac{1}{A_e}\rt\},$$
which implies,
\begin{align*}
\sfP\lt(\cA(\sfP) \setminus D^* \neq \emptyset, D^* \neq \emptyset\rt) &\leq \sum_{e \in \cA(\sfP)}\sfP\lt(\Lambda_{e, \rm iso}\lt(T^*\rt) \leq \frac{1}{A_e}\rt)\\
&\leq \sum_{e \in \cA(\sfP)}\sfP\lt(\Lambda_{e, \rm iso}\lt(n\rt) \leq \frac{1}{A_e}\ \;\; \text{for some}\ \;\; n \in \bN\rt)\\
&\leq\sum_{e \in \cA(\sfP)}\frac{|\cH_{\Psi, e}^{s'_e}|}{A_e} \leq \sum_{e \in \cK} \frac{|\cH_{\Psi, e}^{s'_e}|}{A_e} %= \sum_{e \in \cA(\sfP)}\frac{1}{A},\ \text{letting}\ \frac{|\cP_{e, 0}^{s'_e}|}{A_e} = \frac{1}{A}\ \text{for all}\ e \in \cK\\
%&=\frac{|\cA(\sfP)|}{A} \leq \frac{|\cK|}{A},
\end{align*}
where the first inequality follows from Boole's inequality and the third one follows from Lemma \ref{A_e}.
Thus, if we select $A_e$ as in Theorem \ref{th:errcon} we obtain the desired error control. 

Next, we prove (ii) regarding $\chi^*_{fwer}$, where the thresholds in $\chi^*$ are selected such that $A_e = C_e$ and $B_e = D_e$ for every $e \in \cK$. Fix any arbitrary $\sfP \in \cP_\Psi$. 
Note that, at least one false positive is made if there exists some $e \in D^*_{fwer} \setminus \cA(\sfP)$. Now, $e \notin \cA(\sfP)$ means either $\sfP \in \cH_0$ or $\sfP \in \cH_{\Psi, e}$. %which in turn implies, $\sfP^{s_e} \in \cH_0^{s_e}$ or $\sfP^{s'_e} \in \cH_{\Psi, e}^{s'_e}$. 
Suppose the first case holds, that is, $\sfP \in \cH_0$. Then,
\begin{align} \notag
\sfP\lt(D^*_{fwer} \setminus \cA(\sfP) \neq \emptyset\rt) &= \sfP\lt(D^*_{fwer} \neq \emptyset\rt)\\ \notag
&\leq \sum_{e \in \cK}\sfP\lt(\Lambda_{e, \rm det}\lt(T^*_{fwer}\rt) \geq B_e\rt)\\ \notag
&\leq \sum_{e \in \cK}\sfP\lt(\Lambda_{e, \rm det}\lt(n\rt) \geq B_e\ \;\; \text{for some}\ \;\; n \in \bN\rt)\\ \label{thr_MT1}
&\leq\sum_{e \in \cK}\frac{|\cG_{\Psi, e}^{s_e}|}{B_e} %= \sum_{e \in \cK}\frac{1}{B},\ \text{letting}\ \frac{|\cP_{e, 1}^{s_e}|}{D_e} = \frac{1}{D}\ \text{for all}\ e \in \cK\\
%&=\frac{|\cK|}{D},
\end{align}
where the first inequality follows from Boole's inequality and the third one follows from Lemma \ref{D_e}.
%Thus, let $D = \frac{|\cK|}{\alpha}$ to satisfy the error control for false alarm. 
If the second case holds, that is, $\sfP \in \cH_{\Psi, e}$, then
\begin{align} \notag
\sfP\lt(D^*_{fwer} \setminus \cA(\sfP) \neq \emptyset\rt) &\leq \sum_{e \notin \cA(\sfP)} \sfP\lt(\Lambda_{e, \rm iso}(T^*_{fwer}) \geq B_e\rt)\\ \notag
&\leq \sum_{e \notin \cA(\sfP)} \sfP\lt(\Lambda_{e, \rm iso}\lt(n\rt) \geq B_e\ \;\; \text{for some}\ \;\; n \in \bN\rt)\\ \label{thr_MT2}
&\leq \sum_{e \notin \cA(\sfP)} \frac{|\cG_{\Psi, e}^{s'_e}|}{B_e} \leq \sum_{e \in \cK} \frac{|\cG_{\Psi, e}^{s'_e}|}{B_e},
\end{align}
where the first inequality follows from Boole's inequality and the third one follows from Lemma \ref{B_e}. %Suppose $B > 1$ is such that for all $e \in \cK$,
%$$\frac{|\cP_{e, 1}^{s_e}| \vee |\cP_{e, 1}^{s'_e}|}{B_e} \leq \frac{1}{B}.$$ 
Then, from \eqref{thr_MT1} and \eqref{thr_MT2}, we have
$$\sfP\lt(D^*_{fwer} \setminus \cA(\sfP) \neq \emptyset\rt) \leq \sum_{e \in \cK} \frac{|\cG_{\Psi, e}^{s_e}| \vee |\cG_{\Psi, e}^{s'_e}|}{B_e}.$$ 
Thus, if we select $B_e$ according to (ii), we obtain the desired error control.

Next, observe that at least one false negative is made if either $T_0 < T_{joint}$ and we simply declare $D^*_{fwer} = \emptyset$, or $T_{joint} < T_0$ and we incur the event $\{\cA(\sfP) \setminus D^*_{fwer} \neq \emptyset, D^*_{fwer} \neq \emptyset\}$. Formally,
$$\{\cA(\sfP) \setminus D^*_{fwer} \neq \emptyset\} = \{D^*_{fwer} = \emptyset\} \cup \{\cA(\sfP) \setminus D^*_{fwer} \neq \emptyset, D^*_{fwer} \neq \emptyset\}.$$
For any $e \in \cA(\sfP)$, the probability of the first event can be bounded as follows:
\begin{align} \notag
\sfP\lt(D^*_{fwer} = \emptyset\rt) &\leq \sfP\lt(\Lambda_{e, \rm det}\lt(T^*_{fwer}\rt) \leq \frac{1}{A_e}\rt)\\ \notag
&\leq \sfP\lt(\Lambda_{e, \rm det}\lt(n\rt) \leq \frac{1}{A_e}\ \;\; \text{for some}\ \;\; n \in \bN\rt)\\ \label{thr_MT3}
&\leq \frac{|\cH_0^{s_e}|}{A_e},
\end{align}
where the third inequality follows from Lemma \ref{C_e}. Furthermore, for the second event, we have
\begin{align} \notag
\sfP\lt(\cA(\sfP) \setminus D^*_{fwer} \neq \emptyset, D^*_{fwer} \neq \emptyset\rt) &\leq \sum_{e \in \cA(\sfP)}\sfP\lt(\Lambda_{e, \rm iso}\lt(T^*_{fwer}\rt) \leq \frac{1}{A_e}\rt)\\ \notag
&\leq \sum_{e \in \cA(\sfP)}\sfP\lt(\Lambda_{e, \rm iso}\lt(n\rt) \leq \frac{1}{A_e}\ \;\; \text{for some}\ \;\; n \in \bN\rt)\\ \label{thr_MT4}
&\leq \sum_{e \in \cA(\sfP)}\frac{|\cH_{\Psi, e}^{s'_e}|}{A_e} \leq \sum_{e \in \cK}\frac{|\cH_{\Psi, e}^{s'_e}|}{A_e}, 
\end{align}
where the third inequality follows from Lemma \ref{A_e}.
%Suppose $A > 1$ is such that for all $e \in \cK$,
%$$\frac{|\cH_0^{s_e}| \vee |\cP_{e, 0}^{s_e}|}{A_e} \leq \frac{1}{A}.$$
Then from \eqref{thr_MT3} and \eqref{thr_MT4}, we have
\begin{align*}
\sfP\lt(\cA(\sfP) \setminus D^*_{fwer} \neq \emptyset\rt) &\leq \frac{|\cH_0^{s_e}|}{A_e} + \sum_{e \in \cK}\frac{|\cH_{\Psi, e}^{s'_e}|}{A_e}\\
&\leq \frac{|\cH_0^{s_e}| \vee |\cH_{\Psi, e}^{s'_e}|}{A_e} + \sum_{e \in \cK}\frac{|\cH_0^{s_e}| \vee |\cH_{\Psi, e}^{s'_e}|}{A_e}.
\end{align*}
Thus, if we select $A_e$ according to (ii), we obtain the desired error control. 

Further, note that, for any $e \in \cK$ if $s_e = s'_e = e$ then $b_e = d_e = |\cG^e|$, and $a_e = c_e = |\cH^e|$. Thus, $B_e \geq (b_e \vee d_e)|\cK|/\gamma = b_e|\cK|/\gamma$.
%\begin{align*}
%T^* = \inf\{ n \in \bN : \; \Lambda_{e}(n) \notin \lt(1/A_e, B_e\rt) \; \; \text{for every}  \, e \in \cK \}, \quad \text{and} \quad D^* &= D(T^*). 
%\end{align*}
%Therefore, 
On the other hand, one can bound the probability of at least one false negative as follows.
\begin{align*}
\sfP\lt(\cA(\sfP) \setminus D^*_{fwer} \neq \emptyset\rt) &\leq \sum_{e \in \cA(\sfP)} \sfP\lt(\Lambda_e(T^*_{fwer}) \leq \frac{1}{A_e}\rt)\\
&\leq \sum_{e \in \cA(\sfP)} \sfP\lt(\Lambda_e(n) \leq \frac{1}{A_e}\ \;\; \text{for some}\ \;\; n \in \bN\rt)\\
&\leq \sum_{e \in \cA(\sfP)} \frac{|\cH^e|}{A_e} \leq \sum_{e \in \cK} \frac{a_e}{A_e} %= \sum_{e \in \cA(\sfP)} \frac{1}{A}\ \text{letting}\ \frac{|\cH_0^e|}{A_e} = \frac{1}{A}\ \text{for all}\ e \in \cK\\
%&=\frac{|\cA(\sfP)|}{A} \leq \frac{|\cK|}{A},
\end{align*}
where the first inequality follows from Boole's inequality and the last one follows form Lemma \ref{A_e}. Thus, if we select $A_e$ according to (ii), we obtain the desired error control. The proof is complete.
\end{proof}

\section{Some properties regarding the Information Numbers}\label{app:info_num}

In this section, we present some important lemmas about the information numbers introduced in Section \ref{sec:asymp_opt}, which are critical to establishing Proposition \ref{prop:P1} in Appendix \ref{app:asymp_opt} as well as may be of independent interest to the reader.

\begin{lemma}\label{lem_info_mono}
Suppose that \eqref{max_series} holds for two distinct probability distributions $\sfP, \sfQ \in \cP$, %which are mutually absolutely continuous 
when $s = s_1, s_2 \subseteq [K]$.
\begin{equation*}%\label{info_mono}
\text{If} \quad s_1 \subseteq s_2, \quad \text{then} \quad \cI(\sfP, \sfQ; s_1) \leq \cI\lt(\sfP, \sfQ; s_2\rt).
\end{equation*}
\end{lemma}
%\subsection{Proof of Lemma \ref{lem_info_mono}}
\begin{proof}
Consider the hypotheses testing problem 
$$H_0 : X \sim \sfP\quad \text{versus}\quad H_1 : X \sim \sfQ,$$
where we control both type 1 error and type 2 error below $\alpha$, for some $\alpha \in (0, 1)$.
A test, in this context, is a  pair $ (T, d)$ that consists of an $\{\cF_n\}$-stopping time $T$ that represents the time at which the decision is taken and an $\cF_T$-measurable bernoulli random variable $d$, such that $\{d = i\}$ represents hypothesis $H_i$ is selected  upon stopping, $i = 0, 1$.
Let $\cC^{s}$ be the class of sequential procedures $(T, d)$ which satisfy the desired error control based on $\cF^s$, that is, observations only from $s$. Since  $s_1 \subseteq s_2$, we have $\cF_n^{s_1} \subseteq \cF_n^{s_2}$ for every $n \in \bN$, which implies $\cF^{s_1} \subseteq \cF^{s_2}$ and therefore, $\cC^{s_1} \subseteq \cC^{s_2}$. 
This leads to
\begin{equation}\label{lem_info_mono1}
\inf_{(T, d) \in \cC^{s_1}}\sfE_\sfP\lt[T\rt] \geq \inf_{(T, d) \in \cC^{s_2}}\sfE_\sfP\lt[T\rt].
\end{equation}
Now, from \cite[Lemma 3.4.1]{tartakovsky2014sequential}, as $\alpha \to 0$, we have
$$\inf_{(T, d) \in \cC^{s}} \sfE_\sfP\lt[T\rt] \sim \frac{|\log \alpha|}{\cI\lt(\sfP, \sfQ; s\rt)}, \quad s = s_1, s_2,$$
which combining with \eqref{lem_info_mono1} completes the proof.
\end{proof}

\begin{lemma}\label{info_ineq1}
Consider $\sfP \in \cP$ and class of probability distributions $\cQ \subset \cP$, where $\sfP \notin \cQ$.
Suppose that \eqref{max_series} holds for $\sfP$ and every $\sfQ \in \cQ$, %which are mutually absolutely continuous 
when $s = s_1, s_2 \subseteq [K]$.
\begin{equation*}%\label{info_mono}
\text{If} \quad s_1 \subseteq s_2, \quad \text{then} \quad \cI(\sfP, \cQ; s_1) \leq \cI\lt(\sfP, \cQ; s_2\rt).
\end{equation*}
\end{lemma}

%\subsection{Proof of Lemma \ref{info_ineq1}}

\begin{proof}
From Lemma \ref{lem_info_mono}, for every $\sfQ \in \cQ$, we have
$$\cI(\sfP, \sfQ; s_1) \leq \cI(\sfP, \sfQ; s_2).$$
Therefore, taking minimum over the class $\cQ$ on both sides, we obtain
$$\cI\lt(\sfP, \cQ; s_1\rt) = \min_{\sfQ \in \cQ}\cI\lt(\sfP, \sfQ; s_1\rt) \leq \min_{\sfQ \in \cQ} \cI\lt(\sfP, \sfQ; s_2\rt) = \cI\lt(\sfP, \cQ; s_2\rt),$$
and the proof is complete.
\end{proof}
%\textcolor{red}{Well, I don't think the following lemma can be proven without SLLN, so including it with the SLLN assumption for the time being.}

%\subsection{Proof of Lemma \ref{info_ineq2}}

\begin{lemma}\label{info_ineq2}
Consider two disjoint subsets $s_1, s_2 \subset [K]$.
Suppose that \eqref{max_series} and \eqref{comp_conv_onesided} hold for two distinct probability distribution $\sfP, \sfQ \in \cP$ %which are mutually absolutely continuous 
when $s = s_1, s_2$. If $s_1$ and $s_2$ are independent both under $\sfP$ and $\sfQ$, then \eqref{max_series} and \eqref{comp_conv_onesided} hold for $\sfP$, $\sfQ$ and $s = s_1 \cup s_2$ with
\begin{equation}\label{eq:info_deco}
\cI(\sfP, \sfQ; s) = \cI\lt(\sfP, \sfQ; s_1\rt) + \cI\lt(\sfP, \sfQ; s_2\rt).
\end{equation}
\end{lemma}

\begin{proof}
Let $s = s_1 \cup s_2$.
Since $s_1$ and $s_2$ are independent both under $\sfP$ and $\sfQ$, we have
$$\sfP^s = \sfP^{s_1} \otimes \sfP^{s_2}, \quad \text{and} \quad \sfQ^s = \sfQ^{s_1} \otimes \sfQ^{s_2}.$$ 
Now, suppose for a measurable set $A$, $\sfQ^s(A) = 0$, which means
\begin{align*}
&\sfQ^{s_1}(A) \times \sfQ^{s_2}(A) = 0\\
\Rightarrow\;\; &\sfQ^{s_1}(A) = 0 \;\; \text{or}, \;\; \sfQ^{s_2}(A) = 0.
\end{align*}
By the definition of absolute continuity, if the first case happens, then $\sfP^{s_1}(A) = 0$, otherwise, $\sfP^{s_2}(A) = 0$. Anyway, both cases lead to $\sfP^s(A) = \sfP^{s_1}(A) \times \sfP^{s_2}(A) = 0$. An exactly similar argument also shows that if $\sfP(A) = 0$, then $\sfQ(A) = 0$.
Therefore, $\sfP$ and $\sfQ$ are mutually absolutely continuous when restricted to $\cF_n^s$.
Furthermore, we have the following decomposition of the likelihood process $\Lambda_n\lt(\sfP^s, \sfQ^s\rt)$:
\begin{align*}
\Lambda_n\lt(\sfP^s, \sfQ^s\rt) = \frac{d\sfP^{s_1} \otimes d\sfP^{s_2}}{d\sfQ^{s_1} \otimes d\sfQ^{s_2}}\lt(\cF_n^s\rt) &= \frac{d\sfP^{s_1}}{d\sfQ^{s_1}}\lt(\cF_n^{s_1}\rt) \times \frac{d\sfP^{s_2}}{d\sfQ^{s_2}}\lt(\cF_n^{s_2}\rt)\\ &= \Lambda_n\lt(\sfP^{s_1}, \sfQ^{s_1}\rt) \times \Lambda_n\lt(\sfP^{s_2}, \sfQ^{s_2}\rt),
\end{align*}
which subsequently leads us to
$$Z_n\lt(\sfP^s, \sfQ^s\rt) %= \log\Lambda_n\lt(\sfP^s, \sfQ^s\rt) = \log \Lambda\lt(\sfP^{s_1}, \sfQ^{s_1}; n\rt) + \log \Lambda\lt(\sfP^{s_2}, \sfQ^{s_2}; n\rt) 
= Z_n\lt(\sfP^{s_1}, \sfQ^{s_1}\rt) + Z_n\lt(\sfP^{s_2}, \sfQ^{s_2}\rt).$$
Now, in order to prove \eqref{eq:info_deco}, it suffices to show that
\begin{align}\label{deco_max}
\sfP\lt(\max_{1 \leq k \leq n} Z_k\lt(\sfP^s, \sfQ^s\rt) \geq n \rho  \rt) \to 0 \quad \forall \, \rho> \cI \lt(\sfP, \sfQ; s_1\rt) + \cI\lt(\sfP, \sfQ; s_2\rt),
\end{align}
and at the same time
\begin{align}\label{deco_oneside}
\sum_{n = 1}^{\infty}\sfP\lt(Z_n\lt(\sfP^{s}, \sfQ^{s}\rt) < n \rho \rt) < \infty \quad \forall \, \rho< \cI \lt(\sfP, \sfQ; s_1\rt) + \cI\lt(\sfP, \sfQ; s_2\rt),
\end{align}
since these two conditions uniquely characterize the quantity $\cI(\sfP, \sfQ; s)$.

To start with, we fix an arbitrary $\rho > \cI \lt(\sfP, \sfQ; s_1\rt) + \cI\lt(\sfP, \sfQ; s_2\rt)$. Note that, there exist $\rho_1 > \cI \lt(\sfP, \sfQ; s_1\rt)$ and $\rho_2 > \cI\lt(\sfP, \sfQ; s_2\rt)$ such that $\rho = \rho_1 + \rho_2$. Indeed, by letting $\delta = \rho - (\cI \lt(\sfP, \sfQ; s_1\rt) + \cI\lt(\sfP, \sfQ; s_2\rt)) > 0$, one can consider $\rho_1 = \cI \lt(\sfP, \sfQ; s_1\rt) + \delta/2$ and $\rho_2 = \cI \lt(\sfP, \sfQ; s_2\rt) + \delta/2$. Now,
\begin{align} \notag
&\sfP\lt(\max_{1 \leq k \leq n} Z_k\lt(\sfP^s, \sfQ^s\rt) \geq n \rho  \rt)\\ \notag
&= \sfP\lt(\max_{1 \leq k \leq n} \lt\{Z_k\lt(\sfP^{s_1}, \sfQ^{s_1}\rt) + Z_k\lt(\sfP^{s_2}, \sfQ^{s_2}\rt)\rt\} \geq n \rho  \rt)\\ \notag
&\leq \sfP\lt(\max_{1 \leq k \leq n} Z_k\lt(\sfP^{s_1}, \sfQ^{s_1}\rt) + \max_{1 \leq k \leq n}Z_k\lt(\sfP^{s_2}, \sfQ^{s_2}\rt) \geq n \rho_1 + n \rho_2\rt)\\ \label{deco_max1}
&\leq \sfP\lt(\max_{1 \leq k \leq n} Z_k\lt(\sfP^{s_1}, \sfQ^{s_1}\rt) \geq n\rho_1\rt) + \sfP\lt(\max_{1 \leq k \leq n} Z_k\lt(\sfP^{s_2}, \sfQ^{s_2}\rt) \geq n\rho_2\rt),
\end{align}
where the first inequality follows from the fact that, for any two sequences $\{a_n\}_{n \in \bN}$ and $\{b_n\}_{n \in \bN}$, $$\max_{1 \leq k \leq n}\lt\{a_k + b_k\rt\} \leq \max_{1 \leq k \leq n}a_k + \max_{1 \leq k \leq n}b_k, \quad \text{for every} \;\; n \in \bN,$$
and the last inequality is an application of Boole's inequality. Since $\rho_1 > \cI \lt(\sfP, \sfQ; s_1\rt)$ and $\rho_2 > \cI\lt(\sfP, \sfQ; s_2\rt)$, due to condition \eqref{max_series}, both the quantities on the right hand side of \eqref{deco_max1} converge to $0$ as $n \to \infty$. Hence, \eqref{deco_max} follows.  

In order to prove \eqref{deco_oneside}, we fix an arbitrary $\rho < \cI \lt(\sfP, \sfQ; s_1\rt) + \cI\lt(\sfP, \sfQ; s_2\rt)$. Note that, in a similar way like the previous, there exist $\rho_1 < \cI \lt(\sfP, \sfQ; s_1\rt)$ and $\rho_2 < \cI\lt(\sfP, \sfQ; s_2\rt)$ such that $\rho = \rho_1 + \rho_2$. Now,
\begin{align} \notag
&\sum_{n = 1}^{\infty}\sfP\lt( Z_n\lt(\sfP^{s}, \sfQ^{s}\rt) < n \rho \rt)\\ \notag
&= \sum_{n = 1}^{\infty}\sfP\lt(\lt\{Z_n\lt(\sfP^{s_1}, \sfQ^{s_1}\rt) + Z_n\lt(\sfP^{s_2}, \sfQ^{s_2}\rt)\rt\}  < n\rho_1 + n\rho_2 \rt)\\ \notag
&\leq \sum_{n = 1}^{\infty} \sfP\lt(Z_n\lt(\sfP^{s_1}, \sfQ^{s_1}\rt) < n \rho_1 \rt) + \sfP\lt(Z_n\lt(\sfP^{s_2}, \sfQ^{s_2}\rt) < n \rho_2 \rt)\\ \label{deco_oneside1}
&= \sum_{n = 1}^{\infty} \sfP\lt(Z_n\lt(\sfP^{s_1}, \sfQ^{s_1}\rt) < n \rho_1 \rt) + \sum_{n = 1}^{\infty}\sfP\lt(Z_n\lt(\sfP^{s_2}, \sfQ^{s_2}\rt) < n \rho_2 \rt),
\end{align}
where the inequality follows from Boole's inequality. Since $\rho_1 < \cI \lt(\sfP, \sfQ; s_1\rt)$ and $\rho_2 < \cI\lt(\sfP, \sfQ; s_2\rt)$, due to condition \eqref{comp_conv_onesided}, both the quantities on the right hand side of \eqref{deco_oneside1} are finite. Hence, \eqref{deco_oneside} follows and the proof is complete. 

\end{proof}

\section{Proofs regarding Upper Bound on the expected sample size}\label{app:upp_bnd}

In this section we only prove Theorem \ref{ub_DI} since Theorem \ref{ub_I} can be established in a similar fashion.

\subsection{Proof of Theorem \ref{ub_DI}}\label{pf:ub_DI}

\begin{proof}
It is sufficient to prove the results only for $\chi^*$ since it can be proved for the other tests in a similar way. 
Fix the thresholds $A_e, B_e, C_e, D_e$ for all $e \in \cK$ according to \eqref{A,B,asym} and \eqref{C,D,asym}. Now for any $\sfP \in \cH_0$, consider the stopping time
\begin{align*}
T_\sfP := \inf\lt\{n \geq 1: \min\lt\{\Lambda_n\lt(\sfP^{s_e}, \sfQ\rt) : \sfQ \in \cG_{\Psi, e}^{s_e}\rt\} \geq C_e\ \;\; \text{for every}\ \;\; e \in \cK\rt\}.
\end{align*}
Under the assumption \eqref{comp_conv_onesided}, from \cite[Lemma F.2]{song_fell2019_supp}, it follows that as $\beta \to 0$, 
$$\sfE_\sfP\lt[T_\sfP\rt] \; \lesssim \; \max_{e \in \cK} \lt\{\frac{|\log\beta|}{\cI\lt(\sfP, \cG_{\Psi, e}; s_e\rt)}\rt\}.$$
Since $T^* \leq T_0$, it suffices to show that $T_0 \leq T_\sfP$ for any given set of thresholds $\{C_e : e \in \cK\}$. 
By the definition of $T_0$, it suffices to show that
\begin{equation}\label{ub1}
\Lambda_{e, \rm det}(T_\sfP) \leq \frac{1}{C_e}\ \;\; \text{for every}\ \;\; e \in \cK.
\end{equation}
To this end, fix any arbitrary $e \in \cK$ and
by the definition of $T_\sfP$, we have
$$\min\lt\{\Lambda_{T_\sfP}\lt(\sfP^{s_e}, \sfQ\rt) : \sfQ \in \cG_{\Psi, e}^{s_e}\rt\} \geq C_e.$$
As a result,
\begin{align*}
\frac{1}{C_e} \geq \max\lt\{\Lambda_{T_\sfP}\lt(\sfQ, \sfP^{s_e}\rt) : \sfQ \in \cG_{\Psi, e}^{s_e}\rt\} &= \frac{\max\lt\{\Lambda_{T_\sfP}\lt(\sfQ, \sfP_0\rt) : \sfQ \in \cG_{\Psi, e}^{s_e}\rt\}}{\Lambda_{T_\sfP}\lt(\sfP^{s_e}, \sfP_0\rt)}\\ 
&\geq \frac{\max\lt\{\Lambda_{T_\sfP}\lt(\sfQ, \sfP_0\rt) : \sfQ \in \cG_{\Psi, e}^{s_e}\rt\}}{\max\lt\{\Lambda_{T_\sfP}\lt(\sfQ, \sfP_0\rt) : \sfQ \in \cH_0^{s_e}\rt\}} \\
%\frac{\max_{\sfQ \in \cG_{\Psi, e}^{s_e}} \Lambda_{T_\sfP}\lt(\sfQ, \sfP_0\rt)}{\max_{\sfP \in \cH_0^{s_e}} \Lambda\lt(\sfP, \sfP_0; T_\sfP\rt)} 
&= \Lambda_{e, \rm det}(T_\sfP),
\end{align*}
and the result follows in view of \eqref{ub1}.

Now for any $\sfP \in \cP_\Psi \setminus \cH_0$, consider the stopping time
\begin{align*}
T'_\sfP := \inf\bigg\{n \geq 1: &\min\lt\{\Lambda_n\lt(\sfP^{s_e}, \sfQ\rt) : \sfQ \in \cH_0^{s_e}\rt\} \geq D_e\ \;\; \text{for some}\ \;\; e \in \cA(\sfP),\\ %\text{and}\\
&\min\lt\{\Lambda_n\lt(\sfP^{s'_e}, \sfQ\rt) : \sfQ \in \cH_{\Psi, e}^{s'_e}\rt\} \geq B_e\ \;\; \text{for every}\ \;\; e \in \cA(\sfP),\quad \text{and}\\ 
&\min\lt\{\Lambda_n\lt(\sfP^{s'_e}, \sfQ\rt) : \sfQ \in \cG_{\Psi, e}^{s'_e}\rt\} \geq A_e\ \;\; \text{for every}\ \;\; e \notin \cA(\sfP)
%&\qquad \qquad \qquad \quad D_{\rm iso}(n) \in \\sfPsi
\bigg\}.
\end{align*}
Under the assumption \eqref{comp_conv_onesided}, from \cite[Lemma F.2]{song_fell2019_supp}, it follows that as $\alpha, \gamma, \delta \to 0$,
$$\sfE_\sfP\lt[T'_{\sfP}\rt] \lesssim \min_{e \in \cA(\sfP)}  \lt\{ \frac{|\log \alpha|}{\cI\lt(\sfP, \cH_0; s_e \rt)} \rt\} \bigvee \max_{e \in \cA(\sfP)}  \lt\{\frac{|\log \gamma|}{\cI \lt(\sfP, \cH_{\Psi, e}; s'_e \rt) }   \rt\}  \bigvee \max_{e \notin \cA(\sfP)}  \lt\{ \frac{|\log \delta|}{\cI \lt(\sfP, \cG_{\Psi,e}; s'_e \rt)}  \rt\}.$$
Since $T^* \leq T_{joint}$, it suffices to show that $T_{joint} \leq T'_{\sfP}$ for any given set of thresholds $\{A_e, B_e, D_e : e \in \cK\}$. By the definition of $T_{joint}$ it suffices to show that
\begin{align}\label{ub21}
\begin{split}
&\Lambda_{e, \rm det}\lt(T'_\sfP\rt) \geq D_e\ \;\; \text{for some}\ \;\; e \in \cA(\sfP),\\
&\Lambda_{e, \rm iso}\lt(T'_\sfP\rt) \geq B_e\ \;\; \text{for every}\ \;\; e \in \cA(\sfP),\\ 
&\Lambda_{e, \rm iso}\lt(T'_\sfP\rt) \leq 1/A_e\ \;\; \text{for every}\ e \notin \cA(\sfP),\\
\text{and}\ \;\; &D_{\rm iso}\lt(T'_\sfP\rt) \in \Psi. 
%\text{if}\ e \in \cA(\sfP)\ \text{is such that}\ \min_{\sfQ \in \cH_0^{s_e}}\Lambda^{s_e}\lt(\sfP^{s_e}, \sfQ; T'_\sfP\rt) \geq D_e,\ \text{then}\ \Lambda_{e, \rm det}^{s_e}\lt(T'_\sfP\rt) \geq D_e,
\end{split}
\end{align}
%and
%\begin{equation}\label{ub22}
%\Lambda_{e, \rm det}^{s_e}\lt(T'_P\rt) \geq D_e, \quad 
%\Lambda_{e, \rm iso}^{s'_e}\lt(T'_\sfP\rt) \geq B_e\ \forall\ e \in \cA(\sfP)\ \text{and}\ \Lambda_{e, \rm iso}^{s'_e}\lt(T'_\sfP\rt) \leq \frac{1}{A_e}\ \forall\ e \notin \cA(\sfP).
%\end{equation}
Now, note that if $e \in \cA(\sfP)$, which means under $\sfP$ alternative is true in unit $e$, then $\sfP^{s} \in \cG_{\Psi, e}^s$ for $s \in \{s_e, s'_e\}$. 

To this end, consider that particular $e \in \cA(\sfP)$ for which by definition of $T'_\sfP$,
$$\min\lt\{\Lambda_{T'_\sfP}\lt(\sfP^{s_e}, \sfQ\rt) : \sfQ \in \cH_0^{s_e}\rt\} \geq D_e.$$
%$$\min_{\sfQ \in \cH_0^{s_e}}\Lambda^{s_e}\lt(\sfP^{s_e}, \sfQ; n\rt) \geq D_e.$$
As a result,
\begin{align*}
D_e \leq \min\lt\{\Lambda_{T'_\sfP}\lt(\sfP^{s_e}, \sfQ\rt) : \sfQ \in \cH_0^{s_e}\rt\} &= \frac{\Lambda_{T'_\sfP}\lt(\sfP^{s_e}, \sfP_0\rt)}{\max\lt\{\Lambda_{T'_\sfP}\lt(\sfQ, \sfP_0\rt) : \sfQ \in \cH_0^{s_e}\rt\}}\\ 
&\leq \frac{\max\lt\{\Lambda_{T'_\sfP}\lt(\sfQ, \sfP_0\rt) : \sfQ \in \cG_{\Psi, e}^{s_e}\rt\}}{\max\lt\{\Lambda_{T'_\sfP}\lt(\sfQ, \sfP_0\rt) : \sfQ \in \cH_0^{s_e}\rt\}} = \Lambda_{e, \rm det}(T'_\sfP),
\end{align*}
where the last inequality follows because $\sfP \in \cG_{\Psi, e}$. 

Now fix any arbitrary $e \in \cA(\sfP)$. Then
by definition of $T'_\sfP$, we have
$$\min\lt\{\Lambda_{T'_\sfP}\lt(\sfP^{s'_e}, \sfQ\rt) : \sfQ \in \cH_{\Psi, e}^{s'_e}\rt\} \geq B_e.$$
As a result,
\begin{align*}
B_e \leq \min\lt\{\Lambda_{T'_\sfP}\lt(\sfP^{s'_e}, \sfQ\rt) : \sfQ \in \cH_{\Psi, e}^{s'_e}\rt\} &= \frac{\Lambda_{T'_\sfP}\lt(\sfP^{s'_e}, \sfP_0\rt)}{\max\lt\{\Lambda_{T'_\sfP}\lt(\sfQ, \sfP_0\rt) : \sfQ \in \cH_{\Psi, e}^{s'_e}\rt\}}\\
&\leq \frac{\max\lt\{\Lambda_{T'_\sfP}\lt(\sfQ, \sfP_0\rt) : \sfQ \in \cG_{\Psi, e}^{s'_e} \rt\}}{\max\lt\{\Lambda_{T'_\sfP}\lt(\sfQ, \sfP_0\rt) : \sfQ \in \cH_{\Psi, e}^{s'_e}\rt\}} = \Lambda_{e, \rm iso}(T'_\sfP),
\end{align*}
where the last inequality follows because $\sfP \in \cG_{\Psi, e}$.

Next, fix any arbitrary $e \notin \cA(\sfP)$. Then by definition of $T'_\sfP$, we have
$$\min\lt\{\Lambda_{T'_\sfP}\lt(\sfP^{s'_e}, \sfQ\rt) : \sfQ \in \cG_{\Psi, e}^{s'_e}\rt\} \geq A_e.$$
As a result,
\begin{align*}
\frac{1}{A_e} \geq \max\lt\{\Lambda_{T'_\sfP}\lt(\sfQ, \sfP^{s'_e}\rt) : \sfQ \in \cG_{\Psi, e}^{s'_e}\rt\} &= \frac{\max\lt\{\Lambda_{T'_\sfP}\lt(\sfQ, \sfP_0\rt) : \sfQ \in \cG_{\Psi, e}^{s'_e}\rt\}}{\Lambda_{T'_\sfP}\lt(\sfP^{s'_e}, \sfP_0\rt)}\\ 
&\geq \frac{\max\lt\{\Lambda_{T'_\sfP}\lt(\sfQ, \sfP_0\rt) : \sfQ \in \cG_{\Psi, e}^{s'_e}\rt\}}{\max\lt\{\Lambda_{T'_\sfP}\lt(\sfQ, \sfP_0\rt) : \sfQ \in \cH_{\Psi, e}^{s'_e} \rt\}} = \Lambda_{e, \rm iso}(T'_\sfP),
\end{align*}
where the last inequality follows because $\sfP \in \cH_{\Psi, e}$. Therefore, 
$$\Lambda_{e, \rm iso}(T'_\sfP) > 1 \quad \text{if and only if} \quad e \in \cA(\sfP),$$
which implies that $D_{\rm iso}(T'_{\sfP}) = \cA(\sfP) \in \Psi$.
Thus, the proof is complete in view of \eqref{ub21}.
\end{proof}

\section{Proofs regarding Asymptotic Optimality}\label{app:asymp_opt}

For any random variable $Y$ with underlying probability distribution $\sfP$ and any event $\Gamma$, we introduce the following notation:
$$\sfE_\sfP\lt[Y; \Gamma\rt] := \int_{\Gamma} Yd\sfP,$$
which will be extensively used in this section.

\subsection{Proof of Theorem \ref{P1P2}}\label{pf:P1P2}

First, we establish a lemma which is critical in establishing an asymptotic lower bound on the optimal performance.
 
\begin{lemma}\label{lem:P1P2}
Suppose that \eqref{max_series} holds,  $\emptyset \notin \Psi$ and $\sfP \in \cP_\Psi$.
Then  as $\gamma, \delta \to 0$ we have 
\begin{align*}
\inf \lt\{ \sfE_{\sfP}[T]:\; 
(T, D) \in \cC_{\Psi}(\gamma, \delta) \rt\} \; \gtrsim \; \max_{e \in \cA(\sfP)} \lt\{ \frac{|\log\gamma|}{  \cI(\sfP, \cH_{\Psi,e}) }\rt\}  \bigvee 
\max_{e \notin \cA(\sfP)} \lt\{  \frac{|\log\delta|}{  \cI(\sfP, \cG_{\Psi,e} )} \rt\} .
\end{align*}
When in particular   $\Psi = \Psi_{m, m}$ for some $0<m < |\cK|$, 
then  as $\gamma, \delta \to 0$ we have 
\begin{equation*} 
\inf \lt\{ \sfE_{\sfP}[T]:\; 
(T, D) \in \cC_{\Psi}(\gamma, \delta) \rt\} \; \gtrsim \;
 \frac{|\log(\gamma \wedge \delta)|}{\cI(\sfP, \cP_\Psi \setminus \{\sfP\})}.
\end{equation*}
\end{lemma}

\begin{proof}
Fix $\sfP \in \cP_\Psi$ and denote the quantity $\inf \lt\{ \sfE_{\sfP}[T]:\; 
(T, D) \in \cC_{\Psi}(\gamma, \delta) \rt\}$ by $N(\gamma, \delta)$. Then we need to show that as $\gamma, \delta \to 0$,
$$N(\gamma, \delta) \gtrsim \max_{e \in \cA(\sfP)} \lt\{ \frac{|\log\gamma|}{  \cI(\sfP, \cH_{\Psi,e}) }\rt\}  \bigvee 
\max_{e \notin \cA(\sfP)} \lt\{  \frac{|\log\delta|}{  \cI(\sfP, \cG_{\Psi,e} )} \rt\}.$$
We set
$$L_{\gamma} := \max_{e \in \cA(\sfP)} \lt\{ \frac{|\log\gamma|}{  \cI(\sfP, \cH_{\Psi,e}) }\rt\},\quad \text{and} \quad L_{\delta} := \max_{e \notin \cA(\sfP)} \lt\{  \frac{|\log\delta|}{  \cI(\sfP, \cG_{\Psi,e} )} \rt\}.$$
By Markov's inequality, for any stopping time $T$ and $\gamma, \delta, \epsilon \in (0, 1)$, 
$$\sfE_\sfP[T] \geq (1-\epsilon)L_{\gamma}\sfP\lt(T \geq (1-\epsilon)L_{\gamma}\rt) \quad \text{and} \quad \sfE_\sfP[T] \geq (1-\epsilon)L_{\delta}\sfP\lt(T \geq (1-\epsilon)L_{\delta}\rt).$$
Thus, it suffices to show for every $\epsilon \in (0, 1)$ we have
\begin{subequations}
\label{lb_iso}
\begin{align}%\label{lb21_iso}
%&\liminf_{\alpha, \beta \to 0} \inf_{(T, D) \in \cC_\Psi(\alpha, \beta, \gamma, \delta)} \sfP\lt(T \geq (1-\epsilon)L_{\alpha}\rt) \geq 1,\\ 
\label{lb22_iso}
&\liminf_{\gamma, \delta \to 0} \inf_{(T, D) \in \cC_\Psi(\gamma, \delta)} \sfP\lt(T \geq (1-\epsilon)L_{\gamma}\rt) \geq 1,\quad \text{and}\\ \label{lb23_iso}
&\liminf_{\gamma, \delta \to 0} \inf_{(T, D) \in \cC_\Psi(\gamma, \delta)} \sfP\lt(T \geq (1-\epsilon)L_{\delta}\rt) \geq 1,
\end{align}
\end{subequations}
since this will further imply that 
$$\liminf_{\gamma, \delta \to 0} \; N(\gamma, \delta)/(L_{\gamma} \vee L_{\delta}) \geq (1-\epsilon),$$ 
and the desired result will follow by letting $\epsilon \to 0$. 
In order to prove \eqref{lb_iso}, we start by fixing arbitrary $(T, D) \in \cC_\Psi(\gamma, \delta)$ and $\gamma, \delta, \epsilon \in (0, 1)$. 
To show \eqref{lb22_iso}, consider the probability distribution $\sfQ_1^*$ which attains
$$\min \lt\{\cI(\sfP, \sfQ) : \sfQ \in \cup_{e \in \cA(\sfP)} \cH_{\Psi, e}\rt\}.$$
Thus, there exists $e \in \cA(\sfP)$ such that $\sfQ_1^* \in \cH_{\Psi, e}$, i.e., $e \notin \cA(\sfQ_1^*)$. Then,
\begin{align}\notag
\sfP\lt(T \geq (1-\epsilon)L_{\gamma}\rt) &\geq \sfP\lt(e \in D, T \geq (1-\epsilon)L_{\gamma}\rt)\\ \notag
&= \sfP\lt(e \in D\rt) - \sfP\lt(e \in D, T < (1-\epsilon)L_{\gamma}\rt)\\ \notag
&= 1 - \sfP\lt(e \notin D\rt) - \sfP\lt(e \in D, T < (1-\epsilon)L_{\gamma}\rt)\\ \label{lb4_iso}
&\geq 1- \delta - \sfP\lt(e \in D, T < (1-\epsilon)L_{\gamma}\rt)
\end{align}
where the last inequality follows from the fact that 
$$\{e \notin D\} \subseteq \{\cA(\sfP) \setminus D \neq \emptyset, D \neq \emptyset\}$$ and the error control that $\sfP\lt(\cA(\sfP) \setminus D \neq \emptyset, D \neq \emptyset\rt) \leq \delta$. We can now decompose the third term as
$$\sfP\lt(e \in D, T < (1-\epsilon)L_{\gamma}\rt) \leq p_1 + p_2,$$
such that 
\begin{align*}
p_1 &:= \sfP\lt(e \in D, \Lambda_T(\sfP, \sfQ_1^*) < \frac{\eta}{\gamma}\rt)\quad \text{and}\\ p_2 &:= \sfP\lt(T < (1-\epsilon)L_{\gamma}, \Lambda_T(\sfP, \sfQ_1^*) \geq \frac{\eta}{\gamma}\rt),
\end{align*}
where $\eta$ is an arbitrary constant in $(0, 1)$. For the first term, by a change of measure from $\sfP$ to $\sfQ_1^*$, we have
\begin{align*}
p_1 &= \sfE_{\sfQ_1^*}\lt[\Lambda_T(\sfP, \sfQ_1^*); e \in D, \Lambda_T(\sfP, \sfQ_1^*) < \frac{\eta}{\gamma}\rt]\\
&\leq \frac{\eta}{\gamma}\sfQ_1^*\lt(e \in D\rt) \leq \eta,
\end{align*}
where the last inequality follows from the fact that 
$$\{e \in D\} \subseteq \{D \setminus \cA(\sfQ_1^*) \neq \emptyset\}$$ and the error control that $\sfQ_1^*\lt(D \setminus \cA(\sfQ_1^*) \neq \emptyset\rt) \leq \gamma$. For the second term,
\begin{align*}
p_2 \leq \sfP\lt(\frac{1}{(1-\epsilon)L_{\gamma}}\max_{1 \leq n \leq (1-\epsilon)L_{\gamma}} \log\Lambda_n(\sfP, \sfQ_1^*) \geq \frac{\log\eta}{(1-\epsilon)L_{\gamma}} + \frac{\cI(\sfP, \sfQ_1^*)}{(1-\epsilon)}\rt) =: \xi_{\epsilon, \eta}(\gamma)
\end{align*}
%\begin{align*}
%\epsilon_{\epsilon, \eta}(\beta, \delta) := \sfP\lt(T < (1-\epsilon)\frac{|\log\beta|}{\cI(\sfP, \sfQ^*)}, \log\Lambda(\sfP, \sfQ^*; T) \geq \log\eta + |\log\beta|\rt)
%\end{align*}
Due to \eqref{max_series}, %the SLLN \eqref{SLLN}, we have
%$$\sfP\lt(\lim_{n \to \infty}\frac{1}{n}\log\Lambda(\sfP, \sfQ_1^*, n) = \cI(\sfP, \sfQ_1^*)\rt) = 1.$$
%Therefore, it follows that 
$\xi_{\epsilon, \eta}(\gamma) \to 0$ as $\gamma \to 0$. 

Combining everything together in \eqref{lb4_iso} we have
\begin{align*}
\sfP(T \geq (1-\epsilon)L_{\gamma}) \geq 1 - \delta - \eta - \xi_{\epsilon, \eta}(\gamma).
\end{align*}
Since $(T, D) \in \cC_\Psi(\gamma, \delta)$ and $\gamma \in (0, 1)$ were chosen arbitrarily, taking infimum over $(T, D)$ and letting $\gamma, \delta \to 0$ we obtain
\begin{equation*}
\liminf_{\gamma, \delta \to 0} \inf_{(T, D) \in \cC_\Psi(\gamma, \delta)} \sfP\lt(T \geq (1-\epsilon)L_{\gamma}\rt) \geq 1 - \eta.
\end{equation*}
Finally, letting $\eta \to 0$ we obtain \eqref{lb22_iso}.

To show \eqref{lb23_iso}, consider the probability distribution $\sfQ_2^*$ which attains
$$\min \lt\{\cI(\sfP, \sfQ) : \sfQ \in \cup_{e \notin \cA(\sfP)} \cG_{\Psi, e}\rt\}.$$
Thus, there exists $e \notin \cA(\sfP)$ such that $\sfQ_2^* \in \cG_{\Psi, e}$, i.e., $e \in \cA(\sfQ_2^*)$. Then,
\begin{align}\notag
\sfP\lt(T \geq (1-\epsilon)L_{\delta}\rt) &\geq \sfP\lt(e \notin D, D \neq \emptyset, T \geq (1-\epsilon)L_{\delta}\rt)\\ \notag
&= \sfP\lt(e \notin D, D \neq \emptyset\rt) - \sfP\lt(e \notin D, D \neq \emptyset, T < (1-\epsilon)L_{\delta}\rt)\\ \notag
&= 1 - \sfP\lt(e \in D, D \neq \emptyset\rt) - \sfP\lt(e \notin D, D \neq \emptyset, T < (1-\epsilon)L_{\delta}\rt)\\ \label{lb5_iso}
&\geq 1 - \gamma - \sfP\lt(e \notin D, D \neq \emptyset, T < (1-\epsilon)L_{\delta}\rt)
\end{align}
where the last inequality follows from the desired error control that $\sfP\lt(D \setminus \cA(\sfP) \neq \emptyset\rt) \leq \gamma$. We can now decompose the third term as
$$\sfP\lt(e \notin D, D \neq \emptyset, T < (1-\epsilon)L_{\delta}\rt) \leq p_1 + p_2,$$
such that 
\begin{align*}
p_1 &:= \sfP\lt(e \notin D, D \neq \emptyset, \Lambda_T(\sfP, \sfQ_2^*) < \frac{\eta}{\delta}\rt)\quad \text{and}\\ p_2 &:= \sfP\lt(T < (1-\epsilon)L_{\delta}, \Lambda_T(\sfP, \sfQ_2^*) \geq \frac{\eta}{\delta}\rt),
\end{align*}
where $\eta$ is an arbitrary constant in $(0, 1)$. For the first term, by a change of measure from $\sfP$ to $\sfQ_2^*$, we have
\begin{align*}
p_1 &= \sfE_{\sfQ_2^*}\lt[\Lambda_T(\sfP, \sfQ_2^*); e \notin D, D \neq \emptyset, \Lambda_T(\sfP, \sfQ_2^*) < \frac{\eta}{\delta}\rt]\\
&\leq \frac{\eta}{\delta}\sfQ_2^*\lt(e \notin D, D \neq \emptyset\rt) \leq \eta,
\end{align*}
where the last inequality follows from the desired error control that $\sfQ_2^*\lt(\cA(\sfQ_2^*) \setminus D \neq \emptyset, D \neq \emptyset\rt) \leq \delta$. For the second term, we can bound it above by the similar manner as shown previously
\begin{align*}
p_2 \leq \sfP\lt(\frac{1}{(1-\epsilon)L_{\delta}}\max_{1 \leq n \leq (1-\epsilon)L_{\delta}} \log\Lambda_n(\sfP, \sfQ_2^*) \geq \frac{\log\eta}{(1-\epsilon)L_{\delta}} + \frac{\cI(\sfP, \sfQ_2^*)}{(1-\epsilon)}\rt) =: \xi_{\epsilon, \eta}(\delta)
\end{align*}
%\begin{align*}
%\epsilon_{\epsilon, \eta}(\beta, \delta) := \sfP\lt(T < (1-\epsilon)\frac{|\log\beta|}{\cI(\sfP, \sfQ^*)}, \log\Lambda(\sfP, \sfQ^*; T) \geq \log\eta + |\log\beta|\rt)
%\end{align*}
Due to \eqref{max_series}, %the SLLN \eqref{SLLN}, we have
%$$\sfP\lt(\lim_{n \to \infty}\frac{1}{n}\log\Lambda(\sfP, \sfQ_2^*, n) = \cI(\sfP, \sfQ_2^*)\rt) = 1.$$
%Therefore, it follows that 
$\xi_{\epsilon, \eta}(\delta) \to 0$ as $\delta \to 0$. 

Combining everything together in \eqref{lb5_iso} we have
\begin{align*}
\sfP(T \geq (1-\epsilon)L_{\delta}) \geq 1 - \gamma - \eta - \xi_{\epsilon, \eta}(\delta).
\end{align*}
Since $(T, D) \in \cC_\Psi(\gamma, \delta)$ and $\delta \in (0, 1)$ were chosen arbitrarily, taking infimum over $(T, D)$ and letting $\gamma, \delta \to 0$ we obtain
\begin{equation*}
\liminf_{\gamma, \delta \to 0} \inf_{(T, D) \in \cC_\Psi(\gamma, \delta)} \sfP\lt(T \geq (1-\epsilon)L_{\delta}\rt) \geq 1 - \eta.
\end{equation*}
Finally, letting $\eta \to 0$ we obtain \eqref{lb23_iso}, which along with \eqref{lb22_iso} completes the first part of the proof.

For the second part, we first show that when $\Psi = \Psi_{m, m}$ we have
$$\cup_{e \in \cA(\sfP)} \cH_{\Psi, e} = \cup_{e \notin \cA(\sfP)} \cG_{\Psi, e} = \cP_\Psi \setminus \{\sfP\}.$$
Indeed, the above holds since for any $\sfQ \in \cP_\Psi \setminus \{\sfP\}$, as $|\cA(\sfP)| = |\cA(\sfQ)| = m$, there exists $e \in \cA(\sfP)$ such that $e \notin \cA(\sfQ)$, and also there exists $e \notin \cA(\sfP)$ such that $e \in \cA(\sfQ)$, i.e., both
$$\sfQ \in \cup_{e \in \cA(\sfP)} \cH_{\Psi, e} \quad \text{and} \quad \sfQ \in \cup_{e \notin \cA(\sfP)} \cG_{\Psi, e}$$
hold equivalently. 

Combined with this and using the fact that, for any $(T, D) \in \cC_\Psi(\gamma, \delta)$,
$\sfP \lt(D \setminus \cA(\sfP) \neq \emptyset\rt) = \sfP\lt(\cA(\sfP) \setminus D \neq \emptyset  \rt) = \sfP\lt(D \neq \cA(\sfP)\rt) \leq \gamma \wedge \delta$, one can follow the similar technique as above to complete the proof.
\end{proof}

\begin{proof}[Proof of Theorem \ref{P1P2}]
Since both \eqref{condP1} and \eqref{condP2} hold, from Theorem \ref{ub_I} as $\gamma, \delta \to 0$, we have
\begin{align*}
&\inf \lt\{ \sfE_{\sfP}[T]:\; (T, D) \in \cC_{\Psi}(\gamma, \delta) \rt\}\\ 
&\leq \sfE_\sfP\lt[T^*\rt] \lesssim \max_{e \in \cA(\sfP)} \lt\{ \frac{|\log\gamma|}{  \cI(\sfP, \cH_{\Psi,e}) }\rt\}  \bigvee 
\max_{e \notin \cA(\sfP)} \lt\{  \frac{|\log\delta|}{  \cI(\sfP, \cG_{\Psi,e} )} \rt\}.
\end{align*}
Thus, the result follows by using Lemma \ref{lem:P1P2}. When $\Psi = \Psi_{m, m}$ the result can be established in a similar fashion. 
\end{proof}

\subsection{Some important lemmas}

Now we prove three important lemmas which will be useful in the proofs of Theorems \ref{PminusP0}, \ref{P0} and \ref{P0P1P2}. The next lemma provides an asymptotic lower bound on the optimal performance corresponding to joint detection and isolation.

\begin{lemma}\label{lem:DI_lb}
Suppose that \eqref{max_series} holds, $\emptyset \in \Psi$ and let $\sfP \in \cP_\Psi$.
\begin{enumerate}
\item[(i)] If $\sfP \in \cH_0$, then as $\alpha, \beta \to 0$ while $\gamma$ and $\delta$ are either fixed or go to $0$, we have
$$ \inf \lt\{ \sfE_{\sfP}[T]:\; 
(T, D) \in \cC_{\Psi}(\alpha, \beta,\gamma, \delta) \rt\} \; \gtrsim \; \max_{e \in \cK}  \lt\{ \frac{|\log\beta|}{ \cI \lt( \sfP , \cG_{\Psi,e} \rt)} \rt\}.$$
\item[(ii)] If $\sfP \notin \cH_0$, then as $\alpha, \beta \to 0$ while $\gamma$ and $\delta$ are either fixed or go to $0$, we have
$$ \inf \lt\{ \sfE_{\sfP}[T]:\; 
(T, D) \in \cC_{\Psi}(\alpha, \beta,\gamma, \delta) \rt\} \; \gtrsim \; \frac{|\log\alpha|}{ \cI \lt( \sfP , \cH_0 \rt)}.$$
\item[(iii)] If $\sfP \notin \cH_0$, then as $\alpha, \beta, \gamma, \delta \to 0$, we have
\begin{align*}
&\inf \lt\{ \sfE_{\sfP}[T]:\; 
(T, D) \in \cC_{\Psi}(\alpha, \beta,\gamma, \delta) \rt\}\\ 
&\gtrsim \; \frac{|\log\alpha|}{\cI(\sfP, \cH_0)} \bigvee \max_{e \in \cA(\sfP)}  \lt\{\frac{|\log\gamma|}{ \cI(\sfP, \cH_{\Psi,e} ) } \rt\} \bigvee   \max_{e \notin \cA(\sfP)}  \lt\{ \frac{|\log\delta|}{ \cI(\sfP, \cG_{\Psi,e} ) } \rt\}.
\end{align*}
\end{enumerate}
\end{lemma}
\begin{proof}
In order to prove (i) fix $\sfP \in \cH_0$ and denote the quantity $ \inf \lt\{ \sfE_{\sfP}[T]:\; 
(T, D) \in \cC_{\Psi}(\alpha, \beta,\gamma, \delta) \rt\}$ by $N(\alpha, \beta, \gamma, \delta)$. Then we need to to show that as $\alpha, \beta \to 0$, while  $\gamma$ and $\delta$ are either  fixed or go to $0$,
$$N(\alpha, \beta, \gamma ,\delta) \; \gtrsim \; \max_{e \in \cK}   \lt\{ \frac{|\log \beta|}{\cI \lt(\sfP, \cG_{\Psi,e}\rt)}  \rt\}.$$
We set $$L_{\beta} := \max_{e \in \cK}   \lt\{ \frac{|\log \beta|}{\cI \lt(\sfP, \cG_{\Psi,e}\rt)}  \rt\},\quad \beta \in (0, 1).$$
By Markov's inequality, for any stopping time $T$, and $\beta, \epsilon \in (0, 1)$, 
$$\sfE_\sfP[T] \geq (1-\epsilon)L_{\beta}\sfP\lt(T \geq (1-\epsilon)L_{\beta}\rt).$$
Thus, it suffices to show for every $\epsilon \in (0, 1)$ we have
\begin{equation}\label{lb11}
\liminf_{\alpha, \beta \to 0} \inf_{(T, D) \in \cC_\Psi(\alpha, \beta, \gamma, \delta)} \sfP\lt(T \geq (1-\epsilon)L_{\beta}\rt) \geq 1,
\end{equation}
since this will further imply that 
$$\liminf_{\alpha, \beta \to 0} \; N(\alpha, \beta, \gamma, \delta)/L_{\beta} \geq (1-\epsilon),$$ 
and the desired result will follow by letting $\epsilon \to 0$. In order to prove \eqref{lb11}, we start by fixing arbitrary $(T, D) \in \cC_\Psi(\alpha, \beta, \gamma, \delta)$ and $\beta, \epsilon \in (0, 1)$. Consider the probability distribution $\sfQ^*$ which attains
$$\min\{\cI(\sfP, \sfQ) : \sfQ \in \cup_{e \in \cK} \cG_{\Psi, e}\}.$$ 
Thus, there exists $e \in \cK$ such that $\sfQ^* \in \cG_{\Psi, e}$. i.e., $e \in \cA(\sfQ^*)$, which means $\sfQ^* \in \cP_\Psi \setminus \cH_0$.
Then,
\begin{align}\notag
\sfP\lt(T \geq (1-\epsilon)L_{\beta}\rt) &\geq \sfP\lt(D = \emptyset, T \geq (1-\epsilon)L_{\beta}\rt)\\ \notag
&= \sfP\lt(D = \emptyset\rt) - \sfP\lt(D = \emptyset, T < (1-\epsilon)L_{\beta}\rt)\\ \notag
&= 1 - \sfP\lt(D \neq \emptyset\rt) - \sfP\lt(D = \emptyset, T < (1-\epsilon)L_{\beta}\rt)\\ \label{lb12}
&\geq 1- \alpha - \sfP\lt(D = \emptyset, T < (1-\epsilon)L_{\beta}\rt)
\end{align}
where the last inequality follows from the desired error control that $\sfP(D \neq \emptyset) \leq \alpha$. We can now decompose the third term as
$$\sfP\lt(D = \emptyset, T < (1-\epsilon)L_{\beta}\rt) \leq p_1 + p_2,$$
such that 
\begin{align*}
p_1 &:= \sfP\lt(D = \emptyset, \Lambda_T(\sfP, \sfQ^*) < \frac{\eta}{\beta}\rt)\quad \text{and}\\ p_2 &:= \sfP\lt(T < (1-\epsilon)L_{\beta}, \Lambda_T(\sfP, \sfQ^*) \geq \frac{\eta}{\beta}\rt),
\end{align*}
where $\eta$ is an arbitrary constant in $(0, 1)$. For the first term, by a change of measure from $\sfP$ to $\sfQ^*$, we have
\begin{align*}
p_1 &= \sfE_{\sfQ^*}\lt[\Lambda_T(\sfP, \sfQ^*); D = \emptyset, \Lambda_T(\sfP, \sfQ^*) < \frac{\eta}{\beta}\rt]\\
&\leq \frac{\eta}{\beta}\sfQ^*\lt(D = \emptyset\rt) \leq \eta,
\end{align*}
where the last inequality follows from the error control that $\sfQ^*\lt(D = \emptyset\rt) \leq \beta$. For the second term, we can bound it above by
\begin{align*}
p_2 &= \sfP\lt(T < (1-\epsilon)L_{\beta},  \log\Lambda_T(\sfP, \sfQ^*) \geq \log\eta + |\log\beta|\rt)\\
&\leq \sfP\lt(\max_{1 \leq n \leq (1-\epsilon)L_{\beta}} \log\Lambda_n(\sfP, \sfQ^*) \geq \log\eta + |\log\beta|\rt)\\
&\leq \sfP\lt(\frac{1}{(1-\epsilon)L_{\beta}}\max_{1 \leq n \leq (1-\epsilon)L_{\beta}} \log\Lambda_n(\sfP, \sfQ^*) \geq \frac{\log\eta}{(1-\epsilon)L_{\beta}} + \frac{\cI(\sfP, \sfQ^*)}{(1-\epsilon)}\rt) =: \xi_{\epsilon, \eta}(\beta)
\end{align*}
%\begin{align*}
%\epsilon_{\epsilon, \eta}(\beta, \delta) := \sfP\lt(T < (1-\epsilon)\frac{|\log\beta|}{\cI(\sfP, \sfQ^*)}, \log\Lambda(\sfP, \sfQ^*; T) \geq \log\eta + |\log\beta|\rt)
%\end{align*}
Due to \eqref{max_series}, %the SLLN \eqref{SLLN}, we have
%$$\sfP\lt(\lim_{n \to \infty}\frac{1}{n}\log \Lambda(\sfP, \sfQ^*, n) = \cI(\sfP, \sfQ^*)\rt) = 1.$$
%Therefore, it follows that 
$\xi_{\epsilon, \eta}(\beta) \to 0$ as $\beta \to 0$. 

Combining everything together in \eqref{lb12} we have
\begin{align*}
\sfP(T \geq (1-\epsilon)L_{\beta}) \geq 1 - \alpha - \eta - \xi_{\epsilon, \eta}(\beta).
\end{align*}
Since $(T, D) \in \cC_\Psi(\alpha, \beta, \gamma, \delta)$ and $\beta \in (0, 1)$ were chosen arbitrarily, taking infimum over $(T, D)$ and letting $\alpha, \beta \to 0$ we obtain
\begin{equation*}
\liminf_{\alpha, \beta \to 0} \inf_{(T, D) \in \cC_\Psi(\alpha, \beta, \gamma, \delta)} \sfP\lt(T \geq (1-\epsilon)L_{\beta}\rt) \geq 1 - \eta.
\end{equation*}
Finally, letting $\eta \to 0$ we obtain \eqref{lb11}, which completes the proof of (i).

Next, in order to prove (ii) and (iii), we fix any arbitrary $\sfP \in \cP_\Psi \setminus \cH_0$ and set
$$L_{\alpha} := \frac{|\log\alpha|}{\cI(\sfP, \cH_0)},\quad L_{\gamma} := \max_{e \in \cA(\sfP)} \; \lt\{\frac{|\log\gamma|}{ \cI(\sfP, \cH_{\Psi,e} ) } \rt\},\quad \text{and} \quad L_{\delta} := \max_{e \notin \cA(\sfP)}  \lt\{ \frac{|\log\delta|}{ \cI(\sfP, \cG_{\Psi,e} ) } \rt\}.$$
In a similar manner, to prove the second and third part of the proof, we will show that for every $\epsilon \in (0, 1)$ we have
\begin{subequations}
\label{lb2ndpart}
\begin{align}\label{lb21}
&\liminf_{\alpha, \beta \to 0} \inf_{(T, D) \in \cC_\Psi(\alpha, \beta, \gamma, \delta)} \sfP\lt(T \geq (1-\epsilon)L_{\alpha}\rt) \geq 1,\\ \label{lb22}
&\liminf_{\beta, \gamma, \delta \to 0} \inf_{(T, D) \in \cC_\Psi(\alpha, \beta, \gamma, \delta)} \sfP\lt(T \geq (1-\epsilon)L_{\gamma}\rt) \geq 1,\quad \text{and}\\ \label{lb23}
&\liminf_{\beta, \gamma, \delta \to 0} \inf_{(T, D) \in \cC_\Psi(\alpha, \beta, \gamma, \delta)} \sfP\lt(T \geq (1-\epsilon)L_{\delta}\rt) \geq 1
\end{align}
\end{subequations}
In order to prove \eqref{lb2ndpart}, we start by fixing arbitrary $(T, D) \in \cC_\Psi(\alpha, \beta, \gamma, \delta)$ and $\alpha, \beta, \gamma, \delta, \epsilon \in (0, 1)$. 
To show \eqref{lb21}, consider the probability distribution $\sfQ_0^*$ which attains
$$\min\{\cI(\sfP, \sfQ) : \sfQ \in \cH_0\}.$$ Then,
\begin{align}\notag
\sfP\lt(T \geq (1-\epsilon)L_{\alpha}\rt) &\geq \sfP\lt(D \neq \emptyset, T \geq (1-\epsilon)L_{\alpha}\rt)\\ \notag
&= \sfP\lt(D \neq \emptyset\rt) - \sfP\lt(D \neq \emptyset, T < (1-\epsilon)L_{\alpha}\rt)\\ \notag
&= 1 - \sfP\lt(D = \emptyset\rt) - \sfP\lt(D \neq \emptyset, T < (1-\epsilon)L_{\alpha}\rt)\\ \label{lb3}
&\geq 1- \beta - \sfP\lt(D \neq \emptyset, T < (1-\epsilon)L_{\alpha}\rt)
\end{align}
where the last inequality follows from the desired error control that $\sfP(D = \emptyset) \leq \beta$. We can now decompose the third term as
$$\sfP\lt(D \neq \emptyset, T < (1-\epsilon)L_{\alpha}\rt) \leq p_1 + p_2,$$
such that 
\begin{align*}
p_1 &:= \sfP\lt(D \neq \emptyset, \Lambda_T(\sfP, \sfQ_0^*) < \frac{\eta}{\alpha}\rt)\quad \text{and}\\ p_2 &:= \sfP\lt(T < (1-\epsilon)L_{\alpha}, \Lambda_T(\sfP, \sfQ_0^*) \geq \frac{\eta}{\alpha}\rt),
\end{align*}
where $\eta$ is an arbitrary constant in $(0, 1)$. For the first term, by a change of measure from $\sfP$ to $\sfQ_0^*$, we have
\begin{align*}
p_1 &= \sfE_{\sfQ_0^*}\lt[\Lambda_T(\sfP, \sfQ_0^*); D \neq \emptyset, \Lambda_T(\sfP, \sfQ_0^*) < \frac{\eta}{\alpha}\rt]\\
&\leq \frac{\eta}{\alpha}\sfQ_0^*\lt(D \neq \emptyset\rt) \leq \eta,
\end{align*}
where the last inequality follows from the error control that $\sfQ_0^*\lt(D \neq \emptyset\rt) \leq \alpha$. For the second term, we can bound it above by the similar manner as shown previously
\begin{align*}
p_2 \leq \sfP\lt(\frac{1}{(1-\epsilon)L_{\alpha}}\max_{1 \leq n \leq (1-\epsilon)L_{\alpha}} \log\Lambda_n(\sfP, \sfQ_0^*) \geq \frac{\log\eta}{(1-\epsilon)L_{\alpha}} + \frac{\cI(\sfP, \sfQ_0^*)}{(1-\epsilon)}\rt) =: \xi_{\epsilon, \eta}(\alpha)
\end{align*}
%\begin{align*}
%\epsilon_{\epsilon, \eta}(\beta, \delta) := \sfP\lt(T < (1-\epsilon)\frac{|\log\beta|}{\cI(\sfP, \sfQ^*)}, \log\Lambda(\sfP, \sfQ^*; T) \geq \log\eta + |\log\beta|\rt)
%\end{align*}
Due to \eqref{max_series}, %the SLLN \eqref{SLLN}, we have
%$$\sfP\lt(\lim_{n \to \infty}\frac{1}{n}\log\Lambda(\sfP, \sfQ_0^*, n) = \cI(\sfP, \sfQ_0^*)\rt) = 1.$$
%Therefore, it follows that 
$\xi_{\epsilon, \eta}(\alpha) \to 0$ as $\alpha \to 0$. 

Combining everything together in \eqref{lb3} we have
\begin{align*}
\sfP(T \geq (1-\epsilon)L_{\alpha}) \geq 1 - \beta - \eta - \xi_{\epsilon, \eta}(\alpha).
\end{align*}
Since $(T, D) \in \cC_\Psi(\alpha, \beta, \gamma, \delta)$ and $\alpha \in (0, 1)$ were chosen arbitrarily, taking infimum over $(T, D)$ and letting $\alpha, \beta \to 0$ we obtain
\begin{equation*}
\liminf_{\alpha, \beta \to 0} \inf_{(T, D) \in \cC_\Psi(\alpha, \beta, \gamma, \delta)} \sfP\lt(T \geq (1-\epsilon)L_{\alpha}\rt) \geq 1 - \eta.
\end{equation*}
Finally, letting $\eta \to 0$ we obtain \eqref{lb21}, which completes the proof of (ii).

To show \eqref{lb22}, consider the probability distribution $\sfQ_1$ which attains
$$\min \lt\{\cI(\sfP, \sfQ) : \sfQ \in \cup_{e \in \cA(\sfP)} \cH_{\Psi, e}\rt\}.$$
Thus, there exists $e \in \cA(\sfP)$ such that $\sfQ_1^* \in \cH_{\Psi, e}$, i.e., $e \notin \cA(\sfQ_1^*)$. Then,
\begin{align}\notag
\sfP\lt(T \geq (1-\epsilon)L_{\gamma}\rt) &\geq \sfP\lt(e \in D, T \geq (1-\epsilon)L_{\gamma}\rt)\\ \notag
&= \sfP\lt(e \in D\rt) - \sfP\lt(e \in D, T < (1-\epsilon)L_{\gamma}\rt)\\ \notag
&= 1 - \sfP\lt(e \notin D\rt) - \sfP\lt(e \in D, T < (1-\epsilon)L_{\gamma}\rt)\\ \label{lb4}
&\geq 1- \beta - \delta - \sfP\lt(e \in D, T < (1-\epsilon)L_{\gamma}\rt)
\end{align}
where the last inequality follows from the fact that 
$$\{e \notin D\} \subseteq \{D = \emptyset\} \cup \{\cA(\sfP) \setminus D \neq \emptyset, D \neq \emptyset\}$$ and the error controls that $\sfP(D = \emptyset) \leq \beta$ and $\sfP\lt(\cA(\sfP) \setminus D \neq \emptyset, D \neq \emptyset\rt) \leq \delta$. We can now decompose the third term as
$$\sfP\lt(e \in D, T < (1-\epsilon)L_{\gamma}\rt) \leq p_1 + p_2,$$
such that 
\begin{align*}
p_1 &:= \sfP\lt(e \in D, \Lambda_T(\sfP, \sfQ_1^*) < \frac{\eta}{\gamma}\rt)\quad \text{and}\\ p_2 &:= \sfP\lt(T < (1-\epsilon)L_{\gamma}, \Lambda_T(\sfP, \sfQ_1^*) \geq \frac{\eta}{\gamma}\rt),
\end{align*}
where $\eta$ is an arbitrary constant in $(0, 1)$. For the first term, by a change of measure from $\sfP$ to $\sfQ_1^*$, we have
\begin{align*}
p_1 &= \sfE_{\sfQ_1^*}\lt[\Lambda_T(\sfP, \sfQ_1^*); e \in D, \Lambda_T(\sfP, \sfQ_1^*) < \frac{\eta}{\gamma}\rt]\\
&\leq \frac{\eta}{\gamma}\sfQ_1^*\lt(e \in D\rt) \leq \eta,
\end{align*}
where the last inequality follows from the fact that 
$$\{e \in D\} \subseteq \{D \setminus \cA(\sfQ_1^*) \neq \emptyset\}$$ and the error control that $\sfQ_1^*\lt(D \setminus \cA(\sfQ_1^*) \neq \emptyset\rt) \leq \gamma$. For the second term, we can bound it above by the similar manner as shown previously
\begin{align*}
p_2 \leq \sfP\lt(\frac{1}{(1-\epsilon)L_{\gamma}}\max_{1 \leq n \leq (1-\epsilon)L_{\gamma}} \log\Lambda_n(\sfP, \sfQ_1^*) \geq \frac{\log\eta}{(1-\epsilon)L_{\gamma}} + \frac{\cI(\sfP, \sfQ_1^*)}{(1-\epsilon)}\rt) =: \xi_{\epsilon, \eta}(\gamma)
\end{align*}
%\begin{align*}
%\epsilon_{\epsilon, \eta}(\beta, \delta) := \sfP\lt(T < (1-\epsilon)\frac{|\log\beta|}{\cI(\sfP, \sfQ^*)}, \log\Lambda(\sfP, \sfQ^*; T) \geq \log\eta + |\log\beta|\rt)
%\end{align*}
Due to \eqref{max_series}, %the SLLN \eqref{SLLN}, we have
%$$\sfP\lt(\lim_{n \to \infty}\frac{1}{n}\log\Lambda(\sfP, \sfQ_1^*, n) = \cI(\sfP, \sfQ_1^*)\rt) = 1.$$
%Therefore, it follows that 
$\xi_{\epsilon, \eta}(\gamma) \to 0$ as $\gamma \to 0$. 

Combining everything together in \eqref{lb4} we have
\begin{align*}
\sfP(T \geq (1-\epsilon)L_{\gamma}) \geq 1 - \beta - \delta - \eta - \xi_{\epsilon, \eta}(\gamma).
\end{align*}
Since $(T, D) \in \cC_\Psi(\alpha, \beta, \gamma, \delta)$ and $\gamma \in (0, 1)$ were chosen arbitrarily, taking infimum over $(T, D)$ and letting $\beta,\gamma, \delta \to 0$ we obtain
\begin{equation*}
\liminf_{\beta, \gamma, \delta \to 0} \inf_{(T, D) \in \cC_\Psi(\alpha, \beta, \gamma, \delta)} \sfP\lt(T \geq (1-\epsilon)L_{\gamma}\rt) \geq 1 - \eta.
\end{equation*}
Finally, letting $\eta \to 0$ we obtain \eqref{lb22}.

To show \eqref{lb23}, consider the probability distribution $\sfQ_2^*$ which attains
$$\min \lt\{\cI(\sfP, \sfQ) : \sfQ \in \cup_{e \notin \cA(\sfP)} \cG_{\Psi, e}\rt\}.$$
Thus, there exists $e \notin \cA(\sfP)$ such that $\sfQ_2^* \in \cG_{\Psi, e}$, i.e., $e \in \cA(\sfQ_2^*)$.Then,
\begin{align}\notag
\sfP\lt(T \geq (1-\epsilon)L_{\delta}\rt) &\geq \sfP\lt(e \notin D, D \neq \emptyset, T \geq (1-\epsilon)L_{\delta}\rt)\\ \notag
&= \sfP\lt(e \notin D, D \neq \emptyset\rt) - \sfP\lt(e \notin D, D \neq \emptyset, T < (1-\epsilon)L_{\delta}\rt)\\ \notag
&= 1 - \sfP\lt(\lt(e \in D\rt) \cup \lt(D = \emptyset\rt)\rt) - \sfP\lt(e \notin D, D \neq \emptyset, T < (1-\epsilon)L_{\delta}\rt)\\ \notag 
&= 1- \sfP\lt(e \in D\rt) - \sfP \lt(D = \emptyset\rt) - \sfP\lt(e \notin D, D \neq \emptyset, T < (1-\epsilon)L_{\delta}\rt)\\ \label{lb5}
&\geq 1 - \gamma - \beta - \sfP\lt(e \notin D, D \neq \emptyset, T < (1-\epsilon)L_{\delta}\rt)
\end{align}
where the third equality holds because the events $\{e \in D\}$ and $\{D = \emptyset\}$ are disjoint and the last inequality follows from the desired error controls that $\sfP\lt(D \setminus \cA(\sfP) \neq \emptyset\rt) \leq \gamma$ and $\sfP\lt(D = \emptyset\rt) \leq \beta$. We can now decompose the third term as
$$\sfP\lt(e \notin D, D \neq \emptyset, T < (1-\epsilon)L_{\delta}\rt) \leq p_1 + p_2,$$
such that 
\begin{align*}
p_1 &:= \sfP\lt(e \notin D, D \neq \emptyset, \Lambda_T(\sfP, \sfQ_2^*) < \frac{\eta}{\delta}\rt)\quad \text{and}\\ p_2 &:= \sfP\lt(T < (1-\epsilon)L_{\delta}, \Lambda_T(\sfP, \sfQ_2^*) \geq \frac{\eta}{\delta}\rt),
\end{align*}
where $\eta$ is an arbitrary constant in $(0, 1)$. For the first term, by a change of measure from $\sfP$ to $\sfQ_2^*$, we have
\begin{align*}
p_1 &= \sfE_{\sfQ_2^*}\lt[\Lambda_T(\sfP, \sfQ_2^*); e \notin D, D \neq \emptyset, \Lambda_T(\sfP, \sfQ_2^*) < \frac{\eta}{\delta}\rt]\\
&\leq \frac{\eta}{\delta}\sfQ_2^*\lt(e \notin D, D \neq \emptyset\rt) \leq \eta,
\end{align*}
where the last inequality follows from the desired error control that $\sfQ_2^*\lt(\cA(\sfQ_2^*) \setminus D \neq \emptyset, D \neq \emptyset\rt) \leq \delta$. For the second term, we can bound it above by the similar manner as shown previously
\begin{align*}
p_2 \leq \sfP\lt(\frac{1}{(1-\epsilon)L_{\delta}}\max_{1 \leq n \leq (1-\epsilon)L_{\delta}} \log\Lambda_n(\sfP, \sfQ_2^*) \geq \frac{\log\eta}{(1-\epsilon)L_{\delta}} + \frac{\cI(\sfP, \sfQ_2^*)}{(1-\epsilon)}\rt) =: \xi_{\epsilon, \eta}(\delta)
\end{align*}
%\begin{align*}
%\epsilon_{\epsilon, \eta}(\beta, \delta) := \sfP\lt(T < (1-\epsilon)\frac{|\log\beta|}{\cI(\sfP, \sfQ^*)}, \log\Lambda(\sfP, \sfQ^*; T) \geq \log\eta + |\log\beta|\rt)
%\end{align*}
Due to \eqref{max_series}, %the SLLN \eqref{SLLN}, we have
%$$\sfP\lt(\lim_{n \to \infty}\frac{1}{n}\log\Lambda(\sfP, \sfQ_2^*, n) = \cI(\sfP, \sfQ_2^*)\rt) = 1.$$
%Therefore, it follows that 
$\xi_{\epsilon, \eta}(\delta) \to 0$ as $\delta \to 0$. 

Combining everything together in \eqref{lb5_iso} we have
\begin{align*}
\sfP(T \geq (1-\epsilon)L_{\delta}) \geq 1 - \gamma - \beta - \eta - \xi_{\epsilon, \eta}(\delta).
\end{align*}
Since $(T, D) \in \cC_\Psi(\alpha, \beta, \gamma, \delta)$ and $\delta \in (0, 1)$ were chosen arbitrarily, taking infimum over $(T, D)$ and letting $\beta, \gamma, \delta \to 0$ we obtain
\begin{equation*}
\liminf_{\beta, \gamma, \delta \to 0} \inf_{(T, D) \in \cC_\Psi(\alpha, \beta, \gamma, \delta)} \sfP\lt(T \geq (1-\epsilon)L_{\delta}\rt) \geq 1 - \eta.
\end{equation*}
Finally, letting $\eta \to 0$ we obtain \eqref{lb23}, which along with \eqref{lb22} and \eqref{lb21} completes the proof of (iii).
\end{proof}

The next lemma provides an asymptotic lower bound on the optimal performance corresponding to pure detection.

\begin{lemma}\label{lem:D_lb}
Suppose that \eqref{max_series} holds, $\emptyset \in \Psi$ and let $\sfP \in \cP_\Psi$.
\begin{enumerate}
\item[(i)] If $\sfP \in \cH_0$, then as $\alpha, \beta \to 0$, we have
$$ \inf \lt\{ \sfE_{\sfP}[T]:\; 
(T, D) \in \cD_{\Psi}(\alpha, \beta) \rt\} \; \gtrsim \; \max_{e \in \cK}  \lt\{ \frac{|\log\beta|}{ \cI \lt( \sfP , \cG_{\Psi,e} \rt)} \rt\}.$$
\item[(ii)] If $\sfP \notin \cH_0$, then as $\alpha, \beta \to 0$, we have
$$ \inf \lt\{ \sfE_{\sfP}[T]:\; 
(T, D) \in \cD_{\Psi}(\alpha, \beta) \rt\} \; \gtrsim \; \frac{|\log\alpha|}{ \cI \lt( \sfP , \cH_0 \rt)}.$$
\end{enumerate}
\end{lemma}
\begin{proof}
Let $\gamma = \delta = 1$. Then $\cC_\Psi(\alpha, \beta, \gamma, \delta) = \cD_\Psi(\alpha, \beta)$, which leads to 
$$\inf \lt\{ \sfE_{\sfP}[T]:\; (T, D) \in \cD_{\Psi}(\alpha, \beta) \rt\} = \inf \lt\{ \sfE_{\sfP}[T]:\; (T, D) \in \cC_{\Psi}(\alpha, \beta, 1, 1) \rt\}.$$
Thus, the results follow by using Lemma \ref{lem:DI_lb}(i) and Lemma \ref{lem:DI_lb}(ii) respectively.
\end{proof}

The next lemma provides an asymptotic lower bound on the optimal performance when we are interested to control the FWER.

\begin{lemma}\label{lem:FW_lb}
Suppose that \eqref{max_series} holds, $\emptyset \in \Psi$ and let $\sfP \in \cP_\Psi$.
\begin{enumerate}
\item[(i)] If $\sfP \in \cH_0$, then as $\gamma, \delta \to 0$, we have
$$ \inf \lt\{ \sfE_{\sfP}[T]:\; 
(T, D) \in \cE_{\Psi}(\gamma, \delta) \rt\} \; \gtrsim \; \max_{e \in \cK}  \lt\{ \frac{|\log\delta|}{ \cI \lt( \sfP , \cG_{\Psi,e} \rt)} \rt\}.$$
\item[(ii)] If $\sfP \notin \cH_0$, then as $\gamma, \delta \to 0$, we have
\begin{align*}
&\inf \lt\{ \sfE_{\sfP}[T]:\; 
(T, D) \in \cE_{\Psi}(\gamma, \delta) \rt\}\\ 
&\gtrsim \; \max_{e \in \cA(\sfP)} \lt\{ \frac{|\log\gamma|}{\cI(\sfP, \cH_0)  \wedge  \cI(\sfP, \cH_{\Psi,e} ) }\rt\} \bigvee \max_{e \notin \cA(\sfP)}  \lt\{ \frac{|\log\delta|}{ \cI(\sfP, \cG_{\Psi,e} ) }  \rt\}.
\end{align*}
\end{enumerate}
\end{lemma}
\begin{proof}
If $\alpha \geq \gamma$ and $\beta \geq \delta$, then by \eqref{fwer_and_det}, we have $\cE_\Psi(\gamma, \delta) \subseteq  \cD_\Psi(\alpha, \beta)$. Furthermore, for any arbitrary $(T, D) \in \cE_\Psi(\gamma, \delta)$ and $\sfP \in \cP_\Psi \setminus \cH_0$, we have from definition,
\begin{align*}
&\sfP \lt(D \setminus \cA(\sfP) \neq \emptyset\rt) \leq \gamma, \quad \text{and}\\
&\sfP\lt(\cA(\sfP) \setminus D \neq \emptyset, D \neq \emptyset \rt) \leq \sfP\lt(\cA(\sfP) \setminus D \neq \emptyset \rt) \leq \delta,
\end{align*}
which implies $(T, D) \in \cC_{\Psi \setminus \emptyset} (\gamma, \delta)$ and thus, $\cE_\Psi(\gamma, \delta) \subseteq \cC_{\Psi \setminus \emptyset} (\gamma, \delta)$. This gives,
$$\cE_\Psi(\gamma, \delta) \subseteq \cD_\Psi(\alpha, \beta) \cap \cC_{\Psi \setminus \emptyset} (\gamma, \delta) = \cC_\Psi(\alpha, \beta, \gamma, \delta).$$
Therefore,
\begin{align*}
\inf \lt\{ \sfE_{\sfP}[T]:\; 
(T, D) \in \cC_{\Psi}(\alpha, \beta,\gamma, \delta) \rt\} \leq  \inf \lt\{ \sfE_{\sfP}[T]:\; 
(T, D) \in \cE_{\Psi}(\gamma, \delta) \rt\}.
\end{align*}
Now, let $\alpha = \gamma$ and $\beta = \delta$, and the results follow immediately by using Lemma \ref{lem:DI_lb}(i) and Lemma \ref{lem:DI_lb}(iii) respectively. 
\end{proof}

\subsection{Proof of Theorem \ref{PminusP0}}\label{pf:PminusP0}
\begin{proof}
Fix $\sfP \in \cH_0$. %We only prove the result for $\chi^*$ since it can be proved for the other tests in a similar way. 
Since \eqref{condPminusP0} holds, from Theorem \ref{ub_DI}(i), as $\beta \to 0$, we have
\begin{align*}
\inf \lt\{ \sfE_{\sfP}[T]:\; (T, D) \in \cD_{\Psi}(\alpha, \beta) \rt\}
\leq \sfE_\sfP[T^*_{det}] \; \lesssim \;  \max_{e \in \cK}   \lt\{ \frac{|\log \beta|}{\cI \lt(\sfP, \cG_{\Psi,e}\rt)}  \rt\},
\end{align*}
\begin{align*}
\inf \lt\{ \sfE_{\sfP}[T]:\; (T, D) \in \cC_{\Psi}(\alpha, \beta, \gamma, \delta) \rt\}
\leq \sfE_\sfP\lt[T^*\rt] \; \lesssim \;  \max_{e \in \cK}   \lt\{ \frac{|\log \beta|}{\cI \lt(\sfP, \cG_{\Psi,e}\rt)}  \rt\},
\end{align*}
and as $\delta \to 0$ we have
$$\inf \lt\{ \sfE_{\sfP}[T]:\; (T, D) \in \cE_{\Psi}(\gamma, \delta) \rt\} \leq \sfE_\sfP\lt[T^*_{fwer}\rt] \; \lesssim \;  \max_{e \in \cK}   \lt\{ \frac{|\log \delta|}{\cI \lt(\sfP, \cG_{\Psi,e}\rt)}  \rt\}.$$
Thus, the results follow by using Lemma \ref{lem:D_lb}(i), Lemma \ref{lem:DI_lb}(i) and Lemma \ref{lem:FW_lb}(i) respectively. 
\end{proof}

\subsection{Proof of Theorem \ref{P0}}\label{pf:P0}

\begin{proof}
Fix $\sfP \in \cP_\Psi \setminus \cH_0$. Since \eqref{condP0} holds, from Theorem \ref{ub_DI}(ii), as $\alpha \to 0$ we have
\begin{equation*} 
\inf \lt\{ \sfE_{\sfP}[T]:\; (T, D) \in \cD_{\Psi}(\alpha, \beta) \rt\} \leq \sfE_\sfP\lt[T^*_{det}\rt] \lesssim \frac{|\log \alpha|}{\cI\lt(\sfP, \cH_0\rt)}.
\end{equation*}
Thus, the result follows by using Lemma \ref{lem:D_lb}(ii). 
Next, as $\alpha,  \gamma, \delta \to 0$ such that $|\log \alpha| >> |\log \gamma| \vee |\log \delta|$, again from Theorem \ref{ub_DI}(ii), we have
\begin{align*} 
\inf \lt\{ \sfE_{\sfP}[T]:\; (T, D) \in \cC_{\Psi}(\alpha, \beta, \gamma, \delta) \rt\} \leq \sfE_\sfP\lt[T^*\rt] \lesssim \frac{|\log \alpha|}{\cI\lt(\sfP, \cH_0\rt)},
\end{align*}
and on the other hand, as $\alpha, \beta, \gamma, \delta \to 0$ such that $|\log \alpha| >> |\log \gamma| \vee |\log \delta|$, by using Lemma \ref{lem:DI_lb}(iii) we have
\begin{align*}
\inf \lt\{ \sfE_{\sfP}[T]:\; (T, D) \in \cC_{\Psi}(\alpha, \beta, \gamma, \delta) \rt\}  \geq \frac{|\log\alpha|}{\cI(\sfP, \cH_0)},
\end{align*} 
which leads to the result.
Furthermore, when both \eqref{condP1} and \eqref{condP0} hold, as $\gamma, \delta \to 0$ such that $|\log \gamma| >> |\log \delta|$, from Theorem \ref{ub_DI}(ii), we have
\begin{align*} 
\inf \lt\{ \sfE_{\sfP}[T]:\; (T, D) \in \cE_{\Psi}(\gamma, \delta) \rt\}
\leq \sfE_\sfP\lt[T^*_{fwer}\rt] \; \lesssim \max_{e \in \cA(\sfP)} \lt\{ \frac{|\log\gamma|}{\cI(\sfP, \cH_0)  \wedge  \cI(\sfP, \cH_{\Psi,e} ) }\rt\},
\end{align*}
and on the other hand, from Lemma \ref{lem:FW_lb}(ii), we have
\begin{align*}
\inf \lt\{ \sfE_{\sfP}[T]:\; 
(T, D) \in \cE_{\Psi}(\gamma, \delta) \rt\}
\gtrsim \; \max_{e \in \cA(\sfP)} \lt\{ \frac{|\log\gamma|}{\cI(\sfP, \cH_0)  \wedge  \cI(\sfP, \cH_{\Psi,e} ) }\rt\},
\end{align*}
which completes the proof.
\end{proof}

\subsection{Proof of Theorem \ref{P0P1P2}}\label{pf:P0P1P2}

\begin{proof}
Fix $\sfP \in \cP_\Psi \setminus \cH_0$. Since \eqref{condP1} - \eqref{condP2} - \eqref{condP0} hold, from Theorem \ref{ub_DI}(ii),  as $\gamma, \delta \to 0$  we have
\begin{align*} 
&\inf \lt\{ \sfE_{\sfP}[T]:\; (T, D) \in \cE_{\Psi}(\gamma, \delta) \rt\}\\
&\leq \sfE_\sfP\lt[T^*_{fwer}\rt] \; \lesssim \max_{e \in \cA(\sfP)} \lt\{ \frac{|\log\gamma|}{\cI(\sfP, \cH_0)  \wedge  \cI(\sfP, \cH_{\Psi,e} ) }\rt\}  \bigvee \max_{e \notin \cA(\sfP)}  \lt\{ \frac{|\log \delta|}{\cI \lt(\sfP, \cG_{\Psi,e}\rt)}  \rt\},
\end{align*}
and as $\alpha,  \gamma, \delta \to 0$  we have
\begin{align*}
&\inf \lt\{ \sfE_{\sfP}[T]:\; (T, D) \in \cC_{\Psi}(\alpha, \beta, \gamma, \delta) \rt\}\\
&\leq \sfE_\sfP\lt[T^*\rt] \; \lesssim \; \frac{|\log \alpha|}{\cI\lt(\sfP, \cH_0\rt)} \bigvee \max_{e \in \cA(\sfP)}  \lt\{\frac{|\log \gamma|}{\cI \lt(\sfP, \cH_{\Psi, e}\rt) }   \rt\}  \bigvee \max_{e \notin \cA(\sfP)}  \lt\{ \frac{|\log \delta|}{\cI \lt(\sfP, \cG_{\Psi,e}\rt)}  \rt\}.
\end{align*}
Thus, the results follow by using Lemma \ref{lem:FW_lb}(ii) and  Lemma \ref{lem:DI_lb}(iii) respectively. 
\end{proof}

\subsection{Proof of Remark \ref{rem:P1orP2}}\label{pf:rem:P1orP2}

\begin{proof}
We only prove for the case of Theorem \ref{P1P2} since for the other case, it follows in a similar fashion. Without loss of generality, suppose that only  \eqref{condP1} holds and $\gamma, \delta \to 0$ such that $|\log \gamma| >> |\log \delta|$. Then
from Theorem \ref{ub_I} we have
\begin{align*}
\inf \lt\{ \sfE_{\sfP}[T]:\; (T, D) \in \cC_{\Psi}(\gamma, \delta) \rt\} \; \leq \; \sfE_\sfP\lt[T^*\rt] \; \lesssim \; \max_{e \in \cA(\sfP)} \lt\{ \frac{|\log\gamma|}{  \cI(\sfP, \cH_{\Psi,e}) }\rt\},
\end{align*}
and from Lemma \ref{lem:P1P2} we have
\begin{align*}
\inf \lt\{ \sfE_{\sfP}[T]:\; 
(T, D) \in \cC_{\Psi}(\gamma, \delta) \rt\} \; \gtrsim \; \max_{e \in \cA(\sfP)} \lt\{ \frac{|\log\gamma|}{  \cI(\sfP, \cH_{\Psi,e}) }\rt\},
\end{align*}
which completes the proof.
\end{proof}

\subsection{Proof of Remark \ref{rem:fwer}}\label{pf:rem:fwer}

\begin{proof}
We only prove for the case of Theorem \ref{P0} since for the other case, it follows in a similar fashion. Without loss of generality, suppose that \eqref{FWERcondP0} is satisfied. Then as $\gamma, \delta \to 0$, from  Theorem \ref{ub_DI}(ii) we have
\begin{align*}
&\inf \lt\{ \sfE_{\sfP}[T]:\; 
(T, D) \in \cE_{\Psi}(\gamma, \delta) \rt\}\\
&\leq \sfE_\sfP\lt[T^*_{fwer}\rt] \; \lesssim \min_{e \in \cA(\sfP)}  \lt\{ \frac{|\log \gamma|}{\cI\lt(\sfP, \cH_0; s_e \rt)} \rt\}  \bigvee \max_{e \notin \cA(\sfP)}  \lt\{ \frac{|\log \delta|}{\cI \lt(\sfP, \cG_{\Psi,e}; s'_e \rt)}  \rt\},
\end{align*}
and from Lemma \ref{lem:FW_lb}(ii) we have
\begin{align*}
\inf \lt\{ \sfE_{\sfP}[T]:\; 
(T, D) \in \cE_{\Psi}(\gamma, \delta) \rt\}
\gtrsim \; \frac{|\log\gamma|}{\cI(\sfP, \cH_0) } \bigvee \max_{e \notin \cA(\sfP)}  \lt\{ \frac{|\log\delta|}{ \cI(\sfP, \cG_{\Psi,e} ) }  \rt\}.
\end{align*}
As $|\log \gamma| >> |\log \delta|$ the second terms from both bounds vanish and since \eqref{condP0} holds, the result follows.
\end{proof}

The conditions \eqref{condP1}-\eqref{condP2}-\eqref{condPminusP0}-\eqref{condP0}  may still be satisfied even if  $s_e$ or $s'_e$  is a strict subset of $[K]$ for every $e \in \cK$. The following proposition provides sufficient conditions for this in the presence of some independence and is crucial in establishing results of asymptotic optimality in Appendix \ref{app:appl_dep_source} and \ref{app:appl_dep}.  

\begin{proposition}\label{prop:P1}
Suppose that \eqref{max_series}-\eqref{comp_conv_onesided} hold and $\sfP \in \cP_\Psi$.
\begin{enumerate}

\item[(i)] If $\sfP \notin \cH_0$, \eqref{condP1} holds when $s = s'_{\bar{e}}$ is  independent  of its complement  under $\sfP$ and  
\begin{align}\label{prop:P1_1}
 %\sfQ_1^{s_{e'}} \otimes \sfP^{s_{e'}^c} \in \cP_\Psi \setminus \cH_0,
 \sfQ_1^{s} \otimes \sfP^{s^c} \in \cP_\Psi \setminus \cH_0,
\end{align}
where  $\bar{e}$ and $\sfQ_1$  are such that
\begin{equation*}
\bar{e} \in \argmin_{e \in \cA(\sfP)} \; \cI \lt(\sfP, \cH_{\Psi, e}; s'_e \rt) \qquad \text{and} \qquad \sfQ_1 \in \argmin_{\sfQ \in \cH_{\Psi, \bar{e}}} \; \cI \lt(\sfP, \sfQ; s'_{\bar{e}} \rt).
\end{equation*}

\item[(ii)] If $\sfP \notin \cH_0$, \eqref{condP2} holds when $s = s'_{\bar{e}}$ is  independent  of its complement  under $\sfP$ and  
\begin{align}\label{prop:P2_1}
 %\sfQ_1^{s_{e'}} \otimes \sfP^{s_{e'}^c} \in \cP_\Psi \setminus \cH_0,
 \sfQ_1^{s} \otimes \sfP^{s^c} \in \cP_\Psi \setminus \cH_0,
\end{align}
where  $\bar{e}$ and $\sfQ_1$  are such that
\begin{equation*}
\bar{e} \in \argmin_{e \notin \cA(\sfP)} \; \cI \lt(\sfP, \cG_{\Psi, e}; s'_e \rt) \qquad \text{and} \qquad \sfQ_1 \in \argmin_{\sfQ \in \cG_{\Psi, \bar{e}}} \; \cI \lt(\sfP, \sfQ; s'_{\bar{e}} \rt).
\end{equation*}

\item[(iii)] If $\sfP \in \cH_0$, \eqref{condPminusP0} holds when $s = s_{\bar{e}}$ is independent of its complement under $\sfP$ and
\begin{align}\label{prop:PminusP0_1} 
\sfQ_1^{s} \otimes \sfP^{s^c} \in \cP_\Psi \setminus \cH_0,
\end{align}
where  $\bar{e}$ and $\sfQ_1$ are such that 
$$\bar{e} \in  \argmin_{e \in \cK} \; \cI \lt(\sfP, \cG_{\Psi, e}; s_e \rt)  \qquad \text{and} \qquad 
\sfQ_1 \in  \argmin_{\sfQ \in \cG_{\Psi, \bar{e}}} \; \cI \lt(\sfP, \sfQ; s_{\bar{e}} \rt),$$

%\begin{equation}
%\label{prop:PminusP0_2}
%$Q_1^{W_{e'}} \otimes \sfP^{W_{e'}^c} \in \cP(\Psi) \setminus \cP_0,$
%\end{equation}
%\end{proposition}\begin{proposition}\label{prop:P0}
%Suppose that %\eqref{comp_conv_onesided}-\eqref{SLLN} 
%\eqref{max_series}-\eqref{comp_conv_onesided} hold and consider the substreams $\{W_e : e \in \cE\}$.
\item[(iv)] If $\sfP \notin \cH_0$, \eqref{condP0} holds when there exists $e \in \cA(\sfP)$ such that  $s = s_{e}$ is independent of its complement under $\sfP$ and 
\begin{align} \label{prop:P0_1} 
\sfQ_1^{s} \otimes \sfP^{s^c} \in \cH_0, \quad \text{where} \quad 
\sfQ_1 \in  \argmin_{\sfQ \in \cH_0} \cI \lt(\sfP, \sfQ; s_{e} \rt).
\end{align}
\end{enumerate}
\end{proposition}

\begin{proof}
The proofs are as follows.
\begin{enumerate}
\item[(i)]
\eqref{condP1} holds because of the following observation that,
\begin{align*}
\min_{e \in \cA(\sfP)} \cI \lt( \sfP , \cH_{\Psi,e} \rt) &\leq \cI\lt(\sfP, \sfQ_1^s \otimes \sfP^{s^c}\rt) = \cI\lt(\sfP^s \otimes \sfP^{s^c}, \sfQ_1^{s} \otimes \sfP^{s^c}\rt)\\
&= \cI\lt(\sfP^s, \sfQ_1^s\rt) = \cI(\sfP, \sfQ_1; s'_{\bar{e}})\\
&= \cI(\sfP, \cH_{\Psi, \bar{e}}; s'_{\bar{e}}) = \min_{e \in \cA(\sfP)} \cI(\sfP, \cH_{\Psi, e}; s'_e) \leq \min_{e \in \cA(\sfP)} \cI(\sfP, \cH_{\Psi, e}).
\end{align*}
In the above, the first inequality holds because $\bar{e} \in \cA(\sfP)$ and $\sfQ_1 \in \cH_{\Psi, \bar{e}}$, which combined with \eqref{prop:P1_1}, implies that
$$\sfQ_1^{s} \otimes \sfP^{s^c} \in \cH_{\Psi, \bar{e}} \subseteq \cup_{e \in \cA(\sfP)} \cH_{\Psi, e}.$$ 
%and the first and fourth equalities follow from \eqref{prop:P1_1}. Moreover, 
The second equality follows from Lemma \ref{info_ineq2} under the assumptions \eqref{max_series}-\eqref{comp_conv_onesided}, and the last inequality follows from Lemma \ref{info_ineq1}.

\item[(ii)]
\eqref{condP2} holds because of the following observation that,
\begin{align*}
\min_{e \notin \cA(\sfP)} \cI \lt( \sfP , \cG_{\Psi,e} \rt) &\leq \cI\lt(\sfP, \sfQ_1^s \otimes \sfP^{s^c}\rt) = \cI\lt(\sfP^s \otimes \sfP^{s^c}, \sfQ_1^{s} \otimes \sfP^{s^c}\rt)\\
&= \cI\lt(\sfP^s, \sfQ_1^s\rt) = \cI(\sfP, \sfQ_1; s'_{\bar{e}})\\
&= \cI(\sfP, \cG_{\Psi, \bar{e}}; s'_{\bar{e}}) = \min_{e \notin \cA(\sfP)} \cI(\sfP, \cG_{\Psi, e}; s'_e) \leq \min_{e \notin \cA(\sfP)} \cI(\sfP, \cG_{\Psi, e}).
\end{align*}
In the above, the first inequality holds because $\bar{e} \notin \cA(\sfP)$ and $\sfQ_1 \in \cG_{\Psi, \bar{e}}$, which combined with \eqref{prop:P2_1}, implies that
$$\sfQ_1^{s} \otimes \sfP^{s^c} \in \cG_{\Psi, \bar{e}} \subseteq \cup_{e \notin \cA(\sfP)} \cG_{\Psi, e}.$$ 
%and the first and fourth equalities follow from \eqref{prop:P1_1}. Moreover, 
The second equality follows from Lemma \ref{info_ineq2} under the assumptions \eqref{max_series}-\eqref{comp_conv_onesided}, and the last inequality follows from Lemma \ref{info_ineq1}.

\item[(iii)]
\eqref{condPminusP0} holds because of the following observation that,
\begin{align*}
\min_{e \in \cK} \cI \lt( \sfP , \cG_{\Psi,e} \rt) &\leq \cI\lt(\sfP, \sfQ_1^s \otimes \sfP^{s^c}\rt) = \cI\lt(\sfP^s \otimes \sfP^{s^c}, \sfQ_1^{s} \otimes \sfP^{s^c}\rt)\\
&= \cI\lt(\sfP^s, \sfQ_1^s\rt) = \cI(\sfP, \sfQ_1; s_{\bar{e}})\\
&= \cI(\sfP, \cG_{\Psi, \bar{e}}; s_{\bar{e}}) = \min_{e \in \cK} \cI(\sfP, \cG_{\Psi, e}; s_e) \leq \min_{e \in \cK} \cI(\sfP, \cG_{\Psi, e}).
\end{align*}
In the above, the first inequality holds because $\sfQ_1 \in \cG_{\Psi, \bar{e}}$, which combined with \eqref{prop:PminusP0_1}, implies that
$$\sfQ_1^{s} \otimes \sfP^{s^c} \in \cG_{\Psi, \bar{e}} \subseteq \cup_{e \in \cK} \cG_{\Psi, e}.$$ 
%and the first and fourth equalities follow from \eqref{prop:P1_1}. Moreover, 
The second equality follows from Lemma \ref{info_ineq2} under the assumptions \eqref{max_series}-\eqref{comp_conv_onesided}, and the last inequality follows from Lemma \ref{info_ineq1}.

\item[(iv)]
\eqref{condP0} holds because of the following observation that,
\begin{align*}
\cI(\sfP, \cH_0) &\leq \cI\lt(\sfP, \sfQ_1^{s} \otimes \sfP^{s^c}\rt) = \cI\lt(\sfP^s \otimes \sfP^{s^c}, \sfQ_1^{s} \otimes \sfP^{s^c}\rt)\\
&= \cI\lt(\sfP^s, \sfQ_1^s\rt) = \cI(\sfP, \sfQ_1; s_e) = \cI(\sfP, \cH_0; s_e) \leq \cI(\sfP, \cH_0).
\end{align*}
In the above, the first inequality, and the last equality hold because of \eqref{prop:P0_1}. The second equality follows from Lemma \ref{info_ineq2} under the assumptions \eqref{max_series}-\eqref{comp_conv_onesided} and the last inequality follows from Lemma \ref{info_ineq1}.
\end{enumerate}
\end{proof}

\section{Detection and isolation of anomalous sources: the case of independence}\label{app:appl_indep}

%\subsection{Proof of Proposition \ref{propo:sufficient}}\label{pf:propo:sufficient}
\subsection{An important proposition}
\begin{proposition} \label{propo:sufficient}
If  for every  $k \in [K]$, $\sfP^k \in \cG^k$, $\sfQ^k \in \cH^k$  there  are positive numbers 
$$ \cI (\sfP^k, \sfQ^k ) \quad \text{ and } \quad   \cI (\sfQ^k, \sfP^k ),$$ 
such that  as $n \to \infty$ we have 
\begin{align}\label{max_series_ind}
\begin{split}
&\sfP^k\lt(\max_{1 \leq m \leq n} Z_m (\sfP^k, \sfQ^k) \geq n \rho  \rt) \to 0 \qquad \forall \, \rho> \cI (\sfP^k, \sfQ^k), \\ %\label{max_series1}
&\sfQ^k\lt(\max_{1 \leq m \leq n} Z_m (\sfQ^k, \sfP^k ) \geq  n \rho \rt) \to 0
 \qquad \forall \, \rho > \cI (\sfQ^k, \sfP^k),
\end{split} 
\end{align}
and 
\begin{align}\label{comp_conv_onesided_ind}
\begin{split}
& \sum_{n = 1}^{\infty}\sfP^k\lt( Z_n (\sfP^k, \sfQ^k) < n \rho \rt) < \infty \qquad \forall \, \rho< \cI (\sfP^k, \sfQ^k ), \\ %\label{comp_conv_onesided1}
&\sum_{n = 1}^{\infty}\sfQ^k\lt( Z_n (\sfQ^k, \sfP^k )  < n \rho \rt) < \infty \qquad \forall \, \rho< \cI (\sfQ^k, \sfP^k),  %\label{comp_conv_onesided2}
\end{split}
\end{align}
then  \eqref{max_series} and \eqref{comp_conv_onesided} hold,  for any $s \subseteq [K]$, with
$$\cI(\sfP^s, \sfQ^s) = \sum_{k\in s}\cI(\sfP^k, \sfQ^k) \qquad \text{and} \qquad \cI(\sfQ^s, \sfP^s) = \sum_{k \in  s}\cI(\sfQ^k, \sfP^k).$$
\end{proposition}

\begin{proof}
Fix $e \in \cK$,  $\sfP \in \cG_{\Psi,e}$, $\sfQ \in \cH_{\Psi,e} \cup \cH_0$,  $e  \subseteq s \subseteq  [K].$
Since the sources belonging to $s$ are independent both under $\sfP$ and $\sfQ$, and \eqref{max_series_ind}-\eqref{comp_conv_onesided_ind} hold for every source $k \in [K]$, by an application of Lemma \ref{info_ineq2}, we observe that conditions \eqref{max_series}-\eqref{comp_conv_onesided} hold for $\sfP$, $\sfQ$ and $s$.
Furthermore, we apply the same lemma once again to have
\begin{align*}
\cI(\sfP^s, \sfQ^s) = \sum_{k\in s}\cI(\sfP^k, \sfQ^k) \qquad \text{and} \qquad \cI(\sfQ^s, \sfP^s) = \sum_{k \in  s}\cI(\sfQ^k, \sfP^k),
\end{align*}
which completes the proof.
\end{proof}

In order to prove the next results we introduce some further notations.
For every $k \in [K]$, we order the set of likelihood ratio statistics excluding $\Lambda_k(n)$, i.e., $\{\Lambda_{j}(n) : j \neq k\}$ and denote them by
$$\Lambda_{-k, (1)}(n) \geq \dots \geq \Lambda_{-k, (K-1)}(n),$$
along with the corresponding indices $i_{-k, 1}(n), \dots, i_{-k, K-1}(n)$, i.e.,
$$\Lambda_{-k, (j)}(n) = \Lambda_{i_{-k, j}(n)}(n),\ \text{for every}\ j \in [K-1].$$
Furthermore, we denote by $p_k(n)$ the number of above likelihood ratio statistics that are greater than 1. Formally,
$$\Lambda_{-k, (1)}(n) \geq \dots \geq \Lambda_{-k, (p_k(n))} > 1 \geq \Lambda_{-k, (p_k(n)+1)} \geq \dots \geq \Lambda_{-k, (K-1)}(n).$$
Then it is not difficult to observe that
\begin{equation}\label{p_k(n)}
p_k(n) = 
\begin{cases}
p(n) - 1\ &\text{if}\ k \in \{i_1(n), \dots, i_{p(n)}(n)\}\\
p(n) &\text{otherwise}
\end{cases}.
\end{equation}
Furthermore, when the prior information is $\Psi_{l, u}$, with $0 \leq l \leq u \leq K$, we define the following two quantities for every $k \in [K]$: $$m^1_k(n) := \text{median}\{(l-1) \vee 0, p_k(n), u-1\} \quad \text{and} \quad m^0_k(n) := \text{median}\{l \vee 1, p_k(n), u\}.$$
%In case, $p_e(n) = 0$ for some $e \in \cK$, we simply define $\Lambda_{-e, (0)}(n) \equiv 1$.

For the rest of this section we consider $\Psi = \Psi_{l, u}$, where $l \leq u$.
\subsection{Some important lemmas}

In this section we present some important lemmas which plays a crucial role to simplify the detection and isolation statistics in the proposed procedures so that one can obtain them in a more implementable version.
\begin{lemma}\label{max_char_gap_int}
Suppose $0 \leq l \leq u \leq K$. Then for every $k \in [K]$, $n \in \bN$, and any arbitrary reference distribution $\sfP_0 \in \cH_0$, we have 
\begin{equation}\label{max_char_gap_int1}
\frac{ \max  \lt\{ \Lambda_n(\sfP, \sfP_0) : \sfP \in \cG_{\Psi, k}\rt\} }
 {\max \lt\{ \Lambda_n(\sfP, \sfP_0): \sfP \in \cH_0  \rt\} } = \Lambda_k(n)\prod_{j = 1}^{m^1_k(n)}\Lambda_{-e, (j)}(n),
\end{equation}
and 
\begin{equation}\label{max_char_gap_int2}
\frac{ \max  \lt\{ \Lambda_n(\sfP, \sfP_0) : \sfP \in \cH_{\Psi, k} \rt\} }
 {\max \lt\{ \Lambda_n(\sfP, \sfP_0): \sfP \in \cH_0  \rt\} } = \prod_{j = 1}^{m^0_k(n)}\Lambda_{-e, (j)}(n).
\end{equation}
\end{lemma}

\begin{proof}
We only prove \eqref{max_char_gap_int1} since \eqref{max_char_gap_int2} can be proved in a similar way. Fix any arbitrary reference distribution $\sfP_0 \in \cH_0$. For every $j \in [K]$ define
$$\Lambda^1_j(n) := \max\{\Lambda_n(\sfP, \sfP_0^j) : \sfP \in \cG^j\} \quad \text{and} \quad \Lambda^0_j(n) := \max\{\Lambda_n(\sfP, \sfP_0^j) : \sfP \in \cH^j\}.$$
Note that, for every $j \in [K]$,
\begin{equation}\label{lam1/lam0}
\Lambda_j(n) = \Lambda^1_j(n)/\Lambda^0_j(n).
\end{equation}
Furthermore, for any $\sfP \neq \sfP_0$, we have the following decomposition due to independence across the streams
\begin{equation}\label{decomp-LR}
\Lambda_n(\sfP, \sfP_0) = \prod_{j = 1}^K \Lambda_n(\sfP^j, \sfP_0^j) = \prod_{\{j\} \in \cA(\sfP)} \Lambda_n(\sfP^j, \sfP_0^j) \prod_{\{j\} \notin \cA(\sfP)} \Lambda_n(\sfP^j, \sfP_0^j).
\end{equation}
Fix any arbitrary $k \in [K]$ and a subset of units $A \in \Psi_{l, u}$ such that $\{k\} \in A$. Then from \eqref{decomp-LR} one can obtain that
\begin{align*}
&\max\{\Lambda_n(\sfP, \sfP_0) : \sfP \in \cG_{\Psi, k}, \;\; \cA(\sfP) = A\} = \prod_{\{j\} \in A} \Lambda^1_{j}(n) \prod_{\{j\} \notin A} \Lambda^0_{j}(n), \;\; \text{and}\\
&\max\{\Lambda_n(\sfP, \sfP_0) : \sfP \in \cH_0\} = \prod_{j = 1}^K \Lambda_{j}^0(n).
\end{align*}
Due to \eqref{lam1/lam0}, the above two expressions lead to
$$\frac{\max\{\Lambda_n(\sfP, \sfP_0) : \sfP \in \cG_{\Psi, k}, \;\; \cA(\sfP) = A\}}{\max\{\Lambda_n(\sfP, \sfP_0) : \sfP \in \cH_0\}} = \prod_{\{j\} \in A} \Lambda_{j}(n).$$
Now we further maximize the above quantity over all possible $A \in \Psi_{l, u}$.
Since $\{k\} \in A$ and $l \leq |A| \leq u$, $\argmax_{A \in \Psi_{l, u}} \prod_{\{j\} \in A} \Lambda_j(n)$ consists of sources with indices
\begin{align*}
\{k, i_{-k, 1}(n), \dots, i_{-k, p_k(n)}(n)\} \quad &\text{if} \;\;\; l-1 \vee 0 \leq p_k(n) \leq u-1,\\
\{k, i_{-k, 1}(n), \dots, i_{-k, l-1}(n)\} \quad &\text{if} \;\;\; p_k(n) \leq l - 1 \vee 0, \quad \text{or}\\ 
\{k, i_{-k, 1}(n), \dots, i_{-k, u-1}(n)\} \quad& \text{if} \;\;\; p_k(n) \geq u-1.
\end{align*}
This complete the proof.
\end{proof}

In the next lemma, we provide a simplification for the maximum of the detection statistics when $l = 0$ and $u \leq K$.

\begin{lemma}\label{max_det_char} 
Consider $0 = l < u \leq K$ and $s_k = [K]$ for every $k \in [K]$. Then for every $k \in [K]$ and $n \in \bN$, we have
\begin{equation}\label{max_det_char1}
\max_{k \in [K]} \; \Lambda_{k, \rm det}(n) = \prod_{j = 1}^{m(n)} \Lambda_{(j)}(n),
\end{equation}
where $m(n) := median\{1, p(n), u\}$. Furthermore, when $s_k = [K]$ for every $k \in [K]$ or $s_k = \{k\}$ for every $k \in [K]$, then
$$\max_{k \in [K]} \; \Lambda_{k, \rm det}(n) > 1 \quad \text{if and only if} \quad p(n) \geq 1.$$
\end{lemma}
\begin{proof}
Let $s_k = [K]$ for every $k \in [K]$. For proving \eqref{max_det_char1}, we only consider the case when $p(n) < u$ since for the case when $p(n) \geq u$, it can be proved in a similar way. Now, in this case, 
$$m(n) = \begin{cases}
p(n)\;\; &\text{if} \;\; p(n) \geq 1\\
1\;\; &\text{if} \;\; p(n) = 0
\end{cases}.$$ 

When $p(n) \geq 1$, for every $k \in \{i_1(n), \dots, i_{p(n)}(n)\}$, we have $m_k^1(n) = p_k(n) = p(n) - 1$, and therefore, from \eqref{max_char_gap_int1} in Lemma \ref{max_char_gap_int},
$$\Lambda_{k, \rm det}(n) = \Lambda_k(n)\prod_{j = 1}^{p(n) - 1}\Lambda_{-e, (j)}(n) = \prod_{j = 1}^{p(n)} \Lambda_{(j)}(n).$$
On the other hand, when $k \notin \{i_1(n), \dots, i_{p(n)}(n)\}$, we have $m_k^1(n) = p_k(n) = p(n)$, and therefore, again from \eqref{max_char_gap_int1} in Lemma \ref{max_char_gap_int},
$$\Lambda_{k, \rm det}(n) =  \Lambda_k(n)\prod_{j = 1}^{p(n)}\Lambda_{-e, (j)}(n) = \Lambda_k(n)\prod_{j = 1}^{p(n)} \Lambda_{(j)}(n) \leq \prod_{j = 1}^{p(n)} \Lambda_{(j)}(n),$$
where the inequality follows from the fact that $\Lambda_k(n) \leq 1$. Thus, 
$$\max_{k \in [K]} \; \Lambda_{k, \rm det}(n) = \prod_{j = 1}^{p(n)} \Lambda_{(j)}(n),$$

When $p(n) = 0$, we have $m_k^1(n) = p_k(n) = 0$ for every $k \in [K]$, and therefore, using \eqref{max_char_gap_int1} in Lemma \ref{max_char_gap_int},
$$\Lambda_{k, \rm det}(n) = \Lambda_k(n).$$
This implies $\max_{k \in [K]} \; \Lambda_{k, \rm det}(n) = \Lambda_{(1)}(n),$
and this completes the first part of the proof.

Now we prove the second part. Suppose that $s_k = [K]$ for every $k \in [K]$. Then from \eqref{max_det_char1} that if $p(n) \geq 1$, we have
\begin{equation*}
\max_{k \in [K]} \; \Lambda_{k, \rm det}(n) = \prod_{j = 1}^{p(n) \wedge u} \Lambda_{(j)}(n) > 1,
\end{equation*}
where the inequality follows
since $\Lambda_{(j)}(n) > 1$ for every $j \in [p(n) \wedge u]$. On the other hand, if $p(n) = 0$, we have $\Lambda_{(1)}(n) < 1$ and $m(n) = 0$. Therefore,
from \eqref{max_det_char1},
$$\max_{k \in [K]} \; \Lambda_{k, \rm det}(n) = \Lambda_{(1)}(n) < 1.$$
Now suppose that $s_k = \{k\}$ for every $k \in [K]$. Then $\Lambda_{k, \rm det}(n) = \Lambda_k(n)$ for every $k \in [K]$, and thus, 
$$\max_{k \in [K]} \; \Lambda_{k, \rm det}(n) = \Lambda_{(1)}(n),$$
which is bigger than $1$ if and only if $p(n) \geq 1$.
The proof is complete.
\end{proof}

In the next lemmas, we provide simplification for the isolation statistics depending on the values of $l$, $u$ and $p(n)$.
\begin{lemma}\label{iso_char_gap_int1}
Consider $l \geq 0$ and $s'_k = [K]$ for every $k \in [K]$. Then we have $p(n) < l \vee 1$ if and only if for every $k \in [K]$,
$$\Lambda_{k, \rm iso}(n) = \frac{\Lambda_k(n)}{\Lambda_{-k, (l \vee 1)}(n)},$$
and in this case $D_{\rm iso}(n)$ consists of sources with indices $\{i_1(n), \dots, i_{l \vee 1}(n)\}$.
\end{lemma}
\begin{proof}
From \eqref{p_k(n)}, we have $p(n) \leq l-1$ if and only if $p_k(n) \leq l-1$ for every $k \in [K]$. Equivalently, for every $k \in [K]$, $m_k^1(n) = (l-1) \vee 0$, $m_k^0(n) = l \vee 1$ and therefore, dividing \eqref{max_char_gap_int1} by \eqref{max_char_gap_int2} in Lemma \ref{max_char_gap_int}, we have
$$\Lambda_{k, \rm iso}(n) = \frac{\Lambda_k(n)\prod_{j = 1}^{(l-1) \vee 0}\Lambda_{-k, (j)}(n)}{\prod_{j = 1}^{l \vee 1}\Lambda_{-k, (j)}(n)} = \frac{\Lambda_k(n)}{\Lambda_{-k, (l \vee 1)}(n)}.$$
Therefore, $\Lambda_{k, \rm iso}(n) > 1$ if and only if $k \in \{i_1(n), \dots, i_{l \vee 1}(n)\}$. 
\end{proof}

\begin{lemma}\label{iso_char_gap_int2}
Consider $l \geq 0$ and $s'_k = [K]$ for every $k \in [K]$. Then we have $p(n) = l \vee 1$ if and only if
$$\Lambda_{k, \rm iso}(n) = 
\begin{cases}
\frac{\Lambda_k(n)}{\Lambda_{((l+1) \vee 2)}(n)}\ &\text{if}\ k \in \{i_1(n), \dots, i_{l \vee 1}(n)\}\\
\Lambda_k(n)\ &\text{otherwise}
\end{cases},$$
and in this case $D_{\rm iso}(n)$ consists of sources with indices $\{i_1(n), \dots, i_{l \vee 1}(n)\}$.
\end{lemma}
\begin{proof}
From \eqref{p_k(n)}, we have
\begin{align*}
p(n) = l \vee 1 \quad 
\Leftrightarrow \quad p_k(n) = 
\begin{cases} 
(l-1) \vee 0 \quad &\text{if} \quad  k \in \{i_1(n), \dots, i_{l \vee 1}(n)\},\\
l \vee 1 \quad \quad \; \; &\text{otherwise}.
\end{cases}
\end{align*}
Equivalently, if $k \in \{i_1(n), \dots, i_{l \vee 1}(n)\}$, $m_k^1(n) = (l-1) \vee 0$, $m_k^0(n) = l \vee 1$, and dividing \eqref{max_char_gap_int1} by \eqref{max_char_gap_int2} in Lemma \ref{max_char_gap_int} gives that
$$\Lambda_{k, \rm iso}(n) = \frac{\Lambda_k(n)\prod_{j = 1}^{(l-1) \vee 0}\Lambda_{-k, (j)}(n)}{\prod_{j = 1}^{l \vee 1}\Lambda_{-k, (j)}(n)} = \frac{\Lambda_k(n)}{\Lambda_{-k, (l \vee 1)}(n)} = \frac{\Lambda_k(n)}{\Lambda_{((l+1) \vee 2)}(n)} > 1,$$
otherwise, $m_k^1(n) = m_k^0(n) = l \vee 1$, and
$$\Lambda_{k, \rm iso}(n) = \frac{\Lambda_k(n)\prod_{j = 1}^{l \vee 1}\Lambda_{-k, (j)}(n)}{\prod_{j = 1}^{l \vee 1}\Lambda_{-k, (j)}(n)} = \Lambda_k(n) \leq 1.$$
%Thus, $|D_{\rm iso}(n)| = l \vee 1$ and the proof is complete.
\end{proof}

\begin{lemma}\label{iso_char_gap_int3}
Consider $0 \leq l < u \leq K$ and $s'_k = [K]$ for every $k \in [K]$. Then
we have $l \vee 1 < p(n) < u$ if and only if $\Lambda_{k, \rm iso}(n) = \Lambda_k(n)$ for every $k \in [K]$ and in this case  $D_{\rm iso}(n)$ consists of sources with indices $\{i_1(n), \dots, i_{p(n)}(n)\}$.
\end{lemma}
\begin{proof}
From \eqref{p_k(n)}, we have
\begin{align*}
(l+1) \vee 2 \leq p(n) \leq u-1  \;\; &\Leftrightarrow \;\; l \vee 1 \leq p_k(n) \leq u-1 \;\; \text{for every $k \in [K]$},\\
&\Leftrightarrow \;\; m_k^1(n) = m_k^0(n) = p_k(n) \;\; \text{for every $k \in [K]$}.
\end{align*}
Equivalently, for every $k \in [K]$, dividing \eqref{max_char_gap_int1} by \eqref{max_char_gap_int2} in Lemma \ref{max_char_gap_int} gives that,
$$\Lambda_{k, \rm iso}(n) = \frac{\Lambda_k(n)\prod_{j = 1}^{p_k(n)}\Lambda_{-k, (j)}(n)}{\prod_{j = 1}^{p_k(n)}\Lambda_{-k, (j)}(n)} = \Lambda_k(n),$$
which also implies $D_{\rm iso}(n)$ consists of sources with indices $\{i_1(n), \dots, i_{p(n)}(n)\}$.
\end{proof}

\begin{lemma}\label{iso_char_gap_int4}
Consider $u \leq K$ and $s'_k = [K]$ for every $k \in [K]$. Then we have
$p(n) = u$ if and only if
$$\Lambda_{k, \rm iso}(n) = 
\begin{cases}
\Lambda_k(n)\ &\text{if}\ k \in \{i_1(n), \dots, i_u(n)\}\\
\frac{\Lambda_k(n)}{\Lambda_{(u)}(n)}\ &\text{otherwise}
\end{cases},$$
and in this case  $D_{\rm iso}(n)$ consists of sources with indices $\{i_1(n), \dots, i_u(n)\}$.
\end{lemma}
\begin{proof}
From \eqref{p_k(n)}, we have
\begin{align*}
p(n) = u \;\;
\Leftrightarrow \;\; p_k(n) = 
\begin{cases} 
u-1 \quad &\text{if} \quad  k \in \{i_1(n), \dots, i_u(n)\}\\
u \quad \quad \; \; &\text{otherwise}
\end{cases},
\end{align*}
Equivalently, if $k \in \{i_1(n), \dots, i_u(n)\}$, $m_k^1(n) = m_k^0(n) = u-1$ and dividing \eqref{max_char_gap_int1} by \eqref{max_char_gap_int2} in Lemma \ref{max_char_gap_int} gives that,
$$\Lambda_{k, \rm iso}(n) = \frac{\Lambda_k(n)\prod_{j = 1}^{u-1}\Lambda_{-k, (j)}(n)}{\prod_{j = 1}^{u-1}\Lambda_{-k, (j)}(n)} = \Lambda_{k}(n) > 1,$$
otherwise, $m_k^1(n) = u-1$, $m_k^0(n) = u$, and
$$\Lambda_{k, \rm iso}(n) = \frac{\Lambda_k(n)\prod_{j = 1}^{u-1}\Lambda_{-k, (j)}(n)}{\prod_{j = 1}^{u}\Lambda_{-k, (j)}(n)} = \frac{\Lambda_k(n)}{\Lambda_{-k, (u)}(n)} = \frac{\Lambda_k(n)}{\Lambda_{(u)}(n)} \leq 1.$$
%Thus, $|D_{\rm iso}(n)| = u$ and the proof is complete.
\end{proof}

\begin{lemma}\label{iso_char_gap_int5}
Consider $u < K$ and $s'_k = [K]$ for every $k \in [K]$.
Then we have $p(n) > u$ if and only if for every $k \in [K]$,
$$\Lambda_{k, \rm iso}(n) = \frac{\Lambda_k(n)}{\Lambda_{-k, (u)}(n)},$$
and in this case  $D_{\rm iso}(n)$ consists of sources with indices $\{i_1(n), \dots, i_{u}(n)\}$.
\end{lemma}
\begin{proof}
From \eqref{p_k(n)}, we have $p(n) \geq u+1 \;\; \Leftrightarrow \;\; p_k(n) \geq u$ for every $k \in [K]$. Equivalently, for every $k \in [K]$, $m_k^1(n) = u-1$, $m_k^0(n) = u$, and dividing \eqref{max_char_gap_int1} by \eqref{max_char_gap_int2} in Lemma \ref{max_char_gap_int} gives that,
$$\Lambda_{k, \rm iso}(n) = \frac{\Lambda_k(n)\prod_{j = 1}^{u-1}\Lambda_{-k, (j)}(n)}{\prod_{j = 1}^{u}\Lambda_{-k, (j)}(n)} = \frac{\Lambda_k(n)}{\Lambda_{-k, (u)}(n)}.$$
Therefore, $\Lambda_{k, \rm iso}(n) > 1$ if and only if $k \in \{i_1(n), \dots, i_u(n)\}$. %which implies $|D_{\rm iso}(n)| = u$. 
\end{proof}

\begin{lemma}\label{iso_char_gap}
Consider $0 < l = u < K$ and $s'_k = [K]$ for every $k \in [K]$.
Then we have for every $k \in [K]$,
$$\Lambda_{k, \rm iso}(n) = \frac{\Lambda_k(n)}{\Lambda_{-k, (u)}(n)},$$
and in this case,  $D_{\rm iso}(n)$ consists of sources with indices $\{i_1(n), \dots, i_{u}(n)\}$.
\end{lemma}
\begin{proof}
For every $k \in [K]$, we have $m_k^1(n) = u-1$, $m_k^0(n) = u$, and therefore, dividing \eqref{max_char_gap_int1} by \eqref{max_char_gap_int2} in Lemma \ref{max_char_gap_int} gives that
$$\Lambda_{k, \rm iso}(n) = \frac{\Lambda_k(n)\prod_{j = 1}^{u-1}\Lambda_{-k, (j)}(n)}{\prod_{j = 1}^{u}\Lambda_{-k, (j)}(n)} = \frac{\Lambda_k(n)}{\Lambda_{-k, (u)}(n)},$$
which is greater that $1$ if and only if $k \in \{i_1(n), \dots, i_u(n)\}$. %Thus, we have $|D_{\rm iso}(n)| = m$.
\end{proof}

\subsection{Proof of Theorem \ref{decision_indep}}\label{pf:decision_indep}

\begin{proof}
Fix $(T, D)$ to be either $\chi^*$ or $\chi^*_{fwer}$.
First consider the case when $l = 0$. Then by Lemma \ref{max_det_char} if $p(T) = 0$, then
$$\max_{k \in [K]} \; \Lambda_{k, \rm det}(T) < 1,$$
which further implies that the test stops by $T_0$, i.e., $T = T_0$. Thus, by definition $D = \emptyset$.

Next, we consider the general case when $l \geq 0$ and $p(T) \geq 0$. Then by Lemmas \ref{iso_char_gap_int1}, \ref{iso_char_gap_int2}, \ref{iso_char_gap_int3}, \ref{iso_char_gap_int4}, \ref{iso_char_gap_int5}, we have
$$
D(T) =
\begin{cases}
\{i_1(T), \dots, i_{l \vee 1}(T)\} \quad &\text{if} \;\; p(T) \leq l \vee 1\\
\{i_1(T), \dots, i_{p(T)}(T)\} \quad &\text{if} \;\; l \vee 1 < p(T) < u\\
\{i_1(T), \dots, i_{u}(T)\} \quad &\text{if} \;\; u \leq p(T)
\end{cases}.$$
This completes the proof.
\end{proof}

\subsection{Proof of Theorem \ref{th:iso_omega}}\label{pf:th:iso_omega}

\begin{proof}
We have  $ l\geq 1$ and $s'_k = [K]$ for every  $k \in [K]$.
\begin{enumerate}
\item[(i)] By using Lemma \ref{iso_char_gap} and the fact that $A_k = A$ and $B_k = B$ for every $k \in [K]$, one can write the stopping time $T^*$ as
\begin{align*}
T^*&= \inf\bigg\{n \geq 1: \frac{\Lambda_k(n)}{\Lambda_{-k, (u)}(n)} \notin \lt(\frac{1}{A}, B\rt)\;\; \text{for every} \, k \in [K] \quad \text{and} \quad |D_{\rm iso}(n)| = u\bigg\}.
\end{align*}
Now note that, 
$$\Lambda_{-k, (u)}(n) = \begin{cases}
\Lambda_{(u+1)}(n)\ &\text{if}\ k \in \{i_1(n), \dots, i_u(n)\},\\
\Lambda_{(u)}(n)\ &\text{otherwise}.
\end{cases}$$
Therefore,
\begin{align*}
&\lt\{\frac{\Lambda_k(n)}{\Lambda_{-k, (u)}(n)} \notin \lt(\frac{1}{A}, B\rt)\;\; \text{for every} \, k \in [K] \quad \text{and} \quad |D_{\rm iso}(n)| = u\rt\}\\
= &\bigg\{\frac{\Lambda_k(n)}{\Lambda_{(u+1)}(n)} \geq B\;\; \text{for every}\ k \in \{i_1(n), \dots, i_u(n)\}\\
&\qquad \frac{\Lambda_k(n)}{\Lambda_{(u)}(n)} \leq \frac{1}{A}\;\; \text{for every}\ k \notin \{i_1(n), \dots, i_u(n)\}  \quad \text{and} \quad |D_{\rm iso}(n)| = u\bigg\}\\
= &\lt\{\frac{\Lambda_{(u)}(n)}{\Lambda_{(u+1)}(n)} \geq \max\{A, B\}\rt\}.
\end{align*}
Note that, the condition $|D_{\rm iso}(n)| = u$ disappears from the second equality by another application of Lemma \ref{iso_char_gap}.  The proof is complete.
%since the conditions
%$$\Lambda_{k, \rm iso}(n) = 
%\frac{\Lambda_k(n)}{\Lambda_{(m+1)}(n)} \geq B > 1\;\; \text{for every}\ k \in \{i_1(n), \dots, i_m(n)\},$$
%and
%$$\Lambda_{k, \rm iso}(n) = \frac{\Lambda_k(n)}{\Lambda_{(m)}(n)} \leq \frac{1}{A} < 1\;\; \text{for every}\ e \notin \{i_1(n), \dots, i_m(n)\}$$
%together imply that $|D_{\rm iso}(n)| = m$. Therefore, at $T^*_{iso}$ one must have exactly $m$ isolation statistics greater than $1$ and from the above two conditions, this leads to
%$$D^{*}_{iso} = \{i_1(T^*_{iso}), \dots, i_m(T^*_{iso})\}.$$ 
\item[(ii)] Since $A_k = A$ and $B_k = B$ for every $k \in [K]$, one can write the stopping time $T^*$ as
\begin{align*}
T^* = \inf\lt\{n \geq 1 : \Lambda_{k, \rm iso}(n) \notin \lt(\frac{1}{A}, B\rt)\;\; \text{for every}\ k \in [K]\;\; \text{and}\;\; l \leq |D_{\rm iso}(n)| \leq u\rt\}.
\end{align*}
%Define another stopping time
%\begin{align*}
% T^{**}_{I} := \inf\lt\{n \geq 1 : \Lambda_{k, \rm iso}(n) \notin \lt(\frac{1}{A_3}, B_3\rt)\quad \text{for every}\ k \in [K],\ \text{and}\ l \leq p(n) \leq u\rt\}.
%\end{align*}
Now from Lemmas \ref{iso_char_gap_int1}, \ref{iso_char_gap_int2}, \ref{iso_char_gap_int3}, \ref{iso_char_gap_int4} and \ref{iso_char_gap_int5}, one can observe that at any time $n$ and for any $u \leq K$, the condition $l \leq |D_{\rm iso}(n)| \leq u$ is always satisfied irrespective of the value of $p(n)$. Hence, we omit this condition from the stopping event of $T^*$.
Next, let us consider the case when $u < K$ and in this case, we will show that $T^* = \min\{T_1, T_2, T_4, T_5, T_6\}$. In order to do so,
we further intersect the stopping event of $T^*$ with the disjoint decomposition of these following events:
$$\{p(n) < l\} \cup \{p(n) = l\} \cup \{l < p(n) < u\} \cup \{p(n) = u\} \cup \{p(n) > u\}.$$
%and show that each event corresponds to the stopping events of the rules $T_1, \dots, T_5$. 

Now, for the first event, i.e., when $p(n) < l$, by Lemma \ref{iso_char_gap_int1}, we have
\begin{align*}
&\lt\{p(n) < l\ \text{and}\ \Lambda_{k, \rm iso}(n) \notin \lt(\frac{1}{A}, B\rt)\quad \text{for every}\ k \in [K]\rt\}\\
= &\bigg\{p(n) < l, \frac{\Lambda_k(n)}{\Lambda_{(l + 1)}(n)} \geq B\quad \text{for every}\ k \in \{i_1(n), \dots, i_l(n)\}\\
&\qquad \frac{\Lambda_k(n)}{\Lambda_{(l)}(n)} \leq \frac{1}{A}\quad \text{for every}\ k \notin \{i_1(n), \dots, i_l(n)\}\bigg\}\\
= &\lt\{p(n) < l\ \text{and}\ \frac{\Lambda_{(l)}(n)}{\Lambda_{(l + 1)}(n)} \geq \max\{A, B\}\rt\},
\end{align*}
which is the stopping event of $T_1$.
For the second event, i.e., when $p(n) = l$, by Lemma \ref{iso_char_gap_int2}, we have
\begin{align*}
&\lt\{p(n) = l\ \text{and}\ \Lambda_{k, \rm iso}(n) \notin \lt(\frac{1}{A}, B\rt)\quad \text{for every}\ k \in [K]\rt\}\\
= &\bigg\{p(n) = l, \frac{\Lambda_k(n)}{\Lambda_{(l + 1)}(n)} \geq B\quad \text{for every}\ k \in \{i_1(n), \dots, i_l(n)\}\\
&\qquad \Lambda_k(n) \leq \frac{1}{A}\quad \text{for every}\ k \notin \{i_1(n), \dots, i_l(n)\}\bigg\}\\
= &\lt\{p(n) = l, \frac{\Lambda_{(l)}(n)}{\Lambda_{(l + 1)}(n)} \geq B\;\; \text{and} \;\; \Lambda_{(l + 1)}(n) \leq \frac{1}{A}\rt\},
\end{align*}
which is the stopping event of $T_2$.
For the third even, i.e., when $l < p(n) < u$, by Lemma \ref{iso_char_gap_int3}, we have
\begin{align*}
&\lt\{l < p(n) < u\ \text{and}\ \Lambda_{k, \rm iso}(n) \notin \lt(\frac{1}{A}, B\rt)\quad \text{for every}\ k \in [K]\rt\}\\
= &\lt\{l < p(n) < u\ \text{and}\ \Lambda_k(n) \notin \lt(\frac{1}{A}, B\rt)\quad \text{for every}\ k \in [K]\rt\},
\end{align*}
which is the stopping event of $T_4$.
For the fourth event, i.e., when $p(n) = u$, by Lemma \ref{iso_char_gap_int4}, we have
\begin{align*}
&\lt\{p(n) = u\quad \text{and}\quad \Lambda_{k, \rm iso}(n) \notin \lt(\frac{1}{A}, B\rt)\quad \text{for every}\ k \in [K]\rt\}\\
= &\bigg\{p(n) = u, \quad \Lambda_k(n) \geq B\quad \text{for every}\ k \in \{i_1(n), \dots, i_u(n)\}\\
&\qquad \frac{\Lambda_k(n)}{\Lambda_{(u)}(n)} \leq \frac{1}{A}\quad \text{for every}\ k \notin \{i_1(n), \dots, i_u(n)\bigg\}\\
= &\lt\{p(n) = u,\quad \Lambda_{(u)}(n) \geq B\quad \text{and}\ \frac{\Lambda_{(u + 1)}(n)}{\Lambda_{(u)}(n)} \leq \frac{1}{A}\rt\},
\end{align*}
which is the stopping event of $T_5$. 
For the fifth event, i.e. when $p(n) > u$, by Lemma \ref{iso_char_gap_int5}, we have
\begin{align*}
&\lt\{p(n) > u\ \text{and}\ \Lambda_{k, \rm iso}(n) \notin \lt(\frac{1}{A}, B\rt)\quad \text{for every}\ k \in [K]\rt\}\\
= &\bigg\{p(n) > u, \frac{\Lambda_k(n)}{\Lambda_{(u + 1)}(n)} \geq B\quad \text{for every}\ k \in \{i_1(n), \dots, i_u(n)\}\\
&\qquad \frac{\Lambda_k(n)}{\Lambda_{(u)}(n)} \leq \frac{1}{A}\quad \text{for every}\ k \notin \{i_1(n), \dots, i_u(n)\}\bigg\}\\
= &\lt\{p(n) > u\ \text{and}\ \frac{\Lambda_{(u)}(n)}{\Lambda_{(u + 1)}(n)} \geq \max\{A, B\}\rt\},
\end{align*}
which is the stopping event of $T_6$.
Hence, the result follows.

Next consider the case when $u = K$. We similarly intersect the stopping event of $T^*$ with the disjoint decomposition of these following events:
$$\{p(n) < l\} \cup \{p(n) = l\} \cup \{l < p(n) \leq K\}.$$
For the third event in above, i.e., when $l < p(n) \leq K$, by Lemma \ref{iso_char_gap_int3} and \ref{iso_char_gap_int4}, we have
\begin{align*}
&\lt\{l < p(n) \leq K, \;\; \Lambda_{k, \rm iso}(n) \notin \lt(\frac{1}{A}, B\rt) \;\; \text{for every}\ k \in [K]\rt\}\\
= &\lt\{l < p(n) \leq K, \;\;  \Lambda_k(n) \notin \lt(\frac{1}{A}, B\rt) \;\; \text{for every}\ k \in [K]\rt\},
\end{align*}
which is the stopping event of $T_3$. Therefore, in this case, $T^* = \min\{T_1, T_2, T_3\}$.  This completes the proof.
\end{enumerate}
\end{proof}

\subsection{Proof of Theorem \ref{th:dete_indep}}\label{pf:th:dete_indep}

\begin{proof}
Since $C_k = C$ and $D_k = D$ for every $k \in [K]$, one can write the stopping times $T_0$ and $T_{det}$ as
\begin{align*}
T_0 &= \inf\lt\{n  \in \bN : \max_{k \in [K]} \; \Lambda_{k, \rm det}(n) \leq 1/C \rt\},\\
T_{det} &= \inf\lt\{n \in \bN: \max_{k \in [K]} \; \Lambda_{k, \rm det}(n) \geq D\rt\}.
\end{align*}
Since $C, D > 1$, stopping happens by $T_0$ (resp. $T_{det}$) only if $\max_{k \in [K]} \; \Lambda_{k, \rm det}(n) \leq 1$ (resp. $\max_{k \in [K]} \; \Lambda_{k, \rm det}(n) > 1$), and by Lemma \ref{max_det_char}, this happens only if $p(n) = 0$ (resp. $p(n) \geq 1$). Therefore, the stopping events in $T_0$ and $T_{det}$ remain the same even if they are intersected with $p(n) = 0$ and $p(n) \geq 1$ respectively. 
Thus,
\begin{align*}
T_0 &= \inf\lt\{n  \in \bN : \max_{k \in [K]} \; \Lambda_{k, \rm det}(n) \leq 1/C, \;\; p(n) = 0 \rt\}\\
&= \inf\lt\{n  \in \bN : \Lambda_{(1)}(n) \leq 1/C\rt\},
\end{align*}
where the last equality follows from Lemma \ref{max_det_char}. Similarly, we have
\begin{align*}
T_{det} &= \inf\lt\{n  \in \bN : \max_{k \in [K]} \; \Lambda_{k, \rm det}(n) \geq D, \;\; p(n) \geq 1 \rt\}\\
&= \inf\lt\{n  \in \bN : \prod_{i = 1}^{p(n) \wedge u} \Lambda_{(i)}(n) \geq D\rt\},
\end{align*}
where the last equality again follows from Lemma \ref{max_det_char}, and this completes the first part.

Next we consider joint detection and isolation. Since $A_k = A$, $B_k = B$, $C_k = C$ and $D_k = D$ for every $k \in [K]$, one can write the stopping times $T_{joint}$ as
\begin{align*}
%T_0 &= \inf\lt\{n  \in \bN : \max_{k \in [K]} \; \Lambda_{k, \rm det}(n) \leq 1/C \rt\},\\
T_{joint} &= \inf\{ n \in \bN : \; \max_{k \in [K]} \; \Lambda_{k, \rm det}(n) \geq D,\\
& \qquad \qquad \qquad \;\; \Lambda_{k, \rm iso}(n) \notin \lt(1/A, B\rt) \; \; \text{for every}  \, k \in [K] \quad  \text{and} \quad 1 \leq |D_{\rm iso}(n)| \leq u \}.
\end{align*}
Since $D > 1$, stopping happens by $T_{joint}$ only if $\max_{k \in [K]} \; \Lambda_{k, \rm det}(n) > 1$, and by Lemma \ref{max_det_char}, this happens only if $p(n) \geq 1$. Therefore, the stopping event in $T_{joint}$ remains the same even if it is intersected with $p(n) \geq 1$. 
Thus,
\begin{align*}
T_{joint} &= \inf\{ n \in \bN : \; \max_{k \in [K]} \; \Lambda_{k, \rm det}(n) \geq D, \;\; p(n) \geq 1,\\
& \qquad \qquad \qquad \;\; \Lambda_{k, \rm iso}(n) \notin \lt(1/A, B\rt) \; \; \text{for every}  \, k \in [K] \quad  \text{and} \quad 1 \leq |D_{\rm iso}(n)| \leq u \},\\
&= \inf\bigg\{n  \in \bN : \prod_{i = 1}^{p(n) \wedge u} \Lambda_{(i)}(n) \geq D,  \;\; p(n) \geq 1,\\
& \qquad \qquad \qquad \;\; \Lambda_{k, \rm iso}(n) \notin \lt(1/A, B\rt) \; \; \text{for every}  \, k \in [K] \quad  \text{and} \quad 1 \leq |D_{\rm iso}(n)| \leq u \bigg\},
\end{align*}
where the last equality again follows from Lemma \ref{max_det_char}.

Now from Lemmas \ref{iso_char_gap_int2}, \ref{iso_char_gap_int3}, \ref{iso_char_gap_int4} and \ref{iso_char_gap_int5}, one can observe that for any $n \in \bN$ and $u \leq K$, $p(n) \geq 1$ implies $1 \leq |D_{\rm iso}(n)| \leq u$. Hence, we omit this condition from the stopping event of $T_{joint}$.
First, let us consider the case when $u < K$ and in this case, we will show that $T_{joint} = \min\{T_1, T_3, T_4, T_5\}$. In order to do so,
we further intersect the stopping event of $T_{joint}$ with the disjoint decomposition of $\{p(n) \geq 1\}$ into these following events:
$$\{p(n) \geq 1\} = \{p(n) = 1\} \cup \{1 < p(n) < u\} \cup \{p(n) = u\} \cup \{p(n) > u\}.$$
%and show that each event corresponds to the stopping events of the rules $T_1, \dots, T_5$. 
Now, for the first event, i.e., when $p(n) = 1$, by Lemma \ref{iso_char_gap_int2}, we have
\begin{align*}
&\lt\{\prod_{i = 1}^{p(n) \wedge u} \Lambda_{(i)}(n) \geq D,  \;\; p(n) = 1, \;\; \Lambda_{k, \rm iso}(n) \notin \lt(1/A, B\rt) \; \; \text{for every}  \, k \in [K]\rt\}\\
= &\lt\{\Lambda_{(1)}(n) \geq D, \;\; \frac{\Lambda_k(n)}{\Lambda_{(2)}(n)} \geq B \;\; \text{for}\;\; k  = i_1(n), \;\;
\Lambda_k(n) \leq \frac{1}{A} \;\; \text{for every}\ k \neq i_1(n)\rt\}\\
= &\lt\{\Lambda_{(1)}(n) \geq D, \;\; \frac{\Lambda_{(1)}(n)}{\Lambda_{(2)}(n)} \geq B, \;\;
\Lambda_{(2)}(n) \leq \frac{1}{A}\rt\},
\end{align*}
which is the stopping event of $T_1$.
For the second event, i.e., when $1 < p(n) < u$, by Lemma \ref{iso_char_gap_int3}, we have
\begin{align*}
&\lt\{\prod_{i = 1}^{p(n) \wedge u} \Lambda_{(i)}(n) \geq D, \;\; 1 < p(n) < u, \;\; \Lambda_{k, \rm iso}(n) \notin \lt(\frac{1}{A}, B\rt) \;\; \text{for every}\ k \in [K]\rt\}\\
= &\lt\{\prod_{i = 1}^{p(n)} \Lambda_{(i)}(n) \geq D, \;\; 1 < p(n) < u, \;\;  \Lambda_k(n) \notin \lt(\frac{1}{A}, B\rt) \;\; \text{for every}\ k \in [K]\rt\},
\end{align*}
which is the stopping event of $T_3$.
For the third event, i.e., when $p(n) = u$, by Lemma \ref{iso_char_gap_int4}, we have
\begin{align*}
&\lt\{\prod_{i = 1}^{p(n) \wedge u} \Lambda_{(i)}(n) \geq D, \;\; p(n) = u, \;\; \Lambda_{k, \rm iso}(n) \notin \lt(\frac{1}{A}, B\rt) \;\; \text{for every}\ k \in [K]\rt\}\\
= &\Bigg\{\prod_{i = 1}^{u} \Lambda_{(i)}(n) \geq D, \;\; p(n) = u, \;\; \Lambda_k(n) \geq B \;\; \text{for every}\ k \in \{i_1(n), \dots, i_u(n)\},\\
&\qquad \qquad \qquad \frac{\Lambda_k(n)}{\Lambda_{(u)}(n)} \leq \frac{1}{A} \;\; \text{for every}\ k \notin \{i_1(n), \dots, i_u(n)\Bigg\}\\
= &\lt\{\prod_{i = 1}^{u} \Lambda_{(i)}(n) \geq D, \;\; p(n) = u, \;\; \Lambda_{(u)}(n) \geq B \;\; \text{and} \;\; \frac{\Lambda_{(u + 1)}(n)}{\Lambda_{(u)}(n)} \leq \frac{1}{A}\rt\},
\end{align*}
which is the stopping event of $T_4$. 
For the fourth event, i.e., when $p(n) > u$ by Lemma \ref{iso_char_gap_int5}, we have
\begin{align*}
&\lt\{\prod_{i = 1}^{p(n) \wedge u} \Lambda_{(i)}(n) \geq D, \;\; p(n) > u, \;\; \Lambda_{k, \rm iso}(n) \notin \lt(\frac{1}{A}, B\rt) \;\; \text{for every}\ k \in [K]\rt\}\\
= &\Bigg\{\prod_{i = 1}^{u} \Lambda_{(i)}(n) \geq D, \;\; p(n) > u, \;\; \frac{\Lambda_k(n)}{\Lambda_{(u + 1)}(n)} \geq B \;\; \text{for every}\ k \in \{i_1(n), \dots, i_u(n)\}\\
&\qquad \qquad \qquad \frac{\Lambda_k(n)}{\Lambda_{(u)}(n)} \leq \frac{1}{A} \;\; \text{for every}\ k \notin \{i_1(n), \dots, i_u(n)\}\Bigg\}\\
= &\lt\{\prod_{i = 1}^{u} \Lambda_{(i)}(n) \geq D, \;\; p(n) > u, \;\; \frac{\Lambda_{(u)}(n)}{\Lambda_{(u + 1)}(n)} \geq \max\{A, B\}\rt\},
\end{align*}
which is the stopping event of $T_5$. Hence, the result follows.

Next we consider the case when $u = K$, and similarly intersect the stopping event of $T_{joint}$ with the disjoint decomposition of $\{p(n) \geq 1\}$ into these following events:
$$\{p(n) \geq 1\} = \{p(n) = 1\} \cup \{1 < p(n) \leq K\}.$$
For the second event in above, i.e., when $1 < p(n) \leq K$, by Lemma \ref{iso_char_gap_int3} and \ref{iso_char_gap_int4}, we have
\begin{align*}
&\lt\{\prod_{i = 1}^{p(n) \wedge K} \Lambda_{(i)}(n) \geq D, \;\; 1 < p(n) \leq K, \;\; \Lambda_{k, \rm iso}(n) \notin \lt(\frac{1}{A}, B\rt) \;\; \text{for every}\ k \in [K]\rt\}\\
= &\lt\{\prod_{i = 1}^{p(n)} \Lambda_{(i)}(n) \geq D, \;\; 1 < p(n) \leq K, \;\;  \Lambda_k(n) \notin \lt(\frac{1}{A}, B\rt) \;\; \text{for every}\ k \in [K]\rt\},
\end{align*}
which is the stopping event of $T_2$. Therefore, in this case $T_{joint} = \min\{T_1, T_2\}$. 
%Finally, if we let $A = C$ and $B = D$, then for every $k \in [K]$, we have
%$$\log A_e = \log A = \log C = \log C_e \;\quad \text{and} \;\quad \log B_e = \log B = \log D = \log D_e,$$
%which essentially implies $\chi^* = \chi^*_{fwer}$.
This completes the proof.
\end{proof}

\subsection{Some important lemmas}
Consider $l = 0$, i.e., $\Psi = \Psi_{0, u}$.
From Proposition \ref{propo:sufficient}, for any $\sfP, \sfQ \in \cP_\Psi$, we have
\begin{equation}\label{sum_of_info1}
\cI(\sfP, \sfQ) = \sum_{\{k\} \in \cA(\sfP)} \cI(\sfP^k, \sfQ^k) + \sum_{\{k\} \notin \cA(\sfP)} \cI(\sfP^k, \sfQ^k).
\end{equation}
Based on the above observation, we now present two important lemmas which will be useful to establish Theorem \ref{th:indep_fwer}.
\begin{lemma}\label{P-P0}
For any distribution $\sfP \in \cP_\Psi \setminus \cH_0$,
$$\cI(\sfP, \cH_0) = \sum_{\{k\} \in \cA(\sfP)} \cI(\sfP^k, \cH^k).$$
\end{lemma}

\begin{proof}

Due to \eqref{sum_of_info1} it is not difficult to observe that the minimum of $\cI(\sfP, \sfQ)$ over $\sfQ \in \cH_0$ is obtained by $\sfQ_*$ that satisfies the following: for any $k \in [K]$
$$\sfQ_*^k = 
\begin{cases}
\sfP^k \quad &\text{if} \;\; \{k\} \notin \cA(\sfP)\\
\argmin_{\sfQ \in \cH^k} \cI\lt(\sfP^k, \sfQ\rt) &\text{otherwise}
\end{cases},$$
which yields $$\cI(\sfP, \cH_0) = \cI(\sfP, \sfQ_*) = \sum_{\{k\} \in \cA(\sfP)} \cI(\sfP^k, \sfQ_*^k) + \sum_{\{k\} \notin \cA(\sfP)} \cI(\sfP^k, \sfQ_*^k) = \sum_{\{k\} \in \cA(\sfP)} \cI(\sfP^k, \cH^k).$$
\end{proof}

\begin{lemma}\label{P-P1}
For any distribution $\sfP \in \cP_\Psi \setminus \cH_0$,
$$\min_{\{k\} \in \cA(\sfP)}\cI(\sfP, \cH_{\Psi, k}) = \min_{\{k\} \in \cA(\sfP)} \cI(\sfP^k, \cH^k) + \mathbbm{1}_{\{|\cA(\sfP)| = 1\}}\min_{\{k\} \notin \cA(\sfP)} \cI(\sfP^k, \cG^k),$$
where $\mathbbm{1}_A$ is the indicator of $A$, i.e., it takes the value $1$ if $A$ is true, or $0$, otherwise.
\end{lemma}

\begin{proof}
Since we want to minimize $\cI(\sfP, \sfQ)$ over all $\sfQ$ such that there exists $\{k\} \in \cA(\sfP)$ for which $\sfQ \in \cH_{\Psi, k}$. i.e., all $\sfQ$ such that $\cA(\sfP) \setminus \cA(\sfQ) \neq \emptyset$ and $|\cA(\sfQ)| \geq 1$. First, consider the case when $|\cA(\sfP)| = 1$ and let $\cA(\sfP) = \{\{k_1\}\}$ for some $k_1 \in [K]$. Then due to \eqref{sum_of_info1} we must restrict our attention to all possible $\sfQ$ such that $\{k_1\} \notin \cA(\sfQ)$ and $|\cA(\sfQ)| = 1$. Fix some $k_2 \in \argmin_{k \neq k_1} \cI\lt(\sfP^k, \cG^k\rt)$. Then, the minimum is obtained by $\sfQ_*$ that satisfies the following: for any $k \in [K]$
$$\sfQ_*^k = 
\begin{cases}
\argmin_{\sfQ \in \cH^k} \cI\lt(\sfP^k, \sfQ\rt) \quad &\text{if} \;\; k = k_1\\
\argmin_{\sfQ \in \cG^k} \cI\lt(\sfP^k, \sfQ\rt) &\text{if} \;\; k = k_2\\
\sfP^k &\text{otherwise}
\end{cases},$$
which yields
$$\min_{\{k\} \in \cA(\sfP)}\cI(\sfP, \cH_{\Psi, k}) = \cI(\sfP, \sfQ_*) = \cI(\sfP^{k_1}, \cH^{k_1}) + \min_{k \neq k_1} \cI(\sfP^k, \cG^k).$$

Next, consider the case when $|\cA(\sfP)| > 1$. Then due to \eqref{sum_of_info1} we must restrict our attention to all possible $\sfQ$ such that its distribution differs from $\sfP$ in exactly one source, which is a signal under $\sfP$ and a non-signal under $\sfQ$. Fix some $k_1 \in  \argmin_{\{k\} \in \cA(\sfP)} \cI(\sfP^k, \cH^k)$ Then, the minimum is obtained by $\sfQ_*$ that satisfies the following: for any $k \in [K]$
$$\sfQ_*^k = 
\begin{cases}
\argmin_{\sfQ \in \cH^k} \cI\lt(\sfP^k, \sfQ\rt) &\text{if} \;\; k = k_1\\
\sfP^k &\text{otherwise}
\end{cases},$$
which yields
$$\min_{\{k\} \in \cA(\sfP)}\cI(\sfP, \cH_{\Psi, k}) = \cI(\sfP, \sfQ_*) = \min_{\{k\} \in \cA(\sfP)} \cI(\sfP^k, \cH^k).$$
The proof is complete.
\end{proof}

\subsection{Proof of Theorem \ref{th:indep_fwer}}\label{pf:th:indep_fwer}

\begin{proof}
We have  $l=0$ and $s_k = \{k\}$  for every $k \in [K]$.
\begin{enumerate}
\item[(i)]
Since $u = K$, $s_k = s'_k = \{k\}$, and $A_k = A$ and $B_k = B$ for every $k \in [K]$, we have
\begin{align*}
T^*_{fwer} &= \inf\{ n \in \bN : \; \Lambda_k(n) \notin \lt(1/A, B\rt) \; \; \text{for every}  \, k \in [K] \;\;  \text{and} \;\; D_{\rm iso}(n) \in \Psi_{0, K} \},
\end{align*}
As $\Psi_{0, K}$ is the powerset of  $[K]$, we can omit the condition $\{D_{\rm iso}(n) \in \Psi_{0, K}\}$ from the stopping event of $T^*_{fwer}$.
 This finishes the first part of the proof. Next, we prove its asymptotic optimality.

 %we will use Theorem \ref{P0}, Theorem \ref{gamlesdeltaforMT} and Theorem \ref{P2}. 
First, let us fix any arbitrary $\sfP \in \cH_0$. In order to prove that this rule is asymptotically optimal under $\sfP$, it suffices to show that \eqref{prop:PminusP0_1} holds when $s_k = \{k\}$ for every $k \in [K]$. Then by Proposition \ref{prop:P1}(iii), this will further imply that \eqref{condPminusP0} in Theorem \ref{PminusP0} is satisfied and the result will follow. 
Fix some $\bar{k} \in [K]$ such that
\begin{equation*}
\bar{k} \in \argmin_{k \in [K]} \cI (\sfP^k, \cG^k),
\end{equation*}
and let $s = s_{\bar{k}} = \{\bar{k}\}$.
Note that, under $\sfP$, $s$ is independent of $s^c$ due to independence across the sources. Furthermore, for any $\sfQ_1$ such that
$$\sfQ_1 \in \argmin_{\sfQ \in \cG_{\Psi, \bar{k}}} \cI \lt(\sfP, \sfQ; \bar{k}\rt) \quad \text{implies that} \quad \sfQ_1^{\bar{k}} \in 
\cG^{\bar{k}}. %= \argmin_{\sfQ \in \cG^{\bar{k}}} \cI(\sfP^{\bar{k}}, \sfQ).
$$
Thus, $|\cA\lt(\sfQ_1^{s} \otimes \sfP^{s^c}\rt)| = 1$, and also, under $\sfQ_1^{s} \otimes \sfP^{s^c}$ the sources are still independent, which imply that
$$\sfQ_1^{s} \otimes \sfP^{s^c} \in \cP_\Psi \setminus \cH_0.$$
%\begin{equation}
%\label{prop:PminusP0_2}
%$Q_1^{W_{e'}} \otimes \sfP^{W_{e'}^c} \in \cP_\Psi \setminus \cH_0,$
%\end{equation}
Therefore, \eqref{prop:PminusP0_1} holds. 

Next, we fix any arbitrary $\sfP \in \cP_\Psi \setminus \cH_0$. In order to prove that this rule is asymptotically optimal under $\sfP$, we use Theorem \ref{P0P1P2}. We first show that \eqref{prop:P2_1} holds when $s'_k = \{k\}$ for every $k \in [K]$. Then by Proposition \ref{prop:P1}(ii), this will further imply that \eqref{condP2} is satisfied.
Fix $\{\bar{k}\} \notin \cA(\sfP)$ be the unit such that
\begin{equation*}
\{\bar{k}\} \in \argmin_{\{k\} \notin \cA(\sfP)} \cI (\sfP^k, \cG^k),
\end{equation*}
and let $s = s_{\bar{k}} = \{\bar{k}\}$.
Note that, under $\sfP$, $s$ is independent of $s^c$ due to independence across the sources. Furthermore, for any $\sfQ_1$ such that
$$\sfQ_1 \in \argmin_{\sfQ \in \cG_{\Psi, \bar{k}}} \cI \lt(\sfP, \sfQ; \bar{k}\rt) \quad \text{implies that} \quad \sfQ_1^{\bar{k}} \in 
\cG^{\bar{k}}. %= \argmin_{\sfQ \in \cG^{\bar{k}}} \cI(\sfP^{\bar{k}}, \sfQ).
$$
Thus, $|\cA\lt(\sfQ_1^{s} \otimes \sfP^{s^c}\rt)| = |\cA(\sfP)| + 1 \geq 1$, and also, under $\sfQ_1^{s} \otimes \sfP^{s^c}$ the sources are still independent, which imply that
$$\sfQ_1^{s} \otimes \sfP^{s^c} \in \cP_\Psi \setminus \cH_0.$$
%\begin{equation}
%\label{prop:PminusP0_2}
%$Q_1^{W_{e'}} \otimes \sfP^{W_{e'}^c} \in \cP_\Psi \setminus \cH_0,$
%\end{equation}
Therefore, \eqref{prop:P2_1} holds. Now, we also use Remark \ref{rem:fwer} to complete the rest of this proof.

First consider the case when $|\cA(\sfP)| = 1$ and let $\cA(\sfP) = \{\{k_1\}\}$. Since $s_k = s'_k = \{k\}$ for every $k \in [K]$,
$$\max_{\{k\} \in \cA(\sfP)} \cI(\sfP, \cH_0; s_k) = \cI(\sfP^{k_1}, \cH^{k_1}) = \min_{\{k\} \in \cA(\sfP)} \cI(\sfP,  \cH_{\Psi, k}; s'_k).$$
Furthermore, from Lemma \ref{P-P0} and Lemma \ref{P-P1}, we have
$$\cI(\sfP, \cH_0)= \cI(\sfP^{k_1}, \cH^{k_1}) \leq \cI(\sfP^{k_1}, \cH^{k_1}) + \min_{k \neq k_1} \cI(\sfP^k, \cG^{k}) = \min_{\{k\} \in \cA(\sfP)}\cI(\sfP, \cH_{\Psi, k}).$$
Thus, \eqref{FWERcondP0} holds, and
also, the first equality in above implies that \eqref{condP0} is satisfied. Therefore, by using Remark \ref{rem:fwer} the result follows.

Next, we consider the case when $|\cA(\sfP)| > 1$. Since $s_k = s'_k = \{k\}$ for every $k \in [K]$,
$$\max_{\{k\} \in \cA(\sfP)} \cI(\sfP, \cH_0; s_k) = \max_{\{k\} \in \cA(\sfP)} \cI(\sfP^k, \cH^k) \geq \min_{\{k\} \in \cA(\sfP)} \cI(\sfP^k, \cH^k) = \min_{\{k\} \in \cA(\sfP)} \cI(\sfP,  \cH_{\Psi, k}; s'_k).$$
Furthermore, from Lemma \ref{P-P0} and Lemma \ref{P-P1}, we have
$$\cI(\sfP, \cH_0) =  \sum_{\{k\} \in \cA(\sfP)} \cI(\sfP^k, \cH^k) \geq \min_{\{k\} \in \cA(\sfP)} \cI(\sfP^k, \cH^k) = \min_{\{k\} \in \cA(\sfP)} \cI(\sfP, \cH_{\Psi, k}).$$
Thus, \eqref{FWERcondP1} holds, and also, the last equality in above implies that \eqref{condP1} is satisfied.
Therefore, by using Remark \ref{rem:fwer} the result follows.
This completes the proof.

\item[(ii)]
Since $s_k = \{k\}$, $s'_k = [K]$, and $A_k = A$, $B_k = B$ for every $k \in [K]$, we have $T^*_{fwer} = \min\{T_0, T_{joint}\}$, where one can write the stopping times $T_0$ and $T_{joint}$ as
\begin{align*}
T_0 &= \inf\lt\{n  \in \bN : \max_{k \in [K]} \; \Lambda_k(n) \leq 1/A \rt\} = \inf\lt\{n  \in \bN : \Lambda_{(1)}(n) \leq 1/A\rt\},\\
T_{joint} &= \inf\{ n \in \bN : \; \max_{k \in [K]} \; \Lambda_{e}(n) \geq B,\\
& \qquad \qquad \qquad \;\; \Lambda_{k, \rm iso}(n) \notin \lt(1/A, B\rt) \; \; \text{for every}  \, k \in [K] \quad  \text{and} \quad 1 \leq |D_{\rm iso}(n)| \leq u \}\\
&= \inf\{ n \in \bN : \; \Lambda_{(1)}(n) \geq D,\\
& \qquad \qquad \qquad \;\; \Lambda_{k, \rm iso}(n) \notin \lt(1/A, B\rt) \; \; \text{for every}  \, k \in [K] \quad  \text{and} \quad 1 \leq |D_{\rm iso}(n)| \leq u \}.
\end{align*}
Since $\Lambda_{(1)}(n) \geq D > 1$ implies that $p(n) \geq 1$, the stopping events of $T_{joint}$ remains the same if it is intersected with $p(n) \geq 1$. Again, from Lemmas \ref{iso_char_gap_int2}, \ref{iso_char_gap_int3}, \ref{iso_char_gap_int4} and \ref{iso_char_gap_int5}, one can observe that for any $n \in \bN$ and $u \leq K$, $p(n) \geq 1$ implies $1 \leq |D_{\rm iso}(n)| \leq u$.
Thus, one can safely replace the latter condition with the first and write $T_{joint}$ as
\begin{align*}
T_{joint} &= \inf\{ n \in \bN : \; \Lambda_{(1)}(n) \geq D, \\
& \qquad \qquad \qquad \;\; \Lambda_{k, \rm iso}(n) \notin \lt(1/A, B\rt) \; \; \text{for every}  \, k \in [K] \quad  \text{and} \quad p(n) \geq 1 \}.
\end{align*}
%Hence, we can safely omit the later condition from the stopping event of $T_{joint}$.
Now, it is enough to show that $T_{joint} = \min\{T_1, T_2, T_3, T_4\}$ and in order to do so,
we further intersect the stopping event of $T_{joint}$ with the disjoint decomposition of $\{p(n) \geq 1\}$ into these following events:
$$\{p(n) \geq 1\} = \{p(n) = 1\} \cup \{1 < p(n) < u\} \cup \{p(n) = u\} \cup \{p(n) > u\}.$$
The rest of the proof for the first part can be done in a similar manner as employed in the proof of Theorem \ref{th:dete_indep}. Next, we prove its asymptotic optimality.

First, let us fix any arbitrary $\sfP \in \cH_0$. In order to prove that this rule is asymptotically optimal under $\sfP$, it suffices to show that \eqref{prop:PminusP0_1} holds when $s_k = \{k\}$ for every $k \in [K]$. Then by Proposition \ref{prop:P1}(iii), this will further imply that \eqref{condPminusP0} in Theorem \ref{PminusP0} is satisfied and the result will follow. 
Fix some $\bar{k} \in [K]$ such that
\begin{equation*}
\bar{k} \in \argmin_{k \in [K]} \cI (\sfP^k, \cG^k),
\end{equation*}
and let $s = s_{\bar{k}} = \{\bar{k}\}$.
Note that, under $\sfP$, $s$ is independent of $s^c$ due to independence across the sources. Furthermore, for any $\sfQ_1$ such that
$$\sfQ_1 \in \argmin_{\sfQ \in \cG_{\Psi, \bar{k}}} \cI \lt(\sfP, \sfQ; \bar{k}\rt) \quad \text{implies that} \quad \sfQ_1^{\bar{k}} \in 
\cG^{\bar{k}}. %= \argmin_{\sfQ \in \cG^{\bar{k}}} \cI(\sfP^{\bar{k}}, \sfQ).
$$
Thus, $|\cA\lt(\sfQ_1^{s} \otimes \sfP^{s^c}\rt)| = 1$, and also, under $\sfQ_1^{s} \otimes \sfP^{s^c}$ the sources are still independent, which imply that
$$\sfQ_1^{s} \otimes \sfP^{s^c} \in \cP_\Psi \setminus \cH_0.$$
%\begin{equation}
%\label{prop:PminusP0_2}
%$Q_1^{W_{e'}} \otimes \sfP^{W_{e'}^c} \in \cP_\Psi \setminus \cH_0,$
%\end{equation}
Therefore, \eqref{prop:PminusP0_1} holds. 

Next, we fix any arbitrary $\sfP \in \cP_\Psi \setminus \cH_0$. In order to prove that this rule is asymptotically optimal under $\sfP$, we use Theorem \ref{P0P1P2}. Since $s'_k = [K]$ for every $k \in [K]$, both \eqref{condP1} and \eqref{condP2} are satisfied, however \eqref{condP0} is not satisfied when $|\cA(\sfP)| > 1$.
In this case, we use Remark \ref{rem:fwer} to solve this issue.

Since $|\cA(\sfP)| > 1$ and $s_k = \{k\}$, $s'_k = [K]$ for every $k \in [K]$,
\begin{align*}
\max_{\{k\} \in \cA(\sfP)} \cI(\sfP, \cH_0; s_k) &= \max_{\{k\} \in \cA(\sfP)} \cI(\sfP^k, \cH^k)
 \geq \min_{\{k\} \in \cA(\sfP)} \cI(\sfP^k, \cH^k)\\
 &= \min_{\{k\} \in \cA(\sfP)} \cI(\sfP, \cH_{\Psi, k}) = \min_{\{k\} \in \cA(\sfP)} \cI(\sfP,  \cH_{\Psi, k}; s'_k).
\end{align*}
Furthermore, from Lemma \ref{P-P0} and Lemma \ref{P-P1}, we have
$$\cI(\sfP, \cH_0) =  \sum_{\{k\} \in \cA(\sfP)} \cI(\sfP^k, \cH^k) \geq \min_{\{k\} \in \cA(\sfP)} \cI(\sfP^k, \cH^k) = \min_{\{k\} \in \cA(\sfP)} \cI(\sfP, \cH_{\Psi, k}).$$
Thus, \eqref{FWERcondP1} holds, and also, the last equality in above implies that \eqref{condP1} is satisfied.
Therefore, by using Remark \ref{rem:fwer} the result follows.
This completes the proof.
\end{enumerate}
\end{proof}

%The above rule is known as the \textit{Gap Rule} in the literature. 

\section{Detection and isolation of anomalous sources: the general case}\label{app:appl_dep_source}

\subsection{Proof of Theorem \ref{w_eP1_marg}}\label{pf:w_eP1_marg}

\begin{proof}
We have $\emptyset \in \Psi$, $\sfP \in \cP_\Psi \setminus \cH_0$ and $s'_k = \{k\}$ for every $k \in [K]$.
\begin{enumerate}
\item[(i)] If $\alpha, \beta, \gamma, \delta \to 0$ such that $|\log \gamma| >> |\log \alpha| \vee |\log \delta|$, from Lemma \ref{lem:DI_lb}(iii) we have
\begin{align*}
\inf \lt\{ \sfE_{\sfP}[T]:\; 
(T, D) \in \cC_{\Psi}(\alpha, \beta,\gamma, \delta) \rt\}
\gtrsim \; \max_{\{k\} \in \cA(\sfP)}  \lt\{\frac{|\log\gamma|}{ \cI(\sfP, \cH_{\Psi, k} ) } \rt\},
\end{align*}
and on the other hand, from Theorem \ref{ub_DI}(ii) we have
\begin{align*} %\label{ub_DI_2}
\sfE_\sfP\lt[T^*\rt] \lesssim \max_{\{k\} \in \cA(\sfP)}  \lt\{\frac{|\log \gamma|}{\cI \lt(\sfP^k, \cH^k\rt) }\rt\}.
\end{align*}
Thus, the result follows from the definition of ${\sf{ARE}}_{\sfP}[\chi^*]$.

Now, note that, by definition ${\sf{ARE}}_{\sfP}[\chi^*] \geq 1$. Therefore, to prove that ${\sf{ARE}}_{\sfP}[\chi^*] = 1$, it suffices to show that
\eqref{condP1} holds, i.e., 
$$ \min_{\{k\} \in \cA(\sfP)} \cI (\sfP^k , \cH^k) = \min_{\{k\} \in \cA(\sfP)} \cI (\sfP, \cH_{\Psi, k}),$$
which again is true from Proposition \ref{prop:P1} if \eqref{prop:P1_1} holds.
Fix the source $\bar{k} \in [K]$ such that
\begin{equation*}
\{\bar{k}\} \in \argmin_{\{k\} \in \cA(\sfP)} \cI (\sfP^k, \cH^k),
\end{equation*}
and for which $s$ is independent of $s^c$ under $\sfP$, where $s = s'_{\bar{k}} = \{\bar{k}\}$. Furthermore, for any $\sfQ_1$ such that
$$\sfQ_1 \in \argmin_{\sfQ \in \cH_{\Psi, \bar{k}}} \cI \lt(\sfP, \sfQ; \bar{k}\rt) \quad \text{implies that} \quad \sfQ_1^{\bar{k}} \in 
\cH^{\bar{k}}. %= \argmin_{\sfQ \in \cG^{\bar{k}}} \cI(\sfP^{\bar{k}}, \sfQ).
$$
Thus, $1 \leq |\cA\lt(\sfQ_1^{s} \otimes \sfP^{s^c}\rt)| = |\cA(\sfP)| - 1 \leq u-1$, which implies that
$$\sfQ_1^{s} \otimes \sfP^{s^c} \in \cP_\Psi \setminus \cH_0.$$
%\begin{equation}
%\label{prop:PminusP0_2}
%$Q_1^{W_{e'}} \otimes \sfP^{W_{e'}^c} \in \cP_\Psi \setminus \cH_0,$
%\end{equation}
Therefore, \eqref{prop:P1_1} holds and the result follows.
\item[(ii)] If $\alpha, \beta, \gamma, \delta \to 0$ such that $|\log \delta| >> |\log \alpha| \vee |\log \gamma|$, from Lemma \ref{lem:DI_lb}(iii) we have
\begin{align*}
\inf \lt\{ \sfE_{\sfP}[T]:\; 
(T, D) \in \cC_{\Psi}(\alpha, \beta,\gamma, \delta) \rt\}
\gtrsim \; \max_{\{k\} \notin \cA(\sfP)}  \lt\{\frac{|\log\delta|}{ \cI(\sfP, \cG_{\Psi, k} ) } \rt\},
\end{align*}
and on the other hand, from Theorem \ref{ub_DI}(ii) we have
\begin{align*} %\label{ub_DI_2}
\sfE_\sfP\lt[T^*\rt] \lesssim \max_{\{k\} \notin \cA(\sfP)}  \lt\{\frac{|\log \delta|}{\cI \lt(\sfP^k, \cG^k\rt) }\rt\}.
\end{align*}
Thus, the result follows from the definition of ${\sf{ARE}}_{\sfP}[\chi^*]$.

Now, note that, by definition ${\sf{ARE}}_{\sfP}[\chi^*] \geq 1$. Therefore, to prove that ${\sf{ARE}}_{\sfP}[\chi^*] = 1$, it suffices to show that
\eqref{condP2} holds, i.e., 
$$ \min_{\{k\} \notin \cA(\sfP)} \cI (\sfP^k , \cG^k) = \min_{\{k\} \notin \cA(\sfP)} \cI (\sfP, \cG_{\Psi, k}),$$
which again is true from Proposition \ref{prop:P1} if \eqref{prop:P2_1} holds.
Fix the source $\bar{k} \in [K]$ such that
\begin{equation*}
\{\bar{k}\} \in \argmin_{\{k\} \notin \cA(\sfP)} \cI (\sfP^k, \cG^k),
\end{equation*}
and for which $s$ is independent of $s^c$ under $\sfP$, where $s = s'_{\bar{k}} = \{\bar{k}\}$. Furthermore, for any $\sfQ_1$ such that
$$\sfQ_1 \in \argmin_{\sfQ \in \cG_{\Psi, \bar{k}}} \cI \lt(\sfP, \sfQ; \bar{k}\rt) \quad \text{implies that} \quad \sfQ_1^{\bar{k}} \in 
\cG^{\bar{k}}. %= \argmin_{\sfQ \in \cG^{\bar{k}}} \cI(\sfP^{\bar{k}}, \sfQ).
$$
Thus, $2 \leq |\cA\lt(\sfQ_1^{s} \otimes \sfP^{s^c}\rt)| = |\cA(\sfP)| + 1 \leq u$, which implies that
$$\sfQ_1^{s} \otimes \sfP^{s^c} \in \cP_\Psi \setminus \cH_0.$$
%\begin{equation}
%\label{prop:PminusP0_2}
%$Q_1^{W_{e'}} \otimes \sfP^{W_{e'}^c} \in \cP_\Psi \setminus \cH_0,$
%\end{equation}
Therefore, \eqref{prop:P2_1} holds and the result follows.
\end{enumerate}
\end{proof}

\subsection{Proof of Theorem \ref{w_ePminusP0_marg}}\label{pf:w_ePminusP0_marg}

\begin{proof}
We have $\emptyset \in \Psi$ and $s_k = \{k\}$ for every $k \in [K]$.
\begin{enumerate}
\item[(i)] If $\sfP \in \cP_\Psi \setminus \cH_0$ and $\alpha, \beta, \gamma, \delta \to 0$ such that $|\log \alpha| >> |\log \gamma| \vee |\log \delta|$, from Lemma \ref{lem:DI_lb}(iii) we have
\begin{align*}
\inf \lt\{ \sfE_{\sfP}[T]:\; 
(T, D) \in \cC_{\Psi}(\alpha, \beta,\gamma, \delta) \rt\}
\gtrsim \; \frac{|\log\alpha|}{ \cI(\sfP, \cH_0) },
\end{align*}
and on the other hand, from Theorem \ref{ub_DI}(ii) we have
\begin{align*} %\label{ub_DI_2}
\sfE_\sfP\lt[T^*\rt] \lesssim \min_{\{k\} \in \cA(\sfP)}  \lt\{ \frac{|\log \alpha|}{\cI\lt(\sfP^k, \cH^k\rt)} \rt\}.
\end{align*}
Thus, the result follows from the definition of ${\sf{ARE}}_{\sfP}[\chi^*]$.

Now, note that, by definition ${\sf{ARE}}_{\sfP}[\chi^*] \geq 1$. Therefore, to prove that ${\sf{ARE}}_{\sfP}[\chi^*] = 1$, it suffices to show that
\eqref{condP0} holds, i.e., 
$$ \max_{\{k\} \in \cA(\sfP)} \cI (\sfP^k , \cH^k) = \cI (\sfP, \cH_0),$$
which again is true from Proposition \ref{prop:P1} if \eqref{prop:P0_1} holds.
Let the index of the only signal under $\sfP$ be $\bar{k}$, i.e., $\cA(\sfP) = \{\{\bar{k}\}\}$. Then $s$ is independent of $s^c$ under $\sfP$, where $s = s_{\bar{k}} = \{\bar{k}\}$. Furthermore, for any $\sfQ_1$ such that
$$\sfQ_1 \in \argmin_{\sfQ \in \cH_0} \cI \lt(\sfP, \sfQ; \bar{k}\rt) \quad \text{implies that} \quad \sfQ_1^{\bar{k}} \in 
\cH^{\bar{k}}. %= \argmin_{\sfQ \in \cG^{\bar{k}}} \cI(\sfP^{\bar{k}}, \sfQ).
$$
Thus, $|\cA\lt(\sfQ_1^{s} \otimes \sfP^{s^c}\rt)| = 0$, which implies that
$$\sfQ_1^{s} \otimes \sfP^{s^c} \in \cH_0.$$
%\begin{equation}
%\label{prop:PminusP0_2}
%$Q_1^{W_{e'}} \otimes \sfP^{W_{e'}^c} \in \cP_\Psi \setminus \cH_0,$
%\end{equation}
Therefore, \eqref{prop:P0_1} holds and the result follows.
\item[(ii)] If $\sfP \in \cH_0$ and $\alpha, \beta \to 0$, while  $\gamma$ and $\delta$ are either  fixed or go to $0$, from Lemma \ref{lem:DI_lb}(i) we have
$$ \inf \lt\{ \sfE_{\sfP}[T]:\; 
(T, D) \in \cC_{\Psi}(\alpha, \beta,\gamma, \delta) \rt\} \; \gtrsim \; \max_{k \in [K]}  \lt\{ \frac{|\log\beta|}{ \cI \lt( \sfP , \cG_{\Psi, k} \rt)} \rt\},$$
and on the other hand, from Theorem \ref{ub_DI}(i) we have
$$\sfE_\sfP\lt[T^*\rt] \; \lesssim \;  \max_{k \in [K]}   \lt\{ \frac{|\log \beta|}{\cI \lt(\sfP^k, \cG^k\rt)}\rt\}.$$
Thus, the result follows from the definition of ${\sf{ARE}}_{\sfP}[\chi^*]$.

Now, note that, by definition ${\sf{ARE}}_{\sfP}[\chi^*] \geq 1$. Therefore, to prove that ${\sf{ARE}}_{\sfP}[\chi^*] = 1$, it suffices to show that
\eqref{condPminusP0} holds, i.e., 
$$ \min_{k \in [K]} \cI (\sfP^k , \cG^k) = \min_{k \in [K]} \cI (\sfP, \cG_{\Psi, k}),$$
which again is true from Proposition \ref{prop:P1} if \eqref{prop:PminusP0_1} holds.
Fix the source $\bar{k} \in [K]$ such that
\begin{equation*}
\bar{k} \in \argmin_{k \in [K]} \cI (\sfP^k, \cG^k),
\end{equation*}
and for which $s$ is independent of $s^c$ under $\sfP$, where $s = s'_{\bar{k}} = \{\bar{k}\}$. Furthermore, for any $\sfQ_1$ such that
$$\sfQ_1 \in \argmin_{\sfQ \in \cG_{\Psi, \bar{k}}} \cI \lt(\sfP, \sfQ; \bar{k}\rt) \quad \text{implies that} \quad \sfQ_1^{\bar{k}} \in 
\cG^{\bar{k}}. %= \argmin_{\sfQ \in \cG^{\bar{k}}} \cI(\sfP^{\bar{k}}, \sfQ).
$$
Therefore, $|\cA\lt(\sfQ_1^{s} \otimes \sfP^{s^c}\rt)| = 1$, and since $\Psi_{1, 1} \subseteq \Psi$, this implies that
$$\sfQ_1^{s} \otimes \sfP^{s^c} \in \cP_\Psi \setminus \cH_0.$$
%\begin{equation}
%\label{prop:PminusP0_2}
%$Q_1^{W_{e'}} \otimes \sfP^{W_{e'}^c} \in \cP_\Psi \setminus \cH_0,$
%\end{equation}
Therefore, \eqref{prop:PminusP0_1} holds and the result follows.
\end{enumerate}
\end{proof}

\subsection{Proof of Corollary \ref{cor:dep_source_gauss}}\label{pf:cor:dep_source_gauss}

\begin{proof}
We have $\Psi = \Psi_{0, u}$ and $s_k = s'_k = \{k\}$ for every $k \in [K]$. Fix any $\sfP \in \cP_\Psi$, and let $\mu_\sfP$ be the mean vector and $\Sigma$ be the covariance matrix of $X_n$ under $\sfP$ for every $n \in \bN$. Then we have the decomposition that $\Sigma = URU$, where $U$ is the diagonal matrix with elements of the main diagonal being the standard deviations of the sources, i.e.,
$$U = {\rm{diag}}(\sigma_1, \dots, \sigma_K).$$
Also, in the following we denote by $I_K$ the $K \times K$ identity matrix, and by $a_k$ the column vector in $\bR^K$ with only the $k^{th}$ element being $1$ and the rest being $0$.
\begin{itemize}
\item[(i)]
Fix any source $\bar{k} \in [K]$ such that
$$\{\bar{k}\} \in \argmin_{\{k\} \in \cA(\sfP)} ({\mu_k^2}/{\sigma_k^2}),$$
and consider the distribution $\sfQ$ such that it differs from $\sfP$ only in the mean of observations from the source $\bar{k}$, i.e., under $\sfQ$, we have $X^{\bar{k}} \overset{\text{iid}}{\sim} N(0, \sigma_{\bar{k}}^2)$. To be more specific, if $\mu_\sfQ$ denotes the mean vector of $X_n$ under $\sfQ$, then it differs from $\mu_\sfP$ only in its $\bar{k}^{th}$ element, such that it is $\mu_{\bar{k}}$ in $\mu_\sfP$ but $0$ in $\mu_\sfQ$. However, the covariance matrix of $X_n$ is the same under both $\sfP$ and $\sfQ$. Thus, $\sfQ^{\bar{k}} \in \cH^{\bar{k}}$ and $u \geq |\cA(\sfQ)| = |\cA(\sfP)| - 1 \geq 1$, which implies $\sfQ \in \cH_{\Psi, \bar{k}}$.
Now, from \eqref{gauss_info} we have
\begin{align*}
\cI(\sfP, \sfQ) &= \frac{1}{2}\lt\{\log\frac{det(\Sigma)}{det(\Sigma)} - K + tr(\Sigma^{-1}\Sigma) + (\mu_\sfQ - \mu_\sfP)^t\Sigma^{-1}(\mu_\sfQ - \mu_\sfP)\rt\}\\
&=\frac{1}{2}\lt\{- K + tr(I_K) + (\mu_\sfQ - \mu_\sfP)^t\Sigma^{-1}(\mu_\sfQ - \mu_\sfP)\rt\},\\ %\quad \text{let $I_K$ be $K \times K$ identity matrix}\\
&=\frac{\mu_{\bar{k}}^2}{2}a_{\bar{k}}^tU^{-1}R^{-1}U^{-1}a_{\bar{k}} %\quad \text{where} \;\; a = \underbrace{(0, \dots, 0, 1, 0, \dots, 0)}_{\text{1 at $\bar{k}^{th}$ position, rest are 0}}\\
=\frac{\mu_{\bar{k}}^2}{2\sigma_{\bar{k}}^2}a_{\bar{k}}^tR^{-1}a_{\bar{k}}
= \frac{\mu_{\bar{k}}^2}{2\sigma_{\bar{k}}^2}R^{-1}_{\bar{k}\bar{k}}.
\end{align*}
Therefore, as $\alpha, \beta, \gamma, \delta \to 0$ such that $|\log \gamma| >> |\log \alpha| \vee |\log \delta|$, from Theorem \ref{w_eP1_marg}(i) we have
\begin{align*}
{\sf{ARE}}_{\sfP}[\chi^*] &\leq \frac{\min_{\{k\} \in \cA(\sfP)} \cI \lt( \sfP , \cH_{\Psi, k} \rt)}{\min_{\{k\} \in \cA(\sfP)} \cI \lt( \sfP^k , \cH^k \rt)} = \frac{\min_{\{k\} \in \cA(\sfP)} \cI \lt( \sfP , \cH_{\Psi, k} \rt)}{(1/2)\min_{\{k\} \in \cA(\sfP)} (\mu_k^2/\sigma_k^2)}\\
&\leq \frac{\cI \lt( \sfP , \cH_{\Psi, \bar{k}} \rt)}{\mu_{\bar{k}}^2/2\sigma_{\bar{k}}^2} \leq \frac{\cI(\sfP, \sfQ)}{\mu_{\bar{k}}^2/2\sigma_{\bar{k}}^2} = R^{-1}_{\bar{k}\bar{k}}.
\end{align*}
Finally, the proof is complete by taking minimum over all such $\bar{k}$.
\item[(ii)]
Fix any source $\bar{k} \in [K]$ such that
$$\{\bar{k}\} \in \argmin_{\{k\} \notin \cA(\sfP)} \min_{\theta \in \cM_k \setminus \{0\}} ({\theta^2}/{\sigma_k^2}),$$
and consider the distribution $\sfQ$ such that it differs from $\sfP$ only in the mean of observations from the source $\bar{k}$, i.e., under $\sfQ$, we have $X^{\bar{k}} \overset{\text{iid}}{\sim} N(\theta_{\bar{k}}, \sigma_{\bar{k}}^2)$, where
$$\theta_{\bar{k}} := \argmin_{\theta \in \cM_{\bar{k}} \setminus \{0\}} ({\theta^2}/{\sigma_{\bar{k}}^2}).$$
 To be more specific, if $\mu_\sfQ$ denotes the mean vector of $X_n$ under $\sfQ$, then it differs from $\mu_\sfP$ only in its $\bar{k}^{th}$ element, such that it is $0$ in $\mu_\sfP$ but $\theta_{\bar{k}}$ in $\mu_\sfQ$. However, the covariance matrix of $X_n$ is the same under both $\sfP$ and $\sfQ$. Thus, $\sfQ^{\bar{k}} \in \cG^{\bar{k}}$ and $u \geq |\cA(\sfQ)| = |\cA(\sfP)| + 1 \geq 1$, which implies $\sfQ \in \cG_{\Psi, \bar{k}}$.
Now, from \eqref{gauss_info} we have
\begin{align*}
\cI(\sfP, \sfQ) &= \frac{1}{2}\lt\{\log\frac{det(\Sigma)}{det(\Sigma)} - K + tr(\Sigma^{-1}\Sigma) + (\mu_\sfQ - \mu_\sfP)^t\Sigma^{-1}(\mu_\sfQ - \mu_\sfP)\rt\}\\
&=\frac{1}{2}\lt\{- K + tr(I_K) + (\mu_\sfQ - \mu_\sfP)^t\Sigma^{-1}(\mu_\sfQ - \mu_\sfP)\rt\},\\ %\quad \text{let $I_K$ be $K \times K$ identity matrix}\\
&=\frac{\theta_{\bar{k}}^2}{2}a_{\bar{k}}^tU^{-1}R^{-1}U^{-1}a_{\bar{k}} %\quad \text{where} \;\; a = \underbrace{(0, \dots, 0, 1, 0, \dots, 0)}_{\text{1 at $\bar{k}^{th}$ position, rest are 0}}\\
=\frac{\theta_{\bar{k}}^2}{2\sigma_{\bar{k}}^2}a_{\bar{k}}^tR^{-1}a_{\bar{k}}
=\frac{\theta_{\bar{k}}^2}{2\sigma_{\bar{k}}^2}R^{-1}_{\bar{k}\bar{k}}.
\end{align*}
Therefore, as $\alpha, \beta, \gamma, \delta \to 0$ such that $|\log \delta| >> |\log \alpha| \vee |\log \gamma|$, from Theorem \ref{w_eP1_marg}(ii) we have
\begin{align*}
{\sf{ARE}}_{\sfP}[\chi^*] &\leq \frac{\min_{\{k\} \notin \cA(\sfP)} \cI \lt( \sfP , \cG_{\Psi, k} \rt)}{\min_{\{k\} \notin \cA(\sfP)} \cI \lt( \sfP^k , \cG^k \rt)} = \frac{\min_{\{k\} \notin \cA(\sfP)} \cI \lt( \sfP , \cG_{\Psi, k} \rt)}{(1/2)\min_{\{k\} \notin \cA(\sfP)} \min_{\theta \in \cM_k \setminus \{0\}} ({\theta^2}/{\sigma_k^2})}\\
&\leq \frac{\cI \lt( \sfP , \cG_{\Psi, \bar{k}} \rt)}{\theta_{\bar{k}}^2/2\sigma_{\bar{k}}^2} \leq \frac{\cI(\sfP, \sfQ)}{\theta_{\bar{k}}^2/2\sigma_{\bar{k}}^2} = R^{-1}_{\bar{k}\bar{k}}.
\end{align*}
Finally, the proof is complete by taking minimum over all such $\bar{k}$.
\item[(iii)]
We make the following observation about $\mu_\sfP$ which will be used in this proof:
$$\mu_{\sfP} = \sum_{\{i\} \in \cA(\sfP)} \mu_ia_i.$$
Consider the distribution $\sfQ$ such that the mean of observations from every source is $0$. To be more specific, if $\mu_\sfQ$ denotes the mean vector of $X_n$ under $\sfQ$, then every element of $\mu_\sfQ$ is $0$. However, the covariance matrix of $X_n$ is the same under both $\sfP$ and $\sfQ$. Thus, $\sfQ \in \cH_0$.
Now, from \eqref{gauss_info} we have
\begin{align*}
\cI(\sfP, \sfQ) &= \frac{1}{2}\lt\{\log\frac{det(\Sigma)}{det(\Sigma)} - K + tr(\Sigma^{-1}\Sigma) + (\mu_\sfQ - \mu_\sfP)^t\Sigma^{-1}(\mu_\sfQ - \mu_\sfP)\rt\}\\
&=\frac{1}{2}\lt\{- K + tr(I_K) + \mu_\sfP^t\Sigma^{-1}\mu_\sfP\rt\},\\ %\quad \text{let $I_K$ be $K \times K$ identity matrix}\\
&=\frac{1}{2}\lt(\sum_{\{i\} \in \cA(\sfP)} \mu_ia_i\rt)^tU^{-1}R^{-1}U^{-1}\lt(\sum_{\{j\} \in \cA(\sfP)} \mu_ja_j\rt)\\
&=\frac{1}{2}\sum_{\{i\}, \{j\} \in \cA(\sfP)}\mu_i\mu_ja_i^tU^{-1}R^{-1}U^{-1}a_j\\
&=\frac{1}{2}\sum_{\{i\}, \{j\} \in \cA(\sfP)}\frac{\mu_i\mu_j}{\sigma_i\sigma_j}a_i^tR^{-1}a_j = \frac{1}{2}\sum_{\{i\}, \{j\} \in \cA(\sfP)}\frac{\mu_i\mu_j}{\sigma_i\sigma_j}R^{-1}_{ij}
\end{align*}
Therefore, as $\alpha, \beta, \gamma, \delta \to 0$ such that $|\log \alpha| >> |\log \gamma| \vee |\log \delta|$, from Theorem \ref{w_ePminusP0_marg}(i) we have
\begin{align*}
{\sf{ARE}}_{\sfP}[\chi^*] &\leq \min_{\{k\} \in \cA(\sfP)}  \left\{ \frac{\cI(\sfP, \cH_0) }{\cI(\sfP^k, \cH^k)} \right\} \leq \frac{\cI(\sfP, \sfQ)}{(1/2)\max_{\{k\} \in \cA(\sfP)} (\mu_k^2/\sigma_k^2)}\\
&=\sum_{\{i\}, \{j\} \in \cA(\sfP)} \frac{\mu_iR^{-1}_{ij}\mu_j}{\sigma_i\sigma_j}\bigg/\max_{\{k\} \in \cA(\sfP)} \frac{\mu_k^2}{\sigma_k^2}.
\end{align*}
\item[(iv)]
Since $\sfP \in \cH_0$ all elements of $\mu_\sfP$ are $0$.
Fix any source $\bar{k} \in [K]$ such that
$$\bar{k} \in \argmin_{k \in [K]}  \min_{\theta \in \cM_k \setminus \{0\}} ({\theta^2}/{\sigma_k^2}),$$
and consider the distribution $\sfQ$ such that it differs from $\sfP$ only in the mean of observations from the source $\bar{k}$, i.e., under $\sfQ$, we have $X^{\bar{k}} \overset{\text{iid}}{\sim} N(\theta_{\bar{k}}, \sigma_{\bar{k}}^2)$, where
$$\theta_{\bar{k}} := \argmin_{\theta \in \cM_{\bar{k}} \setminus \{0\}} ({\theta^2}/{\sigma_{\bar{k}}^2}).$$ 
To be more specific, if $\mu_\sfQ$ denotes the mean vector of $X_n$ under $\sfQ$, then its $\bar{k}^{th}$ element is $\theta_{\bar{k}}$ and the rest are $0$, or $\mu_\sfQ = \theta_{\bar{k}}a_{\bar{k}}$. However, the covariance matrix of $X_n$ is the same under both $\sfP$ and $\sfQ$. Thus, $\sfQ^{\bar{k}} \in \cG^{\bar{k}}$ and $|\cA(\sfQ)| = 1$, which implies $\sfQ \in \cG_{\Psi, \bar{k}}$.
Now, from \eqref{gauss_info} we have
\begin{align*}
\cI(\sfP, \sfQ) &= \frac{1}{2}\lt\{\log\frac{det(\Sigma)}{det(\Sigma)} - K + tr(\Sigma^{-1}\Sigma) + (\mu_\sfQ - \mu_\sfP)^t\Sigma^{-1}(\mu_\sfQ - \mu_\sfP)\rt\}\\
&=\frac{1}{2}\lt\{- K + tr(I_K) + \mu_\sfQ^t\Sigma^{-1}\mu_\sfQ\rt\},\\ %\quad \text{let $I_K$ be $K \times K$ identity matrix}\\
&=\frac{\theta_{\bar{k}}^2}{2}a_{\bar{k}}^tU^{-1}R^{-1}U^{-1}a_{\bar{k}} %\quad \text{where} \;\; a = \underbrace{(0, \dots, 0, 1, 0, \dots, 0)}_{\text{1 at $\bar{k}^{th}$ position, rest are 0}}\\
=\frac{\theta_{\bar{k}}^2}{2\sigma_{\bar{k}}^2}a_{\bar{k}}^tR^{-1}a_{\bar{k}}
= \frac{\theta_{\bar{k}}^2}{2\sigma_{\bar{k}}^2}R^{-1}_{\bar{k}\bar{k}}.
\end{align*}
Therefore, as $\alpha, \beta \to 0$, while  $\gamma$ and $\delta$ are either  fixed or go to $0$, from Theorem \ref{w_ePminusP0_marg}(ii) we have
\begin{align*}
{\sf{ARE}}_{\sfP}[\chi^*] &\leq \frac{\min_{k \in [K]} \cI \lt( \sfP , \cG_{\Psi,k} \rt)}{\min_{k \in [K]} \cI \lt( \sfP^k , \cG^k \rt)} = \frac{\min_{k \in [K]} \cI \lt( \sfP , \cG_{\Psi,k} \rt)}{(1/2)\min_{k \in [K]} \min_{\theta \in \cM_k \setminus \{0\}} ({\theta^2}/{\sigma_k^2})}\\
&\leq \frac{\cI \lt( \sfP , \cG_{\Psi, \bar{k}} \rt)}{\theta_{\bar{k}}^2/2\sigma_{\bar{k}}^2} \leq \frac{\cI(\sfP, \sfQ)}{\theta_{\bar{k}}^2/2\sigma_{\bar{k}}^2} = R^{-1}_{\bar{k}\bar{k}}.
\end{align*}
Finally, the proof is complete by taking minimum over all such $\bar{k}$.
\end{itemize}
\end{proof}

\section{Detection and isolation of a  dependence structure}\label{app:appl_dep}

\subsection{Proof of Theorem \ref{pure_det_dep}}\label{pf:pure_det_dep}

\begin{proof}
Fix any arbitrary reference distribution $\sfP_0 \in \cH_0$. For every $e \in \cK$ define
$$\Lambda^1_e(n) := \max\{\Lambda_n(\sfP, \sfP_0^e) : \sfP \in \cG^e\} \quad \text{and} \quad \Lambda^0_e(n) := \max\{\Lambda_n(\sfP, \sfP_0^e) : \sfP \in \cH^e\}.$$
Then, for every $e \in \cK$,
\begin{equation}\label{lam1/lam0}
\Lambda_e(n) = \Lambda^1_e(n)/\Lambda^0_e(n).
\end{equation}
Furthermore, since the dependent pairs are disjoint and \eqref{indep_subs} holds, we can write 
$$\sfP = \bigotimes_{e \in \cA(\sfP)} \sfP^e \otimes \sfP^{v_0}, \quad \text{and} \quad \sfP_0 = \bigotimes_{k \in [K]} \sfP^k = \bigotimes_{e \in \cA(\sfP)} \sfP_0^e \otimes \sfP_0^{v_0},$$
where $v_0$ is the only subset of $[K]$ that consists of the independent sources under $\sfP$.
This implies
\begin{equation}\label{decomp-LR}
\Lambda_n(\sfP, \sfP_0) = \prod_{e \in \cA(\sfP)} \Lambda_n(\sfP^e, \sfP_0^e) \times \Lambda_n(\sfP^{v_0}, \sfP_0^{v_0}).
\end{equation}
Now, fix any arbitrary $\bar{e} \in \cK$ and a subset of units $A \in \Psi$ such that $\bar{e} \in A$. Then under any $\sfP \in \cG_{\Psi, \bar{e}}$ such that $\cA(\sfP) = A$, the subset of independent sources is $$v_0 = [K] \setminus \cup_{e \in A} e.$$  
Thus, from \eqref{decomp-LR} one can obtain that
\begin{align*}
&\max\{\Lambda_n(\sfP, \sfP_0) : \sfP \in \cG_{\Psi, \bar{e}}, \;\; \cA(\sfP) = A\} = \prod_{e \in A} \Lambda^1_e(n) \times \max\{\Lambda_n(\sfP, \sfP_0^{v_0}) : \sfP \in \cH_0^{v_0}\}, \;\; \text{and}\\
&\max\{\Lambda_n(\sfP, \sfP_0) : \sfP \in \cH_0\} = \prod_{e \in A} \Lambda_e^0(n) \times \max\{\Lambda_n(\sfP, \sfP_0^{v_0}) : \sfP \in \cH_0^{v_0}\}.
\end{align*}
Due to \eqref{lam1/lam0}, the above two expressions lead to
\begin{equation*}
\frac{\max\{\Lambda_n(\sfP, \sfP_0) : \sfP \in \cG_{\Psi, \bar{e}}, \;\; \cA(\sfP) = A\}}{\max\{\Lambda_n(\sfP, \sfP_0) : \sfP \in \cH_0\}} = \prod_{e \in A} \Lambda_e(n),
\end{equation*}
and thus, if we further maximize the above quantity over all possible $A \in \Psi$, we have
\begin{equation}\label{max_det_dis}
\Lambda_{\bar{e}, \rm det}(n) = \max_{A \in \Psi} \prod_{e \in A} \Lambda_e(n) = \Lambda_{\bar{e}}(n) \times \max_{\mathcal{B} \subseteq \{i_1(n), \dots i_{1 \vee p(n)}(n)\}  , \; \bar{e} \cup \mathcal{B} \in \Psi}\prod_{e \in \mathcal{B}} \Lambda_e(n).
\end{equation}
When $p(n) \geq 1$, using \eqref{max_det_dis} and that $\Lambda_{\bar{e}}(n) \leq 1$ for any $\bar{e} \notin \{i_1(n), \dots i_{p(n)}(n)\}$, we have
$$\max_{\bar{e} \in \cK} \Lambda_{\bar{e}, \rm det}(n) = \max_{\mathcal{B} \subseteq \{i_1(n), \dots i_{p(n)}(n)\}  , \; \mathcal{B} \in \Psi}\prod_{e \in \mathcal{B}} \Lambda_e(n).$$ 
Furthermore, again from \eqref{max_det_dis} one can observe that
$$\max_{\bar{e} \in \cK} \Lambda_{\bar{e}, \rm det}(n) = \max_{\bar{e} \in \cK} \Lambda_{\bar{e}}(n) = \Lambda_{(1)}(n) < 1 \quad \Leftrightarrow \quad p(n) = 0.$$

Now, since $C_e = C$ and $D_e = D$ for every $e \in \cK$, one can write the stopping times $T_0$ and $T_{det}$ as
\begin{align*}
T_0 &= \inf\lt\{n  \in \bN : \max_{e \in \cK} \; \Lambda_{e, \rm det}(n) \leq 1/C \rt\},\\
T_{det} &= \inf\lt\{n \in \bN: \max_{e \in \cK} \; \Lambda_{e, \rm det}(n) \geq D\rt\}.
\end{align*}
Since $C, D > 1$, stopping happens by $T_0$ (resp. $T_{det}$) only if $\max_{e \in \cK} \; \Lambda_{e, \rm det}(n) \leq 1$ (resp. $\max_{e \in \cK} \; \Lambda_{e, \rm det}(n) > 1$), and from the first part of the proof, this happens only if $p(n) = 0$ (resp. $p(n) \geq 1$). Therefore, the stopping events in $T_0$ and $T_{det}$ remain the same even if they are intersected with $p(n) = 0$ and $p(n) \geq 1$ respectively. 
Thus,
\begin{align*}
T_0 &= \inf\lt\{n  \in \bN : \max_{e \in \cK} \; \Lambda_{e, \rm det}(n) \leq 1/C, \;\; p(n) = 0 \rt\}\\
&= \inf\lt\{n  \in \bN : \Lambda_{(1)}(n) \leq 1/C\rt\},
\end{align*}
and
\begin{align*}
T_{det} &= \inf\lt\{n  \in \bN :  \max_{e \in \cK} \; \Lambda_{e, \rm det}(n) \geq D, \;\; p(n) \geq 1 \rt\}\\
&= \inf\lt\{n  \in \bN : \max_{\mathcal{B} \subseteq 
\{i_1(n), \dots i_{p(n)}(n)  \}  , \; \mathcal{B} \in \Psi}  \prod_{e \in \mathcal{B}} \Lambda_{e}(n) \geq D\rt\},
\end{align*}
This completes the proof.
\end{proof}

\subsection{An important lemma} We present an important lemma which will be useful to establish the following theorems.

\begin{lemma}\label{w_e_simpli}
Suppose that $\emptyset \in \Psi$ and $\sfP \in \cP_\Psi$.
\begin{enumerate}
\item[(i)] If $\sfP \notin \cH_0$ and for every $e \in \cA(\sfP)$ there exists $w_e \supseteq e$ such that \eqref{condP1} holds when $s'_e = w_e$ for every $e \in \cA(\sfP)$, i.e., 
$$\min_{e \in \cA(\sfP)} \cI \lt( \sfP , \cH_{\Psi,e}; w_e \rt)  = \min_{e \in \cA(\sfP)} \cI \lt( \sfP , \cH_{\Psi,e} \rt),$$
then it also holds when $s'_e \supseteq w_e$ for every $e \in \cA(\sfP)$.
\item[(ii)] If $\sfP \notin \cH_0$ and for every $e \notin \cA(\sfP)$ there exists $w_e \supseteq e$ such that \eqref{condP2} holds when $s'_e = w_e$ for every $e \notin \cA(\sfP)$, i.e., 
$$\min_{e \notin \cA(\sfP)} \cI \lt( \sfP , \cG_{\Psi,e}; w_e \rt)  = \min_{e \notin \cA(\sfP)} \cI \lt( \sfP , \cG_{\Psi,e} \rt),$$
then it also holds when $s'_e \supseteq w_e$ for every $e \notin \cA(\sfP)$.
\item[(iii)] If $\sfP \in \cH_0$ and for every $e \in \cK$ there exists $w_e \supseteq e$ such that \eqref{condPminusP0} holds when $s_e = w_e$ for every $e \in \cK$, i.e., 
$$\min_{e \in \cK} \; \cI \lt( \sfP , \cG_{\Psi,e} ; w_e \rt) =\min_{e \in \cK} \; \cI \lt( \sfP , \cG_{\Psi,e}  \rt),$$
then it also holds when $s_e \supseteq w_e$ for every $e \in \cK$.
\item[(iv)] If $\sfP \notin \cH_0$ and for some $e \in \cA(\sfP)$ there exists $w_e \supseteq e$ such that
$$\cI(\sfP, \cH_0; w_e) = \cI(\sfP, \cH_0),$$
then \eqref{condP0} holds when $s_e \supseteq w_e$.
\end{enumerate}
\end{lemma}
\begin{proof}
We only prove (i) and (iv) since the other results can be shown in a similar way. Let $s'_e \supseteq w_e$ for every $e \in \cA(\sfP)$. Then,
\begin{align*}
\min_{e \in \cA(\sfP)} \cI \lt( \sfP , \cH_{\Psi,e} \rt) &\geq \min_{e \in \cA(\sfP)} \cI \lt( \sfP , \cH_{\Psi,e}; s'_e \rt)\\
&\geq \min_{e \in \cA(\sfP)} \cI \lt( \sfP , \cH_{\Psi,e}; w_e \rt) = \min_{e \in \cA(\sfP)} \cI \lt( \sfP , \cH_{\Psi,e} \rt),
\end{align*}
where the first and the second inequality follow by using Lemma \ref{info_ineq1} and the fact that $[K] \supseteq s'_e \supseteq w_e$ for every $e \in \cA(\sfP)$. Thus, equality holds throughout the above expression and \eqref{condP1} is implied. 

In order to prove (iv), let $s_{\bar{e}} \supseteq w_{\bar{e}}$ for some $\bar{e} \in \cA(\sfP)$ such that
$$\cI(\sfP, \cH_0; w_{\bar{e}}) = \cI(\sfP, \cH_0).$$
Then,
\begin{align*}
\cI(\sfP, \cH_0) &\geq \max _{e \in \cA(\sfP)} \cI(\sfP, \cH_0; s_e)\\
&\geq \cI(\sfP, \cH_0; s_{\bar{e}}) \geq \cI(\sfP, \cH_0; w_{\bar{e}}) = \cI(\sfP, \cH_0),
\end{align*}
where the first and third inequality follow by using Lemma \ref{info_ineq1} and the facts that $[K] \supseteq s_e$  for every $e \in \cA(\sfP)$ as well as  $s_{\bar{e}} \supseteq w_{\bar{e}}$.
\end{proof}

\subsection{Proof of Theorem \ref{w_eP1}}\label{pf:w_eP1}

\begin{proof}
We have $\emptyset \in \Psi$,  $\sfP \in \cP_\Psi \setminus \cH_0$, and  $\alpha, \beta, \gamma, \delta \to 0$ such that $|\log \gamma| >> |\log \alpha| \vee |\log \delta|$.
\begin{enumerate}

\item[(i)] From Lemma \ref{lem:DI_lb}(iii) we have
\begin{align*}
\inf \lt\{ \sfE_{\sfP}[T]:\; 
(T, D) \in \cC_{\Psi}(\alpha, \beta,\gamma, \delta) \rt\}
\gtrsim \; \max_{e \in \cA(\sfP)}  \lt\{\frac{|\log\gamma|}{ \cI(\sfP, \cH_{\Psi, e} ) } \rt\},
\end{align*}
and on the other hand, from Theorem \ref{ub_DI}(ii) we have
\begin{align*} %\label{ub_DI_2}
\sfE_\sfP\lt[T^*\rt] \lesssim \max_{e \in \cA(\sfP)}  \lt\{\frac{|\log \gamma|}{\cI \lt(\sfP, \cH_{\Psi, e}; s'_e \rt) }\rt\}.
\end{align*}
Thus, from definition we have 
$${\sf{ARE}}_{\sfP}[\chi^*] \leq \frac{\min_{e \in \cA(\sfP)} \cI \lt( \sfP , \cH_{\Psi,e} \rt)}{\min_{e \in \cA(\sfP)} \cI \lt(\sfP, \cH_{\Psi, e}; s'_e \rt)}.$$
Since by definition ${\sf{ARE}}_{\sfP}[\chi^*] \geq 1$, to prove that ${\sf{ARE}}_{\sfP}[\chi^*] = 1$, it suffices to show that
\eqref{condP1} holds, i.e., 
$$\min_{e \in \cA(\sfP)} \cI \lt( \sfP , \cH_{\Psi,e}; s'_e \rt)  = \min_{e \in \cA(\sfP)} \cI \lt( \sfP , \cH_{\Psi,e} \rt).$$
In view of Lemma \ref{w_e_simpli}(i), the above is true if \eqref{prop:P1_1} in Proposition \ref{prop:P1} holds when for every $e \in \cA(\sfP)$ 
\begin{equation*}
s'_e = v_l \;\; \text{where} \;\; v_l  \supseteq e \;\; \text{for some} \;\; l \in [L(\sfP)].
\end{equation*}
Fix any arbitrary $e \in \cA(\sfP)$.
Then for every $e \in \cA(\sfP)$ we have $(s'_e)^c = \cup_{k \neq l} v_{k}$ which is independent of $v_l$, that is, $s'_e$.  Furthermore, letting
$$\sfQ_1 = \argmin_{\sfQ \in \cH_{\Psi, e}} \cI \lt(\sfP, \sfQ; s'_e\rt) \quad \text{implies that} \quad \sfQ_1^{s'_e} = \argmin_{\sfQ \in \cH_{\Psi, e}^{s'_e}} \cI(\sfP^{s'_e}, \sfQ).$$
Note that, since $L(\sfP) > 1$, there exists $l' \neq l$, $l' \geq 1$ and $e' \in \cA(\sfP)$ such that $e' \subseteq v_{l'} \subseteq (s'_e)^c$. Therefore, denoting $s = s'_e$, we have $e' \in \cA\lt(\sfQ_1^{s} \otimes \sfP^{s^c}\rt) \in \Psi$. Thus,
$$\sfQ_1^{s} \otimes \sfP^{s^c} \in \cP_\Psi \setminus \cH_0,$$
%\begin{equation}
%\label{prop:PminusP0_2}
%$Q_1^{W_{e'}} \otimes \sfP^{W_{e'}^c} \in \cP_\Psi \setminus \cH_0,$
%\end{equation}
which implies that \eqref{prop:P1_1} holds and the result follows.
\item[(ii)] If $s'_e = e$ for every $e \in \cK$ then the upper bound on ${\sf{ARE}}$ follows from part (i).

\begin{itemize}
\item Since $\Psi = \Psi_{dis}$ we have $L(\sfP) = |\cA(\sfP)| > 1$. Furthermore, due to \eqref{indep_subs} for every $e \in \cA(\sfP)$ 
there exists $l \in [L(\sfP)]$ such that $v_l = e = s'_e$. Therefore, the condition in part (i) is satisfied and the result follows. 
\item It suffices to show that \eqref{condP1} holds, i.e.,
$$\min_{e \in \cA(\sfP)} \cI \lt( \sfP^e , \cH^e\rt)  = \min_{e \in \cA(\sfP)} \cI \lt( \sfP , \cH_{\Psi,e} \rt),$$
which again is true from Proposition \ref{prop:P1} if \eqref{prop:P1_1} holds.
Fix the pair $\bar{e} \in \cK$ such that
\begin{equation*}
\bar{e} \in \argmin_{e \in \cA(\sfP)} \cI (\sfP^e, \cH^e),
\end{equation*}
and for which $s$ is independent of $s^c$ under $\sfP$, where $s = s'_{\bar{e}} = \bar{e}$. Furthermore, for any $\sfQ_1$ such that
$$\sfQ_1 \in \argmin_{\sfQ \in \cH_{\Psi, \bar{e}}} \cI \lt(\sfP, \sfQ; \bar{e}\rt) \quad \text{implies that} \quad \sfQ_1^{\bar{e}} \in 
\cH^{\bar{e}}. %= \argmin_{\sfQ \in \cG^{\bar{k}}} \cI(\sfP^{\bar{k}}, \sfQ).
$$
Since $s$ is independent of $s^c$, any $e \in \cA(\sfP)$ such that $e \neq \bar{e}$ must satisfy $e \subseteq s^c$, and the number of such pairs is greater than or equal to $1$ as $L(\sfP) > 1$. Combined with the above, the fact that $\cA(\sfP) \in \Psi$ leads to $\emptyset \neq \cA\lt(\sfQ_1^s \otimes \sfP^{s^c}\rt) \in \Psi$ when $\Psi$ is either $\Psi_{dis}$, or $\Psi_{clus}$, or the powerset of $\cK$. Thus,
$$\sfQ_1^{s} \otimes \sfP^{s^c} \in \cP_\Psi \setminus \cH_0,$$
which implies that \eqref{prop:P1_1} holds and the result follows.
\end{itemize}

\end{enumerate}

\end{proof}

\subsection{Proof of Theorem \ref{w_eP2}}\label{pf:w_eP2}

\begin{proof}
We have $\emptyset \in \Psi$,  $\sfP \in \cP_\Psi \setminus \cH_0$, and  $\alpha, \beta, \gamma, \delta \to 0$ such that $|\log \delta| >> |\log \alpha| \vee |\log \gamma|$.
\begin{enumerate}

\item[(i)] From Lemma \ref{lem:DI_lb}(iii) we have
\begin{align*}
\inf \lt\{ \sfE_{\sfP}[T]:\; 
(T, D) \in \cC_{\Psi}(\alpha, \beta,\gamma, \delta) \rt\}
\gtrsim \; \max_{e \notin \cA(\sfP)}  \lt\{\frac{|\log \delta|}{\cI \lt(\sfP, \cG_{\Psi, e}\rt) }\rt\},
\end{align*}
and on the other hand, from Theorem \ref{ub_DI}(ii) we have
\begin{align*} %\label{ub_DI_2}
\sfE_\sfP\lt[T^*\rt] \lesssim \max_{e \notin \cA(\sfP)}  \lt\{\frac{|\log \delta|}{\cI \lt(\sfP, \cG_{\Psi, e}; s'_e \rt) }\rt\}.
\end{align*}
Thus, from definition we have 
$${\sf{ARE}}_{\sfP}[\chi^*] \leq \frac{\min_{e \notin \cA(\sfP)} \cI \lt( \sfP , \cG_{\Psi,e} \rt)}{\min_{e \notin \cA(\sfP)} \cI \lt(\sfP, \cG_{\Psi, e}; s'_e \rt)}.$$
Since by definition ${\sf{ARE}}_{\sfP}[\chi^*] \geq 1$, to prove that ${\sf{ARE}}_{\sfP}[\chi^*] = 1$, it suffices to show that
\eqref{condP2} holds, i.e., 
$$\min_{e \notin \cA(\sfP)} \cI \lt( \sfP , \cG_{\Psi,e}; s'_e \rt)  = \min_{e \notin \cA(\sfP)} \cI \lt( \sfP , \cG_{\Psi,e} \rt).$$
In view of Lemma \ref{w_e_simpli}(ii), the above is true if \eqref{prop:P2_1} in Proposition \ref{prop:P1} holds when for every $e = \{i, j\} \notin \cA(\sfP)$ we have
\begin{equation*}
s'_e =
\begin{cases}
v_l \quad &\text{if} \quad v_l \supseteq e \quad \text{for some} \quad l \in [L(\sfP)],\\
v_{l} \cup v_{l'} \quad &\text{if} \quad i \in v_l \quad \text{and} \quad j \in v_{l'} \quad \text{for some} \quad l \neq l',  \quad l, l' \in [L(\sfP)],\\
e \cup v_l \quad &\text{if} \quad i \;\; \text{or} \;\; j \in v_0 \quad \text{and} \quad e \setminus v_0 \in v_l \quad \text{for some} \quad l \in [L(\sfP)],\\
e \quad &\text{if} \quad v_0 \supseteq e.
\end{cases}
\end{equation*}
Fix any arbitrary $e \notin \cA(\sfP)$.
Then in the first case, we have $(s'_e)^c = \cup_{k \neq l} v_{k}$ which is independent of $v_l$. In the second case, we have $(s'_e)^c = \cup_{k \neq l, l'} v_{k}$ which is independent of both $v_l$ and $v_{l'}$. In the third case, we have $(s'_e)^c = \cup_{k \neq l} v_k \cup \{v_0 \setminus e\}$, which is independent of $v_l$ and $e \setminus v_0$. Finally in the fourth case, we have $(s'_e)^c = \cup_{k} v_k \cup \{v_0 \setminus e\}$, which is independent of $e$. Therefore, from the above it is clear that $(s'_e)^c$ is independent of $s'_e$.
Furthermore, letting
$$\sfQ_1 = \argmin_{\sfQ \in \cG_{\Psi, e}} \cI \lt(\sfP, \sfQ; s'_e\rt) \quad \text{implies that} \quad \sfQ_1^{s'_e} = \argmin_{\sfQ \in \cG_{\Psi, e}^{s'_e}} \cI(\sfP^{s'_e}, \sfQ).$$
Since $\sfQ_1^{s'_e} \in \cG_{\Psi, e}^{s'_e}$, denoting $s = s'_e$, we have $e \in \cA(\sfQ_1^{s} \otimes \sfP^{s^c}) \in \Psi$. Thus,
$$\sfQ_1^s \otimes \sfP^{s^c} \in \cP_\Psi \setminus \cH_0,$$
%\begin{equation}
%\label{prop:PminusP0_2}
%$Q_1^{W_{e'}} \otimes \sfP^{W_{e'}^c} \in \cP_\Psi \setminus \cH_0,$
%\end{equation}
which implies that \eqref{prop:P2_1} holds and the result follows. 
\item[(ii)] If $s'_e = e$ for every $e \in \cK$ then the upper bound on ${\sf{ARE}}$ follows from part (i).
Now, it suffices to show that \eqref{condP1} holds, i.e.,
$$\min_{e \notin \cA(\sfP)} \cI \lt( \sfP^e , \cG^e\rt)  = \min_{e \notin \cA(\sfP)} \cI \lt( \sfP , \cG_{\Psi,e} \rt),$$
which again is true from Proposition \ref{prop:P1} if \eqref{prop:P2_1} holds.
Fix the pair $\bar{e} \in \cK$ such that
\begin{equation*}
\bar{e} \in \argmin_{e \notin \cA(\sfP)} \cI (\sfP^e, \cG^e),
\end{equation*}
and for which $s$ is independent of $s^c$ under $\sfP$, where $s = s'_{\bar{e}} = \bar{e}$. Furthermore, for any $\sfQ_1$ such that
$$\sfQ_1 \in \argmin_{\sfQ \in \cG_{\Psi, \bar{e}}} \cI \lt(\sfP, \sfQ; \bar{e}\rt) \quad \text{implies that} \quad \sfQ_1^{\bar{e}} \in 
\cG^{\bar{e}}. %= \argmin_{\sfQ \in \cG^{\bar{k}}} \cI(\sfP^{\bar{k}}, \sfQ).
$$
Since $\sfQ_1^{\bar{e}} \in \cG^{\bar{e}}$, we have  $\bar{e} \in \cA\lt(\sfQ_1^s \otimes \sfP^{s^c}\rt) \in \Psi$ when $\Psi$ is either $\Psi_{dis}$, or $\Psi_{clus}$, or the powerset of $\cK$. Thus,
$$\sfQ_1^{s} \otimes \sfP^{s^c} \in \cP_\Psi \setminus \cH_0,$$
which implies that \eqref{prop:P1_1} holds and the result follows.
\end{enumerate}

\end{proof}

\subsection{Proof of Theorem \ref{w_eP0}}\label{pf:w_eP0}

\begin{proof}
We have $\emptyset \in \Psi$,  $\sfP \in \cP_\Psi \setminus \cH_0$, and  $\alpha, \beta, \gamma, \delta \to 0$ such that $|\log \alpha| >> |\log \gamma| \vee |\log \delta|$.
\begin{enumerate}

\item[(i)] From Lemma \ref{lem:DI_lb}(iii) we have
\begin{align*}
\inf \lt\{ \sfE_{\sfP}[T]:\; 
(T, D) \in \cC_{\Psi}(\alpha, \beta,\gamma, \delta) \rt\}
\gtrsim \; \frac{|\log\alpha|}{\cI(\sfP, \cH_0)},
\end{align*}
and on the other hand, from Theorem \ref{ub_DI}(ii) we have
\begin{align*} %\label{ub_DI_2}
\sfE_\sfP\lt[T^*\rt] \lesssim \min_{e \in \cA(\sfP)}  \lt\{ \frac{|\log \alpha|}{\cI\lt(\sfP, \cH_0; s_e \rt)} \rt\}.
\end{align*}
Thus, from definition we have 
$${\sf{ARE}}_{\sfP}[\chi^*] \leq \min_{e \in \cA(\sfP)}  \frac{\cI \lt( \sfP , \cH_0 \rt)}{\cI \lt(\sfP, \cH_0; s_e \rt)}.$$
Since by definition ${\sf{ARE}}_{\sfP}[\chi^*] \geq 1$, to prove that ${\sf{ARE}}_{\sfP}[\chi^*] = 1$, it suffices to show that
\eqref{condP0} holds, i.e., 
$$\max _{e \in \cA (\sfP)} \cI(\sfP, \cH_0; s_e) = \cI(\sfP, \cH_0).$$
In view of Lemma \ref{w_e_simpli}(iv), the above is true if \eqref{prop:P0_1} in Proposition \ref{prop:P1} holds when for some $e \in \cA(\sfP)$ 
\begin{equation*}
s_{e} = v_1 \cup \ldots \cup v_{L(\sfP)}.
\end{equation*}
Then for that $e$, we have $s_e^c = v_0$ which is independent of $s_e$. 
Furthermore, letting
$$\sfQ_1 = \argmin_{\sfQ \in \cH_0} \cI \lt(\sfP, \sfQ; s_e\rt) \quad \text{implies that} \quad \sfQ_1^{s_e} = \argmin_{\sfQ \in \cH_0^{s_e}} \cI(\sfP^{s_e}, \sfQ).$$
Therefore, denoting $s = s_e$, we have $\cA(\sfQ_1^{s} \otimes \sfP^{s^c}) = \emptyset$. Thus,
$$\sfQ_1^{s} \otimes \sfP^{s^c} \in \cH_0,$$
which implies that \eqref{prop:P0_1} holds and the result follows.
\item[(ii)] If $s_e = e$ for every $e \in \cK$ then the upper bound on ${\sf{ARE}}$ follows from part (i).
Now, suppose the only dependent pair under $\sfP$ is $\bar{e}$, i.e., $\cA(\sfP) = \{\bar{e}\}$ and the rest of the sources are independent. Clearly this is the case when \eqref{indep_subs} holds and $|\cA(\sfP)| = 1$. Then
$$\sfP = \sfP^{v_0} \otimes \sfP^{v_1},$$
where $v_1 = \bar{e}$ and $v_0 = (\bar{e})^c$.
Therefore, when $s_e = e$ for every $e \in \cK$, we have $s_{\bar{e}} = \bar{e} = v_1$. Thus, asymptotic optimality follows by using part (i). 
\end{enumerate}

\end{proof}

\subsection{Proof of Theorem \ref{depcase_P0}}\label{pf:depcase_P0}

\begin{proof}
We have $\emptyset \in \Psi$,  $\sfP \in \cH_0$, and $\alpha, \beta \to 0$ while $\gamma$ and $\delta$ are either fixed or go to $0$.
Now, from Lemma \ref{lem:DI_lb}(iii) we have
\begin{align*}
\inf \lt\{ \sfE_{\sfP}[T]:\; 
(T, D) \in \cC_{\Psi}(\alpha, \beta,\gamma, \delta) \rt\}
\gtrsim \; \max_{e \in \cK}  \lt\{ \frac{|\log\beta|}{ \cI \lt( \sfP , \cG_{\Psi,e} \rt)} \rt\},
\end{align*}
and on the other hand, since $s_e = e$ for every $e \in \cK$, from Theorem \ref{ub_DI}(ii) we have
\begin{align*} %\label{ub_DI_2}
\sfE_\sfP\lt[T^*\rt] \lesssim \max_{e \in \cK}   \lt\{ \frac{|\log \beta|}{\cI \lt(\sfP^e, \cG^e\rt)}  \rt\}.
\end{align*}
Thus, from definition we have 
$${\sf{ARE}}_{\sfP}[\chi^*] \leq \frac{\min_{e \in \cK} \cI \lt( \sfP , \cG_{\Psi,e} \rt)}{\min_{e \in \cK} \cI \lt( \sfP^e , \cG^e \rt)}.$$
Since by definition ${\sf{ARE}}_{\sfP}[\chi^*] \geq 1$, to prove that ${\sf{ARE}}_{\sfP}[\chi^*] = 1$, it suffices to show that
\eqref{condPminusP0} holds, i.e., 
$$\min_{e \in \cK} \; \cI \lt( \sfP^e , \cG^e\rt) =\min_{e \in \cK} \; \cI \lt( \sfP , \cG_{\Psi,e}  \rt).$$
In view of Lemma \ref{w_e_simpli}(iii), the above is true if \eqref{prop:PminusP0_1} in Proposition \ref{prop:P1} holds when $s_e = e$ for every $e \in \cK$. 
Fix the pair $\bar{e} \in \cK$ such that
\begin{equation*}
\bar{e} \in \argmin_{e \in \cK} \; \cI \lt(\sfP^e, \cG^e\rt),
\end{equation*}
and for which $s$ is independent of $s^c$ under $\sfP$, where $s = s_{\bar{e}} = \bar{e}$. Clearly this is the case when \eqref{indep_subs} holds.
Furthermore, for any $\sfQ_1$ such that
$$\sfQ_1 \in \argmin_{\sfQ \in \cG_{\Psi, \bar{e}}} \cI \lt(\sfP, \sfQ; \bar{e}\rt) \quad \text{implies that} \quad \sfQ_1^{\bar{e}} \in 
\cG^{\bar{e}}. %= \argmin_{\sfQ \in \cG^{\bar{k}}} \cI(\sfP^{\bar{k}}, \sfQ).
$$
Since $\sfQ_1^{\bar{e}} \in \cG^{\bar{e}}$ and $|\cA(\sfP)| = 0$, we have $|\cA\lt(\sfQ_1^s \otimes \sfP^{s^c}\rt)| = 1$. Thus, by using the fact that $\Psi_{1, 1} \subseteq \Psi$, we have
$$\sfQ_1^s \otimes \sfP^{s^c} \in \cP_\Psi \setminus \cH_0,$$
which implies that \eqref{prop:PminusP0_1} holds and the result follows.

\end{proof}

\subsection{Proof of Theorem \ref{dep_case_psi_di_prac_rule}}\label{pf:dep_case_psi_di_prac_rule}

The following lemma is important in establishing Theorem \ref{dep_case_psi_di_prac_rule}.
\begin{lemma}\label{psi_di_prac_rule_lem}
Suppose \eqref{indep_subs} holds, $\Psi = \Psi_{dis}$ and $\sfP \in \cP_\Psi \setminus \cH_0$. Then
$$\cI(\sfP, \cH_0) = \sum_{e \in \cA(\sfP)} \cI(\sfP^e, \cH^e).$$
\end{lemma}
\begin{proof}
Due to \eqref{indep_subs} there exists $\sfP_0 \in \cH_0$, such that
$$\sfP = \bigotimes_{e \in \cA(\sfP)} \sfP^e \otimes \sfP_0^{v_0}, \quad \text{where} \;\; v_0 = [K] \setminus \bigcup_{e \in \cA(\sfP)} e.$$
Fix any $\sfQ \in \cH_0$. Since $\cA(\sfQ) = \emptyset$, again due to \eqref{indep_subs} one can write $\sfQ$ as follows:
$$\sfQ = \bigotimes_{e \in \cA(\sfP)} \sfQ^e \otimes \sfQ^{v_0}.$$
Then by Lemma \ref{info_ineq2} we have
$$\cI(\sfP, \sfQ) = \sum_{e \in \cA(\sfP)} \cI\lt(\sfP^e, \sfQ^e\rt) + \cI(\sfP, \sfQ; v_0).$$
Therefore, the minimum of $\cI(\sfP, \sfQ)$ over $\sfQ \in \cH_0$ is obtained by $\sfQ_*$ such that for every $e \in \cA(\sfP)$ 
$$\sfQ_*^e \in \argmin_{\sfQ \in \cH^e} \; \cI\lt(\sfP^e, \sfQ\rt) \quad \text{and} \quad \sfQ_*^{v_0} = \sfP^{v_0},$$
which yields $$\cI(\sfP, \cH_0) = \cI(\sfP, \sfQ_*) = \sum_{e \in \cA(\sfP)} \cI\lt(\sfP^e, \sfQ_*^e\rt) + \cI(\sfP, \sfQ_*; v_0) = \sum_{e \in \cA(\sfP)} \cI\lt(\sfP^e, \cH^e\rt).$$
\end{proof}
Now we are ready to prove Theorem \ref{dep_case_psi_di_prac_rule}.
\begin{proof}[Proof of Theorem \ref{dep_case_psi_di_prac_rule}]
First, fix any arbitrary $\sfP \in \cH_0$. Then in order to prove that this rule is asymptotically optimal under $\sfP$ as $\gamma, \delta \to 0$, it suffices to show that \eqref{prop:PminusP0_1} holds when $s_e = e$ for every $e \in \cK$ since by Proposition \ref{prop:P1}, this will further imply that \eqref{condPminusP0} in Theorem \ref{PminusP0} is satisfied and the result will follow. 
Fix $\bar{e} \in \cK$ such that
\begin{equation*}
\bar{e} \in \argmin_{e \in \cK} \cI \lt(\sfP^e, \cG^e\rt).
\end{equation*}
Since $\cA(\sfP) = \emptyset$ and \eqref{indep_subs} holds, $s$ is independent of $s^c$, where $s = \bar{e}$. Furthermore, letting
$$\sfQ_1 \in \argmin_{\sfQ \in \cG_{\Psi, \bar{e}}} \cI \lt(\sfP, \sfQ; \bar{e}\rt) \quad \text{implies that} \quad \sfQ_1^{\bar{e}} \in \cG^{\bar{e}}.$$
This means under $\sfQ_1^{s} \otimes \sfP^{s^c}$ there is exactly one dependent pair, i.e., $\cA\lt(\sfQ_1^{s} \otimes \sfP^{s^c}\rt) = \{\bar{e}\} \in \Psi_{dis}$, which further implies
$$\sfQ_1^s \otimes \sfP^{s^c} \in \cP_\Psi \setminus \cH_0.$$
Thus, \eqref{prop:PminusP0_1} holds. 

Next, fix any arbitrary $\sfP \in \cP_\Psi \setminus \cH_0$.
We show that \eqref{prop:P2_1} holds when $s'_e = e$ for every $e \notin \cA(\sfP)$ since by Proposition \ref{prop:P1}, this will further imply that \eqref{condP2} is satisfied.
Fix the pair $\bar{e} \notin \cA(\sfP)$ such that
\begin{equation*}
\bar{e} \in \argmin_{e \notin \cA(\sfP)} \; \cI \lt(\sfP^e, \cG^e\rt),
\end{equation*}
and for which $s$ is independent of $s^c$, where $s = s'_{\bar{e}} = \bar{e}$.
Furthermore, letting
$$\sfQ_1 = \argmin_{\sfQ \in \cG_{\Psi, \bar{e}}} \cI\lt(\sfP, \sfQ; \bar{e} \rt) \quad \text{implies that} \quad \sfQ_1^{\bar{e}} \in \cG^{\bar{e}}.$$
Since under $\sfP^{s^c}$ the dependent pairs among the sources $s^c$ are still disjoint and $\sfQ_1^{\bar{e}} \in \cG^{\bar{e}}$, we have  $\bar{e} \in \cA\lt(\sfQ_1^s \otimes \sfP^{s^c}\rt) \in \Psi_{dis}$. Thus,
$$\sfQ_1^s \otimes \sfP^{s^c} \in \cP_\Psi \setminus \cH_0,$$
which implies \eqref{prop:P2_1} holds. Now we show that the conditions in Remark \ref{rem:fwer} are satisfied.

 First, we consider the case when $|\cA(\sfP)| = 1$. Let $\cA(\sfP) = \{e_1\}$.
Since every $\sfQ \in \cup_{e \in \cA(\sfP)} \cH_{\Psi, e}$ satisfies $\cA(\sfP) \setminus \cA(\sfQ) \neq \emptyset$, we have $e_1 \notin \cA(\sfQ)$, or $\sfQ^{e_1} \in \cH^{e_1}$. Thus, by Lemma \ref{info_ineq1}
$$\cI(\sfP, \sfQ) \geq \cI(\sfP^{e_1}, \sfQ^{e_1}) \geq \cI(\sfP^{e_1}, \cH^{e_1}),$$
and further taking infimum over $\sfQ \in \cup_{e \in \cA(\sfP)} \cH_{\Psi, e}$ in the left hand side yields
$$\min_{e \in \cA(\sfP)} \cI(\sfP, \cH_{\Psi, e}) \geq \cI(\sfP^{e_1}, \cH^{e_1}) = \cI(\sfP, \cH_0),$$
where the equality follows from Lemma \ref{psi_di_prac_rule_lem}. The equality also implies that \eqref{condP0} in Theorem \ref{P0} is satisfied.
Furthermore, since $s_e = s'_e = e$ for every $e \in \cK$, we have
$$\max_{e \in \cA(\sfP)} \cI(\sfP, \cH_0; s_e) = \cI(\sfP^{e_1}, \cH^{e_1}) = \min_{e \in \cA(\sfP)} \cI(\sfP, \cH_{\Psi, e}; s'_e).$$
Thus, \eqref{FWERcondP0} holds.

Next, consider the case when $|\cA(\sfP)| > 1$. We will show that \eqref{prop:P1_1} holds when $s'_e = e$ for every $e \in \cA(\sfP)$.
Fix $\bar{e} \in \cA(\sfP)$ such that
\begin{equation*}
\bar{e} \in \argmin_{e \in \cA(\sfP)} \; \cI\lt(\sfP^e, \cH^e\rt).
\end{equation*}
Now since the dependent pairs are disjoint and \eqref{indep_subs} holds, it is clear that $s$ is independent of $s^c$, where $s = \bar{e}$. Furthermore, letting
$$\sfQ_1 = \argmin_{\sfQ \in \cH_{\Psi, \bar{e}}} \cI \lt(\sfP, \sfQ; \bar{e}\rt) \quad \text{implies that} \quad \sfQ_1^{\bar{e}} \in \cH^{\bar{e}}.$$
Since under $\sfQ_1^{s} \otimes \sfP^{s^c}$, the dependent pairs are still disjoint and $|\cA\lt(\sfQ_1^s \otimes \sfP^{s^c}\rt)| = |\cA(\sfP)| - 1 \geq 1$, we have
$$\sfQ_1^s \otimes \sfP^{s^c} \in \cP_\Psi \setminus \cH_0.$$
%\begin{equation}
%\label{prop:PminusP0_2}
%$Q_1^{W_{e'}} \otimes \sfP^{W_{e'}^c} \in \cP_\Psi \setminus \cH_0,$
%\end{equation}
Thus, \eqref{prop:P1_1} holds and by Proposition \ref{prop:P1} this will further imply that \eqref{condP1} is satisfied. This means
$$\min_{e \in \cA(\sfP)} \cI(\sfP, \cH_{\Psi, e}) = \min_{e \in \cA(\sfP)} \cI(\sfP^e, \cH^e) \leq \sum_{e \in \cA(\sfP)} \cI(\sfP^e, \cH^e) = \cI(\sfP, \cH_0),$$
where the last equality follows from Lemma \ref{psi_di_prac_rule_lem}.
Furthermore, since $s_e = s'_e = e$ for every $e \in \cK$, we have
$$\max_{e \in \cA(\sfP)} \cI(\sfP, \cH_0; s_e) = \max_{e \in \cA(\sfP)} \cI(\sfP^{e}, \cH^{e}) \geq \min_{e \in \cA(\sfP)} \cI(\sfP^e, \cH^e) = \min_{e \in \cA(\sfP)} \cI(\sfP, \cH_{\Psi, e}; s'_e).$$
Thus, \eqref{FWERcondP1} holds. Therefore, the asymptotic optimality of $\chi^*_{fwer}$ follows by using Theorem \ref{P0P1P2}. The proof is complete.
\end{proof}

\subsection{Proof of Corollary \ref{cor:dep_struc}}\label{pf:cor:dep_struc}

\begin{proof}
We have $\Psi = \Psi_{clus}$ and $s_e = s'_e = e$ for every $e \in \cK$. Let $\mu$ be the mean vector and $\Sigma$ be the covariance matrix of $X_n$ under $\sfP$. Then we have the decomposition that $\Sigma = URU$, where $U$ is the diagonal matrix with elements of the main diagonal being the standard deviations of the sources, i.e.,
$$U = {\rm{diag}}(\sigma_1, \dots, \sigma_K),$$
and $R$ is the correlation matrix. 
For any $k \in [K]$ we define 
$$A_k = (1 - \rho)I_k + \rho 1_k1_k^t,$$
where
$I_k$ is the $k \times k$ identity matrix, and $1_k$ is the column vector in $\bR^k$ with all elements being $1$. Furthermore, for any $j, l \in [K]$ we denote by $0_{j \times l}$ the $j \times l$ matrix with all entries being $0$. Without loss of generality let $v_1 = \{1, 2, \dots, m\}$, then
$$
R = 
\begin{bmatrix}
A_m &0_{m \times (K - m)}\\
0_{(K - m) \times m} &I_{K - m}
\end{bmatrix}.
$$
In the following we also use two important results regarding $A_k$:
\begin{align*}
A_k^{-1} &= \frac{1}{1-\rho}I_k - \frac{\rho}{(1 - \rho)(1 + (k - 1)\rho)}1_k1_k^t,\\
det(A_k) &= (1 - \rho)^{k-1}(1 + (k-1)\rho). 
\end{align*} 
\begin{itemize}
\item Consider the distribution $\sfQ$ such that $L(\sfQ) = 1$ and the data sources $\{1, \dots, m-1\}$ marginally follow equicorrelated multivariate Gaussian with correlation $\rho$ and is independent of the rest of the sources. However, the mean vector of $X_n$ and the variances of the data sources are the same under both $\sfP$ and $\sfQ$. Then, denoting by $\Sigma_{\sfQ}$ the covariance matrix of $X_n$ under $\sfQ$, we have $\Sigma_\sfQ = UR_\sfQ U$, where $R_\sfQ$ is the correlation matrix under $\sfQ$ that takes the following form
$$
R_\sfQ = 
\begin{bmatrix}
A_{m-1} &0_{(m-1) \times (K - m +1)}\\
0_{(K - m + 1) \times (m-1)} &I_{K - m + 1}
\end{bmatrix}.
$$
Then we have
$$R_\sfQ^{-1}
=
\begin{bmatrix}
B_\sfQ &0_{m \times (K - m)}\\
0_{(K - m) \times m} &I_{K - m}
\end{bmatrix},
\quad \text{where}
\quad
B_\sfQ = 
\begin{bmatrix}
A_{m-1}^{-1} &0_{(m-1) \times 1}\\
0_{1 \times (m-1)} &1
\end{bmatrix},
$$
which implies
$$R_\sfQ^{-1}R = 
\begin{bmatrix}
B_\sfQ A_m &0_{m \times (K - m)}\\
0_{(K - m) \times m} &I_{K - m}
\end{bmatrix}.
$$
Furthermore, we have
\begin{align*}
B_\sfQ A_m = 
\begin{bmatrix}
A_{m-1}^{-1} &0_{(m-1) \times 1}\\
0_{1 \times (m-1)} &1
\end{bmatrix}
\begin{bmatrix}
A_{m-1}^{-1} &\rho 1_{m-1}\\
\rho 1_{m-1}^t &1
\end{bmatrix}
= 
\begin{bmatrix}
I_{m-1} &\rho A_{m-1}^{-1}1_{m-1}\\
\rho 1_{m-1}^t &1
\end{bmatrix},
\end{align*}
which yields $tr(B_\sfQ A_m) = m$.
Now, from \eqref{gauss_info} we have
\begin{align*}
\cI(\sfP, \sfQ) &= \frac{1}{2}\lt\{\log\frac{det(\Sigma_\sfQ)}{det(\Sigma)} - K + tr(\Sigma_\sfQ^{-1}\Sigma) + (\mu - \mu)^t\Sigma_\sfQ^{-1}(\mu - \mu)\rt\}\\
&= \frac{1}{2}\lt\{\log\frac{det(UR_\sfQ U)}{det(URU)} - K + tr(U^{-1}R_\sfQ^{-1}U^{-1}URU)\rt\}\\
&= \frac{1}{2}\lt\{\log\frac{(det(U))^2det(R_\sfQ)}{(det(U))^2det(R)} - K + tr(UU^{-1}R_\sfQ^{-1}R)\rt\}\\
&= \frac{1}{2}\lt\{\log\frac{det(R_\sfQ)}{det(R)} - K + tr(R_\sfQ^{-1}R)\rt\}\\
&= \frac{1}{2}\lt\{\log\frac{det(A_{m-1})}{det(A_m)} - K + tr(B_\sfQ A_m) + tr(I_{K-m})\rt\}\\
&= \frac{1}{2}\lt\{-\log\frac{det(A_m)}{det(A_{m-1})} - K + m + (K - m)\rt\}\\
&=-\frac{1}{2}\log\frac{(1-\rho)^{m-1}(1 + (m-1)\rho)}{(1-\rho)^{m-2}(1 + (m-2)\rho)}\\
&=-\frac{1}{2}\log\frac{(1-\rho)(1 + (m-1)\rho)}{(1 + (m-2)\rho)}\\
%&= -\frac{1}{2}\log \lt(1 - \rho^2  \lt(\frac{1_{m-1}^t1_{m-1}}{1-\rho} - \rho\frac{1_{m-1}^t 1_{m-1}1_{m-1}^t 1_{m-1}}{(1 - \rho)(1 + (m - 2)\rho)}\rt)\rt)\\
%&= -\frac{1}{2}\log \lt(1 - \frac{\rho^2}{1 - \rho}\lt(m-1 - \frac{\rho}{1 + (m - 2)\rho}(m-1)^2\rt)\rt)\\
&= -\frac{1}{2}\log \lt(1 - \frac{\rho^2(m-1)}{1+(m-2)\rho}\rt)
\end{align*}
Now note that, the pair $\{m-1, m\}$ is independent under $\sfQ$ but dependent under $\sfP$, which implies $\sfQ \in \cup_{e \in \cA(\sfP)} \cH_{\Psi, e}$.
Therefore, as $\alpha, \beta, \gamma, \delta \to 0$ such that $|\log \gamma| >> |\log \alpha| \vee |\log \delta|$, from Theorem \ref{w_eP1}(ii) we have
\begin{align*}
{\sf{ARE}}_{\sfP}[\chi^*] &\leq \frac{\min_{e \in \cA(\sfP)} \cI \lt( \sfP , \cH_{\Psi,e} \rt)}{\min_{e \in \cA(\sfP)} \cI \lt( \sfP^e , \cH^e \rt)} \leq \frac{\cI(\sfP, \sfQ)}{(-1/2)\log (1-\rho^2)}\\
&= \log\lt( 1- \rho \frac{ \rho(m-1)}{1 + \rho(m - 2)}\rt) / \log(1 - \rho^2).
\end{align*}
Since the above upper bound is an increasing function of $m$, one can let $m \to \infty$ to obtain the following upper bound which is independent of $m$:
$${\sf{ARE}}_{\sfP}[\chi^*] \leq \frac{\log(1 - \rho)}{\log(1 - \rho^2)}.$$

\item Consider the distribution $\sfQ$ such that $L(\sfQ) = 1$ and the data sources $\{1, \dots, m+1\}$ marginally follow equicorrelated multivariate Gaussian with correlation $\rho$ and is independent of the rest of the sources. However, the mean vector of $X_n$ and the variances of the data sources are the same under both $\sfP$ and $\sfQ$. Then, denoting by $\Sigma_{\sfQ}$ the covariance matrix of $X_n$ under $\sfQ$, we have $\Sigma_\sfQ = UR_\sfQ U$, where $R_\sfQ$ is the correlation matrix under $\sfQ$ that takes the following form
$$
R_\sfQ = 
\begin{bmatrix}
A_{m+1} &0_{(m+1) \times (K - m -1)}\\
0_{(K - m - 1) \times (m+1)} &I_{K - m - 1}
\end{bmatrix}.
$$
Then we have
$$R_\sfQ^{-1}
=
\begin{bmatrix}
A_{m+1}^{-1} &0_{(m+1) \times (K - m -1)}\\
0_{(K - m - 1) \times (m+1)} &I_{K - m - 1}
\end{bmatrix},
$$
which implies
$$R_\sfQ^{-1}R = 
\begin{bmatrix}
A_{m+1}^{-1}B_\sfP &0_{(m+1) \times (K - m -1)}\\
0_{(K - m - 1) \times (m+1)} &I_{K - m - 1}
\end{bmatrix}, 
\quad \text{where} \quad
B_\sfP = 
\begin{bmatrix}
A_m &0_{m \times 1}\\
0_{1 \times m} &1
\end{bmatrix}.
$$
Furthermore, we have
\begin{align*}
A_{m+1}^{-1}B_\sfP &= \lt(\frac{1}{1-\rho}I_{m+1} - \frac{\rho}{(1 - \rho)(1 + m\rho)}1_{m+1}1_{m+1}^t\rt)B_\sfP\\
&= \frac{1}{1-\rho}B_\sfP - \frac{\rho}{(1 - \rho)(1 + m\rho)}1_{m+1}1_{m+1}^tB_\sfP,
\end{align*}
which yields
\begin{align*}
tr(A_{m+1}^{-1}B_\sfP) &= \frac{1}{1-\rho}tr(B_\sfP) -  \frac{\rho}{(1 - \rho)(1 + m\rho)} tr(1_{m+1}1_{m+1}^tB_\sfP)\\
&= \frac{1}{1-\rho}(tr(A_m)+1) - \frac{\rho}{(1 - \rho)(1 + m\rho)}tr(1_{m+1}^tB_\sfP1_{m+1})\\
&= \frac{m+1}{1-\rho} - \frac{\rho}{(1 - \rho)(1 + m\rho)}1_{m+1}^tB_\sfP1_{m+1}\\
&= \frac{m+1}{1-\rho} - \rho\frac{m(1 + (m-1)\rho) + 1}{(1 - \rho)(1 + m\rho)}
\end{align*}
Now, from \eqref{gauss_info} we have
\begin{align*}
\cI(\sfP, \sfQ) &= \frac{1}{2}\lt\{\log\frac{det(\Sigma_\sfQ)}{det(\Sigma)} - K + tr(\Sigma_\sfQ^{-1}\Sigma) + (\mu - \mu)^t\Sigma_\sfQ^{-1}(\mu - \mu)\rt\}\\
&= \frac{1}{2}\lt\{\log\frac{det(UR_\sfQ U)}{det(URU)} - K + tr(U^{-1}R_\sfQ^{-1}U^{-1}URU)\rt\}\\
&= \frac{1}{2}\lt\{\log\frac{(det(U))^2det(R_\sfQ)}{(det(U))^2det(R)} - K + tr(UU^{-1}R_\sfQ^{-1}R)\rt\}\\
&= \frac{1}{2}\lt\{\log\frac{det(R_\sfQ)}{det(R)} - K + tr(R_\sfQ^{-1}R)\rt\}\\
&= \frac{1}{2}\lt\{\log\frac{det(A_{m+1})}{det(A_m)} - K + tr(A_{m+1}^{-1}B_\sfP) + tr(I_{K-m-1})\rt\}\\
&= \frac{1}{2}\bigg\{\log\frac{(1-\rho)^m(1+m\rho)}{(1-\rho)^{m-1}(1+(m-1)\rho)} - K + K - m -1\\ 
&\qquad \qquad \qquad \qquad \qquad+  \frac{m+1}{1-\rho} - \rho\frac{m(1 + (m-1)\rho) + 1}{(1 - \rho)(1 + m\rho)}\bigg\}\\
&=\frac{1}{2}\bigg\{\log\frac{(1-\rho)(1+m\rho)}{(1+(m-1)\rho)}\\ 
&+ \frac{(m+1)(1+m\rho) - \rho(1 + m + m(m-1)\rho) - (m+1)(1-\rho)(1 + m\rho)}{(1-\rho)(1 + m\rho)} 
\bigg\}\\
&=\frac{1}{2}\bigg\{-\log\frac{(1+(m-1)\rho)}{(1-\rho)(1+m\rho)} + \rho \frac{(m+1)(1+m\rho) - (1 + m + m(m-1)\rho)}{(1-\rho)(1 + m\rho)} 
\bigg\}\\
&=-\frac{1}{2}\bigg\{-\log\lt(\frac{1+ \rho(m-1)}{1+\rho(m-1)-\rho^2m}\rt) + \frac{2m\rho^2}{(1-\rho)(1 + m\rho)} 
\bigg\}\\
&=-\frac{1}{2}\log\lt(\frac{1 + \rho(m - 1)}{1 + \rho(m - 1) - \rho^2m}\rt) + \frac{m \rho^2}{(1 + m \rho)(1 - \rho)}
\end{align*}
Now note that, the pair $\{m, m+1\}$ is dependent under $\sfQ$ but independent under $\sfP$, which implies $\sfQ \in \cup_{e \notin \cA(\sfP)} \cG_{\Psi, e}$.
Therefore, as $\alpha, \beta, \gamma, \delta \to 0$ such that $|\log \delta| >> |\log \alpha| \vee |\log \gamma|$, from Theorem \ref{w_eP2}(ii) we have
\begin{align*}
&{\sf{ARE}}_{\sfP}[\chi^*] \leq \frac{\min_{e \notin \cA(\sfP)} \cI \lt( \sfP , \cG_{\Psi,e} \rt)}{\min_{e \notin \cA(\sfP)} \cI \lt( \sfP^e , \cG^e \rt)} \leq \frac{\cI(\sfP, \sfQ)}{\log\sqrt{1-\rho^2} + {\rho^2}/{(1 - \rho^2)}}\\
&\leq \lt\{-\frac{1}{2}\log\lt(\frac{1 + \rho(m - 1)}{1 + \rho(m - 1) - \rho^2m}\rt) + \frac{m \rho^2}{(1 + m \rho)(1 - \rho)}\rt\}\bigg/\lt\{\log\sqrt{1-\rho^2} + \frac{\rho^2}{1 - \rho^2}\rt\}.
\end{align*}
Since the above upper bound is an increasing function of $m$, one can let $m \to \infty$ to obtain the following upper bound which is independent of $m$:
$${\sf{ARE}}_{\sfP}[\chi^*] \leq \lt( \log(1 - \rho) + \frac{2\rho}{1 - \rho}\rt)\bigg/
 \lt( \log(1 - \rho^2) + \frac{2\rho^2}{1 - \rho^2}\rt).$$

\item Consider the distribution $\sfQ$ such that all sources are independent. However, the mean vector of $X_n$ and the variances of the data sources are the same under both $\sfP$ and $\sfQ$. Then, denoting by $\Sigma_{\sfQ}$ the covariance matrix of $X_n$ under $\sfQ$, we have $\Sigma_\sfQ = U^2$.
Now, from \eqref{gauss_info} we have
\begin{align*}
\cI(\sfP, \sfQ) &= \frac{1}{2}\lt\{\log\frac{det(\Sigma_\sfQ)}{det(\Sigma)} - K + tr(\Sigma_\sfQ^{-1}\Sigma) + (\mu - \mu)^t\Sigma_\sfQ^{-1}(\mu - \mu)\rt\}\\
&= \frac{1}{2}\lt\{\log\frac{det(U^2)}{det(URU)} - K + tr(U^{-1}U^{-1}URU)\rt\}\\
&= \frac{1}{2}\lt\{\log\frac{(det(U))^2}{(det(U))^2det(R)} - K + tr(UU^{-1}R)\rt\}\\
&= \frac{1}{2}\lt\{\log\frac{1}{det(R)} - K + tr(R)\rt\}\\
&=-\frac{1}{2}\log(det(R)) = -\frac{1}{2}\log(det(A_m))\\
&=-\frac{1}{2}\log((1-\rho)^{m-1}(1 + (m-1)\rho))\\
&=-\frac{1}{2}\lt((m-1)\log(1-\rho) + \log(1 + (m-1)\rho)\rt)
\end{align*}
Now, clearly $\sfQ \in \cH_0$.
Therefore, as $\alpha, \beta, \gamma, \delta \to 0$ such that $|\log \alpha| >> |\log \gamma| \vee |\log \delta|$, from Theorem \ref{w_eP0}(ii) we have
\begin{align*}
{\sf{ARE}}_{\sfP}[\chi^*] &\leq \min_{e \in \cA(\sfP)} \left\{ \frac{\cI(\sfP, \cH_0)}{\cI(\sfP^e, \cH^e)} \right\} \leq \frac{\cI(\sfP, \sfQ)}{(-1/2)\log (1-\rho^2)}\\
&= \frac{(m-1)\log (1-\rho) + \log (1 + (m - 1)\rho)}{\log (1-\rho^2)}.
\end{align*}
\end{itemize}
The proof is complete.
\end{proof}

\end{appendix}

\bibliographystyle{imsart-number} % Style BST file (imsart-number.bst or imsart-nameyear.bst)
\bibliography{JSDI.bib}       % Bibliography file (usually '*.bib')

\end{document}